\documentclass[12pt,PhD,twoside]{muthesis}
\usepackage[top=3cm,bottom=3cm,left=40mm, right=20mm]{geometry}
\usepackage{amssymb}
\usepackage{pdflscape}
\usepackage{microtype}
\usepackage{amsmath}
\usepackage{amsthm}
\usepackage{amsfonts}
\usepackage{tikz}
\usetikzlibrary{chains}
\usetikzlibrary{arrows,decorations.markings}
\usepackage{hyperref}
\usepackage{bookmark}
\usepackage{mathtools}
\usepackage{float}
\restylefloat{table}
\usepackage{multirow}
\usepackage{bbm}
\usepackage{array}
\usepackage{setspace}
\newcolumntype{P}[1]{>{\centering\arraybackslash}p{#1}}
\newcolumntype{M}[1]{>{\centering\arraybackslash}m{#1}}
\usepackage{chngcntr}
\usepackage{listings}
\usepackage[hypcap]{caption}
\usepackage{caption}
\counterwithout{table}{chapter}
\usepackage{comment}
\usepackage{enumitem}
\usepackage{makeidx}
\makeindex
\allowdisplaybreaks
\theoremstyle{plain}
\newtheorem{defs}[subsection]{Definition}
\newtheorem{thm}[subsection]{Theorem}
\newtheorem{thms}[]{Theorem}
\newtheorem{prop}[subsection]{Proposition}
\newtheorem{lem}[subsection]{Lemma}
\newtheorem{conjecture}[subsection]{Conjecture}
\newtheorem{coro}[subsection]{Corollary}
\newtheorem{rem}[subsection]{Remark}
\newtheorem{remark}[subsection]{Remark}

\makeatletter
\newcommand{\subalign}[1]{
  \vcenter{
    \Let@ \restore@math@cr \default@tag
    \baselineskip\fontdimen10 \scriptfont\tw@
    \advance\baselineskip\fontdimen12 \scriptfont\tw@
    \lineskip\thr@@\fontdimen8 \scriptfont\thr@@
    \lineskiplimit\lineskip
    \ialign{\hfil$\m@th\scriptstyle##$&$\m@th\scriptstyle{}##$\crcr
      #1\crcr}}}
\makeatother
\newcommand{\Der}{{\mathrm {Der}}}
\newcommand{\Sub}{{\mathrm {Sub}}}
\newcommand{\me}{{\mathrm {\mathfrak{g}_e}}}

\newcommand{\ra}{{\mathrm {\rightarrow}}}
\newcommand{\x}{{\mathrm {\times}}}

\newcommand{\g}{{\mathrm {\mathfrak{g}}}}
\newcommand{\s}{{\mathrm {\mathfrak{s}}}}
\newcommand{\h}{{\mathrm {\mathfrak{h}}}}
\newcommand{\w}{{\mathrm {\mathfrak{w}}}}
\newcommand{\nn}{{\mathrm {\mathfrak{n}}}}
\newcommand{\Gc}{{\mathrm {\mathcal{G}}}}

\newcommand{\mfn}{{\mathrm {\mathfrak{n}}}}

\newcommand{\mft}{{\mathrm {\mathfrak{t}}}}
\newcommand{\bbz}{{\mathrm {\mathbb{Z}}}}

\newcommand{\bbc}{{\mathrm {\mathbb{C}}}}
\newcommand{\bbk}{{\mathrm {\mathbbm{k}}}}
\newcommand{\la}{{\mathrm {\lambda}}}
\newcommand{\al}{{\mathrm {\alpha}}}
\newcommand{\be}{{\mathrm {\beta}}}
\newcommand{\rad}{{\mathrm {rad}}}
\newcommand{\Ad}{{\mathrm {Ad}}}

\newcommand{\nil}{{\mathrm {nil}}}
\newcommand{\rank}{{\mathrm {rank}}}
\newcommand{\dims}{{\mathrm {dim}}}
\newcommand{\degs}{{\mathrm {deg}}}
\newcommand{\gr}{{\mathrm {gr}}}
\newcommand{\ad}{{\mathrm {ad}}}
\newcommand{\Soc}{{\mathrm {Soc}}}
\newcommand{\soc}{{\mathrm {Soc}}}
\newcommand{\TR}{{\mathrm {TR}}}
\newcommand{\Lie}{{\mathrm {Lie}}}
\newcommand{\cha}{{\mathrm {char}}}
\newcommand{\Hom}{{\mathrm {Hom}}}

\newcommand{\mfsp}{{\mathrm {\mathfrak{sp}}}}

\newcommand{\mfsl}{{\mathrm {\mathfrak{sl}}}}
\newcommand{\divs}{{\mathrm {div}}}
\newcommand{\spnd}{{\mathrm {span}}}

\newcommand{\im}{{\mathrm {im}}}

\newcommand{\rt}[1]{{{e_{\al_{#1}}}}}

\newcommand{\Er}[1]{{{\mathrm{Er}(1;1)^{{(#1)}}}}}

\let\dlabel=\alabel
\newcommand{\dnode}[2][chj]{\node[#1,label={below:\dlabel{#2}}] {};}
\newcommand{\dnodenj}[1]{\dnode[ch]{#1}}
\newcommand{\blabel}[1]{\(\widetilde{\alpha}\)}
\let\flabel=\blabel
\newcommand{\hroot}[2][chj]{\node[#1,label={below:\flabel{#2}}] {};}

\newcommand{\dnodebr}[1]{\node[chj,label={below right:\dlabel{#1}}] {};}

\floatstyle{plaintop}
\restylefloat{table}
\makeatletter
 \def\@textbottom{\vskip \z@ \@plus 1500pt}
 \let\@texttop\relax
\makeatother
\newcommand{\autorefs}[1]{\def\sectionautorefname{Table}\autoref{#1}}

\newcommand{\autoreft}[1]{\def\sectionautorefname{Section}\autoref{#1}}
\newcommand{\tref}[1]{\hyperref[#1]{Theorem \ref*{#1}}}
\newcommand{\lref}[1]{\hyperref[#1]{Lemma \ref*{#1}}}
\newcommand{\pref}[1]{\hyperref[#1]{Proposition \ref*{#1}}}
\newcommand{\rref}[1]{\hyperref[#1]{Remark \ref*{#1}}}
\newcommand{\conref}[1]{\hyperref[#1]{Conjecture \ref*{#1}}}
\newcommand{\corref}[1]{\hyperref[#1]{Corollary \ref*{#1}}}
\newcommand{\aref}[1]{\hyperref[#1]{Appendix \ref*{#1}}}
\newcommand{\chref}[1]{\hyperref[#1]{Chapter \ref*{#1}}}
\newcommand{\seref}[1]{\hyperref[#1]{Section \ref*{#1}}}
\newcommand{\liref}[1]{\hyperref[#1]{Chapter \ref*{#1}}}
\def\sectionautorefname{Section}

\begin{document}

\title{Maximal subalgebras of the exceptional Lie algebras in low characteristic}
\author{Thomas Purslow}
\school{Mathematics}
\faculty{Science and Engineering}
\def\wordcount{}
% Uncomment the line below to suppress the `List of Tables' page (optional)
\tablespagefalse

% Uncomment the line below to suppress the `List of Figures' page (optional)
\figurespagefalse

% Uncomment the line below to use a customised Declaration statement
%\def\declaration{All the work in this thesis has been sourced from Google.}
\beforeabstract
In this thesis we consider the maximal subalgebras of the exceptional Lie algebras over algebraically closed fields of characteristic $p>0$. The classification of such subalgebras over the field of complex numbers was achieved long ago by Dynkin in \cite{Dyn52} completely determined by the maximal subgroups in the corresponding Lie groups.

We start by considering the recent works \cite{HS14} and \cite{P15} where the question of maximal subalgebras in exceptional Lie algebras for positive characteristic was first considered, looking at good characteristic. We apply these results to exceptional Lie algebras of type $G_2$, and give the complete classification of their maximal subalgebras.

For $p=5$ and $\g$ of type $E_8$ we give an initial result on non-semisimple maximal subalgebras. We then consider a new maximal subalgebra, which is the $p$-closure of the non-restricted Witt algebra $W(1;\underline{2})$. This appears to be a completely new occurrence. To finish this section we have a brief consideration of the first Witt algebra in $E_8$.

The final two chapters focus on bad primes $p=2$ and $p=3$ for the exceptional Lie algebras. In this setting we lack a complete classification of the simple Lie algebras, so we must find other means of determining isomorphism classes.

We begin by showing that a Lie algebra, that only exists for $p=3$, is in fact a maximal subalgebra in $F_4$. This was the main result of the research article \cite{Pur16}, which is due to appear in Experimental Mathematics, with the online version already available.

The majority of our remaining problems concern nilpotent elements $e$ such that $e^{[p]}=0$, and $\g_e \ne \Lie(G_e)$. This gives many cases of non-semisimple maximal subalgebras in both $p=2$ and $p=3$ along with interesting Weisfeiler filtrations.

We finish the thesis with a conjecture regarding a possible new class of simple Lie algebras that have links to a new maximal simple subalgebra of $E_8$ when $p=2$.
\afterabstract

% The next part is optional; however it is a good place to thank your
% supervisor and the people responsible for providing computer support ;-)

\prefacesection{Acknowledgements}
\def\baselinestretch{1.2}\normalsize
I would like to thank Professor Alexander Premet for his constant encouragement during this research. Without his patience, generosity and mastery of the subject this work may never have been completed. It has been an absolute pleasure to learn from him, and one of the best experiences anyone could hope to have.

To all my friends and family; All the support you've given is a huge reason as to why I have finished.

\def\baselinestretch{\stretch}\normalsize
% The next line is NOT optional and MUST appear
\afterpreface

% Finally, you can start writing about all the new theorems you have proved
% and all the new results that you have discovered

\lstset{language=C}
\tikzset{middlearrow/.style={decoration={markings, mark= at position 0.6 with {\arrow[scale=4]{#1}} ,}, postaction={decorate}}}
\chapter[Introduction]{Introduction}\label{sec:prelim}
We aim to produce new examples of maximal subalgebras in the exceptional Lie algebras over algebraically closed fields of low characteristic. The study of maximal subalgebras has been well-studied over the complex numbers with \cite{Dyn52} producing the full classification.

In the modular case, the topic was relatively untouched until \cite{HS14, HS16} started to consider maximal subalgebras of Cartan type. Since then \cite{P15} achieved the complete classification of non-semisimple maximal subalgebras, and classifies them as parabolic subalgebras --- exactly the same as the complex numbers.

\section[Notation]{Notation}\label{sec:notation}

We always consider $G$ an algebraic group of exceptional type with $\mathfrak{g}=\mathrm{Lie}(G)$ and $\bbk$ an algebraically closed field of positive characteristic. We initially focus on the classification of the simple Lie algebras over fields of positive characteristic, as many of them will appear in our examples of maximal subalgebras.

To define the classical Lie algebras, we will be using root systems. This allows each of them to fit under the same construction. This can then be used to obtain the analogous version in the modular case using a \emph{Chevalley basis}.

First, we consider $\g$ as a simple Lie algebra over $\bbc$, and consider $\mathfrak{h}$ a Cartan subalgebra (or equivalently, a maximal toral subalgebra) of dimension $l$. Let $\Phi=\Phi(\g,\h)$\index{$\Phi$, root system} be the corresponding root system, and $\Phi_0=\{\al_1,\ldots,\al_l\}$\index{$\Phi_0$, basis of simple roots} be a basis of simple roots for $\Phi$.

We express any root as a linear combination of such simple roots using the ordering in \cite{Bour02}. For now, we give the Dynkin diagram to illustrate the order we use in the exceptional Lie algebras. In \aref{appendix:gaproot} we give further details on how to define the root systems of each exceptional Lie algebra with the ordering as below.

\tikzset{node distance=2em, ch/.style={circle,draw,on chain,inner sep=2pt},chj/.style={ch,join},every path/.style={shorten >=0pt,shorten <=0pt},line width=3pt,baseline=-1ex}
\tikzset{middlearrow/.style={decoration={markings, mark= at position 0.6 with {\arrow[scale=3]{#1}} ,}, postaction={decorate}}}
\begin{align}\label{dynkindiagrams}
G_2\quad
&\begin{tikzpicture}[start chain]
\dnode{1}
\dnodenj{2}
\draw[middlearrow={<}]  ([yshift=0pt]chain-1.east) --([yshift=0pt]chain-2.west); \draw[-] ([yshift=-5pt]chain-1.east) edge([yshift=-5pt]chain-2.west); \draw[-] ([yshift=5pt]chain-1.east)edge([yshift=5pt]chain-2.west);
\end{tikzpicture}
\\\label{dynkindiagrams2}
F_4\quad
&\begin{tikzpicture}[start chain]
\dnode{1}
\dnode{2}
\dnodenj{3}
\dnode{4}
\draw[middlearrow={>}] ([yshift=-2pt]chain-2.east) --([yshift=-2pt]chain-3.west); \draw[-] ([yshift=2pt]chain-2.east)edge([yshift=2pt]chain-3.west);
\end{tikzpicture}
\\\label{dynkindiagrams3}
E_6\quad
&\begin{tikzpicture}
\begin{scope}[start chain]
\foreach \dyi in {1,3,4,5,6} {\dnode{\dyi}}
\end{scope}
\begin{scope}[start chain=br going above]
\chainin(chain-3);
\dnodebr{2}
\end{scope}
\end{tikzpicture}
\\\label{dynkindiagrams4}
E_7\quad
&\begin{tikzpicture}
\begin{scope}[start chain]
\foreach \dyi in {1,3,4,5,6,7} {\dnode{\dyi}}
\end{scope}
\begin{scope}[start chain=br going above]
\chainin(chain-3);
\dnodebr{2}
\end{scope}
\end{tikzpicture}
\\\label{dynkindiagrams5}
E_8\quad
&\begin{tikzpicture}
\begin{scope}[start chain]
\foreach \dyi in {1,3,4,5,6,7,8} {\dnode{\dyi}}
\end{scope}
\begin{scope}[start chain=br going above]
\chainin(chain-3);
\dnodebr{2}
\end{scope}
\end{tikzpicture}
\end{align}

To shorten notation further, we specify roots by giving the coefficients. For example, the highest root in $E_8$ will be expressed as ${\subalign{24&65432\\&3}}$. Throughout we will interchange between $e_{\alpha_2+\alpha_4}$ and $e_{\subalign{00&10000\\&1}}$ to mean the same element, and use $f_{\alpha}$ in place of $e_{-\alpha}$.

For $\al,\be \in \Phi$ set \begin{equation}\label{coroot}\langle \be,\al^{\vee}\rangle=\frac{2(\be|\al)}{(\al|\al)}.\end{equation} We always have that $(\al|\al)$ is an inner product defined on the roots. For us, this is understood to be the normal scalar product since we define all the exceptional Lie algebras using vectors in $\mathbb{R}^n$. We give a description of this in \aref{appendix:gaproot}.

Associated to this Cartan subalgebra we have the root space decomposition \begin{equation*}\g=\h \oplus \sum_{\al \in \Phi_0} \g^{\al}.\end{equation*}

\begin{thm}\label{ChevBas}\cite{Che56} Let $\g$ be a simple Lie algebra over $\mathbb{C}$. There is a basis $\{e_{\al}:\al \in \Phi\} \cup \{h_i: 1 \le i \le l \}$ such that \begin{enumerate}
\item{$[h_i,h_j]=0$ for all $1 \le i,j \le l$,}
\item{$[h_i,e_{\be}]=\langle \be,\al_i^{\vee} \rangle e_{\be}$ for all $1 \le i \le l, \be \in \Phi$,}
\item{$[e_{\al},e_{-\al}]=h_\al$ is a $\bbz$-linear combination of $h_1,\ldots,h_l$,}
\item{If $\al,\be$ are independent roots and $\be-r\al,\ldots,\be+q\al$ is the $\al$-string through $\be$, then $[e_{\al},e_{\be}]=0$ if $q=0$, while $[e_{\al},e_{\be}]=\pm(r+1)e_{\al+\be}$ if $\al+\be \in \Phi$. Moreover, $r \in \{0,1,2\}$ if $\al + \be \in \Phi$.}\end{enumerate}\end{thm}

This is the \emph{Chevalley basis} for classical simple Lie algebras, and taking the $\bbz$-span of this basis is a $\bbz$-subalgebra in $\g$. Denoting this as $\g_{\bbz}$ we may consider the tensor product $\g_{\bbz}\otimes_{\bbz} \bbk$. This is now a Lie algebra with the same basis but structure constants reduced modulo $p$.

The simple Lie algebras over $\bbc$ continue to be simple for algebraically closed fields of characteristic $p \ge 5$ except in the case of $A_{n-1}$ for $p|n$. In this case we quotient out by the centre, and denote the new simple Lie algebra as $\mathfrak{psl}(n)$.

Abusing notation, we include the exceptional Lie algebras, in what we call the \emph{classical simple Lie algebras}. The area we consider will be the exceptional Lie algebras over fields of prime characteristic. The starting point for the study of maximal subalgebras is to split into \emph{good} and \emph{bad} primes.

\begin{defs}Let $\g$ be a simple Lie algebra. Then a prime $p$ is called \emph{bad} for $\g$ if $p$ divides the coefficient of some root $\al$ when expressed as a combination $\al=\sum_I n_i \al_i$ for simple roots $\al_i$. A prime is \emph{good} if it is not bad. A prime $p$ is \emph{very good} if it is good and $p$ does not divide $n$ when $\g$ has type $A_{n-1}$. \end{defs}

It is easy to see which primes are bad for the exceptional Lie algebras, and for certain $p$ some exceptional Lie algebras cease to be simple with details given in \autorefs{badp}.

\begin{table}[H]\caption{List of bad primes}\phantomsection\label{badp}\centering
\begin{tabular}{|M{1cm}|M{1cm}|M{2cm}| }
\hline $\g$ & $p$ & Simple? \\ \hline
\multirow{2}{*}{$G_2$} & $3$& No \\& $2$& Yes\\ \hline
\multirow{2}{*}{$F_4$} &$3$ & Yes\\& $2$ & No \\  \hline
\multirow{2}{*}{$E_6$}& $3$ & No \\ & $2$ & Yes\\ \hline
\multirow{2}{*}{$E_7$} &$3$ &Yes\\& $2$ & No \\  \hline
\multirow{3}{*}{$E_8$}&$5$ &Yes\\  & $3$ & Yes \\&$2$&Yes\\ \hline
\end{tabular}\end{table}

For characteristic three, $G_2$ has an ideal generated by the short roots $I$ such that $G_2/I\cong I\cong \mathfrak{psl}(3)$ and $E_6$ has a one-dimensional centre generated by the element $h_{\al_1}+2h_{\al_3}+h_{\al_5}+2h_{\al_6}$. Similarly $F_4$ has an ideal generated by the short roots of dimension $26$ for characteristic two, and $E_7$ has a one-dimensional centre generated by the element $h_{\al_2}+h_{\al_5}+h_{\al_7}$. Finally, we have that $E_8$ is simple for all prime $p$.

\begin{rem}\label{nosimexc2}
To obtain the ideal generated by the short roots in $G_2$, we use the elements indexed by the roots $\{\pm \{10\}, \pm\{11\}, \pm \{21\}\}$. For the ideal generated by the short roots of $F_4$ we refer the reader to \autoref{sec:chartwo}.\end{rem}

Another important notion for this work is restricted Lie algebras, and their $p$-mappings.

\begin{defs}\cite[Chapter V, Section 7]{Jac62} Let $\g$ be a Lie algebra over $\bbk$. A mapping $[p]:\g \ra \g$ such that $x \mapsto x^{[p]}$ is called a \emph{p-mapping} if it satisfies the following conditions \begin{enumerate}\item{$\ad \, x^{[p]}=(\ad \, x)^p$ for all $x \in \g$,}\item{$(\la x)^{[p]}=\la^p x^{[p]}$ for all $x \in \g$ and $\la \in k$,}\item{$(x+y)^{[p]}=x^{[p]}+y^{[p]}+\sum_{i=1}^{p-1}s_i(x,y),$}\end{enumerate} where $s_i(x,y) \in \g$ such that \[\sum_{i=1}^{p-1}is_i(x,y)t^{i-1}=(\ad(tx+y))^{p-1}(x)\] for $x,y \in \g, \la \in k$ and $t$ a variable. The pair $(L,[p])$ is a \emph{restricted Lie algebra}.\index{$[p]$-mapping}\end{defs}

It follows from \cite[Theorem 11]{Jac62} that any $p$-mapping on a restricted Lie algebra $\g$ is unique provided $\ad\,\g=\Der\,\g$ which is true for all the exceptional Lie algebras. These will be used throughout, where nilpotent elements $e\in\g$ such that $e^{[p]}=0$ produce new non-semisimple maximal subalgebras.

In some of the literature, a Lie algebra $L$ is called \emph{restrictable} if there is a mapping $[p]:L \rightarrow L$ such that $(\ad\,x)^p=\ad\,x^{[p]}$ for all $x \in L$.

To see that all the classical simple Lie algebras $\g$ are restricted, we define a $p$-mapping using a basis $\{e_i\}$ for $\g$. For each $e_i$, let $e_i^{[p]}$ be an element of $\g$ such that $(\ad\,e_i)^p=\ad\,e_i^{[p]}$. Then, \cite[Theorem 11]{Jac62} gives that this defines a unique $p$-mapping for every $x \in \g$. In fact, this shows that any restrictable Lie algebra can be turned into a restricted Lie algebra in a straightforward manner.

\section[Modular Lie algebras]{Modular Lie algebras}\label{sec:notation2}

The second type of simple Lie algebras are the first non-classical simple Lie algebras with no analogues in characteristic zero. This becomes clear by definition since we use truncated polynomial rings which are finite-dimensional when $p>0$. They bear some resemblance to the classical Lie algebras over $\bbc$ since there are $4$ infinite families known as, the $W$ series, $S$ series, $H$ series and $K$ series. The Witt algebras take the place of the series $A_n$, since each of the remaining series are all subalgebras of the $W$ series.

We use \cite{Strade04, SF04, BGP09} as our main references for the simple Lie algebras of Cartan type. The first examples of non-classical simple Lie algebras are the Jacobson--Witt algebras. The results from \cite{SF04} and \cite{Strade04} provide us with a nice description of each of these Lie algebras. We also obtain what we shall refer to as the \emph{standard} grading on each of the simple Lie algebras of Cartan type.

To this end, we start by considering \emph{divided power algebras}. These give a natural way of defining the Cartan type Lie algebras.

\begin{defs}Let $\bbk$ be a field of characteristic $p>0$ and $\mathcal{O}(m)$ denote the commutative associative algebra with $1$ over $\bbk$ defined by generators $x_i^{(j)}$ for $i=1,\ldots,m$ satisfying the relations \begin{align*}x_i^{(0)}&=1\\ x_i^{(j)}x_i^{(l)}&=\binom{j+l}{j}x_i^{(j+l)},\end{align*} where $(j) \in \mathbb{N}^n$. For $(a)=(a_1,\ldots,a_n)\in \mathbb{N}^n$\index{$(a)=(a_1,\ldots,a_n)$ a vector in $\mathbb{N}^n$} we have that the $x^{(a)}:=x_1^{(a_1)}\cdots x_m^{(a_n)}$ with $a_i \in \mathbb{N}$ give a basis of $\mathcal{O}(m)$. This is the \emph{divided power algebra}. \end{defs}

We obtain a subalgebra \[\mathcal{O}(m;\underline{n}):=\spnd_{\bbk}\{x^{(a)}:0 \le a_i < p^{n_i}\},\] for $\underline{n}=(n_1,\ldots,n_m) \in \mathbb{N}^m$.\index{$\mathcal{O}(m;\underline{n})$, truncated polynomial ring} For $\mathcal{O}(m;\underline{1})$ this is just the usual truncated polynomial ring with $n$-variables, and the others are obtained by allowing the power of our variables to exceed $p-1$. Using the former we obtain the restricted Cartan type Lie algebras, and the second produces the non-restricted cases.

\begin{defs} A \emph{grading} on a Lie algebra $\g$ is given by $\g:=\bigoplus_{i=-r}^s \g_i$ such that $[\g_i,\g_j]\subseteq [\g_{i+j}]$. We say $r$ is the \emph{depth} and $s$ the \emph{height} of such a grading. \end{defs}

For divided power algebras there is an obvious grading given by $(\mathcal{O}(m))_{i}:=\{x^{(a)}:\sum_{j=1}^n a_j=i\}$ where $(a)=(a_1,\ldots,a_n)$. Then $\mathcal{O}(m;\underline{n}):=\bigoplus_{i=0}^{r+1}(\mathcal{O}(m;n))_{i}$ where $r=p^{n_1}+\ldots+p^{n_m}-m-1$.

We are now in a position to define the Lie algebras of \emph{Witt} type. Let $\partial_i$ be the derivation of $\mathcal{O}(m)$ defined \[\partial_i(x_j^{(r)}):=\delta_{i,j}x_i^{(r-1)}.\] We then take a subalgebra of $\Der(\mathcal{O}(m))$ defined by
\[W(m):=\spnd_{\bbk}\{x^{(a)}\partial_i:a_i \in \mathbb{N}, i=1,\ldots,m\},\]
with Lie bracket given by
\[[x^{(a)}\partial_i,x^{(b)}\partial_j]=\binom{a+b-\epsilon_i}{a}x^{(a+b-\epsilon_i)}\partial_j-\binom{a+b-\epsilon_j}{b}x^{(a+b-\epsilon_j)}\partial_i,\] for $\epsilon_i$ the standard unit vector in $\mathbb{N}^m$ with $1$ in the $i$-th position and zero everywhere else.
Since $\mathcal{O}(m;\underline{n})$ is a subalgebra of $\mathcal{O}(m)$ it follows that
\[W(m;\underline{n}):=\spnd_{\bbk}\{x^{(a)}\partial_i:0 \le a_i < p^{n_i}, i=1,\ldots,m\},\] \index{$W(m;\underline{n})$, the Witt algebra}is a Lie subalgebra of $W(m)$.

\begin{thm}\label{dimwitt}\cite[Theorem 2.4]{SF04} $W(m;\underline{n})$ is a simple Lie algebra except when $p=2$ and $m=1$. The elements $\{x^{(a)}\partial_i:0 \le a_i < p^{n_i}, i=1,\ldots,m\}$ define a basis for the Witt algebra, and hence $\dim\, W(m;\underline{n})=mp^{n_1+\ldots+n_m}.$\end{thm}

Using the grading on divided power algebras we give what we refer to as the \emph{standard} grading of $W(m;\underline{n})$. We consider each graded component \[{W(m;\underline{n})}_i:=\bigoplus_{j=1}^m {\mathcal{O}(m;\underline{n})}_{i+1}\partial_j.\] This gives a depth-one grading, with $W(m;\underline{n})=\bigoplus_{i=-1}^r {W(m;\underline{n})}_i$.

\begin{rem}We have ${W(m;\underline{n})}_{-1}=\{\partial_1,\ldots,\partial_m\}$, and so $\dim\, {W(m;\underline{n})}_{-1}=m$. We have that $x_i^{(1)}\partial_j$ is in the zero-component for $1 \le i,j\le m$ and taking the Lie bracket we have \[[x_i^{(1)}\partial_j,x_k^{(1)}\partial_l]=\delta_{j,k}x_{i}^{(1)}\partial_l-\delta_{i,l}x_k^{(1)} \partial_j \in W(m;\underline{n})_0.\] Forming an isomorphism with $x_i^{(1)}\partial_j \mapsto E_{i,j}$ gives that ${W(m;\underline{n})}_0 \cong \mathfrak{gl}(2)$.\end{rem}

To define the remaining examples of the Cartan type Lie algebras we may give equivalent conditions, and later for simple Lie algebras in characteristic three we make use of them in some detail. We will consider \index{$\Omega^0(m;\underline{n})=\mathcal{O}(m;\underline{n})$}
\index{$\Omega^1(m;\underline{n})=\Hom_{\mathcal{O}(m;\underline{n})}(W(m;\underline{n}),\mathcal{O}(m;\underline{n}))$} \begin{align}\begin{split}\label{omega} \Omega^0(m;\underline{n})&:=\mathcal{O}(m;\underline{n})\\ \Omega^1(m;\underline{n})&:=\Hom_{\mathcal{O}(m;\underline{n})}(W(m;\underline{n}),\mathcal{O}(m;\underline{n})).\end{split}\end{align}

$\Omega^1(m;\underline{n})$ can be thought of as both a $\mathcal{O}(m;\underline{n})$-module and a $W(m;\underline{n})$-module defining the structure as \begin{align}\begin{split}(f \lambda)D&:=f\lambda(D)\\
(D\lambda)(E) &:= D(\lambda(E)) - \lambda([D,E])\end{split}\end{align} for all $D,E \in W(m;\underline{n}), \lambda \in \Omega^1(m;\underline{n})$ and $f \in \mathcal{O}(m;\underline{n})$.

Finally, we define a mapping \begin{align}\label{dmapping}\begin{split}&d: \Omega^0(m;\underline{n}) \rightarrow \Omega^1(m;\underline{n})\\&df(D)=D(f)\end{split}\end{align}\index{$d:\Omega^0(m;\underline{n}) \rightarrow \Omega^1(m;\underline{n})$ with $df(D)=D(f)$} for all $f \in \mathcal{O}(m;\underline{n})$ and $D \in W(m;\underline{n})$.

We now take the opportunity to define two important classes of $W(2;\underline{n})$-modules that reappear later for simple Lie algebras in characteristic three. Consider the map $\divs: W(m;\underline{n}) \rightarrow \mathcal{O}(m;\underline{n}),$ defined by \begin{equation}\divs\left(\sum_{i=1}^m f_i \,\partial_i\right)=\sum_{i=1}^m \partial_i(f_i).\end{equation}\index{$\divs: W(m;\underline{n}) \rightarrow \mathcal{O}(m;\underline{n})$} For any $\alpha \in \mathbbm{k}$, there are two $W(m;\underline{n})$-modules to consider:

\begin{enumerate}\item{$\mathcal{O}(m;\underline{n})_{(\alpha\,\divs)}$\index{$\mathcal{O}(m;\underline{n})_{(\alpha\,\divs)}$, a $W(m;\underline{n})$-module} obtained by taking $\mathcal{O}(m;\underline{n})$ under the action \begin{equation}\label{omod}D\cdot f:=D(f)+\alpha \,\divs(D)f.\end{equation}}\item{$W(m;\underline{n})_{(\alpha\,\divs)}$ obtained by taking $W(m;\underline{n})$ taking the analogous action \begin{equation}\label{wmod}D\cdot E:=[D,E]+\alpha \,\divs(D)E,\end{equation}}\end{enumerate}for all $D,E \in W(m;\underline{n})$ and $f \in \mathcal{O}(m;\underline{n})$.

Consider the subalgebra of $W(m;\underline{n})$ generated by derivations $D=\sum_{i=1}^nf_i\partial_i$ where \begin{equation}\sum_{i=1}^n\partial_i(f_i)=0.\end{equation} In other words, consider the subalgebra generated by all derivations with divergence zero. This is a subalgebra of $W(m;\underline{n})$, since for all $D,E \in W(m;\underline{n})$ we have \[\divs([D,E])=D(\divs(E))-E(\divs(D)).\]

We have $S(m;\underline{n}):=\{D \in W(m;\underline{n}): \divs(D)=0\}$\index{$S(m;\underline{n})$, special Lie algebra}, and obtain a simple Lie algebra by taking the derived subalgebra. We define maps \[D_{i,j}: \mathcal{O}(m;\underline{n})\rightarrow W(m;\underline{n})\] by \begin{equation*}D_{i,j}(f)=\partial_j(f)\partial_i-\partial_i(f)\partial_j\end{equation*} for $1 \le i,j \le n$. It is straightforward to see that the image of $D_{i,j}$ is in $S(m;\underline{n})$ just by checking $\divs(D_{i,j})=D_i(D_j(f))-D_j(D_i(f))=0$. Using \cite[Lemma 2.32]{BGP09} we see that the image of $D_{i,j}$ lies in the derived subalgebra of $S(m;\underline{n})$, and hence the image of $D_{i,j}$ is precisely $S(m;\underline{n})^{(1)}$.

\begin{thm}\label{dimspec}\cite[Proposition 3.3, Theorem 3.5 and 3.7]{SF04} Let $m>2$, then $S(m;\underline{n})^{(1)}$ is simple, and spanned by the elements \[\{D_{i,j}(x^{(a)}):0\le a_k <p^{n_k} \,\,\text{for}\,\, 1 \le k \le m \,\,\text{and}\,\, 1 \le i,j <m\},\] such that $\dim\,S(m;\underline{n})=(m-1)(p^{\sum_{i=1}^m n_i}-1)$.\end{thm}

This is the so-called Lie algebra of \emph{Special} type. The smallest dimension of $S(m;\underline{n})^{(1)}$ is $248$ when $p=5$ for $m=3$ and $\underline{n}=\underline{1}$. This matches the dimension of simple Lie algebras of type $E_8$, but they are non-isomorphic. Although this allows us to rule out $S(m;\underline{n})$ as a maximal subalgebra in all exceptional Lie algebras when $p \ge 5$, we will see this algebra make an appearance in a different way. We finish with the grading inherited from $W(m;\underline{n})$, \[S(m;\underline{n})^{(1)}=\bigoplus_{j=-1}^{s}(S(m;\underline{n})^{(1)})_j\] where $(S(m;\underline{n})^{(1)})_j=S(m;\underline{n})^{(1)}\cap W(m;\underline{n})_j$, and $s=p^{\sum n_j}-m-2$.

Since $D_{i,j}(x_j^{(1)})=\partial_i$ it follows that $(S(m;\underline{n})^{(1)})_{-1}=W(m;\underline{n})_{-1}$. For the zero component, the elements \begin{align*}D_{j,i}(x_i^{(2)})&=x_i^{(1)}\partial_j \quad\text{for all}\quad i \neq j, \\ D_{i,i+1}(x_i^{(1)}x_{i+1}^{(1)})&=x_i^{(1)}\partial_i-x_{i+1}^{(1)}\partial_{i+1}\quad\text{for}\quad i=1,\ldots,m-1\end{align*} all have divergence zero. Any other element from $W(m;\underline{n})_0$ has non-zero divergence, for example $x_i^{(1)}\partial_i$ clearly fails to have divergence zero. It follows that $\dim \,(S(m;\underline{n})^{(1)})_0=m^2-1$, and hence we may show that $(S(m;\underline{n})^{(1)})_0\cong \mathfrak{sl}(m)$.

An alternative definition in some of the literature may be obtained by considering the differential form $\omega_S:=dx_1\wedge\ldots\wedge dx_m$. It can be checked that \[\{D \in W(m;\underline{n}):D(\omega_S)=0\}= \{D \in W(m;\underline{n}):\divs(D)=0\},\] with \cite[\S4.3]{SF04} providing the necessary details to show these are the same construction.

There is another Lie algebra that we can obtain that is closely related to the Special Lie algebra $S(m;\underline{n})$. We adjoin the degree derivation $\sum_{i=1}^m x_i^{(1)}\partial_i$ to $S(m;\underline{n})_0$. Since \[\left[\sum_{i=1}^m x_i^{(1)}\partial_i,E\right]=jE,\] for all $E \in W(m;\underline{n})_j$, it follows that this new construction is also a Lie algebra. We denote this by $CS(m;\underline{n})$\index{$CS(m;\underline{n})$}.

For the Lie algebras of \emph{Hamiltonian} type\index{$H(2m;\underline{n})$, Hamiltonian Lie algebra} we may consider the intersection $H(2m;\underline{n}):=W(2m;\underline{n}) \cap H(2m)$, where $H(2m)$ is the subalgebra of $W(2m)$ consisting of derivations satisfying $D(\omega_H)=0,$ with \[\omega_H=\sum_{i=1}^m dx_j \wedge dx_{j+m}.\] This is equivalent to the following construction. Define \begin{equation}\sigma(j)=\begin{cases} 1 & : 1 \le j <m\\ -1 & : m+1 \le j <2m, \end{cases}\end{equation} and \begin{equation}\label{hamstuff}j'=\begin{cases} j+m &{: 1 \le j \le m}\\ j-m &: m+1 \le j \le 2m. \end{cases}\end{equation}Then, consider
\begin{align*}H(2m;\underline{n})&:=\left\{\sum_{i=1}^{2m}f_i\partial_i \in W(2m;\underline{n}):\sigma(j')\partial_i(f_{j'})=\sigma(i')\partial_j(f_{i'}),\right.\\ &\hspace{8cm}\left. 1 \le i,j \le 2m \vphantom{\sum_{i}^{i}}\right\}.\end{align*}

To obtain the ``usual'' basis for $H(2m;\underline{n})$ we first consider the map $D_H:\mathcal{O}(m)\rightarrow W(m)$ defined by \begin{equation}\label{hamultip2}D_H: f \mapsto \sum_{j=1}^{2m}\sigma(j)\partial_j(f)\partial_{j'}.\index{$D_H(f)= \sum_{j=1}^{2m}\sigma(j)\partial_j(f)\partial_{j'}$}\end{equation}

\begin{rem}\label{hamstructure}The elements $D_H(f)$ lie in $H(2m;\underline{n})$ for all $f \in \mathcal{O}(2m;\underline{n})$. However, not all elements of $H(2m;\underline{n})$ lie in the image of $D_H$. For example; the derivation $x_j^{(p^{n_j}-1)}\partial_{j'}$ for $1 \le j \le 2m$ is not $D_H(f)$ for any $f \in \mathcal{O}(2m;\underline{n})$.\end{rem}

Using \cite[Theorem 2.3 (i) and (iv)]{BGP09} we obtain that for $D, E \in H(2m;\underline{n})$ there is some $u$ such that $[D,E]=D_H(u)$. Further, for $f,g \in \mathcal{O}(2m;\underline{n})$ we have that $[D_H(f),D_H(g)]=D_H(D_H(f)(g))$. Hence, it follows that the derived subalgebra of $H(2m;\underline{n})$ is in the image of $D_H$.

\begin{thm}\label{dimham}\cite[Lemma 4.1 and Theorem 4.5]{SF04} $H(2m;\underline{n})^{(2)}$ is a simple Lie algebra with basis
\[\{D_H(x^{(a)}):a \ne(0,\ldots,0) \,\,\text{and}\,\, a \ne (p^{n_1}-1,\ldots,p^{n_{2m}}-1)\}.\]
Hence, $\dim \,H(2m;\underline{n})^{(2)}=p^{\sum_{i=1}^{2m}n_i}-2.$\end{thm}

Similarly to the Special Lie algebra, we may adjoin the degree derivation $\mathcal{D}:=\sum_{i=1}^{2m} x_i^{(1)}\partial_i$ to $H(2m;\underline{n})$ denoted as $CH(2m;\underline{n})$ to obtain a larger Lie algebra. It should be noted that we also adjoin $\mathcal{D}$ to the simple Lie algebra $H(2m;\underline{n})^{(2)}$. In this case we denote the new Lie algebras as $H(2m;\underline{n})^{(2)}\oplus\bbk\mathcal{D}$ since $(CH(2m;\underline{n}))^{(2)}=H(2m;\underline{n})^{(1)}$.\index{$CH(2m;\underline{n})$}

The standard grading is inherited from $W(m;\underline{n})$, consider \[(H(2m;\underline{n})^{(2)})_j=H(2m;\underline{n})^{(2)}\cap W(2m;\underline{n})_j=D_H(\mathcal{O}(2m;\underline{n})_{j+2}).\] It follows that $H(2m;\underline{n})^{(2)}=\bigoplus_{j=-1}^{r}(H(2m;\underline{n})^{(2)})_j$, and $r=p^{\sum n_j}-2m-3$. We have $\dim\, H(2m;\underline{n})_{-1}=2m$, and zero component spanned by elements $D_H(x_ix_j)=\sigma(j)x_i\partial_{j'}+\sigma(i)x_j\partial_{i'}$. Hence, $(H(2m;\underline{n})^{(2)})_0\cong \mathfrak{sp}(2m)$.

Our final example of the Cartan type Lie algebras is the \emph{contact} Lie algebra \index{$K(2m+1;\underline{n})$, contact Lie algebra}, consider $f \in \mathcal{O}(2m+1)$ with \[D_K(f):=\sum_{i=1}^{2m+1}f_i\partial_i,\] where \begin{align*}f_i &:=x_i\partial_{2n+1}(f)+\sigma(i')\partial_{i'}(f), 1 \le i \le 2m, \\ f_{2m+1}&=\Delta(f):=2f-\sum_{j=1}^{2m}x_j\partial_j(f)\end{align*} The elements $D_K$ are elements of $W(2m+1)$, and $[D_K(f),D_K(g)]=D_K(u)$ where \begin{align*}u&=\Delta(f)\partial_{2m+1}(g)-\Delta(g)\partial_{2m+1}(f)+\{f,g\}\\\text{where}\quad\{f,g\}&=\sum_{j=1}^{2m}\sigma(j)\partial_j(f)\partial_{j'}(g).\end{align*} It is now clear that $D_K(f)$ for $f \in \mathcal{O}(2m+1)$ form a subalgebra of $W(2m+1)$. We denote this by $K(2m+1)$, and the \emph{contact} Lie algebra is the intersection \[K(2m+1;\underline{n})=K(2m+1) \cap W(2m+1;\underline{n}).\]

\begin{thm}\label{dimcon}\cite[Theorem 5.5]{SF04} $K(2m+1;\underline{n})^{(1)}$ is a simple Lie algebra, and \[K(2m+1;\underline{n})^{(1)}=\left\{ \begin{array}{lr} K(2m+1;\underline{n}) & : 2m+4\not\equiv 0 \mod p\\ \spnd_{\bbk}\{D_K(x^{(a)}):a \neq \tau(m)\} & : 2m+4 \equiv 0 \mod p, \end{array} \right.\] where $\tau(m)=(p^{n_1}-1,\ldots,p^{n_{2m+1}}-1).$ The dimension is then either $p^{\sum_{i=1}^{2m+1}n_i}$ or $p^{\sum_{i=1}^{2m+1}n_i}-1$ respectively. \end{thm}

Suppose for $a=(a_1,\ldots,a_{2m+1})$ with $a_i \in \mathbb{N}$ that $\|a\|:=|a|+a_{2m+1}-2$. We use this to define the standard grading of $K(2m+1,\underline{n})$. This differs from the other Cartan type Lie algebras as the grading has depth two. We define spaces as $K(2m+1;\underline{n})_j=\spnd_{\bbk}\{D_K(x^{(a)}):\|a\|=j\}.$ This produces \[K(2m+1,\underline{n})=\bigoplus_{j=-2}^r K(2m+1;\underline{n})_j,\] where $r=\sum_{i=1}^{i=2m}p^{n_i}+2p^{n_{2m+1}}-2m-4$. We have $K(2m+1;\underline{n})_0 \cong \mathfrak{csp}(2m)= \mathfrak{sp}(2m)\oplus \bbk I$ where $I$ is the identity matrix, and $\dim \, K(2m+1;\underline{n})_{-1}=2m$.

This completes the list of simple Lie algebras of Cartan type. It should be noted that we may also consider a differential form in this case to define the contact algebras, in particular consider \[\omega_K:=dx_m+\sum_{i=1}^{2r} \sigma(i)x_idx_{i'}\] for $\sigma$ from \eqref{hamstuff}. For completeness $K(2m+1;\underline{n}):=\{D \in W(2m+1;\underline{n}):D(\omega_K) \in \mathcal{O}(2m+1;\underline{n})\omega_K\}$, with the equivalence obtained in \cite[Theorem 5.1.1(3)]{Strade04}.

The classification of simple Lie algebras over algebraically closed fields of characteristic $p \ge 5$ was achieved through a series of papers by Premet and Strade culminating with \cite{PSt08}. To give the complete classification we consider one final case --- the Melikyan Lie algebras.

Let $\widetilde{W(m;\underline{n})}$ be a copy of $W(m;\underline{n})$. Then, define the \emph{Melikyan algebras}\index{$M(n_1,n_2)$, the Melikyan algebras} as \[M(\underline{n}):=W(2;\underline{n})\oplus {\widetilde{W(2;\underline{n})}}_{(2\,\divs)} \oplus {\mathcal{O}(2;\underline{n})}_{(-2\,\divs)}.\]This space has a Lie algebra structure with Lie bracket, \begin{align*}[f_1\widetilde{D_1}+f_2\widetilde{D_2},g_1\widetilde{D_1}+g_2\widetilde{D_2}]&=f_1g_2-f_2g_1\\ [D,\widetilde{E}]&=\widetilde{[D,E]}+2\,\divs(D)\widetilde{E} \\ [D,f]&=D(f)-2\,\divs(D)f  \\ [f,\widetilde{E}]&=fE \\ [f,g]&=2(f\widetilde{D}_g-g\widetilde{D}_f)\end{align*} where $\widetilde{D}_f=D_1(f)\widetilde{D}_2-D_2(f)\widetilde{D}_1$, for all $D \in W(2;\underline{n})$, $\widetilde{E} \in \widetilde{W(2;\underline{n})}$ and $f_i, g_i, f, g \in \mathcal{O}(2;\underline{n}).$

\begin{prop}\label{melikdef}\cite[Lemma 4.3.1, Theorem 4.3.3]{Strade04} $M(n_1,n_2)$ is a simple Lie algebra only when $p=5$, and has dimension $5^{n_1+n_2+1}$. The Melikyan algebra $M(n_1,n_2)$ is restricted if and only if $(n_1,n_2)=(1,1)$.\end{prop}

This construction could be attempted for all $p$, but the Jacobi identity for this Lie bracket fails unless $p=5$. The smallest dimension of a Melikyan algebra is $125$, and so may appear in $E_7$ or $E_8$ as a maximal subalgebra. However, this was ruled out in \cite[\S4.3]{P15}. We now give the complete classification of simple modular Lie algebras.

\begin{thm}\label{classification}\cite{PSt08} Any finite-dimensional simple Lie algebra over an algebraically closed field of characteristic $p \ge 5$ is of classical, Cartan or Melikyan type.\end{thm}

To finish our description of simple Lie algebras of Cartan type, we give some results from \cite{Strade04} about the restrictedness of such algebras. The idea of \emph{$p$-envelopes} is quite simple, we want all possible $p$-th powers of elements in our Lie algebra. Let $\g$ be a restricted Lie algebra, and $L$ be a subalgebra of $\g$. We denote by $L_{[p]}$ the smallest restricted subalgebra of $\g$ that contains $L$.

\begin{defs}\cite[Definition 1.2.1]{Strade04} Let $L$ to be a finite dimensional Lie algebra. Consider a triple $(\g,[p],i)$ such that $\g$ is a restricted Lie algebra. Suppose $i:L \ra \g$ is an injective Lie algebra homomorphism, then we call this a \emph{$p$-envelope}\index{$L_{[p]}$, the $p$-envelope or $p$-closure of $L$} of $L$ if $i(L)_{[p]}=\g$.\end{defs}

We now consider the minimal $p$-envelope of the simple Lie algebras of Cartan type. Using \cite[Lemma 7.2.1]{Strade04} there are calculations that are used to prove the following result:

\begin{thm}\cite[Theorem 7.2.2]{Strade04} Let $L$ be a Cartan type Lie algebra, and $L_{[p]}$ denote the $p$-envelope of $L$ in $\Der(L)$. Then we have the following \begin{enumerate}\item{$W(m;\underline{n})_{[p]}=W(m;\underline{n})+\sum_{i=1}^m\sum_{0<j_i<n_i}\bbk\partial_i^{p^{j_i}}$,}
\item{${S(m;\underline{n})^{(1)}}_{[p]}=S(m;\underline{n})^{(1)}+\sum_{i=1}^m\sum_{0<j_i<n_i}\bbk\partial_i^{p^{j_i}}$,}
\item{${H(2m;\underline{n})^{(2)}}_{[p]}=H(2m;\underline{n})^{(2)}+\sum_{i=1}^{2m}\sum_{0<j_i<n_i}\bbk\partial_i^{p^{j_i}}$,}
\item{${K(2m+1;\underline{n})^{(1)}}_{[p]}=K(2m+1;\underline{n})^{(1)}+\sum_{i=1}^{2m+1}\sum_{0<j_i<n_i}\bbk\partial_i^{p^{j_i}}$.}
\end{enumerate}\end{thm}

\begin{rem}With the construction as above, we will sometimes refer to $L_{[p]}$ as the \emph{$p$-closure} of $L$. \end{rem}

We note that any simple maximal subalgebra $L$ in $\g$ has to be restricted, where $\g$ is an exceptional Lie algebra defined over an algebraically closed field of characteristic $p>0$. If $L$ is a maximal non-restricted subalgebra of $\g$ then $L_{[p]}\cong \g$ since $\g$ is restricted and $L$ is maximal. Since all Lie algebra $L$ are ideals in $L_{[p]}$, it follows that $L$ is an ideal of $\g$ --- a contradiction to $\g$ being simple.

\begin{coro}\cite[Corollary 7.2.3]{Strade04} The restricted simple Lie algebras of Cartan type are $W(m;\underline{1})$, $S(m;\underline{1})^{(1)}$, $H(2m;\underline{1})^{(2)}$, and $K(2m+1;\underline{1})^{(1)}$.\end{coro}

Given our classical simple Lie algebras are all restricted, we now have examples of non-restricted simple Lie algebras for the first time. In \tref{nonrestrictedwitt} we see a non-restricted subalgebra $L$ that plays a role in a maximal subalgebra for the exceptional Lie algebra of type $E_8$ in characteristic five. In this case we find that the $p$-closure is equal to the normaliser of $L$ in $\g$.

We become interested in derivations when we consider non-semisimple subalgebras, so state a result about the derivation algebra of the Cartan type Lie algebras.

\begin{thm}\cite[Theorem 7.1.2]{Strade04}, where we only consider $p>3$ for now.
\begin{enumerate}
\item{$\Der\,W(m;\underline{n})\cong W(m;\underline{n})\oplus \sum_{i=1}^m\sum_{0<j_i<n_i}\bbk\partial_i^{p^{j_i}}$,}
\item{$\Der\,{S(m;\underline{n})^{(1)}}\cong CS(m;\underline{n})^{(1)}\oplus\sum_{i=1}^m\sum_{0<j_i<n_i}\bbk\partial_i^{p^{j_i}},$}
\item{$\Der\,{H(2m;\underline{n})^{(2)}}\cong CH(2m;\underline{n})^{(2)}\oplus\sum_{i=1}^{2m}\sum_{0<j_i<n_i}\bbk\partial_i^{p^{j_i}}$,}
\item{$\Der\,{K(2m+1;\underline{n})^{(1)}}\cong K(2m+1;\underline{n})^{(1)}\oplus\sum_{i=1}^{2m+1}\sum_{0<j_i<n_i}\bbk\partial_i^{p^{j_i}}$.}
\end{enumerate}
\end{thm}

It is worth noting that we have only considered ``nice'' cases of Cartan type Lie algebras, all of them can be formed using differential forms. It is possible to change the differential form, and find slightly different algebras. We only encounter this in characteristic three, which can be found in \autoreft{sec:charthree}.

\section[Nilpotent Orbits]{Nilpotent Orbits}

Nilpotent orbits are of huge importance when trying to classify maximal subalgebras, since any element of $\g$ can be written as $x_S+x_N$ for semisimple $x_S$ and nilpotent $x_N$. Hence, classifying the maximal subalgebras containing nilpotent elements will go a long way in the attempt to completely classify maximal subalgebras.

In this section we will recall some well-known facts about nilpotent orbits, with our main reference \cite{Jan04}. For every nilpotent element $e \in \g$ we can form the orbit of $e$ under the adjoint action of $G$ on $\g$. Throughout, we will write $\g_e$\index{$\mathfrak{g}_e$, centraliser of $e$ in $\g$} to denote the centraliser of $e$ in our Lie algebra $\g$, and similarly $\mathfrak{n}_e$\index{$\mathfrak{n}_e$, normaliser of $e$ in $\g$} for the normaliser of $e$ in $\g$. For now, we assume $p$ is a good prime for $G$.

\begin{defs} A group $H$ is called \emph{unipotent} if all the elements of $H$ are unipotent. A nilpotent element $e \in \g$ is called \emph{distinguished} if $C_G(e)^0$ is a unipotent group where $C_G(e)^{0}$ is the connected component of the centraliser of $e$ in $G$. \end{defs}

For every nilpotent element $e \in \g$ there is a Levi subgroup $L$ such that $e$ is distinguished in the Lie algebra of $L$. For example, we take a maximal torus $T$ of $C_G(e)^{0}$ and consider the Levi subgroup of $L=C_G(T)$. It then follows from \cite[Lemma 4.6]{Jan04} that $e \in \Lie(L)$.

In fields of characteristic zero or $p\gg0$ we may use Jacobson--Morozov to associate an $\mfsl(2)$-triple to every nilpotent element. This observation is used to prove the Bala--Carter classification of nilpotent orbits for $p$ sufficiently large \cite{BaCar761, BaCar762}. This is extended to good characteristic using cocharacters as replacements for $\mfsl(2)$ triples --- see \cite{Pre03} for further details.

One of the key details in classifying nilpotent orbits in good characteristic is a Springer isomorphism, which gives a bijection between nilpotent orbits in $\g=\Lie(G)$ and unipotent subgroups in $G$.

\begin{defs}Let $e$ be a nilpotent element in $\g$ such that $e$ is distinguished in some Levi subgroup $L$. A cocharacter $\tau: \bbk^{\ast} \rightarrow G$ is \emph{associated} to $e$ if both $e \in \g(\tau,2)$ and $\im(\tau) \subseteq [L,L]$\index{$\tau$ is an associated cocharacter}.\end{defs}

These cocharacters induce $\bbz$-gradings on $\g$ with $\g=\bigoplus_i \g(\tau,i)$ such that \[[\g(\tau,i),\g(\tau,j)]\subseteq \g(\tau,i+j),\] and so we use cocharacters to help identify the isomorphism classes of our maximal subalgebras. This is achievable since a lot is known about the gradings of simple Lie algebras.

 To obtain this grading we follow \cite[\S6]{LT11}. We associate a weighted diagram to every cocharacter $\tau$. For each simple root $\al_i$ there exists $a_i \in \bbz$ such that $\tau(c)=c^{a_i}\al_i$ for all $c \in \bbk^{\ast}$, and write the diagram as the Dynkin diagram for $G$ with nodes labelled $a_i$ for simple roots $\al_i$ for all $i$. We can then calculate the weight of every $x \in \g$, just by looking at the action on the roots.

For example; consider the regular nilpotent element $e_{\al_1}+e_{\al_2}+e_{\al_3}+e_{\al_4}$ in $F_4$ with associated cocharacter $\tau$ given by $2\quad2\quad2\quad2$ from \cite[pg. 79]{LT11}. Then, we calculate the weight of $e_{1111}$ very simply as $2+2+2+2=8$. We say $e_{1111} \in \g(\tau,8)$.

This produces a well-defined grading on our Lie algebra $\g$ since taking the Lie bracket of elements in $\g$ the roots change accordingly. To illustrate this, continue to consider the regular nilpotent element of $F_4$ with associated cocharacter $\tau$ given by $2\quad2\quad2\quad2$. Then, consider $[e_{1111},e_{0010}]=\la e_{1121}$ --- the bracket should have weight $10$, and a quick check shows that $e_{1121}$ satisfies this.

The classification of nilpotent orbits in exceptional Lie algebras is well-known with \cite{LT11} giving the full list of orbits along with representatives, and the corresponding centraliser. From \cite{Pre03} every nilpotent element $e$ has at least one associated cocharacter, and any two associated cocharacters to $e$ are conjugate under $G$.

The second half of this thesis focuses on bad primes, and some of these results fail in bad characteristic. The articles \cite{stuhler, spa83, spa84, hs85} obtain complete lists of all the nilpotent orbits that have no analogue in good characteristic, called \emph{non-standard}, that are encountered in bad characteristic. We list these nilpotent orbits in the tables of centralisers for each orbit. An immediate corollary of all these results is the main result of \cite{hs85}.

\begin{thm}\cite[Theorem 1]{hs85} Let $G$ be an algebraic group of exceptional type with Lie algebra $\g=\Lie(G)$ over an algebraically closed field of characteristic $p>0$. Then, for all $p$, there are only finitely many nilpotent orbits in $\g$.\end{thm}

Many of the usual results on nilpotent orbits break down, mainly in these so-called non-standard nilpotent orbits. For example; the result that every nilpotent orbit has an associated cocharacter breaks down. We only consider standard nilpotent orbits in this work to find new maximal subalgebras. These are analogues of the list of nilpotent orbits in the exceptional Lie algebras $\g$ that arise from the characteristic zero case.

\begin{rem}For nilpotent $e \in \g$ we often say that $e$ lies in the same nilpotent orbit or has the same representative as in the good characteristic case. This should always be assumed to mean that the nilpotent element $e$ lies in the nilpotent orbit with the same label as the good characteristic case.\end{rem}

Most nilpotent orbits we consider have the same representative for all characteristics given in \cite{LT11} and \cite{S16}. In this case it follows from \cite[Theorem 5.2]{CP13} that standard nilpotent orbits still have the same associated cocharacters for all $p$. This is extremely helpful in bad characteristic as it is easier to provide the isomorphism class of a simple Lie algebra just by recognising the grading.

\subsection*{Centralisers of nilpotent elements}
We use the tables of \cite{S16} to give the dimension of the centralisers of nilpotent orbits in the exceptional Lie algebras for all primes $p$. In good characteristic we always have $\dim\,\g_e = \dim\, \Lie(G_e)$, we sometimes call such an orbit \emph{smooth}. The dimension of the centralisers in the good case is given in \cite[Table 2]{LT11}.

For $p$ bad this remains true most of the time, although there are some examples where $\dim\,\g_e>\dim\,\Lie(G_e)$, and these seem to be behind many of our examples of maximal subalgebras.
\renewcommand\arraystretch{1.35}
\begin{table}[H]\caption{Dimension of centralisers in $G_2$ of nilpotent orbits}\phantomsection\label{fig:centralisersg2}\centering\footnotesize
\begin{tabular}{ccc }
\hline Orbit & $p$ & $\dim(\g_e)$ \\ \hline
\multirow{3}{*}{$G_2$} & $\ge 5$&$2$\\&$3$ & $3$ \\ & $2$ & $4$ \\ \hline
\multirow{3}{*}{$G_2(a_1)$} &$\ge 5$ &$4$\\& $3$ & $5$ \\ & $2$ & $4$ \\ \hline
\multirow{1}{*}{$(\widetilde{A_1})^{(3)}$}& $3$ & $6$ \\ \hline
\multirow{3}{*}{$\widetilde{A_1}$} &$\ge 5$ &$6$\\& $3$ & $8$ \\ & $2$ & $8$ \\ \hline
\multirow{3}{*}{$A_1$}&$\ge 5$ &$8$\\  & $3$ & $8$ \\ & $2$ & $8$ \\ \hline\hline
\end{tabular}\end{table}

For nilpotent orbits in bad characteristic we are extremely fortunate for the work of \cite{VIG05}. This looks at the Jordan block structure of nilpotent elements in bad characteristic, with many errors in transcribing orbit representatives corrected in \cite{S16}. From this point on we will only refer to the correct nilpotent orbit representatives.

\begin{table}[H]\caption{Dimension of centralisers in $F_4$ of nilpotent orbits}\phantomsection\label{centralisersf4}\centering\footnotesize
\begin{tabular}[t]{ ccc }
\hline Orbit & $p$ & $\dim(\g_e)$ \\ \hline
\multirow{3}{*}{$F_4$}&$\ge 5$ &$4$\\ & $3$ & $6$ \\ & $2$ & $8$ \\ \hline
\multirow{3}{*}{$F_4(a_1)$} &$\ge 5$ &$6$\\& $3$ & $6$ \\ & $2$ & $10$ \\ \hline
\multirow{3}{*}{$F_4(a_2)$} &$\ge 5$ &$8$\\& $3$ & $8$ \\ & $2$ & $10$ \\ \hline
\multirow{1}{*}{$(C_3)^{(2)}$}& $2$ & $12$ \\ \hline
\multirow{3}{*}{$C_3$}&$\ge 5$ &$10$\\ & $3$ & $10$ \\ & $2$ & $16$ \\ \hline
\multirow{3}{*}{$B_3$} &$\ge 5$ &$10$\\ & $3$ & $10$ \\ & $2$ & $14$ \\ \hline
\multirow{3}{*}{$F_4(a_3)$} &$\ge 5$ &$12$\\ & $3$ & $12$ \\ & $2$ & $14$ \\ \hline
\multirow{1}{*}{$(C_3(a_1))^{(2)}$}& $2$ & $16$ \\ \hline
\multirow{3}{*}{$C_3(a_1)$} &$\ge 5$ &$14$\\ & $3$ & $14$ \\ & $2$ & $20$ \\ \hline
\multirow{1}{*}{$(A_1+\widetilde{A_2})^{(2)}$}& $2$ & $16$ \\ \hline
\multirow{3}{*}{$A_1+\widetilde{A_2}$} &$\ge 5$ &$16$\\ & $3$ & $18$ \\ & $2$ & $20$ \\ \hline\hline\end{tabular}\quad\begin{tabular}[t]{ ccc }
\hline Orbit & $p$ & $\dim(\g_e)$ \\ \hline
\multirow{1}{*}{$(B_2)^{(2)}$} & $2$ & $17$ \\ \hline
\multirow{3}{*}{$B_2$} &$\ge 5$ &$16$\\& $3$ & $16$ \\ & $2$ & $21$ \\ \hline
\multirow{3}{*}{$A_2+\widetilde{A_1}$}&$\ge 5$ &$18$\\ & $3$ & $18$ \\ & $2$ & $18$ \\ \hline
\multirow{3}{*}{$\widetilde{A_2}$} &$\ge 5$ &$22$\\& $3$ & $22$ \\ & $2$ & $28$ \\ \hline
\multirow{1}{*}{$(A_2)^{(2)}$}  & $2$ & $22$ \\ \hline
\multirow{3}{*}{$A_2$} &$\ge 5$ &$22$\\ & $3$ & $22$ \\ & $2$ & $22$ \\ \hline
\multirow{3}{*}{$A_1+\widetilde{A_1}$}&$\ge 5$ &$24$\\  & $3$ & $24$ \\ & $2$ & $28$ \\ \hline
\multirow{1}{*}{$(\widetilde{A_1})^{(2)}$}& $2$ & $31$ \\ \hline
\multirow{3}{*}{$\widetilde{A_1}$} &$\ge 5$ &$30$\\ & $3$ & $30$ \\ & $2$ & $36$ \\ \hline
\multirow{3}{*}{$A_1$}  &$\ge 5$ &$36$\\& $3$ & $36$ \\ & $2$ & $36$ \\ \hline\hline\end{tabular}\end{table}

By \cite[Theorem 1]{hs85}, the number of nilpotent orbits is always finite. It turns out that there are very few non-standard orbits. These are usually denoted by $(\mathcal{O})^{(p)}$ where $p$ is the prime they appear and $\mathcal{O}$ the type of the nilpotent orbit. For example, in $F_4$ when $p=3$ there are no non-standard orbits. However, for $p=2$ there are $6$ non-standard orbits by \cite{spa84} which looked at nilpotent orbits in $F_4$ for $p=2$.

The knock-on effect for these orbits is that we do not have associated cocharacters, it can be the case that none actually exist. We do not consider these orbits, but must be careful that none of our orbit representative choices differ from \cite{S16}. It is crucial that we use these orbits otherwise we cannot be guaranteed of having an associated cocharacter $\tau$.

\begin{rem}Later on we will see a case which uses an orbit representative arising from the good characteristic case in \cite{LT11} that differs from the representative in \cite{S16}. For this, we check that it remains in the same orbit.\end{rem}

\begin{table}[H]\caption{Dimension of centralisers in $E_6$ of nilpotent orbits}\phantomsection\label{centraliserse6}\centering\footnotesize
\scalebox{0.95}{\begin{tabular}[t]{ ccc }
\hline Orbit & $p$ & $\dim(\g_e)$ \\ \hline
\multirow{3}{*}{$E_6$} &$\ge 5$ &$6$\\& $3$ & $9$ \\ & $2$ & $8$ \\ \hline
\multirow{3}{*}{$E_6(a_1)$} &$\ge 5$ &$8$\\& $3$ & $9$ \\ & $2$ & $8$ \\ \hline
\multirow{3}{*}{$D_5$}&$\ge 5$ &$10$\\ & $3$ & $10$ \\ & $2$ & $12$ \\ \hline
\multirow{3}{*}{$E_6(a_3)$} &$\ge 5$ &$12$\\& $3$ & $13$ \\ & $2$ & $12$ \\ \hline
\multirow{3}{*}{$D_5(a_1)$} &$\ge 5$ &$14$\\ & $3$ & $14$ \\ & $2$ & $14$ \\ \hline
\multirow{3}{*}{$A_5$}&$\ge 5$ &$14$\\  & $3$ & $15$ \\ & $2$ & $16$ \\ \hline
\multirow{3}{*}{$A_4+A_1$} &$\ge 5$ &$16$\\ & $3$ & $16$ \\ & $2$ & $16$ \\ \hline
\multirow{3}{*}{$D_4$} &$\ge 5$ &$18$\\ & $3$ & $18$ \\ & $2$ & $20$ \\ \hline
\multirow{3}{*}{$A_4$}&$\ge 5$ &$18$\\ & $3$ & $18$ \\ & $2$ & $18$ \\ \hline
\multirow{3}{*}{$D_4(a_1)$} &$\ge 5$ &$20$\\& $3$ & $20$ \\ & $2$ & $20$ \\ \hline\hline\end{tabular}\quad\begin{tabular}[t]{ccc}
\hline Orbit & $p$ & $\dim(\g_e)$ \\ \hline
\multirow{3}{*}{$A_3+A_1$} &$\ge 5$ &$22$\\& $3$ & $22$ \\ & $2$ & $24$ \\ \hline
\multirow{3}{*}{${A_2}^{2}+A_1$} &$\ge 5$ &$24$\\ & $3$ & $27$ \\ & $2$ & $24$ \\ \hline
\multirow{3}{*}{$A_3$}&$\ge 5$ &$26$\\  & $3$ & $26$ \\ & $2$ & $26$ \\ \hline
\multirow{3}{*}{$A_2+{A_1}^{2}$}&$\ge 5$ &$28$\\  & $3$ & $28$ \\ & $2$ & $28$ \\ \hline
\multirow{3}{*}{${A_2}^{2}$} &$\ge 5$ &$30$\\ & $3$ & $31$ \\ & $2$ & $30$ \\ \hline
\multirow{3}{*}{$A_2+A_1$}&$\ge 5$ &$32$\\  & $3$ & $32$ \\ & $2$ & $32$ \\ \hline
\multirow{3}{*}{$A_2$} &$\ge 5$ &$36$\\ & $3$ & $36$ \\ & $2$ & $36$ \\ \hline
\multirow{3}{*}{${A_1}^{3}$}&$\ge 5$ &$38$\\  & $3$ & $38$ \\ & $2$ & $40$ \\ \hline
\multirow{3}{*}{${A_1}^{2}$} &$\ge 5$ &$46$\\ & $3$ & $46$ \\ & $2$ & $46$ \\ \hline
\multirow{3}{*}{$A_1$}&$\ge 5$ &$56$\\  & $3$ & $56$ \\ & $2$ & $56$ \\ \hline\hline\end{tabular}}\end{table}
The exceptional Lie algebra of type $E_6$ is the only case that has no non-standard nilpotent orbits. This may make the case of $E_6$ in $p=2$ and $p=3$ slightly easier than the other exceptional Lie algebras in terms of a full classification of maximal subalgebras.

Our focus is on standard nilpotent orbits where the dimension of the centraliser changes. This produces some very interesting cases of maximal subalgebras in exceptional Lie algebras.

It is worth noting that most of the non-standard orbits only appear for $p=2$, in fact there are only $2$ non-standard orbits for $p>2$. Namely, $(\widetilde{A_1})^{(3)}$ in $G_2$ and $(A_7)^{(3)}$ in $E_8$.

\begin{table}[H]  \caption{Dimension of centralisers in $E_7$ of nilpotent orbits}\phantomsection\label{centraliserse7}\centering\footnotesize
\begin{tabular}[t]{ ccc }
\hline Orbit & $p$ & $\dim(\g_e)$ \\ \hline
\multirow{3}{*}{$E_7$} &$\ge 5$ &$7$\\& $3$ & $9$ \\ & $2$ & $14$ \\ \hline
\multirow{3}{*}{$E_7(a_1)$}&$\ge 5$ &$9$\\ & $3$ & $9$ \\ & $2$ & $14$ \\ \hline
\multirow{3}{*}{$E_7(a_2)$} &$\ge 5$ &$11$\\& $3$ & $11$ \\ & $2$ & $14$ \\ \hline
\multirow{3}{*}{$E_7(a_3)$}&$\ge 5$ &$13$\\ & $3$ & $13$ \\ & $2$ & $14$ \\ \hline
\multirow{3}{*}{$E_6$} &$\ge 5$ &$13$\\ & $3$ & $15$ \\ & $2$ & $15$ \\ \hline
\multirow{3}{*}{$E_6(a_1)$} &$\ge 5$ &$15$\\ & $3$ & $15$ \\ & $2$ & $15$ \\ \hline
\multirow{3}{*}{$D_6$}&$\ge 5$ &$15$\\  & $3$ & $15$ \\ & $2$ & $22$ \\ \hline
\multirow{3}{*}{$E_7(a_4)$}&$\ge 5$ &$17$\\  & $3$ & $17$ \\ & $2$ & $22$ \\ \hline
\hline\end{tabular}\quad\begin{tabular}[t]{ccc}
\hline Orbit & $p$ & $\dim(\g_e)$ \\ \hline
\multirow{3}{*}{$D_6(a_1)$}&$\ge 5$ &$19$\\ & $3$ & $19$ \\ & $2$ & $22$ \\ \hline
\multirow{3}{*}{$D_5+A_1$} &$\ge 5$ &$19$\\& $3$ & $19$ \\ & $2$ & $22$ \\\hline
\multirow{1}{*}{$(A_6)^{(2)}$} & $2$ & $22$ \\\hline
\multirow{3}{*}{$A_6$} &$\ge 5$ &$19$\\& $3$ & $19$ \\ & $2$ & $23$ \\ \hline
\multirow{3}{*}{$E_7(a_5)$} &$\ge 5$ &$21$\\& $3$ & $21$ \\ & $2$ & $22$ \\ \hline
\multirow{3}{*}{$D_5$}&$\ge 5$ &$21$\\ & $3$ & $21$ \\ & $2$ & $23$ \\ \hline
\multirow{3}{*}{$E_6(a_3)$}&$\ge 5$ &$23$\\ & $3$ & $23$ \\ & $2$ & $23$ \\ \hline
\multirow{3}{*}{$D_6(a_2)$}&$\ge 5$ &$23$\\ & $3$ & $23$ \\ & $2$ & $26$ \\ \hline
\multirow{3}{*}{$D_5(a_1)+A_1$}&$\ge 5$ &$25$\\ & $3$ & $25$ \\ & $2$ & $29$ \\ \hline
\hline\end{tabular}\end{table}

\begin{table}[H] \centering\footnotesize
\begin{tabular}[t]{ ccc }
\hline Orbit & $p$ & $\dim(\g_e)$ \\ \hline
\multirow{3}{*}{$A_5+A_1$} &$\ge 5$ &$25$\\& $3$ & $27$ \\ & $2$ & $26$ \\ \hline
\multirow{3}{*}{$(A_5)'$}&$\ge 5$ &$25$\\  & $3$ & $25$ \\ & $2$ & $27$ \\ \hline
\multirow{3}{*}{$A_4+A_2$} &$\ge 5$ &$27$\\ & $3$ & $27$ \\ & $2$ & $27$ \\ \hline
\multirow{3}{*}{$D_5(a_1)$} &$\ge 5$ &$27$\\ & $3$ & $27$ \\ & $2$ & $27$ \\ \hline
\multirow{3}{*}{$A_4+A_1$}&$\ge 5$ &$29$\\  & $3$ & $29$ \\ & $2$ & $29$ \\ \hline
\multirow{3}{*}{$D_4+A_1$} &$\ge 5$ &$31$\\& $3$ & $31$ \\ & $2$ & $38$ \\ \hline
\multirow{3}{*}{$(A_5)''$} &$\ge 5$ &$31$\\& $3$ & $31$ \\ & $2$ & $32$ \\\hline
\multirow{3}{*}{$A_3+A_2+A_1$}&$\ge 5$ &$33$\\ & $3$ & $33$ \\ & $2$ & $38$ \\ \hline
\hline\end{tabular}\quad\begin{tabular}[t]{ccc}\hline Orbit & $p$ & $\dim(\g_e)$ \\ \hline
\multirow{3}{*}{$A_4$}&$\ge 5$ &$33$\\ & $3$ & $33$ \\ & $2$ & $33$ \\ \hline
\multirow{1}{*}{$(A_3+A_2)^{(2)}$} & $2$ & $38$\\ \hline
\multirow{3}{*}{$A_3+A_2$}&$\ge 5$ &$35$\\ & $3$ & $35$ \\ & $2$ & $39$ \\ \hline
\multirow{3}{*}{$D_4(a_1)+A_1$}&$\ge 5$ &$37$\\ & $3$ & $37$ \\ & $2$ & $38$ \\ \hline
\multirow{3}{*}{$D_4$} &$\ge 5$ &$37$\\& $3$ & $37$ \\ & $2$ & $39$ \\ \hline
\multirow{3}{*}{$A_3+{A_1}^{2}$}&$\ge 5$ &$39$\\ & $3$ & $39$ \\ & $2$ & $42$ \\ \hline
\multirow{3}{*}{$D_4(a_1)$} &$\ge 5$ &$39$\\& $3$ & $39$ \\ & $2$ & $39$ \\ \hline
\multirow{3}{*}{$(A_3+A_1)'$}&$\ge 5$ &$41$\\  & $3$ & $41$ \\ & $2$ & $43$ \\ \hline
\multirow{3}{*}{${A_2}^{2}A_1$}&$\ge 5$ &$43$\\  & $3$ & $45$ \\ & $2$ & $43$ \\ \hline
\hline\end{tabular}\end{table}

\begin{table}[H]\centering\footnotesize
\begin{tabular}[t]{ccc}
\hline Orbit & $p$ & $\dim(\g_e)$ \\ \hline
\multirow{3}{*}{$(A_3+A_1)''$}&$\ge 5$ &$47$\\  & $3$ & $47$ \\ & $2$ & $48$ \\ \hline
\multirow{3}{*}{$A_2+{A_1}^{3}$}&$\ge 5$ &$49$\\  & $3$ & $49$ \\ & $2$ & $50$ \\ \hline
\multirow{3}{*}{${A_2}^{2}$}&$\ge 5$ &$49$\\ & $3$ & $49$ \\ & $2$ & $49$ \\ \hline
\multirow{3}{*}{$A_3$}&$\ge 5$ &$49$\\ & $3$ & $49$ \\ & $2$ & $49$ \\ \hline
\multirow{3}{*}{$A_2+{A_1}^{2}$} &$\ge 5$ &$51$\\& $3$ & $51$ \\ & $2$ & $51$ \\ \hline
\multirow{3}{*}{$A_2+A_1$} &$\ge 5$ &$57$\\& $3$ & $57$ \\ & $2$ & $57$ \\ \hline\hline\end{tabular}\quad\begin{tabular}[t]{ccc}\hline Orbit & $p$ & $\dim(\g_e)$ \\ \hline
\multirow{3}{*}{${A_1}^{4}$}&$\ge 5$ &$63$\\  & $3$ & $63$ \\ & $2$ & $70$ \\ \hline
\multirow{3}{*}{$A_2$}  &$\ge 5$ &$67$\\ & $3$ & $67$ \\ & $2$ & $67$ \\ \hline
\multirow{3}{*}{$({A_1}^{3})'$} &$\ge 5$ &$69$\\  & $3$ & $69$ \\ & $2$ & $71$ \\ \hline
\multirow{3}{*}{$({A_1}^{3})''$}  &$\ge 5$ &$79$\\ & $3$ & $79$ \\ & $2$ & $80$ \\ \hline
\multirow{3}{*}{${A_1}^{2}$} &$\ge 5$ &$81$\\ & $3$ & $81$ \\ & $2$ & $81$ \\ \hline
\multirow{3}{*}{$A_1$} &$\ge 5$ &$99$\\ & $3$ & $99$ \\ & $2$ & $99$ \\ \hline \hline\end{tabular}
\end{table}

\begin{table}[H] \caption{Dimension of centralisers in $E_8$ of nilpotent orbits}\phantomsection\label{fig:centralisers}\centering\footnotesize
\begin{tabular}[t]{ ccc }
\hline Orbit & $p$ & $\dim(\g_e)$ \\ \hline
\multirow{4}{*}{$E_8$} & $\ge 7$ & $8$ \\& $5$ & $10$ \\ & $3$ & $12$ \\ & $2$ & $16$ \\ \hline
\multirow{4}{*}{$E_8(a_1)$} & $\ge 7$ & $10$ \\& $5$ & $10$ \\ & $3$ & $12$ \\ & $2$ & $16$ \\ \hline
\multirow{4}{*}{$E_8(a_2)$} & $\ge 7$ & $12$ \\& $5$ & $12$ \\ & $3$ & $12$ \\ & $2$ & $16$ \\ \hline
\multirow{4}{*}{$E_8(a_3)$} & $\ge 7$ & $14$ \\& $5$ & $14$ \\ & $3$ & $16$ \\ & $2$ & $16$ \\ \hline
\multirow{4}{*}{$E_8(a_4)$} & $\ge 7$ & $16$ \\& $5$ & $16$ \\ & $3$ & $16$ \\ & $2$ & $16$ \\ \hline
\multirow{4}{*}{$E_7$} & $\ge 7$ & $16$ \\&$5$ & $16$ \\ & $3$ & $18$ \\ & $2$ & $24$ \\ \hline
\multirow{4}{*}{$E_8(b_4)$} & $\ge 7$ & $18$ \\& $5$ & $18$ \\ & $3$ & $18$ \\ & $2$ & $24$ \\ \hline\hline\end{tabular}\quad\begin{tabular}[t]{ ccc}
\hline  Orbit & $p$ & $\dim(\g_e)$ \\\hline
\multirow{4}{*}{$E_8(a_5)$} & $\ge 7$ & $20$ \\& $5$ & $20$ \\ & $3$ & $20$ \\ & $2$ & $24$ \\ \hline
\multirow{4}{*}{$E_7(a_1)$} & $\ge 7$ & $20$ \\& $5$ & $20$ \\ & $3$ & $20$ \\ & $2$ & $24$ \\ \hline
\multirow{4}{*}{$E_8(b_5)$} & $\ge 7$ & $22$ \\& $5$ & $22$ \\ & $3$ & $22$ \\ & $2$ & $24$ \\ \hline
{$(D_7)^{(2)}$} & $2$ & $24$ \\ \hline
\multirow{4}{*}{$D_7$}& $\ge 7$ & $22$ \\ & $5$ & $22$ \\ & $3$ & $22$ \\ & $2$ & $32$ \\ \hline
\multirow{4}{*}{$E_8(a_6)$} & $\ge 7$ & $24$ \\& $5$ & $24$ \\ & $3$ & $24$ \\ & $2$ & $24$ \\ \hline
\multirow{4}{*}{$E_7(a_2)$} & $\ge 7$ & $24$ \\& $5$ & $24$ \\ & $3$ & $24$ \\ & $2$ & $28$ \\ \hline
\multirow{4}{*}{$E_6 + A_1$} & $\ge 7$ & $26$ \\& $5$ & $26$ \\ & $3$ & $30$ \\ & $2$ & $28$ \\ \hline\hline
\end{tabular}\end{table}

\begin{table}[H]\centering\footnotesize
\begin{tabular}[t]{ ccc}
\hline  Orbit & $p$ & $\dim(\g_e)$ \\
\hline{$(D_7(a_1))^{(2)}$} & $2$ & $28$ \\ \hline
\multirow{4}{*}{$D_7(a_1)$} & $\ge 7$ & $26$ \\& $5$ & $26$ \\ & $3$ & $26$ \\ & $2$ & $32$ \\ \hline
\multirow{4}{*}{$E_8(b_6)$} & $\ge 7$ & $28$ \\& $5$ & $28$ \\ & $3$ & $34$ \\ & $2$ & $30$ \\ \hline
\multirow{4}{*}{$E_7(a_3)$} & $\ge 7$ & $28$ \\& $5$ & $28$ \\ & $3$ & $28$ \\ & $2$ & $28$ \\ \hline
\multirow{4}{*}{$E_6(a_1) +A_1$} & $\ge 7$ & $30$ \\& $5$ & $30$ \\ & $3$ & $30$ \\ & $2$ & $30$ \\
\hline {$(A_7)^{(3)}$} & $3$ & $30$ \\ \hline
\multirow{4}{*}{$A_7$} & $\ge 7$ & $30$ \\& $5$ & $30$ \\ & $3$ & $34$ \\ & $2$ & $32$ \\ \hline
\multirow{4}{*}{$D_7(a_2)$} & $\ge 7$ & $32$ \\& $5$ & $32$ \\ & $3$ & $32$ \\ & $2$ & $32$ \\ \hline\hline\end{tabular}\quad\begin{tabular}[t]{ ccc}
\hline Orbit & $p$ & $\dim(\g_e)$ \\ \hline
\multirow{4}{*}{$E_6$} & $\ge 7$ &$32$\\& $5$ & $32$ \\ & $3$ & $34$ \\ & $2$ & $34$ \\ \hline
\multirow{4}{*}{$D_6$} & $\ge 7$ &$32$\\& $5$ & $32$ \\ & $3$ & $32$ \\ & $2$ & $40$ \\
\hline{$(D_5+ A_2)^{(2)}$} & $2$ & $40$  \\ \hline
\multirow{4}{*}{$D_5+A_2$} & $\ge 7$ &$34$\\& $5$ & $34$ \\ & $3$ & $34$ \\ & $2$ & $40$ \\ \hline
\multirow{4}{*}{$E_6(a_1)$} & $\ge 7$ &$34$\\& $5$ & $34$ \\ & $3$ & $34$ \\ & $2$ & $34$ \\ \hline
\multirow{4}{*}{$E_7(a_4)$} & $\ge 7$ &$36$\\& $5$ & $36$ \\ & $3$ & $36$ \\ & $2$ & $40$ \\ \hline
\multirow{4}{*}{$A_6 + A_1$} & $\ge 7$ &$36$\\& $5$ & $36$ \\ & $3$ & $36$ \\ & $2$ & $40$ \\ \hline
\multirow{4}{*}{$D_6(a_1)$} & $\ge 7$ &$38$\\& $5$ & $38$ \\ & $3$ & $38$ \\ & $2$ & $40$ \\\hline\hline\end{tabular}\end{table}

\begin{table}[H]\centering\footnotesize\begin{tabular}[t]{ccc}
\hline Orbit & $p$ & $\dim(\g_e)$ \\
\hline{$(A_6)^{(2)}$} & $2$ & $40$ \\ \hline
\multirow{4}{*}{$A_6$} & $\ge 7$ &$38$\\& $5$ & $38$ \\ & $3$ & $38$ \\ & $2$ & $42$ \\ \hline
\multirow{4}{*}{$E_8(a_7)$} & $\ge 7$ &$40$\\& $5$ & $40$ \\ & $3$ & $40$ \\ & $2$ & $40$ \\ \hline
\multirow{4}{*}{$D_5 + A_1$} & $\ge 7$ &$40$\\& $5$ & $40$ \\ & $3$ & $40$ \\ & $2$ & $44$ \\ \hline
\multirow{4}{*}{$E_7(a_5)$} & $\ge 7$ &$42$\\& $5$ & $42$ \\ & $3$ & $42$ \\ & $2$ & $44$ \\ \hline
\multirow{4}{*}{$E_6(a_3)+A_1$} & $\ge 7$ &$44$\\& $5$ & $44$ \\ & $3$ & $46$ \\ & $2$ & $44$ \\ \hline
\multirow{4}{*}{$D_6(a_2)$} & $\ge 7$ &$44$\\& $5$ & $44$ \\ & $3$ & $44$ \\ & $2$ & $48$ \\ \hline
\multirow{4}{*}{$D_5(a_1)+A_2$} & $\ge 7$ &$46$\\& $5$ & $46$ \\ & $3$ & $46$ \\ & $2$ & $48$ \\ \hline\hline\end{tabular} \quad\begin{tabular}[t]{ ccc}
\hline Orbit & $p$ & $\dim(\g_e)$ \\ \hline
\multirow{4}{*}{$A_5+A_1$} & $\ge 7$ &$46$\\& $5$ & $46$ \\ & $3$ & $48$ \\ & $2$ & $48$ \\ \hline
\multirow{4}{*}{$A_4+A_3$} & $\ge 7$ &$48$\\& $5$ & $50$ \\ & $3$ & $48$ \\ & $2$ & $48$ \\ \hline
\multirow{4}{*} {$D_5$} & $\ge 7$ &$48$\\& $5$ & $48$\\ & $3$ & $48$ \\ & $2$ & $50$  \\ \hline
\multirow{4}{*}{$E_6(a_3)$} & $\ge 7$ &$50$\\& $5$ & $50$ \\ & $3$ & $50$ \\ & $2$ & $50$ \\
\hline {$(D_4+A_2)^{(2)}$} & $2$ & $52$ \\ \hline
\multirow{4}{*}{$D_4+A_2$} & $\ge 7$ &$50$\\& $5$ & $50$ \\ & $3$ & $50$ \\ & $2$ & $64$ \\ \hline
\multirow{4}{*}{$A_4+A_2+A_1$} & $\ge 7$ &$52$\\& $5$ & $52$ \\ & $3$ & $52$ \\ & $2$ & $52$ \\ \hline
\multirow{4}{*}{$D_5(a_1)+A_1$} & $\ge 7$ &$52$\\& $5$ & $52$\\ & $3$ & $52$ \\ & $2$ & $52$  \\ \hline\hline
\end{tabular}\end{table}

\begin{table}[H]\centering\footnotesize \begin{tabular}[t]{ ccc }
\hline Orbit & $p$ & $\dim(\g_e)$ \\ \hline
\multirow{4}{*}{$A_5$} & $\ge 7$ &$52$\\& $5$ & $52$ \\ & $3$ & $52$ \\ & $2$ & $54$ \\ \hline
\multirow{4}{*}{$A_4+A_2$} & $\ge 7$ &$54$\\& $5$ & $54$ \\ & $3$ & $54$ \\ & $2$ & $54$ \\ \hline
\multirow{4}{*}{$A_4+{A_1}^{2}$} & $\ge 7$ &$56$\\& $5$ & $56$ \\ & $3$ & $56$ \\ & $2$ & $56$ \\ \hline
\multirow{4}{*}{$D_5(a_1)$} & $\ge 7$ &$58$\\& $5$ & $58$ \\ & $3$ & $58$ \\ & $2$ & $58$ \\ \hline
\multirow{4}{*}{$2A_3$} & $\ge 7$ &$60$\\& $5$ & $60$ \\ & $3$ & $60$ \\ & $2$ & $64$ \\ \hline
\multirow{4}{*}{$A_4+A_1$} & $\ge 7$ &$60$\\& $5$ & $60$ \\& $3$ & $60$\\&$2$ & $60$\\ \hline
\multirow{4}{*}{$D_4(a_1)+A_2$} & $\ge 7$ &$64$\\& $5$ & $64$ \\ & $3$ & $64$ \\ & $2$ & $64$ \\ \hline\hline
\end{tabular}\quad\begin{tabular}[t]{ ccc}
\hline Orbit & $p$ & $\dim(\g_e)$ \\\hline
\multirow{4}{*}{$D_4+A_1$} & $\ge 7$ &$64$\\& $5$ & $64$ \\ & $3$ & $64$ \\ & $2$ & $72$ \\ \hline
\multirow{4}{*}{$A_3+A_2+A_1$} & $\ge 7$ &$66$\\& $5$ & $66$ \\ & $3$ & $66$ \\ & $2$ & $72$ \\ \hline
\multirow{4}{*}{$A_4$} & $\ge 7$ &$68$\\& $5$ & $68$ \\ & $3$ & $68$ \\ & $2$ & $68$ \\
\hline{$(A_3+A_2)^{(2)}$} & $2$ & $72$\\ \hline
\multirow{4}{*}{$A_3+A_2$} & $\ge 7$ &$70$\\& $5$ & $70$ \\ & $3$ & $70$ \\ & $2$ & $74$ \\ \hline
\multirow{4}{*}{$D_4(a_1)+A_1$} & $\ge 7$ &$72$\\& $5$ & $72$ \\ & $3$ & $72$ \\ & $2$ & $72$ \\ \hline
\multirow{4}{*}{$A_3+{A_1}^{2}$} & $\ge 7$ &$76$\\& $5$ & $76$ \\ & $3$ & $76$ \\ & $2$ & $80$ \\ \hline
\multirow{4}{*}{${A_2}^{2}+{A_1}^{2}$} & $\ge 7$ &$80$\\& $5$ & $80$ \\ & $3$ & $84$ \\ & $2$ & $80$ \\ \hline  \hline
\end{tabular}\end{table}

\begin{table}[H]\centering\footnotesize\begin{tabular}[t]{ ccc}
\hline Orbit & $p$ & $\dim(\g_e)$ \\ \hline
\multirow{4}{*}{$D_4$} & $\ge 7$ &$80$\\& $5$ & $80$ \\ & $3$ & $80$ \\ & $2$ & $82$ \\ \hline
\multirow{4}{*}{$D_4(a_1)$} & $\ge 7$ &$82$\\& $5$ & $82$ \\ & $3$ & $82$ \\ & $2$ & $82$ \\ \hline
\multirow{4}{*}{$A_3+A_1$} & $\ge 7$ &$84$\\& $5$ & $84$ \\ & $3$ & $84$ \\ & $2$ & $86$ \\ \hline
\multirow{4}{*}{${A_2}^{2}+A_1$}  & $\ge 7$ &$86$ \\& $5$ & $86$\\ & $3$ & $88$ \\ & $2$ & $86$  \\ \hline
\multirow{4}{*}{${A_2}^{2}$} & $\ge 7$ &$92$\\& $5$ & $92$ \\ & $3$ & $92$ \\ & $2$ & $92$ \\ \hline
\multirow{4}{*}{$A_2+{A_1}^{3}$} & $\ge 7$ &$94$\\& $5$ & $94$ \\ & $3$ & $94$ \\ & $2$ & $96$ \\ \hline
\multirow{4}{*}{$A_3$} & $\ge 7$ &$100$\\& $5$ & $100$ \\ & $3$ & $100$ \\ & $2$ & $100$ \\ \hline\hline \end{tabular}\quad\begin{tabular}[t]{ccc}
\hline Orbit &$p$&$\dim(\g_e)$\\ \hline
\multirow{4}{*}{$A_2+{A_1}^{2}$} & $\ge 7$ &$102$\\& $5$ & $102$ \\ & $3$ & $102$ \\ & $2$ & $102$ \\ \hline
\multirow{4}{*}{$A_2+A_1$} & $\ge 7$ &$112$\\& $5$ & $112$ \\ & $3$ & $112$ \\ & $2$ & $112$ \\ \hline
\multirow{4}{*} {${A_1}^{4}$} & $\ge 7$ &$120$\\& $5$ & $120$\\&$3$&$120$\\&$2$&$128$ \\ \hline
\multirow{4}{*}{$A_2$} & $\ge 7$ &$134$\\& $5$ & $134$ \\ & $3$ & $134$ \\ & $2$ & $134$ \\ \hline
\multirow{4}{*}{${A_1}^{3}$} & $\ge 7$ &$136$\\& $5$ & $136$ \\ & $3$ & $136$ \\ & $2$ & $138$ \\ \hline
\multirow{4}{*}{${A_1}^{2}$} & $\ge 7$ &$156$\\& $5$ & $156$ \\ & $3$ & $156$ \\ & $2$ & $156$ \\ \hline
\multirow{4}{*}{$A_1$} & $\ge 7$ &$190$\\& $5$ & $190$ \\ & $3$ & $190$ \\ & $2$ & $190$ \\ \hline\hline\end{tabular} \end{table}

\section[Block's Theorems]{Block's Theorems}\label{sec:blocktheorem}

We consider finite-dimensional Lie algebras, and give some results about their minimal ideals over fields of characteristic $p>0$. This is a small collection of the results due to Block in \cite{Blo69} using \cite[\S3.3]{Strade04} for our main reference.

The power of these results is that they hold for all $p>0$. This leaves us with the task of identifying minimal ideals in a Lie algebra to produce a description, or even the isomorphism class in some cases. This will be very useful in Chapters \ref{sec:Ermolaev} and \ref{sec:dreams} when we look at non-semisimple maximal subalgebras in exceptional Lie algebras.

\begin{thm}\cite[Corollary 3.3.4]{Strade04} Let $S$ be a finite-dimensional simple Lie algebra. Then \begin{align*}\Der(S \otimes \mathcal{O}(m;\underline{n}))&=(\Der(S)\otimes\mathcal{O}(m;\underline{n})) \rtimes (1_S \otimes \Der(\mathcal{O}(m;\underline{n})))\\&=(\Der(S)\otimes\mathcal{O}(m;\underline{n})) \rtimes (1_S \otimes W(m;\underline{n}))\end{align*}\end{thm}

This result may be restated in terms of any finite-dimensional Lie algebra $L$ with non-abelian minimal ideal $I$. Any non-abelian minimal ideal $I$ is $L$-simple, in the sense that the only $L$-invariant ideals of $I$ are the trivial ones.

\begin{thm}\label{blocktheorem2}\cite{Blo69} Let $L$ be a finite dimensional Lie algebra with non-abelian minimal ideal $I$. Let $J$ be a maximal ideal of $I$, and consider $S:=I/J$. We have that $I\cong S\otimes\mathcal{O}(m;\underline{n})$ for some $m \in \mathbb{N}$ and $\underline{n}\in\mathbb{N}^m$. Then, it follows that \[S\otimes\mathcal{O}(m;\underline{n})\subset L/\mathfrak{z}_L(I)\subset (\Der(S)\otimes\mathcal{O}(m;\underline{n})) \rtimes (1_S \otimes W(m;\underline{n}))\]as Lie algebras.\end{thm}

\begin{rem}Any minimal ideal $I$ has a unique maximal ideal $J$, such that $(I/J) \otimes \mathcal{O}(m;\underline{n})$ has maximal ideal given by $(I/J) \otimes \mathcal{O}(m;\underline{n})_{(1)}$ where $\mathcal{O}(m;\underline{n})_{(1)}$ is a maximal ideal in $\mathcal{O}(m;\underline{n})$. The maximal ideal $(I/J) \otimes \mathcal{O}(m;\underline{n})_{(1)}$ is also nilpotent, and by maximality it follows that it must be the nil radical of $(I/J) \otimes \mathcal{O}(m;\underline{n})$.\end{rem}

The final result we give is the obvious statement for semisimple Lie algebras. The \emph{socle} of a semisimple Lie algebra is defined as the direct sum of its minimal ideals. If $\soc(L)=\bigoplus_{j}I_j$, then the ideals $I_j$ are irreducible $L$-modules.

\begin{coro}\cite{Blo69} Let $L$ be a finite dimensional semisimple Lie algebra. Then there are simple Lie algebras $S_i$ and truncated polynomial rings $\mathcal{O}(m_i;\underline{1})$ such that \[\Soc(L)=\bigoplus_{i=1}^t S_i \otimes \mathcal{O}(m_i,\underline{1}),\] and \[\bigoplus_{i=1}^t S_i \otimes \mathcal{O}(m_i,\underline{1})\subset L\subset \bigoplus_{i=1}^t (\Der(S_i) \otimes \mathcal{O}(m_i,\underline{1})) \rtimes (1_{S_i} \otimes W(m_i;\underline{1})).\]\end{coro}

The main result for our purposes is \tref{blocktheorem2}. This is used in many of our examples of non-semisimple maximal subalgebras $M$ in the exceptional Lie algebras $\g$. We factor out by the solvable radical $A$ of $M$, defined as the maximal solvable ideal of $M$. Then, we find the minimal ideal $I$ of $M/A$, and use \tref{blocktheorem2} to show that we have $I\cong S\otimes\mathcal{O}(m;\underline{n}) \subseteq M/A \subseteq \Der(S \otimes \mathcal{O}(m;\underline{n}))$.

In some cases we can use these inclusions as Lie algebras to provide the isomorphism class of our maximal subalgebra $M$, as the dimension of $M/A$ will be equal to the dimension of $\Der(S \otimes \mathcal{O}(m;\underline{n}))$ where $S$ is some simple Lie algebra, $m \in \mathbb{N}$ and $\underline{n} \in \mathbb{N}^m$.

\renewcommand\arraystretch{1.15}
\section[Weisfeiler filtration]{Weisfeiler filtration}\label{sec:Wei}\index{Weisfeiler filtration}
We construct the Weisfeiler filtration from \cite{Wei78}, which will heavily feature in the non-semisimple maximal subalgebras. Let $\g$ be a finite-dimensional simple Lie algebra, with non-semisimple maximal subalgebra $M_{(0)}$.

For $(\g,M_{(0)})$ choose an $M_{(0)}$-invariant subspace $M_{(-1)} \supset M_{(0)}$ in $\g$ such that $M_{(-1)}/M_{(0)}$ is an irreducible $M_{(0)}$-module. Consider $\g/M_{(0)}$ as a $M_{(0)}$-module given by the action \[x \cdot (y+M_{(0)}):=[x,y]+M_{(0)}.\] This is clearly a well-define action, as we are using the Lie bracket in $\g$. By maximality of $M_{(0)}$, it must be the case that the subalgebra generated by $M_{(-1)}$ is equal to $\g$.

Define subspaces $M_{(-i)}$ of the filtration recursively \begin{align}\label{weisfeiler}\begin{split}
          M_{(-2)}&:=[M_{(-1)},M_{(-1)}]+M_{(-1)},\\
          M_{(-3)}&:=[M_{(-2)},M_{(-1)}]+M_{(-2)},\\
           \vdots\\   M_{(-k+1)}&:=[M_{(-k)},M_{(-1)}]+M_{(-k)}.\end{split}\end{align}

It follows by the maximality of $M_{(0)}$ that there exist a positive integer $q>0$ such that $M_{(-q)}=\g$. For the positive degree components consider $M_{(1)}:=\{x \in M_{(0)}:[x,M_{(-1)}] \subseteq M_{(0)}\}$. Then, define $M_{(i)}$ for $i \ge 2$ recursively as
\begin{align*}M_{(2)}&:=\{x \in M_{(1)}:[x,M_{(-1)}] \subseteq M_{(1)}\},\\
M_{(3)}&:=\{x \in M_{(2)}:[x,M_{(-1)}] \subseteq M_{(2)}\},\\
\vdots\\   M_{(r)}&:=\{x \in M_{(r-1)}:[x,M_{(-1)}] \subseteq M_{(r-1)}\}. \end{align*} There is an integer $t\ge 0$ such that $0=M_{(t+1)} \subsetneq M_{(t)}$, otherwise we would obtain an ideal in $\g$. Since $\g$ is assumed to be simple this cannot occur.
\begin{defs}\cite{Wei78} The \emph{Weisfeiler filtration} associated with the pair $(M_{(0)},M_{(-1)})$ is given by\begin{equation*}\g= M_{(-q)} \supset M_{(-q+1)}\supset \ldots \supset M_{(-1)} \supset M_{(0)}\supset M_{(1)} \supset \ldots \supset M_{(r)}\ne 0,\end{equation*} such that $[M_{(i)},M_{(j)}]\subseteq M_{(i+j)}$ for all $i,j$.\end{defs}

We have that $M_{(-1)}/M_{(0)}$ is irreducible as an $M_{(0)}$-module, and so $\nil(M_{(0)})\subseteq M_{(1)}$ where $\nil(M_{(0)})$ is the \emph{nilradical} of $M_{(0)}$ defined as the maximal ideal consisting of nilpotent elements in $M_{(0)}$. For $x \in M_{(1)}$ we have $(\ad\,x)^{q+r+1}(M_{(-q)})=0$. It now follows that $M_{(1)}$ is an ideal of $M_{(0)}$ consisting of nilpotent elements of $\g$ and so $\nil(M_{(0)})=M_{(1)}$.

We always assume $\nil(M_{(0)})\ne 0$ and so we always have $r>0$. We obtain a corresponding graded Lie algebra denoted $\mathcal{G}$\index{$\Gc$ or $\gr(\g)$, graded Lie algebra} from this filtration \begin{equation*}\Gc:=\gr(\g):=\bigoplus_{i=-q}^{r}\Gc_i\end{equation*} where $\Gc_i=M_{(i)}/M_{(i+1)}$, and multiplication defined in the obvious way with\begin{equation*}[x+M_{(i+1)},y+M_{(j+1)}]:=[x,y]+M_{(i+j-1)}.\end{equation*} By construction $\Gc_{-k}=[\Gc_{-1},\Gc_{-k+1}]$ for all $k\le -2$, and $\Gc_{-1}$ is $\Gc_0$-irreducible.

There will be examples with $\gr(\g)$ a simple Lie algebra itself, and in these cases we will give the isomorphism class. For exceptional Lie algebras in characteristic $p=2$ and $p=3$ our main use of this construction will be to determine maximality of some non-semisimple subalgebras.

To do this we take a non-semisimple subalgebra $L$ in exceptional Lie algebra $\g$ with non-zero nil radical $A$. We try to build a filtration of the exceptional Lie algebra $\g$ as above using $A=M_{(1)}$ and $L=M_{(0)}$. We find an $L$-invariant subspace denoted by $M_{(-1)}$ such that $M_{(-1)}/L$ is irreducible as a $L/A$-module with the subalgebra generated by $M_{(-1)}$ equal to all of $\g$ to give that $L$ is a maximal subalgebra of $\g$.

\begin{rem}\label{whymax}Maximality above follows since we are building a filtration, we may assume any maximal subalgebra $M$ containing $L$ must contain at least one element from $M_{(-1)}\setminus L$. If $M_{(-1)}$ is $L$-invariant and $M_{(-1)}/L$ is irreducible, then $M$ must contain all of $M_{(-1)}$. Hence, if $\langle M_{(-1)} \rangle = \g$ we must have that $M=\g$ as required. \end{rem}

One of the key results is the notion of the Weisfeiler radical and its properties. This was used in \cite{P15} to prove that the non-semisimple maximal subalgebras must be parabolic for good primes $p$ in the exceptional Lie algebras.

\begin{defs}The \emph{Weisfeiler radical}\index{$M(\Gc)$, the Weisfeiler radical} of $\Gc$ is defined as the sum of all ideals of $\Gc$ contained in $\bigoplus_{i \le -1}\Gc_i$ denoted $M(\Gc)$.\end{defs}

We give some key facts about the Weisfeiler radical that will be assumed throughout the remainder of this thesis.

\begin{enumerate}\item{$M(\Gc)$ is the largest graded ideal of $\Gc$ contained in $\Gc_{-}$, and so \begin{equation*}M(\Gc)=M(\Gc)_{-2}\oplus\ldots\oplus M(\Gc)_{-q}.\end{equation*} Since $\Gc_{-1}$ is an irreducible $\Gc_0$-module we have $(\Gc_{-1} \cap M(\Gc))=0$. If the intersection was non-zero, then by irreducibility $\Gc_{-1} \subseteq M(\Gc)$. Then $[\Gc_{-1},\Gc_1] \subseteq M(\Gc)$ --- a contradiction.}\item{$M(\Gc)$ is defined as the radical of $\Gc$, and $\bar{\Gc}=\Gc/M(\Gc)$ is a semisimple Lie algebra with unique minimal ideal usually denoted as $A(\bar{\Gc})$.}\end{enumerate}

The uniqueness of this ideal ensures that it must be graded in $\bar{\Gc}$. By Block's theorem we have that $A(\bar{\mathcal{G}}) \cong S \otimes \mathcal{O}(m;\underline{n})$ for simple Lie algebra $S$, $m \in \mathbb{N}$ and $\underline{n} \in \mathbb{N}^m$. Using \tref{blocktheorem2}, we obtain \[S \otimes \mathcal{O}(m;\underline{n})\subset \bar{\Gc} \subseteq (\Der(S)\otimes \mathcal{O}(m;\underline{n})) \rtimes (1_S \otimes W(m;\underline{n}))\] as Lie algebras. By \cite{Wei78} there are two possibilities for $A(\bar{\Gc})$, the so-called \emph{degenerate} case and \emph{non-degenerate} case. The difference between the degenerate and non-degenerate case is with the differing gradings it gives on $A(\bar{\Gc})$. We use \cite[\S3.4]{P15} as our main reference for this, which differs only in notation from the original in \cite{Wei78}.

Let $\Gc_{+}:=\oplus_{i\ge 0}\Gc_0$. For the \emph{degenerate} case, the grading comes from a grading on the truncated polynomial ring $\mathcal{O}(m;\underline{n})$. Suppose $A(\bar{\Gc})\cap \bar{\Gc}_{+}=0$, $\Gc_{i\ge 2}=0$ and $[[\Gc_{-1},\Gc_1],\Gc_1]=0$. To obtain this grading we fix $s \le m$ with the generators $x_1,\ldots,x_s$ of degree $-1$ and $x_{s+1},\ldots,x_m$ degree $0$. This gives $A_i(\bar{\Gc})=S \otimes \mathcal{O}(m;\underline{n})[i;s]$, where $\mathcal{O}(m;\underline{n})[i;s]$ is the span of all monomials $\Pi_{i=1}^m x_i^{(a_i)}$ with where $(a_i)\in \mathbb{N}^s$, $a_1+\ldots+a_s=-i$ and $0 \le a_i \le p^{n_i}-1$. Hence, $A_0(\bar{\Gc})=S \otimes \mathcal{O}(x_{s+1},\ldots,x_m)$. In this case we also have that $\Gc_1$ is a non-zero subspace of $\sum_{i=1}^s \mathcal{O}(x_{s+1},\ldots,x_m)\partial_i$.

In the \emph{non-degenerate} case our grading on the minimal ideal $A(\bar{\Gc})$ is induced from a grading on our simple Lie algebra $S$. We suppose that $A(\bar{\Gc})_1 \neq 0$ and $S_{\pm 1} \ne 0$, then the grading on $A(\bar{\Gc})$ is induced by the grading on $S$ with $A(\bar{\Gc})_i=S_i \otimes \mathcal{O}(m;\underline{n})$. Hence \[S= \bigoplus_{i=-q}^{r}S_i\] is non-trivially graded, and by \tref{blocktheorem2} we have \[S_0 \otimes \mathcal{O}(m;\underline{n}) \subset \bar{\Gc}_0 \subset (\Der_0(S)\otimes \mathcal{O}(m;\underline{n})) \rtimes (1_S \otimes W(m;\underline{n})).\]

For the main result of \cite{P15} to classify non-semisimple maximal subalgebras of exceptional Lie algebras we need to consider both cases, and rule them out case-by-case. We only need to be aware that these two cases may occur, so we only give the very basic definition of each case above. We refer the reader to \cite{Strade04, Wei78} where the complete results can be found on both cases.

In both \chref{sec:Ermolaev} and \chref{sec:dreams} we produce conjectures regarding some examples of maximal non-semisimple subalgebras in the exceptional Lie algebras in characteristic two and three with some observations about the Weisfeiler filtration and the corresponding graded Lie algebra. We provide some information as to whether we expect these to be \emph{degenerate} or \emph{non-degenerate} examples.

\section[Overview of results]{Overview of results}

\section*{\chref{sec:CartMax}}

Assume the characteristic is good, so $p \ge 5$ for all exceptional Lie algebras except for $E_8$ where we assume $p>5$. This will consist of a review of maximal subalgebras in the exceptional Lie algebras using \cite{HS14} and \cite{P15}. This completely deals with maximal subalgebras of Cartan type and non-semisimple maximal subalgebras in the exceptional Lie algebras for $p$ good.

This review will end with the complete classification of maximal subalgebras in exceptional Lie algebras using the recent arXiv article \cite{PremetStewart}. We use $G_2$ to illustrate the ideas of this result, and give the full classification in $G_2$ as follows:

\begin{thms}(See \tref{MainRes}) Let $G$ be an algebraic group of type $G_2$ over an algebraically closed field of good characteristic with Lie algebra $\g=\Lie(G)$. Suppose $L$ is a maximal subalgebra of $\g$, then one of the following holds unique up to conjugacy by $G$:\begin{enumerate} \item[(a)]{$\rad(L)\ne 0$ and $L$ is a parabolic subalgebra.}\item[(b)]{$\rad(L)=0$ and $L$ has maximal rank, then $L$ is either $\mathfrak{sl}(2)+\mathfrak{sl}(2)$ or $\mathfrak{sl}(3)$.}\item[(c)]{$\rad(L)=0$ and $L$ has rank one, then either $L \cong W(1;\underline{1})$ for $p=7$ or $L \cong \mfsl(2)$ when $p>7$.} \end{enumerate}\end{thms}

\section*{\chref{sec:newchapter}}
Then, we relax our assumption on $p$ in type $E_8$ letting $p=5$. We start with the maximal subalgebra $\w$ constructed in \cite[Theorem 4.2]{P15} as a non-semisimple maximal subalgebra. This gives rise to a Weisfeiler filtration $\mathcal{F}$ such that $\Gc=\gr(\mathcal{F})\cong S(3;\underline{1})^{(1)}$. We also provide the first example of a non-restricted Lie algebra appearing in a maximal subalgebra isomorphic to the $p$-closure of $W(1;2)$. This result also proves the uniqueness of this maximal subalgebra up to conjugation in $G$.

\begin{thms}(See \tref{nonrestrictedwitt}) Let $G$ be an algebraic group of type $E_8$ over an algebraically closed field of characteristic five with Lie algebra $\g=\Lie(G)$ and $e:=e_{\al_1}+e_{\al_2}+e_{\al_2+\al_4}+e_{\al_3+\al_4}+e_{\al_5}+e_{\al_6}+e_{\al_7}+e_{\al_8}$. For $f:=f_{\subalign{24&65432\\&3}}$, we have \[L:=\langle e,f \rangle\] is isomorphic to $W(1;2)$, and the $p$-closure $L_{[p]}$ is a maximal subalgebra of $\g$. Further, $L_{[p]}$ is unique up to conjugation by $G$.\end{thms}

\section*{\chref{sec:Ermolaev}}

This will be a consideration of algebraically closed fields $\bbk$ with characteristic three. We start with a brief exposition on the currently known simple Lie algebras in characteristic three along with information on their gradings.

Then, we focus on the article \cite{Pur16} showing that one of these new simple Lie algebras, namely the Ermolaev algebra, is a maximal subalgebra in the exceptional Lie algebra of type $F_4$.

\begin{thms}(See \cite[Theorem 1.1]{Pur16}) Let $G$ be an algebraic group of type $F_4$ over an algebraically closed field of characteristic three with Lie algebra $\g=\Lie(G)$. Let $e:=e_{1000}+e_{0100}+e_{0001}+e_{0120}$ be a representative for the nilpotent orbit denoted $F_4(a_1)$. For $f:=f_{1232}$, the subalgebra $L:=\langle e,f \rangle \cong \Er{1}$ is a maximal subalgebra of $\g$. \end{thms}

We consider nilpotent orbits where $e^{[p]}=0$ with $\g_e\ne\Lie(G_e)$. There are $5$ new examples of non-semisimple maximal subalgebras, each of these are counterexamples in bad characteristic to \cite[Theorem 1.1]{P15} where it is shown that for good primes $p$ any maximal non-semisimple subalgebra of an exceptional Lie algebra must be parabolic.

Two of these examples occur in $F_4$ and $E_6$ where we find Weisfeiler filtrations $\mathcal{F}_{\g}$ such that $\gr(\mathcal{F}_{\g})\cong S(3;\underline{1})^{(1)}$ or $\mathcal{S}_3(1,\omega_S)$ where $\mathcal{S}_3(n,\omega_S)$ is one of Skryabin's algebras \cite{Sk92} respectively.

The remaining three examples occur in $E_7$ and $E_8$, where we are able to completely describe two examples of maximal non-semisimple subalgebras, and give some initial remarks about the Weisfeiler filtration.

The final example will be the nilpotent orbit $\mathcal{O}({A_2}^2+{A_1}^2)$ in $E_8$, where $\dim\,\g_e=\dim\,\Lie(G_e)+4$. We can show that the quotient of this centraliser by its nilpotent radical lies between $H(4;\underline{1})^{(1)}$ and $H(4;\underline{1})$.

\section*{\chref{sec:dreams}}

This will be the final chapter, and focuses on characteristic two. There are three cases in $E_{6,7,8}$ for the nilpotent orbit denoted ${A_1}^3$ where we obtain a non-semisimple maximal subalgebra denoted by $M_n$ for each exceptional Lie algebra $E_n$ with $n \in \{6,7,8\}$.

\begin{thms}(See \tref{p2newmax}) Let $G$ be an algebraic group of type $E_n$ over an algebraically closed field of characteristic two with Lie algebra $\g=\Lie(G)$. For $e:=\rt{1}+\rt{4}+\rt{6} \in \mathcal{O}({A_1}^3)$ we have that $M_n=N_{\g}(A)$ is a maximal subalgebra of $\g$, where $A$ is the radical of $\mathfrak{n}_e$. \end{thms}

We are fortunate that the results of Block and Weisfeiler hold for $p=2$. Using Block's theorems we are able to identify ideals of the form $S \otimes \mathcal{O}(m;\underline{n})$, and using this we can provide information about $M(\Gc)$ and $\Gc/M(\Gc)$.

The nilpotent orbit denoted ${A_1}^4$ in the exceptional Lie algebras $E_7$ and $E_8$ provides us with our final examples of non-semisimple maximal subalgebras. For $E_8$ the result is very intriguing, we obtain a simple Lie algebra $L$ where $\dim \, L_{-1}=8$ and $L_0 \cong \mathfrak{sp}(8)$. This looks very similar to $M_{-1}$ and $M_0$ in the standard grading of the Hamiltonian Lie algebra, allowing us to show there is a simple subalgebra of $H(2n;\underline{1})^{(1)}$ at least for $n=3$ and $n=4$.

We find a strange simple maximal subalgebra $M$ in $E_8$ with $\dim \, M=124$. There appears to be at least two conjugacy classes arising from the nilpotent orbits $E_8(a_2)$ and $E_8(a_4)$.

\begin{thms}(See \tref{specialmax}) Let $G$ be an algebraic group of type $E_8$ over an algebraically closed field of characteristic two with Lie algebra $\g=\Lie(G)$. Let $e_1$ and $e_2$ be orbit representatives for $E_8(a_2)$ and $E_8(a_4)$ respectively, then there exists $f_1$ and $f_2$ such that $L_i:= \langle e_i,f_i \rangle$ is a maximal subalgebra of $\g$ with dimension $124$ for $i=1,2$.\end{thms}

The thesis finishes with a brief look at Lie superalgebras in characteristic two. We consider a case that produces a class of simple Lie algebras in \cite{KoLei92}. The dimension of these simple Lie algebras matches the dimension of the subalgebras of $H(2n;\underline{1})^{(1)}$ for $n=3$ and $n=4$. We end with a conjecture about the possible isomorphism between these Lie algebras.

%\end{document}

\lstset{language=C}
\tikzset{middlearrow/.style={decoration={markings, mark= at position 0.6 with {\arrow[scale=4]{#1}} ,}, postaction={decorate}}}
\chapter[Maximal subalgebras in fields of good characteristic]{Maximal subalgebras in the exceptional Lie algebras for good primes}\label{sec:CartMax}
There have been recent advances in maximal subalgebras of exceptional Lie algebras over fields of positive characteristic, with \cite{HS14} and \cite{P15} considering this task for good $p$. We consider both of these, using the exceptional Lie algebra of type $G_2$ to illustrate the results in more detail. We finish by considering the recent arXiv article by \cite{PremetStewart} which gives the complete classification of maximal subalgebras in the exceptional Lie algebras for good primes $p$.

\section[Maximal subgroups over the complex numbers]{Maximal subgroups over the field of complex numbers}

The classification of maximal subgroups in the simple Lie groups $G$ over the complex numbers from \cite{Dyn52} readily lifts to the list of maximal subalgebras in $\g=\Lie(G)$. Dynkin considered semisimple subalgebras, which splits into \emph{regular} and \emph{non-regular} subalgebras. By \emph{regular} we mean subalgebras $\h$ which contain a maximal toral subalgebra of $\g$, and we may use \cite{BS49} to obtain the regular semisimple subalgebras.

There is a nice description to obtain all such subalgebras, using root systems. Attach the negative of the highest root $\widetilde{\al}$ to the Dynkin diagram, which gives the \emph{extended Dynkin diagram}. Then, delete vertices corresponding to simple roots $\al_i$, and produce the root system of a regular semisimple subalgebra. In $G_2$ the extended Dynkin diagram is
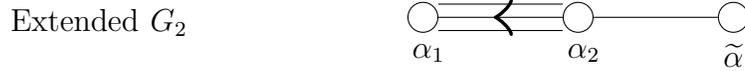
\tikzset{node distance=4em, ch/.style={circle,draw,on chain,inner sep=4pt},chj/.style={ch,join},every path/.style={shorten >=0pt,shorten <=0pt},line width=3pt,baseline=-1ex}
\begin{align*}
&\text{Extended}\,\, G_2&&\begin{tikzpicture}[start chain]
\dnode{1}
\dnodenj{2}
\hroot{}
\draw[middlearrow={<}]  ([yshift=0pt]chain-1.east) --([yshift=0pt]chain-2.west); \draw[-] ([yshift=-5pt]chain-1.east) edge([yshift=-5pt]chain-2.west); \draw[-] ([yshift=5pt]chain-1.east)edge([yshift=5pt]chain-2.west);
\end{tikzpicture}
\end{align*}

Deleting each vertex gives two non-isomorphic maximal subalgebras of maximal rank, with type $A_2$ and ${A_1}^2$. There is a small error in the initial result from \cite{Dyn52}, which allowed some non-maximal subalgebras to be included. This is easily fixed with the next result, stating that the simple roots must have prime coefficients in the highest root $\widetilde{\al}$. A proof can be found in \cite{GG78}.

\begin{prop}Suppose $\widetilde{\al}=\sum_{i} \la_{i}\al_i$ for simple roots $\al_i$, then provided the coefficient of the simple root $\al_i$ is either $1$ or a prime we obtain a maximal subalgebra.\end{prop}

This does not change the maximal subalgebras in the example of $G_2$, as $\widetilde{\al}=3\al_1+2\al_2$. However, we lose some cases in other exceptional Lie algebras as they have simple roots with `bad' coefficients. A good example is $F_4$, where $\widetilde{\al}=2\al_1+3\al_2+4\al_3+2\al_4$. Deleting $\al_3$ gives a semisimple subalgebra of type $A_3+A_1$. This is contained in a subalgebra of type $B_4$, a conjugate to the subalgebra obtained by deleting $\al_1$.

For good primes this result continues to hold, and gives the same regular semisimple maximal subalgebras. For bad prime this holds provided $\la_i \ne p$. In this case we obtain a non-semisimple maximal subalgebra. For example, in $E_8$ for $p=5$ deleting the root $\al_5$ gives a Lie algebra of type $A_4+A_4$, which has a non-trivial centre when $p=5$.

This leaves non-regular semisimple subalgebras. In $G_2$, only $\mfsl(2)$ is a possibility for rank reasons. It turns out the $\mfsl(2)$-triple corresponding to the regular nilpotent element $\rt{1}+\rt{2}$ is maximal, and unique up to conjugation. The other exceptional Lie algebras have maximal $\mfsl(2)$-triples corresponding to regular nilpotent orbits, but in $E_7$ and $E_8$ there is more than one conjugacy class. These correspond to some subregular orbits.

We give the list of the other semisimple maximal subalgebras, where the reader is referred to \cite{Dyn52} for full details on proving these are maximal subalgebras, \begin{enumerate}\item{In $F_4$ there is a maximal subalgebra of type $G_2+A_1$.}
\item{In $E_6$ we have types $G_2$, $C_4$, $G_2+A_2$ and $F_4$.}
\item{$E_7$ contains maximal subalgebras of types $A_1+A_1, A_2, G_2+A_1, G_2+C_3$ and $F_4+A_1$.}
\item{Finally, for $E_8$ there are maximal subalgebras of type $B_2$, $A_2+A_1$ and $G_2+F_4$.}\end{enumerate}

To complete the classification, the non-semisimple maximal subalgebras need to be dealt with. Over the complex numbers \cite{mor56} shows any non-semisimple maximal subgroup of $G$ must be a parabolic subgroup. Hence, any maximal subalgebra of $\g=\Lie(G)$ with non-zero radical must be $\Lie(P)$ for some parabolic subgroup $P$ of $G$.

For Lie algebras over algebraically closed field of prime characteristic the proof fails to work, and so we must try to find an analogous result. These subalgebras can be obtained using the simple roots, giving weight $1$ to $\al_i$ and weight $0$ to the remaining simple roots. Each of these contains a maximal torus of $\g$, and hence are also regular.

\section[The modular analogue of Morozov's theorem]{The modular analogue of Morozov's theorem}\label{sec:Morozovmodular}

Let $G$ be an algebraic group of exceptional type with Lie algebra $\g=\Lie(G)$. Assume $\bbk$ is an algebraically closed field of characteristic $p\ge 5$, unless $\g$ has type $E_8$ where we assume $p \ge 7$.We split the task of classifying maximal subalgebras $M$ into two cases; firstly we classify $M$ such that $\rad(M) \ne 0$, and then consider $\rad(M)=0$. The article \cite{P15} completes the first with the use of the Weisfeiler filtrations to provide an analogous result to \cite{mor56}.

For exceptional Lie algebras of type $G_2$ we will show any maximal subalgebra with non-zero radical is regular, and then show it is parabolic.
\begin{prop}
Let $G$ be an algebraic group of type $G_2$ with Lie algebra $\g=\Lie(G)$ for good $p$, and $M$ be a maximal subalgebra such that $\rad(M)\ne 0$. Then, the degenerate case of the Weisfeiler filtration does not hold in $\g$.
 \end{prop}
\begin{proof}Suppose the degenerate case of the Weisfeiler filtration holds, then $A(\bar{\Gc})=S \otimes O(m;\underline{n})$. It follows \[\dim\,A(\bar{\Gc})=\dim\, S\cdot p^m \ge 3\cdot 5^m,\] which forces $m=0$ in $G_2$. Hence $A(\bar{\Gc})=S$, but this says we have a non-graded simple Lie algebra isomorphic to a graded ideal --- a contradiction.
\end{proof}

Hence, we may assume the non-degenerate case holds with $A(\bar{\Gc})=S \otimes O(m;\underline{n})$. The grading on $A_i(\bar{\Gc})=S_i \otimes O(m;\underline{n})$ is inherited from the grading on $S$. By \tref{blocktheorem2}, $S \otimes O(m;\underline{n}) \subset \bar{\Gc} \subset (\Der(S)\otimes O(m;\underline{n})) \rtimes (1_S\otimes W(m;\underline{n}))$. By dimension reasons $m=0$ in $G_2$, and so $S \subset \bar{\Gc} \subset \Der(S)$.

\begin{defs}\cite[Definition 1.2.1]{Strade04} In a Lie algebra $L$ we define a \emph{torus} $T$ to be a subalgebra consisting of toral elements of $L$. Given a $p$-envelope $\mathcal{L}$ of $L$, the \emph{absolute toral rank} of $L$ is defined as the maximal dimension of tori in the restricted Lie algebra $\mathcal{L}/\mathfrak{z}(\mathcal{L})$ usually denoted $\TR(L).$\index{$\TR(L)$, the absolute toral rank of $L$} \end{defs}

\begin{thm}\cite[Theorem 5.1]{S98} If $L$ is a filtered Lie algebra, then $\TR(L)\ge \TR(\gr(L))$.\end{thm}

\begin{prop}\label{twodtor}Let $G$ be an algebraic group of type $G_2$ with Lie algebra $\g=\Lie(G)$ for $p$ good and $M$ a maximal subalgebra such that $\rad(M)\ne 0$. Then, $M$ is a regular subalgebra in $\g$.\end{prop}\begin{proof}
We have that $\TR(M) \ge \TR(\gr(M))$, and so $\TR(S) \le 2$. If $\TR(S)=1$, then $S$ is classical, $W(1,\underline{1})$, or $H(2;\underline{1})^{(2)}$.

If $\TR(S)=2$, then $S$ is classical, $W(2,\underline{1})$, $H(4,\underline{1})^{(2)}$, $S(3,\underline{1})^{(1)}$, $W(1,\underline{2})$, $K(3,\underline{1})$, $M(1,1)$, $H(2;\underline{1},\Phi)$, or $H(2,(2,1))^{(2)}$ by \cite{PSt01}.

Considering the dimension of these in \cite{Strade04}, we only have \[S \in \{\mfsl(2), \mfsl(3), \mfsp(4), W(1,\underline{1}), G_2\}.\] In particular $\ad(S)=\Der(S)$, and hence $\bar{\Gc}=S$. Further, the grading of $S$ is standard producing the cases \begin{enumerate}\item[(i)]{If $S$ is $\mfsl(2)$ or $\mfsl(3),$ then \[S=S_{-1}\oplus S_0 \oplus S_1.\]} \item[(ii)]{If S is $\mfsp(4),$ then \[S=S_{-2}\oplus S_{-1}\oplus S_0 \oplus S_1 \oplus S_2.\]} \item[(iii)]{Finally, if $S$ is of type $G_2,$ then $\bar{\Gc}=\Gc=G_2$.}\end{enumerate}

Hence, we may assume $S \in \{\mfsl(2), W(1;1)\}$ if $M$ is not regular, then for $S=W(1;1)$ we have the natural grading $W_{-1}+\ldots$, or its reverse. For the natural grading we have $S_{-1}:=\bbk \partial$, and hence $\bar{\Gc}_{-2}=[S_{-1},S_{-1}]=0$. For the reverse of the grading, $\dim\,S_{-1}=1$, but $S_{-2} \ne 0$ --- a contradiction. Hence, for $S=W(1;1)$ it must be the cases $M(\Gc)=0$. A similar argument shows $S \ne \mfsl(2)$.

In any case, $S_0$ contains a two-dimensional torus. It follows there are elements $t_i \in S_0$ such that $t_i-t_i^{[p]}\in \nil(M)$ for $i=1,2$. It follows that $(t_i-t_i^{[p]})^{[p]^N}=0$ for some $N$, and so $t_i^{[p]^N}$ is a toral element in $M$. Under the canonical homomorphism $M \rightarrow S_0$ this new element has the same image as $t$, and so $M$ is regular. \end{proof}

Now, we consider $G$ to be any algebraic group of exceptional type with Lie algebra $\g=\Lie(G)$ for good primes $p$. We use \cite[Lemma 2.4]{P15} to show any maximal subalgebra $M$ of $\g$ containing a maximal toral subalgebra $\mft$ is parabolic. For $p>3$ all maximal toral subalgebras are conjugate. In particular, they are classical Cartan subalgebras by \cite[Ch.II, Section 3]{Se67} with one-dimensional root spaces of $\g$ with respect to $\mft$.

\begin{lem}\cite[Lemma 2.4]{P15} Let $G$ be an algebraic group of exceptional type with Lie algebra $\g=Lie(G)$ for $p$ good. Let $M$ be a maximal subalgebra of $\g$ such that $\rad(M)\ne 0$. If $M$ is regular, then it is a parabolic subalgebra of $\g$.\end{lem}

\begin{proof}Let $\Phi$ be the root system of $G$ with respect to $T$, where $T$ is a maximal torus such that $\mft=\Lie(T)$. As usual we decompose $\g=\mft \oplus \sum_{\al \in \Phi}\g_{\al}$, where each root subspace $\g_{\al}:=\bbk e_{\al}$ is one-dimensional for $e_{\al} \in \g$. For $M$, there exists a closed subset $\Psi$ of $\Phi$ such that $M=\mft \oplus \sum_{\al\in\Psi}\bbk e_{\al}$ by \cite{Se67}.

It follows from above that we may consider $M$ as a subalgebra of $\g_{\bbz}$. We consider $\g_{\bbz}(\Psi)$, and since $M$ is maximal it follows that $\Psi$ is a maximal closed root subsystem. This implies that $\g_{\bbz}(\Psi) \otimes_{\bbz} \bbc$ is a maximal subalgebra of $\g_{\bbc}$. We may apply Morozov's theorem to conclude that $\Psi$ is either symmetric or parabolic. The restriction of the Killing form of $\g$ to $M$ is non-degenerate if $\Psi$ is symmetric, but since $\rad(M) \ne 0$ this cannot be the case. Hence, $\Psi$ is parabolic and $M$ is a parabolic subalgebra of $\g$. \end{proof}

This result together with \pref{twodtor} gives that any maximal subalgebra $M$ of $G_2$ such that $\rad(M)\ne 0$ is parabolic. This argument relies very heavily on the dimension of $G_2$ being small along with its rank, and so extending this idea is not straightforward.

In other exceptional Lie algebras, the dimensions are big enough to contain the degenerate case. This is where \cite{P15} provides the necessary details to consider all exceptional Lie algebras. The previous result allows us to assume any new non-semisimple maximal subalgebra is not regular.

We rule out the degenerate case, and for this we need two results about the nilpotent orbits in the exceptional Lie algebra $\g$.

\begin{lem}\cite[Lemma 2.2]{P15}\label{premlem} Let $G$ be an algebraic group of exceptional type with Lie algebra $\g=\Lie(G)$ for $p$ good, and $L$ be a Lie subalgebra of $\g$ such that $[L,L]$ consists of nilpotent elements of $\g$. Then $L$ is contained in a Borel subalgebra of $\g$. \end{lem}

\begin{coro}\cite[Corollary 2.3]{P15}\label{premcoro} Let $G$ be an algebraic group of exceptional type with Lie algebra $\g$ for $p$ good and $M$ be a maximal subalgebra of $\g$. Suppose $N:=\nil(M)\ne 0$. Let $R$ be any Lie subalgebra of $M$ whose derived ideal $[R,R]$ consists of nilpotent elements of $\g$. Then the centraliser $c_{\g}(N)$ is an ideal of $M$ and there exists $e \in c_{\g}(N) \cap \mathcal{O}_{\min}$ such that $[R,e] \subseteq \bbk e$, where $\mathcal{O}_{\min}$ consists of all nilpotent $e$ such that $[e,[e,\g]]=\bbk e$.\end{coro}

Using these results, it is possible to completely rule out the degenerate case of the Weisfeiler filtration.

\begin{prop}\cite[\S3.5,3.6 and 3.7]{P15} Let $G$ be an algebraic group of exceptional type with Lie algebra $\g=\Lie(G)$ over an algebraically closed field of good characteristic. If $M$ is a maximal subalgebra of $\g$ such that $\rad(M) \ne 0$, then the degenerate case of Weisfeiler's theorem does not hold. \end{prop}

\begin{proof}[Sketch of proof] Assume to the contrary that $A(\bar{\mathcal{G}}) \cong S \otimes O(m;\underline{n})$. Immediately this implies $m \ge 1$, as the grading on the minimal ideal arises from a grading on the polynomial ring. If $m=1$, then $\mathcal{G}_1=\bbk \partial_1$ is one-dimensional. Since $\mathcal{G}_2=0$, we observe $\mathcal{G}_1 \cong M_{(1)}=\nil(M)$.

Hence, $M \subseteq \mfn_{\g}(\bbk e)$ for some non-zero nilpotent element $e$ of $\g$. We know $\mfn_{\g}(\bbk e)$ is contained in a proper parabolic subalgebra for all $e$ when $p$ is very good, and so this case cannot occur. Since $p$ is good we immediately rule out $m>2$ by dimension reasons as $\dim\,\g/p^{m} > 3$. This forces $m=2$ and $n=(1,1)$.

In particular, for all cases $S \in \{W(1;1),\mfsl(2)\}$. The final step is to construct a subalgebra that satisfies the conditions of \lref{premlem}, and apply \corref{premcoro} to rule these out. This subalgebra produces non-zero nilpotent elements $e \in \g$ such that $(\ad\,e)^3=0$ --- a contradiction as we need $(\ad\,e)^{p-1} \ne 0$ for $p \ge 5$.\end{proof}

For the non-degenerate case, this is a prime example of why the classification of simple Lie algebras is so important. It allows a case-by-case check of all possible simple Lie algebras, ruling each out using the above results.

The hardest task is to deal with Hamiltonian Lie algebras. Although many are ruled out by dimension reasons, the remaining cases provide the most difficulty. Since $p \ge 7$ for $E_8$, and both $\dim\,M(1;1)=125=\dim\, K(3;\underline{1})$ for $p=5$ these two cases may only occur in Lie algebras of type $E_7$, and using that $\dim\,\g=133$ allows this to be done in \cite{P15}.

\begin{thm}\cite[Theorem 1.1]{P15} Let $G$ be a simple algebraic $\bbk$-group, where $p=\cha(\bbk)$ is a very good prime for $G$, and let $M$ be a maximal subalgebra of $\g=\Lie(G)$ with $\rad(M) \ne 0$. Then $M=\Lie(P)$ for some maximal parabolic subgroup $P$ of $G$.\end{thm}

We will see this result break down badly if $p$ is bad, where \cite[Theorem 4.2]{P15} provides the first of many examples to do this. We discuss this in greater detail at the end of the chapter. For now, we continue our discussion of maximal subalgebras in the exceptional Lie algebras for good primes.

\section[Maximal subalgebras of Cartan type]{Maximal subalgebras of Cartan type}\label{sec:ge0}

We consider the possibility of maximal subalgebras of Cartan type. The list of possibilities is restricted early on, since the dimension of them becomes large as $p$ grows. This allows us to make a list of possible maximal subalgebras of Cartan type.

\begin{table}[H]\centering\caption{The possible maximal subalgebras of Cartan type}\phantomsection\label{hi} \begin{tabular}{|l |c |c|} \hline $\h$ & $\dim\,\h$ & Possible $p$ \\ \hline \hline $W(1;1)$ & $p$ & $p < \dim\,\g$  \\\hline $W(1;2)$ & $p^2$ & $p \le 13$ \\\hline $W(1;3)$&$p^3$&$p=5$\\\hline $W(2;\underline{1})$&$2p^2$&$p \le 11$\\ \hline $H(2;\underline{1})^{(2)}$ & $p^2-2$& $p \le 13$  \\ \hline $H((2;\underline{1});\Phi(\tau))^{(1)}$ & $p^2-1$&$p\le 13$ \\ \hline $H((2;\underline{1});\Phi(1))$& $p^2$&$p \le 13$ \\ \hline $H(2;(1,2))^{(2)}$ & $p^3-2$ &$p=5$ \\ \hline $K(3;\underline{1})$ & $p^3$ &$p=5$  \\ \hline $M(1;1)$&$p^3$&$p=5$ \\ \hline\end{tabular} \end{table}

We begin with maximal subalgebras of Witt type, and outline some steps involved in \cite{HS14} to prove only $W(1;1)$ is a maximal subalgebra. All other subalgebras of Cartan type can be ruled out.

\begin{thm}\cite[Theorem 1.1]{HS14} Let $G$ be an algebraic group of exceptional type with Lie algebra $\g=\Lie(G)$ and $p=h+1$, where $h$ is the Coxeter number of $G$.

For regular nilpotent element $e=\sum_i e_{\al_i}$ and $f_{\tilde{\al}}$ where $\widetilde{\al}$ is the highest root, we have that the Lie algebra generated by these elements is isomorphic to $W(1;1)$.

Further, the subalgebra is maximal in $G_2, F_4, E_7$ and $E_8$. For $E_6$, there is a $W(1;1)$ subalgebra contained in a subalgebra of type $F_4$. These are the only cases of maximal $W(1;1)$ subalgebras. \end{thm}

\begin{rem}Note that for $G_2$, $F_4$, $E_6$, $E_7$, and $E_8$ we have that the Coxeter number $h$ is $6, 12, 12, 18$ and $30$ respectively. \end{rem}

To prove this result, we find all possible nilpotent elements $e$ where we may find subalgebras of type $W(1;1)$. Then, prove only those containing regular nilpotent elements can be maximal. This is achieved by finding non-zero fixed vectors for a generating set of $W(1;1)$.

\begin{prop}\label{wittg2}Let $G$ be an algebraic group of type $G_2$ with Lie algebra $\g$. For $e=e_{\al_1}+e_{\al_2}$, the subalgebra $\langle e,f_{\widetilde{\al}}\rangle \cong W(1;1)$ is a maximal subalgebra unique up to conjugation for $p=7$. Further, this is the only occurrence in exceptional Lie algebras of type $G_2$. \end{prop}

\begin{proof}We have that $W(1;1) \cong \langle f_{\al_1}+f_{\al_2},e_{\widetilde{\al}}\rangle$ is a subalgebra of $\g$ by \cite[Lemma 13]{Pre85}, and in our case we are simply looking at the same subalgebra with the negative and positive roots interchanged. Suppose $W:=W(1;1)$ was not maximal. Since it contains a regular nilpotent element it cannot be that $W$ lies in a regular semisimple subalgebra.

If $W$ is in a parabolic, then there is a map $W \rightarrow \mathfrak{p}/\rad(\mathfrak{p})$ for parabolic subalgebra $\mathfrak{p}$. In $G_2$ this forces $W(1;1) \cong \mathfrak{sl}(2)$ --- a contradiction. Hence, $W$ is a maximal subalgebra.

Consider the usual basis ${\{e_{i}: -1 \le i \le 5}\}$ for $W(1;1)$ where $e_{-1}:=e=e_{\al_1}+e_{\al_2}$ is a regular nilpotent element in $\g$. Choose $h \in \mathfrak{n}_e$ such that $[h,e_{-1}]=-e_{-1}$, and in $W(1;1)$ we also have $[e_0,e_{-1}]=-e_{-1}$.

We may use the cocharacter $\tau$ given by $2\quad2$ from \cite[pg. 73]{LT11}, and it follows that $h-e_0 \in \g_{e}(\tau,0)$. The same paper tells us that $\g_e(\tau,0)$ is trivial, and so $h=e_0$. Hence, our choice for $e_{-1}$ and $e_0$ are unique.

In $W(1;1)$ we have $[e_0,e_5]=5e_5$, thus in $\g$ we must have $e_5 \in \g(-2p+4)\oplus \g(4)=\g(10)\oplus \g(4)$ for $p=7$. This gives $e_5=f_{\widetilde{\al}}+\la e_{\al_1+\al_2}$ for some $\la \in \bbk$. We calculate $e_4=[e_{-1},e_5]=[e_{\al_1}+e_{\al_2},f_{\widetilde{\al}}+\la e_{\al_1+\al_2}]=a\,f_{3\al_1+\al_2}+b\,e_{2\al_1+\al_2}$. It is clear that $b=0$ if and only if $\la=0$, and hence \[[e_{4},e_5]=[a\,f_{3\al_1+\al_2}+b\,e_{2\al_1+\al_2},f_{\widetilde{\al}}+\la e_{\al_1+\al_2}]=b\,(c\,e_{\al_1+\al_2}+d\,f_{\widetilde{\al}})=0.\] Since this must equal zero, we conclude $b=0$. Hence $\la=0$, and uniqueness follows.

Suppose for $p=5$ we obtain $W(1;1)$ containing regular $e_{-1}$, then $(\ad (e_{-1}))^5=0$. However, the tables of \cite{S16} show this regular element has the property that $e_{-1}^{[5]}\ne 0$. This shows we do not obtain any $W(1;1)$ for regular nilpotent $e_{-1}$. For $p>7$, it follows from \cite{Cha41} that $W(1;1)$ has no irreducible representations of dimension smaller than $p-1$, which rules out $W(1;1)$ in $\g$ as the smallest such representation is $7$-dimensional with $\rho:G_2\hookrightarrow\mathfrak{so}(7)$.

Finally, consider the non-regular nilpotent elements $e$ of $\g$ and suppose $W(1;1)$ contains $e$. It follows we must have that $\rho(e)^{3} \ne 0$, but by \cite{Law95} this implies $e$ is regular. Hence, we may conclude there are no such occurrences in $G_2$.\end{proof}

For the other exceptional Lie algebras, we obtain maximal subalgebras isomorphic to $W(1;1)$ by taking regular nilpotent element $e:=\sum_{i=1}^{\rank(\g)}e_{\al_i}$ along with $f:=f_{\widetilde{\al}}$ in fields of characteristic $p=h+1$ where $h$ is the Coxeter number for exceptional Lie algebra $\g$, that is $p=7, 13, 19, 31$ for $G_2, F_4, E_7$ and $E_8$ respectively.

There is also a $W(1;1)$ subalgebra in $E_6$ for $p=13$, but this lies in a subalgebra of type $F_4$. The subalgebra generated by $e$ and $f$ is isomorphic to $W:=W(1;1)$ with isomorphism given by mapping $e \mapsto \partial$ and $f \mapsto X^{p-1}\partial$.

Maximality follows by checking $\langle W, v\rangle=\g$ for $v \in \g_e$. This is enough since we may consider $\g$ as a $W$-module, then check that the number of composition factors is equal to the dimension of the null space for $\partial$. It then follows that any submodule strictly containing $W$ also contains a null vector, in other words $v \in \g_e$.

The key detail in proving uniqueness is the fact that our representative for $X\partial$ is unique. This requires a slight generalisation of the reasoning in the previous result.

\begin{lem}\cite[Lemma 3.2]{HS14} Let $G$ be an algebraic group of exceptional type with Lie algebra $\g=\Lie(G)$ over an algebraically closed field of good characteristic and $L$ a Levi subgroup. Suppose $e$ is a nilpotent element of $\g$, distinguished in $\Lie(L)$. Then $\im\,\ad\,e\cap \g_e(\tau,0)$ is trivial unless $L$ has a factor of type $A_{p-1}$.\end{lem}

In \pref{wittg2}, the fact that $\g_e(\tau,0)=0$ in the regular nilpotent orbit makes life much easier. However, this will not always be the case. We could have gone further, and shown $h \in \im(\ad\,e)$ as well. This follows since we can find $f'$ in $W(1;1)$ such that $h=[e,f']$. This allows us to assume that our choice of $h=X\partial$ is unique since $h \in\im\,\ad\,e\cap \g_e(\tau,0)=0$, and makes the issue of uniqueness of the $W(1;1)$ subalgebra a question of showing that our choice for $X^{p-1}\partial$ is unique.

\begin{rem}\label{exceptionalorbits}There are eight nilpotent orbits in the exceptional Lie algebras that have label $\mathcal{O}(A_{p-1}+A_r)$ for some $r$ and $p$ good, with a list given in \cite[Table 1]{HS14} along with the isomorphism class for $\g_e(\tau,0)\cap \im\,\ad\,e$.\end{rem}

Our choice for $X\partial$ will always be a toral element $h \in \mathfrak{n}_e \setminus \g_e$. In fact, using \cite[Proposition 3.3]{HS14} it must be a Lie algebra representative for the associated cocharacter to $e$ always denoted as $\tau$. This element is usually denoted as $\Lie(\tau)$ or $h_{\tau}$ in the literature. This element has the property that $[h_{\tau},e]=2e$, and the centraliser $C_G(h_{\tau})$ has Lie algebra $\g(\tau,0)$.

\begin{rem}\label{uniqueco} For further information on the cocharacter $\tau$ we refer the reader to \cite[Section 2.3-2.7]{Pre03}, but for us we find such an element $h=h_{\tau}$ in the normaliser of our nilpotent element $e$. Using \cite[Proposition 3.3]{HS14} we obtain uniqueness, and hence we may use $h_{\tau}$ as our representative for $X\partial$ provided $e$ does not have a factor of type $A_{p-1}$.\end{rem}

This result is given for $p$ good, and the issue in bad characteristic begins with the lack of \cite[Lemma 3.2]{HS14}. When we consider the exceptional Lie algebra of type $E_8$ for $p=5$, we show this holds for the cases we are worried about.

\begin{lem}\cite[Lemma 3.11]{HS14} Let $G$ be an algebraic group of exceptional type with Lie algebra $\g=\Lie(G)$ and $p=h+1$, where $h$ is the Coxeter number of $G$. The $W(1;1)$ subalgebras containing regular nilpotent $e$ are unique, and maximal if and only if $\g$ is not of type $E_6$. \end{lem}
\begin{proof} Uniqueness is a generalisation of the argument used in \pref{wittg2}, as we may assume our choice $h_{\tau}$ for $X\partial$ is unique up to conjugation. To show $f$ is unique, we require $[X\partial,X^{p-1}\partial]=(p-2)X^{p-1}\partial$, and so $f \in \g(\tau,-2p+4) \oplus \g(\tau,4)$. Using relations of the form $[f,[e,f]]=0$ we determine that $f=f_{\widetilde{\al}}$.

Maximality does not occur in $E_6$ since our regular nilpotent element in $F_4$ is also regular in $E_6$, and by the uniqueness of $W(1;1)$ in $F_4$ it follows $W(1;1) \subseteq F_4 \subseteq E_6$.\end{proof}

For other potential cases of $W(1;1)$ subalgebras, we require $e^{[p]}=0$ and the lowest graded component of $e$ with respect to $\tau$ to be non-zero with weight $-2p+4$. A list of all nilpotent orbits $\mathcal{O}$ such that $\g(\tau,-2p+4)\ne 0$ is found in \cite[Table 2]{HS14}. To reduce the list further, we look for $f \in \g(\tau,-2p+4)$ such that $(\ad\,e)^{p-1}(f)=\lambda e$. If no such $f$ exists, then we may rule out a subalgebra of type $W(1;1)$. This reduces the number of cases, and it turns out we require that $e$ is regular in a Levi subalgebra $\mathfrak{l}$ of $\g$ using \cite[Lemma 3.6]{HS14}.

The idea is to show in all cases where we obtain a subalgebra of type $W:=W(1;1)$, we can find a non-trivial abelian subalgebra normalised by $W$. There are some reductions made by considering $\g$ as a $W$-module, and then computing the composition factors from \cite[Table 3]{HS14} for all cases where $\mathcal{O}$ has no factor of type $A_{p-1}$. We may apply \cite[Lemma 3.9]{HS14} to give a restriction on the number of some factors that may force the existence of a fixed vector for $W$ in certain cases. Then $W \subseteq \g_v$, and hence not maximal.

\begin{prop}Let $G$ be an algebraic group of exceptional type with Lie algebra $\g=\Lie(G)$ and $p$ good. There are no maximal subalgebra of type $W(1;1)$ unless $e$ is a regular nilpotent element of $\g$.\end{prop}
\begin{proof}This is a summary of \cite[Appendix]{HS14}. Note that $W(1;1)$ is generated as a Lie algebra by the elements $\partial$ and $X^3\partial$. Hence, we need a generic element $u\in\g$ such that $[X\partial,u]=2u$ to represent $X^3\partial$. We then look for $v \ne 0$ such that $[e,v]=[X\partial,v]=[X^3\partial,v]=0$. This ensures $v$ is a fixed vector for any $W(1;1)$, and hence maximality is ruled out.

We may assume that $e$ and $X\partial$ are unique. Take $f=\sum_i^{\dim\,\g}\lambda_i\, v_i$ where $v_i$ are the basis elements for $\g$ insisting $[e,f]=X\partial$ and $f \in \g(\tau,1)$. Consider $f$ as a candidate for $X^2\partial$, and applying the same idea we can find generic $u$ to represent $X^3\partial$ with $[e,u]=f$ and $u \in \g(\tau,2)$.

Next, we consider $v:=\sum_i^{\dim\,\g}\la_i\, v_i$. We use that $[e,v]=[X\partial,v]=0$ to reduce the number of indeterminates of $v$. This forces many coefficients to be zero, and by considering $[u,v]=0$ we aim to find some non-zero fixed vector $v$.

Using GAP to insist $[u,v]=0$ leaves many linear equations in the coefficients of $v$. Putting this into a matrix $A$, we verify the rank of $A$ is strictly less than the number of indeterminates for $v$. This ensures that all the linear equations can be solved non-trivially, hence finding such a $v$ for all possible $W(1;1)$. There are two exceptions where $v=0$, but these cases are dealt with by finding a non-trivial abelian subalgebra that is normalised by $W(1;1)$. \end{proof}

\begin{rem}We will work through an example in \aref{fixedvectorapp} of the above procedure when we consider $W(1;1)$ subalgebras in the exceptional Lie algebra of type $E_8$ for characteristic $p=5$.\end{rem}

Referring back to \autorefs{hi}, to complete the case of $\mathfrak{h}$ of Witt type we need to rule out the non-restricted cases $W(1;n)$ for $n>1$ and $W(2;\underline{1})$.

\begin{prop}\label{usefulwitt}Let $G$ be an algebraic group of exceptional type with Lie algebra $\g=\Lie(G)$ and $p$ good. There are no subalgebras of type $W(1;n)$ for $n>1$.\end{prop}
\begin{proof} Suppose $\mathfrak{h}$ is a maximal subalgebra of $\g$. For $\h \cong W(1;n)$ we require $e$ such that $(\ad\,e)^{p^n-1}$ is non-zero and that $e$ lies in the image. Since we have associated cocharacters given in \cite{LT11}, we can assume that $e$ has a non-zero graded component of weight $-2p^n+4$. Since $p$ is good we have $-2p^n+4 \le -46$, and so $n=1$ by the tables of \cite{LT11}. Hence $W(1;1)$ is our only possible case.

Both $W(1;2)$ and $W(1;3)$ are ruled out in \cite{HS14} by finding different contradictions. For the first case, \cite[Proof of Theorem 1.3]{HS14} it is observed that the lowest graded piece has weight $2h-2$, forcing that $e$ can be applied at most $h+1$ times. This gives $p^2-1\le h+1$, and hence $p \le 3$ unless $\g$ has type $E_8$ where $p \le 5$. Since $p$ is good this is the necessary contradiction to conclude $\mathfrak{h} \ncong W(1;2)$.
\end{proof}

Possibly the most difficult task is ruling out subalgebras of Hamiltonian type. Both \cite[Theorem 1.3]{HS14} and \cite[Theorem 4.1]{P15} provide arguments of how to do this, but require some understanding about the representation theory of $H:=H(2;1;\Phi)^{(2)}$.

In the first result, we have that $H$ contains a subalgebra of type $W(1;1)$, and so any occurrence in $\g$ is related to the $W(1;1)$ classification. Considering $\g$ as a $W(1;1)$-module, and \cite{HS14} gives an algorithm to compute the composition factors. In most cases are read off from the weight spaces of $\g_e$ with respect to the cocharacter $\tau$.

In certain circumstances when our nilpotent orbit has a factor of type $A_{p-1}$ we need to be careful. There are multiple cases as $X\partial$ is not unique and not necessarily our usual choice. Hence, the number of factors can change. The representation theory considered in \cite[Lemma 2.6, 2.9]{HS14} determines that we require the same number of composition factors for each type $L(\la)$ for $1 \le\la\le p-2$.

\begin{rem}The notation $L(\la)$, is a composition factor of weight $\la$. This weight comes from the fact that in each $L(\la)$ there is a unique vector $v$ killed by $\partial$, and $X\partial\, v=(\la+1)v$. For full details we refer the reader to \cite[Section 2.3]{HS14}.\end{rem}

\begin{prop}Let $G$ be an algebraic group of exceptional type with Lie algebra $\g=\Lie(G)$ and $p$ good. There are no subalgebras of type $H(2;1;\Phi)^{(2)}$.\end{prop}

We apply the algorithm to the nilpotent orbits where we obtain $W(1;1)$. This is given in \cite[Tables 3,5]{HS14}, which allows one to check that there are no cases where we have the same number of composition factors for $L(\la)$ with $1\le\la\le p-2$. This gives an immediate corollary to rule out another type of subalgebra.

\begin{coro}Let $G$ be an algebraic group of exceptional type with Lie algebra $\g=\Lie(G)$ and $p$ good. There are no subalgebras of type $W(2;\underline{1})$.\end{coro}

\begin{proof}This follows since $H(2;\underline{1})^{(2)}$ appears as a subalgebra of $W(2;\underline{1})$, and so ruling out subalgebras of type $H(2;1;\Phi)^{(2)}$ provides this result. \end{proof}

The alternative result \cite[Proposition 4.1]{P15} is useful as it is achieved without GAP or any algorithm on composition factors. It relies on a new concept of $p$-balanced toral elements. We want a list of nilpotent elements where we may obtain a subalgebra of type $H$, and rule them out case by case.

It is worth observing that all the ideas in this section extend to $E_8$ for $p=5$ as we only rely on the representation theory of $H$ which is valid for $p \ge 5$.

\begin{defs}\cite[Definition 2.6]{P15} Let $d$ be a positive integer. A toral element $h \in \g$ is $d$-balanced if $\dim\,\g(h,i)=\dim\,\g(h,j)$ for all $i,j \in \mathbb{F}^{\times}_p$ and all eigenspaces $\g(h,i)$ with $i \ne 0$ have dimension divisible by $d$.\end{defs}

The complete list of all $d$-balanced elements is given in \cite[Proposition 2.7]{P15} for the exceptional Lie algebras in good characteristic. The key observation is $M \cong H(2;1;\Phi)^{(2)}$ contains a non-zero $p$-balanced element. This is due to special properties of the irreducible representations of the Hamiltonian Lie algebra, and although it requires more thought in $H(2;1;\Phi(1))$ the result also holds.

Then, we may show $h$ is a $p$-balanced toral element of $\g$. This reduces the cases to the list in \cite[Proposition 2.7]{P15}. Associated to each $p$-balanced element is its corresponding sheet, which in turn allows one to use \cite{dga09} to build the list of potential orbits, which may produce subalgebras of type $H$. In the majority of cases $e \in \mathcal{O}(A_{p-1})$, and so one may rule these out using the same argument.

To complete the classification of subalgebras of Cartan type we still need to consider the cases of $K(3;\underline{1})$ and $M(1;1)$, which both have dimension $125$ for $p=5$ and so a priori may appear as subalgebras in the exceptional Lie algebra of type $E_7$.

\begin{lem}\cite[Lemma 4.1]{HS14} Let $G$ be an algebraic group of exceptional type with Lie algebra $\g=\Lie(G)$ and $p$ good. If $\h$ is a proper simple subalgebra of $\g$, then $\dim\,\mathfrak{h} \le p^3-4$.\end{lem}

The idea behind this is to use the Weisfeiler filtration to show any subalgebra $\mathfrak{h}$ in $\g$ has dimension at most $121$. This result then completes the proof of \cite[Theorem 1.3]{HS14}.

\section[Finishing the classification in the good case]{Finishing the classification in the good case}\label{sec:classification}

In this section we describe the complete classification of maximal subalgebras in $G_2$ for $p \ge 5$, and finish by stating the full classification of maximal subalgebras in the exceptional Lie algebras over algebraically closed fields of good characteristic due to \cite{PremetStewart}.

For $\g$ of type $G_2$ we avoid many Cartan type subalgebras by dimension reasons. For now, let $G$ be an algebraic group of type $G_2$ with Lie algebra $\g=\Lie(G)$ and $L$ be a maximal subalgebra such that $\rad(L)=0$. Consider the \emph{socle} of $L$, defined as the sum of minimal ideals in $L$.

\begin{prop}\label{keytoralrankuse}Let $G$ be an algebraic group of type $G_2$ with Lie algebra $\g=\Lie(G)$ and $p$ good. If $L$ is a maximal non-regular semisimple subalgebra of $\g$, then $L \in \{W(1;1),\mfsl(2)\}$. \end{prop}
\begin{proof}We have $\Soc(L)=(S_1 \otimes \mathcal{O}(m_1;\underline{1})) \oplus \ldots \oplus (S_r \otimes \mathcal{O}(m_r;\underline{1}))$ for simple Lie algebras $S_i$ and truncated polynomial rings $\mathcal{O}(m_i;\underline{1})$ by \cite{Blo69}. $L$ contains a maximal toral subalgebra if $r \ge 2$, and so we may assume $r=1$. We have $\Soc(L)=S \otimes \mathcal{O}(m;\underline{1})$, but since $\dim\,S\ge 3$, and $\dim\,\mathcal{O}(m;\underline{1})\ge p^m$ we have \[\dim\,\Soc(L) \ge 5^m \cdot 3= 15.\] Hence, $m=0$ since $\dim\,G_2=14$. Thus $\soc(L)$ is a simple Lie algebra.

By \cite{Blo69}, $S \subseteq L \subseteq \Der(S)$. Since $L$ contains a maximal toral subalgebra if $\TR(S)=2$ we may assume $\TR(S)=1$. By the classification of simple modular Lie algebras we find $S \in \{W(1,1),\mfsl(2)\}$. In particular, $\Der(S)=S$ as required. \end{proof}

We are left with maximal $\mfsl(2)$ triples, and we show there is only one such subalgebra appears which allows the classification of maximal subalgebras in $G_2$ to be completed.

Recall \rref{uniqueco}, for this we consider any algebraic group $G$ of exceptional type with Lie algebra $\g$. We take a nilpotent element $e$ along with associated cocharacter $\tau$. We can then find an element $h_{\tau} \in \g$ such that $[h_{\tau},e]=2e$. Further, using \cite{Pre03} we know that $e$ is distinguished in the derived subalgebra of a Levi subalgebra $\mathfrak{l}$. Since $\tau$ is an associated cocharacter it induces a grading on $\mathfrak{l}'$, and $\mathfrak{l}'$ contains $h_{\tau}$.

It follows in both \cite{Jan04,Pre03} that the map $\ad\,e: \mathfrak{l}'(\tau,-2) \rightarrow \mathfrak{l}'(\tau,-2)$ is bijective, and hence there is a unique element $f \in \mathfrak{l}'(\tau,-2)$ such that $[e,f]=h_{\tau}$ and $[h_{\tau},f]=-2f$. Hence, the triple $\{e,h_{\tau},f\}$ is isomorphic to $\mathfrak{sl}(2)$ as a Lie algebra.

\begin{defs}\cite[Definition 2.3]{PremetStewart} An $\mathfrak{sl}(2)$-triple $\{e',h',f'\}$ is called \emph{standard} if it is conjugate to an $\mathfrak{sl}(2)$ triple of the form $\{e,h_{\tau},f\}$. \end{defs}

\begin{prop}There are no non-standard $\mfsl(2)$-triples in $\g$ of exceptional type unless $e$ has type $\mathcal{O}(A_{p-1}+A_r)$ for some $r \ge 0$. \end{prop}

\begin{proof}This follows from \cite[Remark 2.4]{PremetStewart}, which states that if $\g$ contains a non-standard $\mathfrak{sl}(2)$ triple with nilpotent $e \in\g$, then $e$ has type $\mathcal{O}(A_{p-1}+A_r)$ for some $r \ge 0$. This is a consequence of the fact that these are the only nilpotent orbits where $\im\,\ad\,e \cap \g_e(\tau,0)$ are non-zero.

In these cases we may have non-conjugate options for $h'$, which follows since $h_{\tau}-h' \in \im\,\ad\,e \cap \g_e(\tau,0)$. Then we may use \cite{Jan04} to see which nilpotent orbits have such a non-zero intersection, with a list given in \cite[Table 1]{HS14}. \end{proof}

This has an immediate corollary in exceptional Lie algebras of type $G_2$, which allows two simple results to finish the complete classification of maximal subalgebras in $G_2$. This was achieved before the article \cite{PremetStewart} was released, and is the easiest case when classifying maximal subalgebras so we include this here.

\begin{coro}Let $G$ be an algebraic group of type $G_2$ with Lie algebra $\g=\Lie(G)$ and $p$ good. For nilpotent elements $e \in \g$, we have that any $\mathfrak{sl}(2)$ triple is standard. \end{coro}

\begin{prop}\label{sltriple}In Lie algebras of type $G_2$, the $\mfsl(2)$ triple associated to $\rt{1}+\rt{2}$ is maximal for $p>7$. Further, this is the only one we obtain.\end{prop}
\begin{proof}For fields of characteristic $p=5$ and regular nilpotent $e$ we have $e^{[p]}\ne 0$. Hence $e^{[p]}$ commutes with $\h=\spnd(\rt{1}+\rt{2},f,h),$ and so $\bbk e^{[p]} + \spnd(e,f,h)$ strictly contains $\h$ contradicting the maximality of $\h$. Let $\mathfrak{s}:=\mfsl(2)$ associated to $e:=\rt{1}+\rt{2}$. Since any maximal $\mfsl(2)$ triple is standard, we check for $p=7$ that $\s$ is contained in $W(1;1)$.

For $p>7$, suppose $M$ is a maximal subalgebra containing $\mathfrak{s}$. If $\s \subset M$, then $\rad(M)=0$. This follows since $M$ would be a parabolic subalgebra, but this can not occur. If $M$ has rank $2$, then $M$ is regular, and does not contain $\mathfrak{s}$. If $M$ has rank $1$, then $\Soc(M) = \mathfrak{s}\otimes \mathcal{O}(m,\underline{n})$. As $\dim\, \g=14$ this forces $m=0$, and so $M=\s$.

Any remaining maximal $\mfsl(2)$-triples are standard by the previous lemma, and as none of these are maximal in the characteristic zero they will not be maximal here.
\end{proof}

We now provide the full classification in the exceptional Lie algebra of type $G_2$. This result also was known to O.K Ten who announced such a result, but never published.

\begin{thm}\label{MainRes} Let $G$ be an algebraic group of type $G_2$ with Lie algebra $\g=\Lie(G)$ over an algebraically closed field of good characteristic. If $L$ is a maximal subalgebra in $\g$, then it is one of the following: \item[(a)]{semisimple of maximal rank,}\item[(b)]{parabolic,}\item[(c)]{$L \cong W(1;1)$ for $p=7$ or,}\item[(d)]{$L \cong \mfsl(2)$ when $p>7$.} \end{thm}

\begin{proof}If $L$ is a maximal subalgebra such that $\rad(L)\ne 0$, then by \cite[Theorem 1.1]{P15} $M$ is a parabolic of $\g$. Suppose $\rad(L)=0$, if $\mathfrak{t} \subseteq L$ we apply Borel--de Siebenthal to obtain subalgebras of type $A_2$ and ${A_1}^2$.

Hence, we may assume $L$ is semisimple and has toral rank equal to one. By \pref{keytoralrankuse} this implies $L \in \{\mfsl(2), W(1;1)\}$, and \pref{wittg2} gives $W(1;1)$ is only maximal when $p=7$. If $p=5$, then by \pref{sltriple} gives there are no $\mfsl(2)$ triples for $p\le 7$ with the $\mfsl(2)$ corresponding to the regular nilpotent element of $G_2$ a maximal subalgebra for $p>7$ completing the classification.\end{proof}

We finish this section by giving the two key results of the recent arXiv article \cite{PremetStewart}, which complete the classification of maximal subalgebras in the exceptional Lie algebras for good $p$. For the remainder of this section we shall assume $G$ to be any algebraic group of exceptional type with Lie algebra $\g=\Lie(G)$.

For the first of the results we need to define the notion of \emph{exotic semidirect products} in the exceptional Lie algebras.

\begin{defs}Let $\h$ be a semisimple restricted subalgebra of $\g$ such that $\h \cong  (\mathfrak{sl}(2)\otimes\mathcal{O}(1;1))\rtimes(I\otimes \mathcal{D})$ as Lie algebras for some restricted subalgebra $\mathcal{D}$ of $W(1;1)$. Then we call $\h$ an \emph{exotic semidirect product}.\end{defs}

\begin{thm}(Compare with \cite[Theorem 1.1]{PremetStewart}) Let $G$ be an algebraic group of exceptional type with Lie algebra $\g=\Lie(G)$, and $\h$ be an exotic semidirect product in $\g$. Then $p \in \{5,7\}$, $G$ has type $E_7$ or $E_8$, and the following hold:
\begin{enumerate}\item{If $G$ is of type $E_8$, then $\h$ is contained in a regular subalgebra of type $E_7+A_1$. Hence, this is not a maximal subalgebra in $\g$.}\item{If $G$ is of type $E_7$, then $\Soc(\h)\cong\mathfrak{sl}(2)\otimes\mathcal{O}(1;1)$ and $\h \cong (\mathfrak{sl}(2)\otimes\mathcal{O}(1;1))\rtimes(I\otimes \mathcal{D})$ for some restricted subalgebra $\mathcal{D}$ of $W(1;1)$. Further, suppose $\h$ is maximal subalgebra of $\g$. Then $\mathcal{D}\cong \mathfrak{sl}(2)$ when $p=5$ and $\mathcal{D}\cong W(1;1)$ when $p=7$. Any two maximal exotic semidirect produces are $\Ad\,G$-conjugate.}\end{enumerate}\end{thm}

There is one more case from \cite[Theorem 1.1 (iii) and (iv)]{PremetStewart} of a maximal subalgebra $\mathfrak{m}$ arising from a semisimple restricted subalgebra $\h$ in $\g$ containing a minimal ideal which is not simple. This case only occurs when $G$ has type $E_7$, with $\Soc(\h)$ as above. Then $\mathfrak{n}_{\g}(\Soc(\h))$ is a maximal subalgebra of $\g$ for $p=5$ and $p=7$.

\begin{thm}\cite[Theorem 1.2]{PremetStewart} Let $G$ be an algebraic group of exceptional type with Lie algebra $\g=\Lie(G)$ and $p$ good. Let $\mathfrak{m}$ be a maximal subalgebra of $\g$ and suppose that all minimal ideals of $\mathfrak{m}$ are simple Lie algebras. Then one of the following occurs:
\begin{enumerate}\item{There exists a maximal connected reductive subgroup $M$ of $G$ such that $\mathfrak{m}=\Lie(M)$.}\item{The group $G$ does not have type $E_6$, the Coxeter number of $G$ is $p-1$, and $\mathfrak{m}$ is conjugate to the Witt algebra $W(1;1)$ generated by the regular nilpotent element $\sum_{\al\in\Phi_0} f_{\al}$ and the highest root vector $e_{\tilde{\al}}$.}\end{enumerate}\end{thm}

All maximal connected subgroups of the group $G$ are known, see \cite{LS04} for a complete list. Hence, these two results give the complete classification of maximal subalgebras in $\g$ up to conjugacy.

\chapter[Maximal subalgebras in $E_8$ over fields of characteristic five]{Maximal subalgebras in $E_8$ over fields of characteristic five}\label{sec:newchapter}

We have only considered the case of good $p$ for $G$. For the remainder of this thesis, we assume $p$ is bad. If $p=5$, then we still have the complete classification of simple Lie algebras. In this chapter we consider $G$ an algebraic group of type $E_8$ with Lie algebra $\g=\Lie(G)$ for $\bbk$ of characteristic $p=5$. We consider properties of some maximal subalgebras in this case.

\section[Non-semisimple subalgebras in $E_8$ over fields of characteristic five]{Non-semisimple subalgebras in $E_8$ over fields of characteristic five}\label{sec:E8}\label{maxnonseme8}

Classifying the non-semisimple subalgebras over the complex numbers was achieved in \cite{mor56} with the extension to good characteristic by \cite{P15}. This shows the only maximal subalgebras $M$ such that $\rad(M)\ne 0$ are parabolic.

We consider \cite[Theorem 4.2]{P15}, which provides the first counterexample of this result for bad characteristic. The author of this thesis provided some calculations using GAP, that was used to show we obtain a non-semisimple maximal subalgebra that is not parabolic. These calculations are all explained in \aref{appendix:gapmod}.

Consider $e=\rt{1}+\rt{2}+\rt{3}+\rt{4}+\rt{6}+\rt{7}+\rt{8}$, the standard representative for the orbit $\mathcal{O}({A_4}+{A_3})$ in the exceptional Lie algebra of type $E_8$ with cocharacter $\tau$ labelled as \[\subalign{2\quad2\quad&2\quad{-9}\quad2\quad2\quad2\\&2}\] taken from \cite[pg. 148]{LT11}.

\begin{thm}\label{sec:e8p5non}\cite[Theorem 4.2]{P15} Let $G$ be an algebraic group of type $E_8$ with Lie algebra $\g=\Lie(G)$ over an algebraically closed field of characteristic five. Then, the following are true for any $e \in \mathcal{O}(A_4+A_3)$: \begin{enumerate}\item{Let $\g'_e$ be the subalgebra of $\g_e$ generated by $\g_e(\tau,\pm 1)$. Then, $A \subset \g_e'$ and $\g_e'/A \cong H(2;\underline{1})^{(2)}$ as Lie algebras.}\item{$A=\rad(\mathfrak{n}_e)$ and $\mathfrak{n}_e/A \cong \Der\, H(2;\underline{1})^{(2)}$ as Lie algebras.}\item{For $\mathfrak{w}:=N_{\g}(A)$ we have that $A=\rad(\mathfrak{w})$ and $\mathfrak{w}/A \cong W(2;\underline{1})$ as Lie algebras.}\item{$A \cong (O(2;\underline{1})/\bbk1)^{\ast}$ as $W(2;\underline{1})$-modules.}\item{$A \subset \mathcal{N}(\g)$ and $\mathfrak{w}$ is a maximal Lie subalgebra of $\g$.}\end{enumerate}\end{thm}

\begin{proof}[Remarks]Using GAP we can obtain the following information:\begin{enumerate}
\item{The centraliser $\g_e$ has dimension $50$ with a $24$-dimensional abelian radical $A$,}
\item{the normaliser $\mathfrak{w}=\mathfrak{n}_{\g}(A)$ is $74$-dimensional, with $\mathfrak{w}/A$ a simple restricted Lie algebra,}
\item{the subalgebra generated by $\g(\tau,\pm 1)$ is $47$-dimensional.}
\end{enumerate}All of these claims are proved in \cite{P15}, using the information above. We will be looking at very similar subalgebras, and our method of proving maximality is applicable here. Consider $L_{-1}:=\{x \in \g:[x,a]\subseteq \w\}$, and using \aref{sec:weirdwesicalculations} this is $124$-dimensional such that $\langle L_{-1}\rangle=\g$. Crucially, $L_{-1}$ is a $\w$-module and $L_{-1}/\w$ is irreducible.

Hence, any subalgebra properly containing $\w$ must contain $L_{-1}$. This proves maximality, and using $L_{-1}$ we can build a Weisfeiler filtration with some nice properties in \tref{weise8}.\end{proof}

The maximal subalgebra $\mathfrak{w}$ is $74$-dimensional and lies in the short exact sequence \[0 \,\ra\, A \,\ra\, \mathfrak{w} \,\ra\, W(2;\underline{1}) \,\ra\, 0.\] Hence, if $p$ is bad we should expect more non-semisimple maximal subalgebras that are not parabolic. Later, we consider other nilpotent orbits that exhibit similar behaviour producing new maximal subalgebras.

It was noted in \cite[Remark 4.4]{P15} that the numerology suggests the possible isomorphism $\gr_{\mathcal{F}}(\g)\cong S(3;\underline{1})^{(1)}$ as Lie algebras, where $\mathcal{F}$ is the Weisfeiler filtration obtained using $\w$. There is a grading on $S(3;\underline{1})^{(1)}$ given by associating degree $1$ to $x_3$ and zero to both $x_{1}, x_2$. This gives \begin{align}\label{specgrad}\begin{split}S_{-1}&:=\{{x_1}^i{x_2}^j\partial_3: 0 \le i,j<p , (i,j) \ne (p-1,p-1)\}, \\S_0&:=\{i{x_1}^{i-1}{x_2}^jx_3\partial_3-{x_1}^i{x_2}^j\partial_1,j{x_1}^i{x_2}^{j-1}x_3\partial_3-{x_1}^i{x_2}^j\partial_2:0 \le i,j < p\},\end{split}\end{align}
with $S_0 \cong W(2;\underline{1})$, and $S_{-1} \cong (\mathcal{O}(2;\underline{1})/\bbk 1)^{\ast}$ from \cite[pg. 283]{PSt01}. We prove the claim, letting $\mathcal{G}:=\gr_{\mathcal{F}}(\g)$. First, we start with a result used to see when $\mathcal{G}$ is simple.

\begin{prop}\label{simplegraded}Let $\mathcal{G}$ be a graded Lie algebra in the sense of Weisfeiler from \autoreft{sec:Wei}. Suppose $M(\Gc)=0$. If $\Gc_0$ is simple, and $\Gc_i$ is an irreducible $\Gc_0$-module for all $i \ge 1$, then $\Gc$ is simple.\end{prop}
\begin{proof}By \cite[Proposition 1.6.1]{Wei78} that if $I$ is a non-zero ideal of $\mathcal{G}$, then $\mathcal{G}_{-} \subseteq I$. We have $[\mathcal{G}_{-1},\mathcal{G}_1] \subset I$ is a non-zero ideal of $\mathcal{G}_0$, and since $\mathcal{G}_0$ is simple $\mathcal{G}_0=[\mathcal{G}_1,\mathcal{G}_{-1}] \subseteq I$. Consider $[\mathcal{G}_0,\mathcal{G}_1]$. This must be non-zero, otherwise $\mathcal{G}_0 \subseteq \mathcal{G}_1$ --- a contradiction. It follows by the irreducibility of $\Gc_1$ that $\mathcal{G}_1 \subseteq I$. A similar reasoning shows $\Gc_i \subseteq I$ for all $i$. \end{proof}

\begin{thm}\label{weise8}Let $G$ be an algebraic group of type $E_8$ with Lie algebra $\g=\Lie(G)$ over an algebraically closed field of characteristic five. For the maximal subalgebra $\mathfrak{w}$ defined in \tref{sec:e8p5non}, the Weisfeiler filtration $\mathcal{F}$ arising from $\w$ gives rise to a corresponding graded Lie algebra $\mathcal{G}:=\gr_{\mathcal{F}}(\g)$ isomorphic to $S(3;\underline{1})^{(1)}$.\end{thm}
\begin{proof}We use GAP to find the remaining spaces in the filtration $\mathcal{F}$, and verify $\dim\,\mathcal{G}=248$. The basis for $L_{-1}$ is given in \aref{sec:weirdwesicalculations}, with information on how to deduce each $L_{i}/L_{i+1}$ is irreducible.
\begin{enumerate}\item{$\dim \,L_{-4}=248$,}\item{$\dim\, L_{-3}=224$,}\item{$\dim \,L_{-2}=174$,}\item{$\dim \,L_{-1}=124$,}\item{$\dim\, L_{0}=\dim\, \mathfrak{w}=74$,}\item{$\dim\, L_1=\dim \, \rad(\mathfrak{w})=24$.}\end{enumerate}

After writing $\g$ as a $\w$-module we may ask GAP for all submodules --- this gives the list of those above. Since each $\Gc_i$ is irreducible there is some minimal $k$ such that $\Gc_{k}\subseteq M(\Gc)$. However, since $[L_{-i},L_1]=L_{-i+1}$ for all $i=2,3,4$ it follows that $[\Gc_1,\Gc_{-k}]\ne 0 \subseteq \Gc_{-k+1}\cap M(\mathcal{G})$ contradicting the fact $k$ is minimal. Hence $M(\Gc)=0$, and so $\Gc$ is a simple Lie algebra with an unbalanced grading.

It follows that $\Gc$ is restricted since $\Gc_0\cong W(2;\underline{1})$ is and $\Gc_i^{[5]}\subseteq \Gc_{5i}=0$ for all $i \ne 0$. The classification of simple Lie algebras leaves two options, either $\Gc\cong \g$, or $\Gc\cong S(3;\underline{1})^{(1)}$. Any grading of $\g$ is symmetric, and so $\g(i)=\g(-i)$ for all $i$. Hence, we cannot have a grading $\Gc_{-4} \oplus \ldots \oplus \Gc_{1}$. It follows $\Gc \ncong \g$, and since $\Gc$ is a $248$-dimensional simple Lie algebra we must have $\Gc \cong S(3;\underline{1})^{(1)}$.
\end{proof}

\begin{conjecture}\label{conjprem}\cite[Conjecture 4.3]{P15} Suppose $G$ is of type $E_8$ and $p=5$. Any maximal subalgebra $M$ such that $\rad(M) \ne 0$ is either conjugate to $\mathfrak{w}$ under the adjoint action of $G$ or is the Lie algebra of some maximal parabolic subgroup $P$ of $G$ or coincides with the centraliser of a toral element $t$ of $\g$. The last case gives $M=\Lie(G_t)$ and $G_t$ is a semisimple group of type $A_4A_4$. \end{conjecture}

The final case occurs since $p=5$ is equal to one of the coefficients in the highest root, hence the subalgebra of type $A_4+A_4$ obtained via the Borel-de Siebenthal algorithm has a non-trivial centre. This now becomes a non-semisimple maximal subalgebra.

We consider how \cite[Lemma 2.4]{P15} changes for $E_8$ when $p=5$ as \cite{Se67} about root space decompositions still holds. The Killing form is identically zero, but we may use the so-called normalised Killing form $\kappa$ as a replacement. This is defined explicitly in \cite[pg. 661]{CP13}.

\begin{lem} Let $G$ be an algebraic group of type $E_8$ with Lie algebra $\g=\Lie(G)$ over an algebraically closed field of characteristic five. Let $L$ be a maximal subalgebra of $\g$ with $\rad(L)\ne 0$. If $L$ contains a maximal toral subalgebra, then $L$ is either a parabolic subalgebra or $L \cong \Lie(G_t)$.\end{lem}

\begin{proof}The proof of \cite[Lemma 2.4]{P15} contains the basic set-up we need for this. Any maximal toral subalgebra $\mft$ of $\g$ is still a classical Cartan subalgebra satisfying the axioms of \cite{Se67}. It also holds that there exists some maximal torus $T$ of $G$ such that $\mathfrak{t}=\Lie(T)$. All root spaces of $\g$ with respect to $\mathfrak{t}$ are one-dimensional. Let $\Phi$ be the root system of $G$ with respect to $T$, and $\Psi$ a subset such that $L:=\mft\oplus \sum_{\al \in \Psi} \bbk e_{\al}$, where $\Psi$ is a closed subset of $\Phi$.

The maximality of $L$ ensures the Chevalley $\bbz$-form $\g_{\bbz}$ is such that $L_{\bbz}:=\g_{\bbz}(\Psi)\otimes_{\bbz}\bbc$ is maximal in $\g_{\bbc}$. Applying Morozov's theorem \cite{mor56} to $L_{\bbz}$ gives that $\Psi$ is either symmetric or parabolic. Suppose $\Psi$ is not parabolic, then the radical of $\kappa$ restricted to $L$ is equal to $\mathfrak{z}(L)$ by \cite[Lemma 2.2]{premet96}. By our assumptions, $\rad(L) \ne 0$. It follows that $\Psi$ symmetric if $\mathfrak{z}(L) \ne 0$, and so $L$ contains a factor of type $A_{p-1}$. This leaves the only option $L \cong \Lie(G_t)$. \end{proof}

This would be a starting point in tackling this conjecture. The key issue is \cite[Theorem 4.2]{lmt08} gives a subalgebra $L$ such that $[L,L]$ consists of nilpotent elements, but does not lie in a Borel subalgebra. If this is the only place where this happens for $\g$ over fields of characteristic five, then we may be able to reuse many observations from \cite[Theorem 1.1]{P15}.

\section[The non-restricted Witt algebras in $E_8$ for $p=5$]{The non-restricted Witt algebras in $E_8$ for $p=5$}\label{sec:2nwitt}

We consider the non-restricted Witt algebras, and show the $p$-closure of $W(1;2)$ is a maximal subalgebra of $\g$ for $p=5$. For the nilpotent orbit denoted $E_8(a_1)$ with standard representative $e:=e_{\al_1}+e_{\al_2}+e_{\al_2+\al_4}+e_{\al_3+\al_4}+e_{\al_5}+e_{\al_6}+e_{\al_7}+e_{\al_8}$ from \cite[pg. 184-185]{LT11}. Using the tables of \cite{S16} the standard representative still lies in the orbit labelled $E_8(a_1)$ as in the good characteristic case, and so we may continue to use associated cocharacter $\tau$ given by \[\subalign{2\quad2\quad&{0}\quad{2}\quad{2}\quad2\quad{2}\\&{2}}\] in \cite{LT11}.

There is a basis for Lie algebras of type $W(1;\underline{n})$ given by $\{e_i:-1 \le i \le p^n-2\}$ with Lie bracket:
\[[e_i,e_j]=
\begin{cases}
\left(\binom{i+j+1}{j}-\binom{i+j+1}{i}\right)e_{i+j} & {\text{if}\,\, -1 \le i+j \le p^n-2,}\\
 \,0 & {\text{otherwise}.}
\end{cases}\]
Note that $\binom{j}{-1}=0$ for all positive $j$.

\begin{prop}\label{extrawitt}Let $G$ be an algebraic group of type $E_8$ with Lie algebra $\g=\Lie(G)$ over an algebraically closed field of characteristic five. For $e=e_{\al_1}+e_{\al_2}+e_{\al_2+\al_4}+e_{\al_3+\al_4}+e_{\al_5}+e_{\al_6}+e_{\al_7}+e_{\al_8}$ and $L:=\langle e,f_{\widetilde{\al}}\rangle$ where $\tilde{\al}$ is the highest root in $E_8$, we have that $L \cong W(1;2)$.\end{prop}

\begin{proof}Given the basis $\{e_i:-1 \le i \le 23\}$ for $W(1;2)$ with multiplication defined as above, we define $e_{-1}:=e$ and set $e_i:={(\ad\,e)}^{23-i}(f_{\widetilde{\al}})$. The relations of the form $[e_i,e_j]$ are easily checked on GAP, and show that this defines an isomorphism between $L$ and $W(1;2)$.

Alternatively, we could consider the grading from cocharacter $\tau$ observing that this is depth-one graded with one-dimensional zero component. We apply the recognition theorem to see $L \cong W(1;2)$. \end{proof}

We can reduce the number of elements to check maximality using the centraliser of $e$. We know that $e^{[5]}\ne 0$ and $\dim\,\g_e=10$ by \cite{S16}. This centraliser is still smooth in characteristic $p=5$, and so the basis in \cite{LT11} is still a basis in this case --- up to some sign changes.

\begin{prop}\label{newcent}Let $e$ be the standard representative for the nilpotent orbit labelled $E_8(a_1)$ for characteristic five, then $\g_e$ has basis:
\begin{align*}
v_1&=e
\\v_2&=4e_{\subalign{11&11100\\&1}}+2e_{\subalign{11&11110\\&1}}+2e_{\subalign{01&11110\\&1}}+2e_{\subalign{01&11111\\&0}}+e_{\subalign{11&21100\\&1}}-e_{\subalign{01&21110\\&1}}-4e_{\subalign{12&21000\\&1}}\\&-3e_{\subalign{00&11111\\&1}}+e_{\subalign{01&22100\\&1}}
\\v_3&=e_{\subalign{01&22210\\&1}}-e_{\subalign{01&22111\\&1}}+e_{\subalign{12&32100\\&1}}+e_{\subalign{11&22110\\&1}}+e_{\subalign{12&22100\\&1}}-e_{\subalign{12&21110\\&1}}-2e_{\subalign{11&21111\\&1}}\\&+e_{\subalign{11&11111\\&1}}
\\v_4&=e_{\subalign{12&22210\\&1}}+e_{\subalign{12&22111\\&1}}-2e_{\subalign{11&22211\\&1}}-2e_{\subalign{01&22221\\&1}}+e_{\subalign{12&32110\\&2}}-2e_{\subalign{12&32111\\&1}}
\\v_5&=-e_{\subalign{12&32211\\&2}}+e_{\subalign{12&33211\\&1}}-e_{\subalign{12&32221\\&1}}-e_{\subalign{12&43210\\&2}}-2e_{\subalign{12&22221\\&1}}
\\v_6&=e_{\subalign{12&33221\\&2}}-e_{\subalign{12&33321\\&1}}-e_{\subalign{23&43210\\&2}}+e_{\subalign{13&43211\\&2}}+e_{\subalign{12&43221\\&2}}
\\v_7&=e_{\subalign{12&33321\\&2}}+e_{\subalign{13&43221\\&2}}-e_{\subalign{23&43211\\&2}}
\\v_8&=e_{\subalign{13&54321\\&3}}+e_{\subalign{23&54321\\&2}}
\\v_9&=e_{\subalign{24&54321\\&3}}
\\v_{10}&=e_{\subalign{24&65432\\&3}}
\end{align*} \end{prop}

The normaliser $N_{\g}(L)$ is $26$-dimensional, and equal to the $p$-closure of $L$. This follows since $e^{[5]} \in \g_e(\tau,10)$, and since this component is one-dimensional we know $e^{[p]}=v_2$. It is now clear that $v_2 \in N_{\g}(L)$.

\begin{lem}\label{blok2} Let $G$ be an algebraic group of type $E_8$ with Lie algebra $\g$ over an algebraically closed field of characteristic five, and let $L:=\langle e,f_{\widetilde{\al}}\rangle$. If $M$ is a maximal subalgebra containing $L$, then $M$ has exactly two Jordan blocks under the action of $e$.\end{lem}
\begin{proof} We start by showing $M$ must be a graded subalgebra of $\g$. If $\dim\, M' =d$, then the torus $T$ which defines the grading on $\g_e$ acts on the variety of all $d$-dimensional subalgebras of $\g$ denoted $\Sub(\g, d)$, which is a projective variety. It should be noted that $T$ is chosen such that $\Lie(T)$ acts as a Lie algebra representative for cocharacter $\tau$.

The set of all $X\in \Sub(\g,d)$ containing $L$ is a closed subset of $\Sub(\g,d)$ stable under the action of $T$. As $T$ is a connected solvable group we may apply the Borel fixed-point theorem to see this has a fixed point in $X$. Hence, there is a $d$-dimensional graded subalgebra $M$ of $\g$ which contains $L$.

It follows that the number of Jordan blocks is equal to $\dim\,(M \cap \me)$. The subalgebras $M_i:=\langle e, f_{\widetilde{\al}}, v_i \rangle$ with $i \ge 3$ all share the property that $M_i=\g$ using GAP to check. Hence $\left.(\ad \,e)\right|_M$ has at most two Jordan blocks, and we know $N_{\g}(L)$ contains both $e$ and $e^{[p]}.$ It follows that any maximal subalgebra containing $L$ has exactly two Jordan blocks. \end{proof}

\begin{remark}\label{info}$L$ has only one Jordan block of size $25$, and $(\ad\, e)^{25}=0$ by \cite{S16}.\end{remark}

\begin{thm}\label{nonrestrictedwitt} Let $G$ be an algebraic group of type $E_8$ with Lie algebra $\g=\Lie(G)$ for $p=5$ and $e=e_{\al_1}+e_{\al_2}+e_{\al_2+\al_4}+e_{\al_3+\al_4}+e_{\al_5}+e_{\al_6}+e_{\al_7}+e_{\al_8}$ be a nilpotent orbit representative for the orbit denoted $E_8(a_1)$. Then \[L:=\langle e,f_{\widetilde{\al}} \rangle\] is isomorphic to $W(1;2)$, and the $p$-closure is a maximal subalgebra of $\g$. Further, it is unique up to conjugation by $G$.\end{thm}

\begin{proof}
We know $L \cong W(1;2)$, and any maximal subalgebra $M$ containing $L$ has only two Jordan blocks. It follows from \rref{info} that $L$ is a projective submodule of $E_8$ regarded as a module over $\bbk e+\bbk e^{[p]}$, and hence a direct summand of $M$.

The size of the second Jordan block is calculated by finding $u \in \g$ such that $[e,u]=e^{[p]}$. We then compute $\langle u,e,f_{\widetilde{\al}} \rangle$, and check to see this is isomorphic to $\g$. Using GAP, with full calculation given in \aref{w12e8}, we find \begin{align}\label{keyujordan}\begin{split}u&=e_{\subalign{00&1111\\&0}}+3e_{\subalign{11&11000\\&1}}+
3e_{\subalign{11&11100\\&0}}+3e_{\subalign{01&11100\\&1}}+4e_{\subalign{00&11110\\&1}}+
2e_{\subalign{01&11110\\&0}}\\&+2e_{\subalign{00&11111\\&0}}+4e_{\subalign{11&21000\\&1}}+
3e_{\subalign{01&21100\\&1}}.\end{split}\end{align}

We generate a subalgebra of $\g$, and see $\langle u,e,f_{\widetilde{\al}} \rangle \cong\g$. This forces the size of the second Jordan block to be one, and hence $M=L \oplus \bbk e^{[p]}$. In particular, the $p$-closure of $L$ is maximal in $\g$.

Choose $h$ as in \rref{uniqueco} for our representative of $X\partial$, and taking a generic $v \in\g(\tau,-2)$ with $[e,v]=0$ ensures that $v=0$. In particular, $\im(\ad\,e) \cap \g_e(\tau,0)=0$ for $p=5$. It follows that $h$ is unique up to conjugation. If $f$ represents $X^{p^2-1}\partial,$ then we have $[X\partial,{X}^{p^2-1}\partial]=(p^2-2){X}^{p^2-1}\partial.$ Hence $f \in \bigoplus_{n \in \mathbb{Z}} \g(-np+4),$ allowing many possible $f$.

Take $f:=\sum_i x_i\cdot v_i$ for all $v_i \in \g(\tau,4)$, and consider $[(\ad\,e)^{22}(f),f]=0$. Note that many $x \in \g_e(\tau,4)$ already lie in $L$, and so linear combinations of them satisfy our condition. Further details on how to do this in GAP are given in \aref{uniquew12}. For example, we have $x[233]=4\cdot x[234]$. A quick check shows $v_{233}+4\cdot v_{234}$ is already an element of $L$.

Hence, we conclude this does not give a different Lie algebra. We then deduce that it must be the case that \begin{align*}x[233]=x[234]&=0, \\ x[9]=x[13]=x[16]=x[19]=x[20]&=0.\end{align*} We substitute this back in to $f$, and show the remaining coefficients must be zero. Consider $[f,[f,e]]=0$, this forces a lot of coefficients to be zero except for $x[189]$.

Suppose $f=f_{\widetilde{\al}}+f_{\subalign{12&32100\\&2}}$, then $[(\ad\,e)^j(f),f]=0$ for all $j$ in $W(1;2)$. However, for some $j$ we have that \[[(\ad\,e)^j(f),f]=x[189]\cdot \text{something}.\]Hence, $x[189]=0$ and $f=f_{\widetilde{\al}}$ giving uniqueness. \end{proof}

This confirms the restriction on $p$ in \cite[Theorem 1.3]{HS14} for obtaining $W(1;2)$ subalgebras is necessary. We may find more cases when $p \le 5$ for non-restricted subalgebras.

For other $W(1;\underline{n})$ subalgebras we require $e^{[p]}\ne 0$, and $(\ad \,e_{-1})^{p^n-1}(e_{p^n-2})=\la\, e_{-1}$. The list of nilpotent orbits is unchanged for $p=5$ using \cite{hs85}, and by \cite{S16} there is no nilpotent element such that $e^{[p]^n}\ne 0$ for $n \ge 2$. This rules out $W(1;\underline{n})$ for all $n \ge 3$.

For $W(1;2)$ we need nilpotent $e$ such that $\bbk e \subseteq \im\,(\ad\,e)^{p^2-1}$. Hence, we need non-zero graded components of degree $-2(p^2-2)=-46$. This leaves the orbits denoted $\mathcal{O}(E_8)$ and $\mathcal{O}(E_8(a_1))$.

\begin{remark}\label{nowittsub}
Consider $\mathcal{O}(E_8)$ with standard representative $e:=\sum_{i=1}^8 e_{\al_i}$ with associated cocharacter $\tau$ given by \[\subalign{{2}\quad{2}\quad&{2}\quad{2}\quad{2}\quad2\quad{2}\\&{2}}.\] We have \[\g(\tau,-46)= \spnd_{\bbk}\{f_{\subalign{24&54321\\&2}},f_{\subalign{23&54321\\&3}}\},\] by \cite[pg. 185]{LT11}, and for $f:=\la_1 f_{\subalign{24&54321\\&2}}+\la_2 f_{\subalign{23&54321\\&3}}$ we have $(\ad\,e)^{24}(f)=\la e$ provided $\la_1+\la_2 \ne 0 \mod p$. The subalgebra $L:=\langle e,f\rangle$ is $49$-dimensional with abelian solvable radical $A$ of dimension $24$. \end{remark}

It is possible there is $f' \in \g(\tau,4)$ such that $e$ and $f+f'$ generate $W(1;2)$. We check $[f+f',(\ad e)^{22}(f+f')]=0$, and using this we may show $f'+f$ never satisfies $[f+f',[e,f+f']]=0$. This shows we obtain no subalgebras of type $W(1;2)$ for this orbit, but $L/A \cong W(1;2)$.

Note we obtain a simple non-restricted Lie algebra of dimension $25$, and then using cocharacter $\tau$ we obtain a grading. Then, we may conclude they are isomorphic as Lie algebras using \cite[Theorem 1]{BGP09}.

Both orbits have links to the orbit $\mathcal{O}(A_4+A_3)$, since $e^{[p]}$ has this type in both cases. The important difference between these orbits is the regular orbit is no longer smooth. In fact, we have $\dim\,\g_e=\dim\,\Lie(G_e)+2$. These two new elements consist of $e^{[p]}$, but the second element \[v:=f_{\subalign{11&21000\\&1}}-f_{\subalign{11&11100\\&1}}-
f_{\subalign{11&11110\\&0}}+f_{\subalign{01&21100\\&1}}+3f_{\subalign{01&11110\\&1}}+
3f_{\subalign{00&11111\\&1}}+2f_{\subalign{01&11111\\&0}}\in \g(\tau,-14),\] also satisfies $\dim\,\g_v=50$ and $v^{[p]}=0$. This forces the nilpotent element $v$ to have type $A_4+A_3$, and allows us to show where $L$ lies.

\begin{prop}Let $G$ be an algebraic group of type $E_8$ with Lie algebra $\g=\Lie(G)$ for $p=5$ and $L$ be as in \rref{nowittsub}. For $A=\rad(L)$ we have $L \subsetneq N_{\g}(A)$, and further $N_{\g}(A)$ is a maximal subalgebra of type $\mathfrak{w}$ from \tref{sec:e8p5non}.
\end{prop}
\begin{proof}
We have $v\in \mathcal{O}(A_4+A_3)$, and by the calculations in \cite{P15} the dimension of $N:=N_{\g}(v)$ is $51$ and the radical $\rad(N)$ is abelian of dimension $24$. We verify $\rad(N)=A$ in GAP, and then $N_{\g}(A)$ must have type $\mathfrak{w}$. Since $L \cap N_{\g}(A)=L$, it follows that $L \subsetneq N_{\g}(A)$ as required.
\end{proof}
We confirm that there is only one conjugacy class of maximal subalgebras isomorphic to the $p$-closure of $W(1;2)$, this immediately follows from \tref{nonrestrictedwitt} and \rref{nowittsub}.
\begin{thm}\label{allwitt}Let $G$ be an algebraic group of type $E_8$ with Lie algebra $\g=\Lie(G)$ for $p=5$. There is only one conjugacy class of $W(1;2)$ subalgebras in $\g$, and its $p$-closure is maximal.
\end{thm}

An immediate corollary allows us to rule out another Cartan type Lie algebra.

\begin{coro}\label{noresh}Let $G$ be an algebraic group of type $E_8$ with Lie algebra $\g=\Lie(G)$ for $p=5$. There are no subalgebras of type $H(2;(n_1,n_2))^{(2)}$, with $n_1 \ne n_2.$\end{coro}
\begin{proof}Since $p=5$ and $\dim\,\g=248$, we have $n_1,n_2 \in \{1,2\}$. Without losing any generality we may assume $n_1=1,n_2=2$. There is a basis \[\{{X_1}^{i-1}{X_2}^j\partial_2-{X_1}^i{X_2}^{j-1}\partial_1:0 <i+j < p^2+p-2\},\] which has a subalgebra of type $W(1;2)$ taking elements with $j=1$ and $0\le i \le {p}^2-1$.

Therefore any subalgebra of type $H(2;(2,1))^{(2)}$ contains $W(1;2)$ as a subalgebra. By \tref{allwitt} there is only one conjugacy class in $\g$, and the $p$-closure is maximal. Any subalgebra properly containing $W(1;2)$ is either the $p$-closure or $\g$ itself.\end{proof}

\section[Preliminary results on the first Witt algebra in $E_8$ for $p=5$]{The maximal $W(1;1)$ type subalgebras in $E_8$ for $p=5$}\label{sec:witty}

For $p=5$ only two orbits have the property that $\g_e \ne \Lie(G_e)$ --- namely $\mathcal{O}(E_8)$ and $\mathcal{O}(A_4+A_3)$. In the previous section we showed both cases lead us to the maximal subalgebra $\w$ defined in \tref{sec:e8p5non}. This may rule out major surprises for maximal subalgebras in $\g$, and allow for a complete list of the maximal Cartan type subalgebras in $E_8$.

\begin{conjecture}\label{e8conj}Let $G$ be an algebraic group of type $E_8$ with Lie algebra $\g=\Lie(G)$ for $p=5$. If $M$ is a maximal subalgebra of Cartan or Melikyan type in $\g$, then $M \cong W(1;2)_{[p]}.$
\end{conjecture}

Since $\dim\,\g=248$ we only have a small list of possible maximal subalgebras of Cartan type. We will take a brief look at $W(1;1)$ subalgebras, and show in all but two possible cases they are not maximal.

\begin{table}[H]\caption{List of possible Cartan type subalgebras}\phantomsection\centering \begin{tabular}{l c} \hline $\h$ & $\dim\,\h$ \\ \hline\hline $W(1;1)$ & $5$  \\ \hline $W(1;2)$ & $25$  \\ \hline $W(2;\underline{1})$ & $50$  \\ \hline $H(2;\underline{1})$ & $23$  \\ \hline $H((2;\underline{1});\phi(\tau))^{(1)}$ & $24$ \\ \hline $H((2;\underline{1});\phi(1))$& $25$ \\ \hline $H(2;(1,2))$ & $123$ \\ \hline $K(3;\underline{1})$ & $125$  \\ \hline $M(1;1)$&$125$ \\ \hline\end{tabular} \end{table}

By \cite[Section 4.3]{P15} we may rule out Melikyan algebras, and \tref{allwitt} rules out $H(2;(1,2))^{(2)}$. To prove \conref{e8conj} we are left with $W(1;1)$, $W(2;\underline{1})$, $H(2;\underline{1};\Phi)^{(2)}$, and $K(3;\underline{1})$. Since $H(2;\underline{1};\Phi)^{(2)} \subseteq W(2;\underline{1})$ we should be able to deal with both cases at the same time.

We rule out any maximal subalgebras of type $W(1;1)$ in all but two orbits using the same idea as \cite{HS14}. To begin, we note \cite[Lemma 3.4]{HS14} no longer holds in all cases for $p=5$, with the orbit $\mathcal{O}(A_4+A_3)$ having non-zero $\g(-2p+3)$ component. The table below gives the list of all nilpotent orbits where we may find $\bbk e\in\im(\ad\,e)^4$.

\begin{table}[H]\caption{Nilpotent orbits to be checked}\phantomsection\label{fig:nilpcheck}\centering
\begin{tabular}{|ccc|}
\hline &$\mathcal{O}$  & \\ \hline\hline $A_4+A_3$  & $A_4+A_2$  & $A_4+2A_1$  \\ \hline $A_4+A_2+A_1$  & $A_4+A_1$ & $2A_3$ \\ \hline $D_4(a_1)+A_2$ & $A_3+A_2+A_1$ & $A_4$ \\ \hline $A_3+A_2$  & $D_4(a_1)+A_1$  & $A_3+2A_1$  \\ \hline $A_3+A_1$ & $D_4(a_1)$ & $A_3$ \\ \hline\hline
\end{tabular}\end{table}

\begin{prop}\label{nilptbc} From \autorefs{fig:nilpcheck} only four orbits have the property $e\in\im(\ad\,e)^4$. These are $A_3$, $A_4$, ${A_3}^2$, and $A_4+A_3$. \end{prop}
\begin{proof}This is done using GAP and applying $\ad\,e$ four times to a basis for the Lie algebra $\g$. From this, only the above orbits satisfy such a property. Consider $\mathcal{O}(A_3)$ with representative $e_1:=e_{\al_1}+e_{\al_3}+e_{\al_4}$, and associated cocharacter $\tau_1$ given by \[\subalign{{2}\quad{2}\quad&{2}\quad{-3}\quad{0}\quad0\quad{0}\\-&{3}}\] defined on \cite[pg. 126]{LT11}. The list of graded components in $\g$ for $\tau_1$ shows $\g(\tau_1,-6):=\langle f_{\subalign{11&10000\\&0}} \rangle$ is one-dimensional, and $(\ad\,e_1)^4(f_1)=\la e_1$. The subalgebra $L_1:=\langle e_1,f_1\rangle$ is a simple Lie algebra isomorphic to $W(1;1)$.

Similarly, for $\mathcal{O}({A_3}^2)$ represented by $e_2:=e_{\al_1}+e_{\al_3}+e_{\al_4}+e_{\al_6}+e_{\al_7}+e_{\al_8}$ with cocharacter $\tau_2$ given by \[\subalign{{2}\quad{2}\quad&{2}\quad{-6}\quad2\quad2\quad2\\-&{3}}.\] we consider $f_2 \in \g(\tau_2,-6)$. In this case, this space has dimension bigger than one, so we take a generic $f$ and insist $(\ad\, e_2)^4(f_2)=\la e_2$. This gives $f_2:=\lambda(f_{\subalign{11&10000\\&0}}+f_{\subalign{00&00111\\&0}})$, with $L_2:=\langle e_2,f_2 \rangle \cong W(1;1)$.

Following as above we find for $\mathcal{O}({A_4})$ represented by $e_3:=e_{\al_1}+e_{\al_2}+e_{\al_3}+e_{\al_4}$ with cocharacter $\tau_3$ given by \[\subalign{{2}\quad{2}\quad&{2}\quad{-6}\quad{0}\quad0\quad{0}\\&{2}}\] from \cite[pg. 135]{LT11}. Consider $f_3 \in \g(\tau_3,-6)$ where this space has dimension bigger than one, so take a generic $f$ and insist $(\ad\, e_3)^4(f_3)=\la e_3$. This gives $f_3:=\lambda(2f_{\subalign{11&10000\\&0}}+f_{\subalign{01&10000\\&1}})$, with $\langle e_3,f_3\rangle \cong W(1;1)$. Finally, in $\mathcal{O}({A_4}+A_3)$ we find $f_4:=\lambda(2f_{\subalign{11&10000\\&0}}+f_{\subalign{01&10000\\&1}}+f_{\subalign{00&00111\\&0}})$ to produce a simple Lie algebra isomorphic to $W(1;1)$.
\end{proof}

\begin{prop}\label{nomaxwitt} Let $G$ be an algebraic group of type $E_8$ with Lie algebra $\g=\Lie(G)$ for $p=5$. If $\mathcal{O}$ is not nilpotent orbit of type $A_4+A_3$ or ${A_3}^2$, then the $W(1;1)$ subalgebras obtained are not maximal in $\g$. \end{prop}
\begin{proof}
From \pref{nilptbc}, we know the orbit $\mathcal{O}(A_3)$ with representative $e_{\al_1}+e_{\al_3}+e_{\al_4}$ has $f$ such that $\langle e,f \rangle\cong W(1;1)$. We may choose a representative for $h_{\tau}=X\partial$ as usual from \rref{uniqueco}, and so $X\partial \in \g_e(\tau,0) \cap \im(\ad\,e)$ where $\tau$ is our cocharacter from \cite{LT11}. This is still trivial, and to see this we could take a generic $v \in \g(\tau,-2)$ and insist $[e,v]=0$. This shows that $v=0$, and so $\im(\ad\,e)\cap\g_e(0)=0$. Further details of achieving this in GAP can be found in \aref{firstwittgap}.

The remaining orbit has a factor of type $A_{p-1}$, and so we may have more candidates for $X\partial$. In this case $\g_e(0) \cap \im(\ad,e)=\mathfrak{z}(\mathfrak{l})$, where the Levi subalgebra has type $A_4$. This has the effect of producing different $X\partial$, allowing $h+\lambda h_0$ for $h_0 \in \mathfrak{z}(\mathfrak{l})$ and $\lambda \in \mathbb{F}_p$.

To show $W(1;1)$ is not maximal, we find fixed vectors $w\ne 0\in \g$ for a generic generating set of $W(1;1)$. This rules out any hope for maximality. For the generic generating set of $W(1;1)$, we need $u=\sum_i x_iv_i \in \g$ such that $[e,u]=h$ and $[h,u]=u$. This ensures $u$ is a candidate for $X^2\partial$. Since $W(1;1)=\langle\partial,X^3\partial\rangle$ we use $u$ to find a candidate for $X^3\partial$ --- generic $v$ such that $[X\partial,v]=2v$ and $[e,v]=u$.

If $w$ is a fixed vector for $W(1;1)$, then $[X\partial,w]=[e,w]=0$. Hence, we are looking for fixed vector $w \in \g_e(\tau,0)$ as $[X\partial,w]=0$. Let $w=\sum_i x_i v_i$ for $v_i \in \g$ and $x_i \in \bbk$. We consider $[v,w]$, which gives a collection of linear equations all required to be zero. It follows that $[v,w]=0$ if and only if we can solve such a system. The coefficients of our Lie bracket are linear equations in the indeterminates $x_i$, and hence we may form a matrix $A$ as in \cite[Appendix]{HS14}.

It follows that $W(1;1)$ is not maximal if there are more indeterminates than the rank of $A$. For $e$ of type $A_3$ and $A_4$ we have $\dim \,\g_e(0)=55$ and $\dim \,\g_e(0)=24$ respectively. This is enough to see that $w$ always has ``enough'' indeterminates for $[v,w]=0$ to have a non-zero solution. This is the case for each choice of $X\partial$ in the orbit $A_4$. Hence $[W(1;1),w]=0$, and so $W(1;1)\oplus\bbk w$ is a strictly bigger subalgebra.
\end{proof}

Note that, in the Appendix we give a worked through example of finding our fixed vector $w$ in the case of $A_3$. For this the reader is referred to \aref{fixedvectorapp}.
\begin{rem}For the orbits $\mathcal{O}({A_3}^2)$ and $\mathcal{O}(A_4+A_3)$ the idea in \pref{nomaxwitt} is inconclusive. The size of $\g_e(\tau,0)$ is smaller than the number of indeterminates we find. Hence, it appears likely we would need to find a non-trivial abelian subalgebra normalised by $W(1;1)$.\end{rem}

%\end{document}

\chapter[Maximal subalgebras in fields of characteristic three]{Maximal subalgebras over fields of characteristic three}\label{sec:Ermolaev}
We consider exceptional Lie algebras over fields of characteristic three, and produce new examples of maximal subalgebras. This begins with the Ermolaev algebra in the exceptional Lie algebra of type $F_4$, which was the main result of the article \cite{Pur16} for the author of this thesis. After this, we find several examples of non-semisimple maximal subalgebras that behave similarly to \cite[Theorem 4.2]{P15}.

\section[Simple Lie algebras in characteristic three]{Simple Lie algebras in characteristic three}\label{sec:charthree}

There is no complete classification of simple Lie algebras in fields of characteristic three. However, there are some results that help identify simple Lie algebras in certain situations.

For the simple Lie algebras we encounter \cite[\S4.4]{Strade04} will be our main reference for their constructions. This gives a reasonably comprehensive list of the currently known simple Lie algebras for characteristic three.

Then, in addition with \cite{ko95, KU89, Sk92} we are able to identify simple Lie algebras when they have depth-one gradings. We start with a recognition theorem for depth-one graded Lie algebras for $p=3$.

\begin{thm}\cite[Theorem 1]{ko95} Let $L=\bigoplus_{i=-1}^r L_i$ be a simple depth-one graded Lie algebra with classical simple $L_0$ component (modulo a $0, 1$-dimensional centre), and $L_{-1}$ a restricted $L_0$-module over an algebraically closed field of characteristic $p \ne 2$. Then $L$ is either of classical type, or Cartan type.\end{thm}

This is quite a powerful result, and is very similar to the Weak Recognition theorem \cite[Theorem 2.66]{BGP09} in characteristic three in the case of depth-one gradings. We are able to still use the Weak Recognition theorem, as the next result shows using the exact same statement as in \cite[Theorem 2.66]{BGP09}.

\begin{thm}[Weak Recognition Theorem]\index{Weak Recognition Theorem}\label{weakrec} Let $\g=\bigoplus_{i=-2}^r \g_i$ be a finite-dimensional graded Lie algebra over an algebraically closed field of characteristic $p>2$. Assume that: \begin{enumerate}\item{$\g_0$ is isomorphic to $\mathfrak{gl}(m)$, $\mathfrak{sl}(m)$, $\mathfrak{sp}(2m)$, or $\mathfrak{csp}(2m)$.}\item{$\g_{-1}$ is a standard $\g_0$-module of dimension $m$ or $2m$.}\item{If $\g_{-2} \neq 0$, then it is one-dimensional and equal to $[\g_{-1},\g_{-1}]$, and $\g_0 \cong \mathfrak{csp}(2m)$.}\item{If $x \in \bigoplus_{j \ge 0}\g_j$ and $[x,\g_{-1}]=0$, then $x=0$.}\item{If $x \in \bigoplus_{j \ge 0}\g_{-j}$ and $[x,\g_{1}]=0$, then $x=0$.}\end{enumerate} Then either $\g$ is a Cartan type Lie algebra with the standard grading and \[X(m;\underline{n})^{(2)}\subseteq \g\subseteq X(m;\underline{n}),\] where $X(m;\underline{n})$ is a Cartan type Lie algebra; or $\g$ is a classical simple Lie algebra of type $A_n$ or $C_n$.\end{thm}
\begin{proof}Note that this is a slight generalisation to \cite{BGP09} as we can allow for $p=3$. This holds since the results \cite[Proposition 2.7.3 and Lemma 5.2.3]{Strade04} are still true for $p=3$.\end{proof}

These rely on the simple Lie algebra having a depth-one grading. A complete recognition theorem in the sense of \cite[Theorem 0.1]{BGP09} for graded Lie algebras \[L=\bigoplus_{i=-q}^r L_i, \quad\text{and}\quad [L_i,L_j]\subseteq L_{i+j},\] is still out of reach.

Our examples of non-semisimple maximal subalgebras all exhibit depth-one gradings. Hence, we are able to identify the simple Lie algebras that appear in exceptional Lie algebras as subquotients of some maximal subalgebras. There are examples of simple Lie algebras that are neither Cartan nor classical that exhibit depth-one gradings with the Ermolaev algebra a first example. The key difference is to consider non-semisimple $L_0$.

We have the usual map \[\divs: W(2;\underline{1}) \rightarrow \mathcal{O}(2;\underline{1}),\] such that $\divs(\sum f_i \,\partial_i)=\sum \partial_i(f_i)$, and for any $\alpha \in \mathbbm{k}$ there is a $W(2;\underline{1})$-module denoted by $\mathcal{O}(2;\underline{1})_{(\alpha\,\divs)}$ obtained by taking $\mathcal{O}(2;\underline{1})$ under the action \[D\cdot f:=D(f)+\alpha \,\divs(D)f,\] for all $D \in W(2;\underline{1})$ and $f \in \mathcal{O}(2;\underline{1})$ as in \eqref{omod}.

The restricted \emph{Ermolaev} algebra\index{$\mathrm{Er}(1;1)$, Ermolaev algebra} as a vector space is $W(2;\underline{1})\oplus \mathcal{O}(2;\underline{1})_{(\divs)}$ using the particular case of $\alpha=1$ denoted by $\mathrm{Er}(1;1)$. This admits an automorphism of order two with $1$-eigenspace $W(2;\underline{1})$ and $(-1)$-eigenspace $\mathcal{O}(2;\underline{1})_{(\divs)}$. The Lie bracket is given by \begin{equation}\label{erbracket}[f,g]:=(f\partial_2(g)-g\partial_2(f))\partial_1+(g\partial_1(f)-f\partial_1(g))\partial_2,\end{equation} for all $f,g \in \mathcal{O}(2;\underline{1})$ with all other products defined canonically. It should be clear that $W(2;\underline{1})$ is generated as a Lie algebra by taking the Lie brackets of elements in $\mathcal{O}(2;\underline{1})_{(\divs)}$. For example, $[X_1,1]=\partial_1$ and so on.

To obtain a simple Lie algebra from this, we note that if $\alpha=1$, then $\mathcal{O}(2;\underline{1})_{(\alpha\,\divs)}$ has a submodule of codimension $1$ denoted by $\mathcal{O}'(2;\underline{1})_{(\divs)}$ using \cite[Proposition 4.3.2, (1)]{Strade04}. This module with $\alpha=2$ is also used when describing the Melikyan algebra from \pref{melikdef}. The derived subalgebra of $\mathrm{Er}(1,1)$ is equal to \[W(2;\underline{1})\oplus \mathcal{O}'(2;\underline{1})_{(\divs)},\] and is simple of dimension $26$. This inherits a $\mathbb{Z}$-grading from $W(2;\underline{1})$ with
\begin{equation}\label{ermograding}\mathrm{Er}(1,1)_i:=W(2;\underline{1})_i\oplus \mathcal{O}(2;\underline{1})_{i+1},\end{equation}
for all $i \ge -1$. We immediately have
$\mathrm{Er}(1,1)_{-1}=\mathbbm{k}\partial_1+\mathbbm{k}\partial_2+\mathbbm{k}1$, \[\dim\,L_0=6 \quad\text{and}\quad \rad(L_0)\ne 0.\] The radical is $3$-dimensional and non-central such that $L_0/\rad(L_0)\cong \mathfrak{sl}(2)$.

It is possible to generalise this construction to all vectors $\underline{n} \in \mathbb{N}^2$, and consider
\[\mathrm{Er}(n_1,n_2):=W(2;\underline{n}) \oplus \mathcal{O}(2;\underline{n})_{(\divs)}\]
to produce simple Lie algebras of dimension $3^{n_1+n_2+1}-1$ resulting in the so-called \emph{Ermolaev} series. This class of simple Lie algebras was first constructed in \cite{er82}, and is neither classical nor Cartan type. They only appear in algebraically closed fields of characteristic three. To see this consider the Jacobi identity in the Ermolaev series for all $p>0$. This gives \begin{align*}J(x_1\partial_1,x_1,x_2)&=[x_1\partial_1,[x_1,x_2]]+[x_2,[x_1\partial_1,x_1]]+[x_1,[x_2,x_1\partial_1]]\\&=3(x_1\partial_1+x_2\partial_2),\end{align*} which is clearly only zero for $p=3$.

Using the Ermolaev algebra we may define a $10$-dimensional simple Lie algebra first constructed in \cite{Ko70}, and a new series of Lie algebras sometimes referred to as Lie algebras of \emph{Frank} type. The first of these was considered in \cite{fr73}.

We denote these cases by $S$ and $T$ respectively as in \cite{KU89}. For the $10$-dimensional case we define a Lie algebra by $S=S_{-1}\oplus S_0 \oplus S_1$ where $S_{-1}\cong\mathrm{Er}(1,1)_{-1}$, $S_0= \langle 1, x, \partial, x\partial \rangle$ and $S_1=S_0^{(1)}$ with $S_0^{(2)}=0$.

For $T$ we start with the construction in \cite{KU89}, where we consider $T_{-1}\cong\mathrm{Er}(1,1)_{-1}$, $T_0 \cong \langle 1, x, x^2, x\partial , \partial\rangle$ and $T_1=T_0^{(1)}$. This gives the first simple Lie algebra of \emph{Frank type}, with details on how to extend to a full series given in \cite{Sk92}. We do not encounter either $S$ or $T$ in this work but we include the definition for completeness.

It turns out that there are very few simple depth-one graded Lie algebras with $L_0$ component that it is non-semisimple with non-central nilpotent radical. This is seen using the extremely useful result from \cite{KU89}, which classifies such cases.

\begin{thm}\cite[Theorem]{KU89} Let $L=\bigoplus_{i=-1}^r L_i$ be a simple graded Lie algebra with $L_0$ component non-semisimple containing a non-central radical over an algebraically closed field of characteristic $p>2$. Then either $L$ is isomorphic to one of the Lie algebras of Cartan types $W(m;\underline{n})$, $S(m;\underline{n})^{(1)}$, or $K(2m+1;\underline{n})$; or else $p=3$ and $L$ is a Lie algebra of Ermolaev type, Frank type or a simple Lie algebra $S$ of dimension $10$. \end{thm}

This result reduces our problem to identifying the zero component of the gradings. This follows since in all our cases we begin with a nilpotent orbit, and then for each representative there is an associated cocharacter. Hence, we obtain a grading on our non-semisimple subalgebra. If this is depth-one graded, then we can apply the above result.

In the next section of this chapter we consider \cite{Pur16}, which encounters a depth-one graded simple Lie algebra $L$ such that $L_0$ is non-semisimple.

To finish our small list of simple Lie algebras in characteristic three we consider some cases considered in \cite{Sk92}, one of which has links to the exceptional Lie algebra of type $E_6$ (see \tref{none6}). These may be referred to as \emph{Skryabin algebras} \index{$\mathcal{S}_i(\underline{n})$, for $i=1,2,3$ --- Skryabin algebras}. We study graded simple Lie algebras $\g$ such that $[\g_i,\g_{-1}]=\g_{i-1}$ for all $i<0$, and $[a,\g_{-1}]\ne 0$ for $0 \ne a \in \g_i$, $i>0$. The series found in \cite{Sk92} all satisfy the condition $\dim\, \g_{-1}=\dim\, \g_{-2}=3$.

We use the construction in \cite[\S4.4]{Strade04}, and for $n \in \mathbb{N}^3$ let $\widetilde{S(3;\underline{n})}$ be a copy of $S(3;\underline{n})$. Consider the vector space \[\mathcal{S}_1(\underline{n}):=W(3;\underline{n})\oplus\mathcal{O}(3;\underline{n})_{(-\divs)}\oplus\Omega^1(3;\underline{n})_{(\divs)}\oplus\widetilde{S(3;\underline{n})}_{(\divs)},\] for $\Omega^1$ defined as in \eqref{omega}. We define the Lie bracket in $W(3;\underline{n})$ canonically, and for the remaining elements as follows \begin{align*}[f,f']&:=f'df-fdf',\\
\left[f,\sum g_idx_i\right]&:=(\partial_3(fg_2)-\partial_2(fg_3))\tilde{\partial}_1+(\partial_1(fg_3)-\partial_3(fg_1))\tilde{\partial}_2\\&+(\partial_2(fg_1)-\partial_1(fg_2))\tilde{\partial}_3,\\
\left[f,\tilde{D}\right]&:=fD,\\
\left[\sum f_idx_i,\sum g_jdx_j\right]&:=(f_2g_3-f_3g_2)\partial_1+(f_3g_1-f_1g_3)\partial_2+(f_1g_2-f_2g_1)\partial_3,\\
\left[\sum f_idx_i,\sum g_j\tilde{\partial}_j\right]&:=\sum f_ig_i,\\
\left[\sum f_i\tilde{\partial}_i,\sum g_j\tilde{\partial}_j\right]&:=(f_3g_2-f_2g_3)dx_1+(f_3g_1-f_1g_3)dx_2+(f_1g_2-f_2g_1)dx_3,\end{align*} for $f,f' \in \mathcal{O}(3;\underline{n})$, $\sum f_idx_i,\sum g_idx_i\in \Omega^1(3;\underline{n})$, $\sum f_i\tilde{\partial}_i,\sum g_j\tilde{\partial}_j \in \widetilde{S(3;\underline{n})}$ and $D \in W(3;\underline{n})$. This has grading given by \begin{align*}\degs\, x^{(a)}\partial_i&:=4|a|-4,\\
\degs\, x^{(a)}&:=4|a|-3,\\
\degs\, x^{(a)}dx_i&:=4|a|-2,\\
\degs\, x^{(a)}\tilde{\partial}_i&:=4|a|-1.\end{align*}

For the Lie bracket above to be well-defined the Jacobi identity must be satisfied. This is where it becomes imperative that our field has characteristic three. We have $\dim\,\mathcal{S}_{1}(\underline{n})=3^{|n|+2}+1$ and $\mathcal{S}_1(\underline{n})$ is depth-four graded with ${\mathcal{S}_1(\underline{n})}_{-4}=\bbk \partial_1+\bbk \partial_2+\bbk \partial_3$. We also have that \[{\mathcal{S}_1(\underline{n})}_{0}=\sum_{1 \le k,l\le 3} \bbk x_k\partial_l \cong \mathfrak{gl}(3).\] Computing Lie brackets we see that \begin{align*}\left[W(3;\underline{n}),\widetilde{S(3;\underline{n})}_{(\divs)}\right]&\subset \widetilde{S(3;\underline{n})}^{(1)}_{(\divs)},\\
\left[\mathcal{O}(3;\underline{n})_{(-\divs)},\Omega^1(3;\underline{n})_{(\divs)}\right]&\subset \widetilde{S(3;\underline{n})}^{(1)}_{(\divs)},\\
\left[\Omega^1(3;\underline{n})_{(\divs)},\Omega^1(3;\underline{n})_{(\divs)}\right]&\subset \widetilde{S(3;\underline{n})}^{(1)},\end{align*} and so the derived subalgebra \[\mathcal{S}_1(\underline{n})^{(1)}=W(3;\underline{n})\oplus\mathcal{O}(3;\underline{n})_{(-\divs)}\oplus\Omega^1(3;\underline{n})_{(\divs)}\oplus\widetilde{S(3;\underline{n})}^{(1)}_{(\divs)}\] is a direct sum of pairwise non-isomorphic irreducible $W(3;\underline{n})$-modules. Hence $\mathcal{S}_1(\underline{n})^{(1)}$ is simple with $\dim\, \mathcal{S}_1(\underline{n})^{(1)}=3^{|n|+2}-2$.

\begin{thm}\cite[Theorem 4.3]{Sk92} Let $\g$ be a Lie algebra satisfying the general conditions above, such that $\g_i=0$ for $i<-4$, $\g_{-4}$ is $3$-dimensional and $\g_{-3}$ is one-dimensional. Further if $\g_0 \cong \mathfrak{gl}(\g_{-1})$. Then either $\g$ is generated by $\g_{-1}$ and $\g_{1}$, in which case $\g$ is simple of dimension $29$, or $\mathcal{S}_1(\underline{n})^{(1)}\subseteq \g \subseteq \mathcal{S}_1(\underline{n})$.\end{thm}

We define some subalgebras of $\mathcal{S}_1(\underline{n})$ to produce other simple Lie algebras. Consider the Lie algebra \[\mathcal{S}_2(\underline{n}):=W(3;\underline{n})\oplus \Omega^1(3;\underline{n})_{(\divs)},\] which is depth-two graded with $({\mathcal{S}_2(\underline{n})})_{2i}:=W(3;\underline{n})_i$ for $i \ge 2$ and $({\mathcal{S}_2(\underline{n})})_{2i-3}:=\Omega^1(3;\underline{n})_i$ for $i \ge 1$. This satisfies $\dim\,\g_{-3}=\dim\,\g_{-2}=3$, and ${(\mathcal{S}_2(\underline{n}))}_{0}\cong \mathfrak{gl}(3)$. This is a direct sum of non-isomorphic irreducible $W(3;\underline{n})$-modules, and so $\mathcal{S}_2(\underline{n})$ is simple with dimension $2(3^{|n|+1})$.

To be able to state one more key result from \cite{Sk92}, we need to venture slightly into the deformation theory of $S(m;\underline{n})^{(1)}$. This will produce simple subalgebras of $\mathcal{S}_2(\underline{n})$ that depend on which volume form we use --- our main focus will be on $\omega_S=dx_1\wedge dx_2\wedge dx_3$. This is just the usual volume form used in the definition of $S(m;\underline{n})$.

By \cite[Theorem 6.3.4]{Strade04} the only volume forms to consider are $\omega_S$, $(1+x^{(\tau(n))})\omega_S$ and $\exp(x_{i}^{(p^{n_j})})\omega_S$ with $i=1,\ldots,3$. Consider \[\mathcal{S}_3(n;\omega):=S(3;\underline{n};\omega)\oplus ud(u^{-1}\mathcal{O}(3;\underline{n}))_{(\divs)},\] where $S(3;\underline{n};\omega)=(u^{-1}S(3))\cap W(3;\underline{n})$ and $d$ from \eqref{dmapping}. It follows that $ud(u^{-1}) \in \Omega^1(3;\underline{n})$ and $\mathcal{S}_3(\underline{n};\omega)$ is a subalgebra of $\mathcal{S}_2(\underline{n})$ in all cases.

For $\omega_S$, we have $\mathcal{S}_3(\underline{n};\omega)=S(3;\underline{n})\oplus d\mathcal{O}(3;\underline{n})_{(\divs)}$. However, both $\mathcal{O}(3;\underline{n})_{(\divs)}$ and $S(3;\underline{n})$ have submodules. Taking the derived subalgebra we obtain \[\mathcal{S}_3(\underline{n};\omega)^{(1)}=S(3;\underline{n})^{(1)}\oplus d\mathcal{O}'(3;\underline{n})_{(\divs)}.\]This has dimension $3^{|n|+1}-4$. For the remaining two cases we obtain simple Lie algebras of dimensions $3^{|n|+1}-3$ and $3^{|n|+1}-1$ respectively with the reader referred to \cite{Strade04} for further details.

\begin{thm}\cite[Theorem 5.2]{Sk92}. Let $\g$ be a depth-two graded Lie algebra over a perfect field, such that $\g_0$ is either $\mathfrak{sl}(\g_{-1})$ or $\mathfrak{gl}(\g_{-1})$. Suppose that $\g_1 \ne 0$, then one of the following holds:
\begin{enumerate}\item{$\dim\, \g_1=9$, and $\g \cong \mathcal{S}_2(\underline{n})$.}\item{$\dim \,\g_1=6$, and $\mathcal{S}_3(\underline{n},\omega)^{(1)}\subseteq \g\subseteq \mathcal{S}_3(\underline{n},\omega)$.}\end{enumerate}\end{thm}

\section[The Ermolaev algebra and $F_4$]{The Ermolaev algebra and $F_4$}\label{sec:ermoax}

Let $G$ be an algebraic group of type $F_4$ with Lie algebra $\g=\Lie(G)$. In \cite[Remark 1.4]{HS14}, there is an example of a $26$-dimensional simple subalgebra $L$ in $\mathfrak{g}$, that is neither classical nor $W(1;\underline{1})$. The dimension of $L$ equals both the dimension of $\mathrm{Er}(1;1)^{(1)}$ and $K(3;\underline{1})$ of the known simple Lie algebras for fields of characteristic three.

Consider the nilpotent orbit denoted $F_4(a_1)$ with orbit representative $e:=e_{1000}+e_{0100}+e_{0001}+e_{0120}$. This is a different representative to the one given in the tables of \cite{S16}, so we must verify that this still lies in the same orbit for characteristic three.

It is easy to check that $\dim\,\g_e=6$, and so the orbit $\mathcal{O}(e)$ is subregular because there is only one nilpotent orbit of codimension $6$ in $\g$ by \cite[Theorem 4]{hs85}. Hence, we may label the nilpotent orbit $F_4(a_1)$ as in the characteristic zero case and continue to use the associated cocharacter $\tau$ with weighted diagram $2202$ from \cite[pg. 79]{LT11} by \cite{CP13}.

We obtain the following information using GAP.

\begin{enumerate}\item{For $f:=f_{1232}$, the subalgebra $L:=\langle e,f \rangle$ is simple and of dimension $26$.}\item{$N_{\g}(L)=L$.}\item{For $f':=f_{1222}-f_{1242}\in L$, the subalgebra $W:=\langle e,f' \rangle$ is simple with dimension $18$.}\end{enumerate}

To prove $L$ is isomorphic to the Ermolaev algebra we locate a complementary subspace to $W$ in $L$, and use this to give a new grading on $L$ that resembles the standard grading of $\mathrm{Er}(1;1)^{(1)}$. We then apply \cite[Theorem 1]{KU89} to conclude $L \cong \mathrm{Er}(1;1)^{(1)}$.

\begin{prop}\label{ermoisom}\cite[Proposition 3.1]{Pur16} Let $G$ be an algebraic group of type $F_4$ with Lie algebra $\g=\Lie(G)$ for $p=3$. The subalgebra $L=\langle e,f \rangle$ is isomorphic to $\mathrm{Er}(1;1)^{(1)}$.\end{prop}
\begin{proof} Applying the cocharacter $\tau$ from \cite[pg. 79]{LT11} to the basis elements of $L$ establishes the degree of each homogenous element of $L$, and thus produces a grading on $L$. This gives the following table:

\begin{center}\begin{tabular}{|P{1.7cm}|P{0.8cm}|P{0.8cm}|P{0.8cm}|P{0.69cm}|P{0.69cm}|P{0.69cm}|P{0.69cm}|P{0.65cm}|P{0.65cm}|P{0.65cm}|P{0.65cm}|P{0.65cm}}\hline $i$ &$-14$&$-12$&$-10$&$-8$&$-6$&$-4$&$-2$&$0$&$2$&$4$&$6$ \\ \hline $\dim\,L(i)$&$1$&$1$&$3$&$3$&$3$&$3$&$3$&$3$&$3$&$2$&$1$\\\hline\end{tabular}\end{center}

To find a suitable complementary module to $W$ in $L$, we find the element corresponding to $1 \in \mathcal{O}'(2;\underline{1})_{(\divs)}$ as an element of $L$, and construct an $8$-dimensional vector space $V$ with the action of $W$ on such an element.

Consider $\ker(\ad\,e) \cap L(4)$, which is one-dimensional with basis element $v:=e_{0111}-e_{1110}$. Using GAP to compute the $\mathbbm{k}$-span of all brackets $[x,v]$ for $x \in W$ produces a module $V$ such that $L=W \oplus V$. We will show that this is exactly what we need to conclude that $L \cong \mathrm{Er}(1;1)^{(1)}$.

Recall from \eqref{erbracket} that $W(2;\underline{1})$ is generated as a Lie algebra by all brackets of the form $[u,v]$ for $u,v \in \mathcal{O}(2;\underline{1})_{(\divs)}$. We generate a subalgebra in GAP by the set $[V,V]:=\mathrm{span}_{\mathbbm{k}}\{[u,v]: u,v \in V\}$. This produces a simple $18$-dimensional Lie algebra, containing both $e$ and $f'$. In particular, $\langle[V,V]\rangle\cong W$.

The complete details on obtaining $V$, and all the subsequent calculations performed in GAP above are given in \aref{appprop}.

Using $\tau$ we obtain graded components of $V$ such that $V$ is graded in degrees $-10 \le i \le 4$ with each even graded component one-dimensional. For each homogenous component of $V$ we give a new integer $d(i)$ to replace the old degree.

\begin{center}\begin{tabular}{|M{1cm}|P{3.5cm}|M{1cm}|}\hline $i$ & Basis element of $V(\tau,i)$ &$d(i)$ \\ \hline\hline $4$&$e_{0111}-e_{1110}$&$-1$\\ \hline $2$&$e_{0011}-e_{0110}$&$0$\\ \hline $0$&$e_{0010}$&$1$ \\ \hline $-2$&$f_{0011}+f_{0110}$&$0$ \\ \hline $-4$&$f_{0111}+f_{1110}$&$1$\\ \hline $-6$&$f_{1111}$&$2$ \\ \hline $-8$&$f_{1231}$&$1$\\ \hline $-10$&$ f_{1232}$&$2$\\ \hline\end{tabular}\end{center}
Label elements as $v(i)$ to represent the basis element given above. For example, $v(4):=e_{0111}-e_{1110}$. Since all elements of $L$ may be written as Lie brackets of the elements of $V$, this gives a new grading on $L$. We verify this is well-defined via GAP, checking that $[\g(i),\g(j)]=\g(i+j)$. We obtain \[L_{-1}=\spnd_{\bbk}\{e_{0111}-e_{1110}, e_{1121}+e_{0122}-e_{1220}, e_{0001}+e_{1000}+e_{0100}\},\] by computing the non-zero elements $v(4), [v(4),v(2)]$ and $[v(4),v(-2)]$. This gives $\dim \,L_{-1}=3$, and taking the obvious Lie brackets of elements along with the elements of degree $0$ from $V$ we have \[L_0=\langle v(2),v(-2),[v(4),v(0)], [v(4),v(-4)], [v(4),v(-8)], [v(2),v(-2)]\rangle.\] This is a $6$-dimensional Lie algebra consisting of an $\mathfrak{sl}(2)$ triple \[\{e_1,f_1,h_1\},\] with $e_1:=e_{0121}+e_{1120}, f_1:=f_{0121}+f_{1120}$ and $h_1:=h_{\alpha_1}+h_{\alpha_4}$ along with a $3$-dimensional radical with basis $\{v(2),v(-2),[v(-2),v(2)]\}$. This gives a depth-one graded simple Lie algebra with $L_0$ component containing a non-central nilpotent radical. We apply \cite[Theorem 1]{KU89} to conclude $L \cong \mathrm{Er}(1;1)^{(1)}$ as required.
\end{proof}

\begin{coro}\cite[Corollary 3.2]{Pur16} Let $G$ be an algebraic group of type $F_4$ with Lie algebra $\g=\Lie(G)$ for $p=3$ and $L=\langle e,f \rangle$. Then the subalgebra $W:=\langle e,f' \rangle$ is isomorphic to $W(2;\underline{1})$. \end{coro}\begin{proof}Since every element of $W$ is obtained by taking Lie brackets of elements in $V$, we obtain a grading on $W$. We have $W_{-1}$ is two-dimensional with basis elements $e_{1121}+e_{0122}-e_{1220}$ and $e_{0001}+e_{1000}+e_{0100}$.

These are obtained taking the brackets $[v(4),v(2)]$ and $[v(4),v(-2)]$ respectively. We calculate the degree $0$ component in the same way to obtain \[W_0=\langle[v(4),v(0)], [v(4),v(-4)], [v(4),v(-8)], [v(2),v(-2)]\rangle,\] which consists of an $\mathfrak{sl}(2)$ triple \[\{e_1,f_1,h_1\},\] with $e_1:=e_{0121}+e_{1120}, f_1:=f_{0121}+f_{1120}$, $h_1:=h_{\alpha_1}+h_{\alpha_4}$ and central element $h_{\alpha_2}+h_{\alpha_4}$. This gives a depth-one graded simple Lie algebra with classical simple $W_0$ (modulo its centre). Using \cite[Theorem 1]{ko95} in combination with $\dim\,W=18$ gives $W \cong W(2;\underline{1})$.\end{proof}

We will make use of the fact that for $p \ge 3$, we still have a non-degenerate symmetric form for Lie algebras of type $F_4$. This together with the fact that $L=N_{\mathfrak{g}}(L)$ will allow us to conclude that $L$ is a maximal subalgebra in $F_4$.

\begin{lem}\label{selfdual}The adjoint module of the Lie algebra $L:=\mathrm{Er}(1;1)^{(1)}$ is not self-dual.\end{lem}\begin{proof} In the standard grading of the Ermolaev algebra given in \eqref{ermograding}, we have that $\dim\,L_{-1}=3$ and $\dim\,L_3=2$. Hence, $L_{-1}\ncong (L_3)^{\ast}$ as $L_0$-modules. It follows by \cite[Lemma 4]{Pre85} that $L$ is not self-dual. \end{proof}

\begin{thm}\label{ermax}\cite[Theorem 1.1, Proposition 4.2]{Pur16} Suppose $G$ is an algebraic group of type $F_4$ over an algebraically closed field of characteristic three with Lie algebra $\g=\Lie(G)$.

Let $e:=e_{1000}+e_{0100}+e_{0001}+e_{0120}$ be a representative for the nilpotent orbit denoted $F_4(a_1)$. Then, for $f:=f_{1232}$ the subalgebra $L:=\langle e,f \rangle \cong \Er{1}$ is a maximal subalgebra of $\g$.\end{thm}

\begin{proof} By \pref{ermoisom} $L$ is isomorphic to the Ermolaev algebra. It is well-known that the adjoint module of $\g$ is self-dual, and admits a non-degenerate invariant symmetric form denoted by $\kappa$. This is the so-called normalised Killing form defined explicitly \cite[pg. 661]{CP13}. Since $\dim\,L=\frac{1}{2}\dim\,\mathfrak{g}=26$, and the adjoint module of $L$ is not self-dual by \lref{selfdual} we have $L$ is a maximal totally isotropic subspace of $\g$ with respect to $\kappa$.

Suppose $M$ is a proper subalgebra strictly containing $L$, then the restriction of $\kappa$ to $M$ is non-zero with non-zero radical $R$. This follows since $L$ is totally isotropic with $\dim\,L > \frac{1}{2}\dim\,M$. If $L \cap R=0$, then $L \oplus R$ is a totally isotropic subalgebra containing $L$. Hence, $R \cap L=L$ and $R=L$. This now provides the necessary contradiction, since $M \subseteq N_{\mathfrak{g}}(L)=L$.\end{proof}

\section[Counterexamples to Morozov's theorem: $F_4$]{Counterexamples to Morozov's theorem: $F_4$}\label{blahblah}

The classification of the non-semisimple maximal subalgebras in exceptional Lie algebras was recently obtained in good characteristic by \cite{P15}. In the same article \cite[Theorem 4.2]{P15} shows $p$ good is a necessary condition to obtain that they are parabolic subalgebras.

This example is given in \tref{sec:e8p5non} uses the orbit $\mathcal{O}(A_4+A_3)$, and produced a strange example of a non-semisimple maximal subalgebra, which we considered in \autoreft{maxnonseme8}. We aim to produce similar examples, and for this we consider nilpotent orbits with representatives $e$ such that \begin{enumerate}[label=(\alph*)]\item{$\g_e\ne\Lie(G_e)$,}\item{$e^{[p]}=0$.}\end{enumerate}

We continue with $G$ an algebraic group of type $F_4$ with Lie algebra $\g=\Lie(G)$. Consider the nilpotent orbit denoted by $\mathcal{O}(\widetilde{A_2}+A_1)$ with representative $e:=e_{1000}+e_{0010}+e_{0001}$. This representative is the same for all $p$ using the tables of \cite{LT11} and \cite{S16}, hence we may use the associated cocharacter $\tau$ with diagram \[2\quad{-5}\quad2\quad2\] from \cite[pg. 76]{LT11}.

In \cite[Section 4]{lmt08}, there are two new linearly independent non-zero elements of the centraliser given by $X,Y \in \g_e(\tau,-1)$ such that $X=e_{0111}+e_{1110}+2e_{0120}$ and $Y=2f_{1111}-2f_{1120}+f_{0121}$. Since $\Lie(G_e) \subseteq \bigoplus_{i\ge 0} \g_e(\tau,i)$ has dimension $16$, and $\dim\,\g_e=18$ we deduce \[\g_e=\bigoplus_{i \ge -1} \g_e(\tau,i) \quad \text{and}\quad \dim\,\g_e(\tau,-1)=2.\] The vectors contained in $\bigoplus_{i \ge 0} \g_e(\tau,i)$ from \cite[pg. 76]{LT11} are still linearly independent for characteristic three, and so continue to form a basis for $\Lie(G_e)$. We obtain the following using GAP,

\begin{enumerate}
\item{The radical $A$ of $\g_e$ and $\mathfrak{n}_e$ is abelian and of dimension $8$;}
\item{the normalizer $\w_{F_4}:=N_{\g}(A)$ has dimension $26$;}
\item{the Lie algebra $\w_{F_4}/A$ is simple and restricted of dimension $18$.}
\end{enumerate}

\begin{rem}In all our examples, for any nilpotent $e$ the radical $A$ of $\g_e$ and $\mathfrak{n}_e$ will always be the same. This follows since $\dims\,\mathfrak{n}_e=\dims\,\g_e+1$, and the new element $h$ is such that $[h,e]=\la\,e$ for some scalar $\la$. Hence, $h \in \g(\tau,0)$ and then it is easy to verify that the radical $A$ is contained in $\g(\tau,\ge 2)$ using GAP. \end{rem}

\begin{thm}\label{nonf4}Let $G$ be an algebraic group of type $F_4$ over an algebraically closed field of characteristic three with Lie algebra $\g=\Lie(G)$. For any $e \in \mathcal{O}(\widetilde{A_2}+A_1)$ we have the following \begin{enumerate}[label=(\alph*)]
\item{\label{partaf4}$A=\rad (\g_e)$ and $\g_e/A \cong H(2;\underline{1})$ as Lie algebras.}
\item{\label{partbf4}$A=\rad (\mathfrak{n}_e)$ and $\mathfrak{n}_e/A \cong CH(2;\underline{1})$ as Lie algebras.}
\item{$A=\rad(\w_{F_4})$ and $\w_{F_4}/A\cong W(2;\underline{1})$ as Lie algebras.}
\item{$\w_{F_4}$ is a maximal Lie subalgebra of $\g$.}\end{enumerate}\end{thm}

\begin{proof}
Let $M=\g_e/A$, then $M$ is a depth-one graded Lie algebra. This follows since $A \subseteq \g_e(\tau,\ge 2)$, which is easily computed by finding a basis of $A$ in GAP and checking the degree of each element. We also have $M_0 \cong \mathfrak{sl}(2)$ by \cite[pg. 76]{LT11}, and that $\g_e/A$ has a depth-one grading.

We have that $M_{-1}$ is a $2$-dimensional $M_0$-module. It is easy to verify that $[M_{-1},x]=0$ implies that $x=0$ for $x \in M_{>0}$. Similarly one can verify $[M_1,x]=0$ implies $x=0$ for all $x \in M_{-1}$. Hence, we may apply the Weak Recognition theorem from \tref{weakrec} to $M$ obtaining that $H(2;\underline{1})^{(2)} \subseteq M \subseteq H(2;\underline{1})$.

Then, since $\g_e/A$ is $10$-dimensional part \ref{partaf4} follows. Since $\mathfrak{n}_e=\g_e \oplus \bbk h$, it follows that both $\rad(\mathfrak{n}_e)=A$ and $(\mathfrak{n}_e)_0 \cong \mathfrak{sl}(2) \oplus \bbk h$. Hence $\mathfrak{n}_e/A \subseteq CH(2;\underline{1})$, thus giving part \ref{partbf4} since $\dim\,\mathfrak{n}_e/A=\dim\,CH(2;\underline{1})$.

The Lie algebra $\w_{F_4}/A$ is simple and restricted of dimension $18$ with cocharacter $\tau$ inducing a grading on the Lie algebra. This automatically forces $A=\rad(\w_{F_4})$. We have that $\w_{F_4}/A$ is a depth-one graded simple Lie algebra. This can be easily verified using GAP, checking to see that $\g_e(\tau,-1)$ is the lowest graded piece.

We can now identify the isomorphism class of $\w_{F_4}/A$. Since the radical $A$ is contained in $\bigoplus_{i>0}\w_{F_4}(\tau,i)$ we only need to consider $\w_{F_4}(\tau,0)$. This satisfies the requirements of \cite[Theorem 1]{ko95} as $\w_{F_4}(\tau,0)$ has dimension $4$ with a one-dimensional centre. Factoring out by this gives a simple Lie algebra of type $A_1$. Hence, $\w_{F_4}/A \cong W(2;\underline{1})$ by \cite{ko95}.

To prove $\w_{F_4}$ is a maximal subalgebra, we set $L_0:=\w_{F_4}$ and need to compute the set \[L_{-1}:=\{x \in \g: [x,A] \subseteq L_0\}.\] Then we need to show $L_{-1}/L_0$ is an irreducible $L_0/L_1$ module. In GAP we can write $\g$ as a $L_0$-module and consider all submodules in GAP, we are fortunate that there are only $4$ non-zero submodules. Computing $L_{-1}$ in GAP, we find this is a $44$-dimensional vector space such that the Lie algebra generated by $L_{-1}$ is isomorphic to $\g$, and $L_{-1}/L_0$ is an irreducible $L_0/L_1$-module.

The full details on finding this space is given in \aref{AppF42}. Any subalgebra of $\g$ that strictly contains $\w_{F_4}$ must contain $L_{-1}$ by \rref{whymax}, and so we conclude $\w_{F_4}$ is indeed maximal.
\end{proof}

\begin{rem}
If $p=3$, then $CH(2;\underline{1}) \ncong \Der\,H(2;\underline{1})^{(2)}$ by \cite[Theorem 7.1.2]{Strade04}, and hence we have $\mathfrak{n}_e/A \ncong \Der\, H(2;\underline{1})^{(2)}$. The subalgebra $\g'_e$ generated by $\g_e(\tau,\pm 1)$ is isomorphic to $H(2;\underline{1})$.

This differs from \cite[Theorem 4.2]{P15}, and this is because $x_1^2\partial_2, x_2^2\partial_1 \in H(2;\underline{1})$ have degree one in the standard grading when $p=3$. In particular, these are already contained in $\g_e(\tau,1)$. \end{rem}

We may use $L_{-1}$ to produce the Weisfeiler filtration from $\w_{F_4}$. We already have that the zero component is isomorphic to $W(2;\underline{1})$. This observation is similar to the observation regarding a maximal non-semisimple subalgebra in $E_8$ for $p=5$, where we were able to show that the corresponding graded Lie algebra was isomorphic to $S(3;\underline{1})^{(1)}$ in \tref{weise8}. This suggests we may have the exact same situation happening in $F_4$ for $p=3$ where the graded Lie algebra associated to the Weisfeiler filtration is isomorphic to $S(3;\underline{1})^{(1)}$.

There is a grading on $S(3;\underline{1})^{(1)}$ by associating degree $1$ to $x_3$ and zero to both $x_{1}, x_2$ as in \eqref{specgrad}. This gives \begin{align}\label{specf4grad}\begin{split}S_{-1}&:=\{{x_1}^i{x_2}^j\partial_3: 0 \le i,j<p , (i,j) \ne (p-1,p-1)\}, \\S_0&:=\{i{x_1}^{i-1}{x_2}^jx_3\partial_3-{x_1}^i{x_2}^j\partial_1,j{x_1}^i{x_2}^{j-1}x_3\partial_3-{x_1}^i{x_2}^j\partial_2:0 \le i,j<p\}.\end{split}\end{align}This is outlined in \cite[pg. 283]{PSt01}, with $S_0 \cong W(2;\underline{1})$, and $S_{-1} \cong (\mathcal{O}(2;\underline{1})/\bbk 1)^{\ast}.$

\begin{thm}\label{weisf4}Let $G$ be an algebraic group of type $F_4$ over an algebraically closed field of characteristic three with Lie algebra $\g=\Lie(G)$.

Consider the Weisfeiler filtration $\mathcal{F}$ associated to the pair $(L_{-1},\w_{F_4})$ from \tref{nonf4} with corresponding graded Lie algebra $\Gc$. Then, we have that $\Gc\cong S(3;\underline{1})^{(1)}$.\end{thm}
\begin{proof}
In the proof of \tref{weise8} we used the classification of simple Lie algebras to show $\mathcal{G} \cong S(3;\underline{1})^{(1)}$, but since the characteristic is three we no longer have such a classification. Our plan is to use \cite[Theorem 1]{ko95}, and for this we need to regrade our Lie algebra. The full Weisfeiler filtration associated to our irreducible $L_0$-module $L_{-1}$ is given in \aref{AppF42}, and we obtain \begin{enumerate}\item{$\dim \,L_{-2}=52$,}\item{$\dim \,L_{-1}=44$,}\item{$\dim \,L_{0}=\dim \mathfrak{w}_{F_4}=26$,}\item{$\dim\, L_1=\dim\, A=8$.}\end{enumerate}
It follows using GAP that each $\mathcal{G}_i$ is an irreducible $\mathcal{G}_0$-module. It is a simple calculation to show that $[L_{-2},L_1]=L_{-1}$, and hence $M(\Gc)=0$ using the same reasoning as in \tref{weise8}. This gives $\mathcal{G}$ is a simple Lie algebra by \pref{simplegraded}.

Let $\partial_z:=e$, $\partial_{x}:=\widetilde{X}$ and $\partial_{y}:=\widetilde{Y}$ be elements of $\mathcal{G}$ with degree $-1$ where we define $\widetilde{X}:=X+L_{1}$ and $\widetilde{Y}:=Y+L_1$ for $X, Y$ defined earlier. For the top component of $S(3;\underline{1})^{(1)}$ in the standard grading, only $u:={x}^2y{z}^2\partial_x-x{y}^2{z}^2\partial_y$ lies in the top component of the grading \eqref{specf4grad} above. 

Hence, there are $8$ possible elements of $\Gc_{(-2)}$ to satisfy the relations \begin{align}\label{relations}\begin{split}(\ad\,\partial_x)^2(\ad\,\partial_z)^2(\ad\,\partial_y)(u)&=\partial_x\\
(\ad\,\partial_y)^2(\ad\,\partial_z)^2(\ad\,\partial_x)(u)&=\partial_y\end{split}\end{align}

Searching for such elements in GAP we find $f_1:=(2f_{1000}+2f_{0001}+f_{0010})+L_{-1}$ is a representative for this coset. Similarly, we identify two more elements
\[f_2:=(f_{0111}+f_{1110}+f_{0120})+L_0 \quad\text{and}\quad f_3:=(e_{1111}+e_{0121}+e_{1120})+L_0.\] We give $\deg(f_i)=4$, and verify $[f_i,f_j]=0$ for all $i,j$. In $\mathcal{G}$ the multiplication is different, as $[f_{0111}+f_{1110}+f_{0120},e_{1111}+e_{0121}+e_{1120}]\ne 0$ in $\g$ but $[f_{0111}+f_{1110}+f_{0120},e_{1111}+e_{0121}+e_{1120}] \in L_{(-1)}$, which forces $[f_i,f_j]=0$.

Using $\mathcal{G}'_{4}$ and $\mathcal{G}'_{-1}$ we can build a grading on $\mathcal{G}$, which will automatically satisfy $[\mathcal{G}'_{i},\mathcal{G}'_{-1}]\subseteq \mathcal{G}'_{i-1}$. Now consider all Lie brackets of the form $[\partial_{x,y,z},f_i]$ for $i=1,2,3$,
\begin{align*}v_1&:=[\partial_x,f_1]=[X,f_1]+L_{-1}=(e_{1100}+2e_{0110})+L_{-1}
\\ v_2&:=[\partial_x,f_2]=[X,f_2]+L_{0}=(h_{\al_4}+h_{\al_1}+2h_{\al_3}+2h_{\al_2})+L_0
\\ v_3&:=[\partial_x,f_3]=[X,f_3]+L_{0}=(2e_{1231}+2e_{1222})+L_0
\\ v_4&:=[\partial_y,f_1]=[Y,f_1]+L_{-1}=f_{0122}+L_{-1}
\\ v_5&:=[\partial_y,f_2]=[Y,f_2]+L_0=(2f_{1231}+2f_{1222})+L_0
\\ v_6&:=[\partial_y,f_3]=[Y,f_3]+L_0=(2h_{\al_1}+2h_{\al_3}+h_{\al_2})+L_0
\\ v_7&:=[\partial_z,f_1]=[e,f_1]+L_0=(2h_{\al_4}+2h_{\al_1}+h_{\al_3})+L_0
\\ v_8&:=[\partial_z,f_2]=[e,f_2]+L_1=2f_{1100}+L_1
\\ v_9&:=[\partial_z,f_3]=[e,f_3]+L_1=e_{0122}+L_1
\end{align*}
In $S(3;\underline{1})^{(1)}$ we have $v_2+v_6=v_7$, and so in $\mathcal{G}$ these elements lie in the same coset. The element $a:=h_{\al_2}+h_{\al_3} \in L_0$, and clearly $v_7+a=2\cdot(v_2+v_6)$. Hence, the set $\{v_i\}$ is an $8$-dimensional vector space in $\mathcal{G}$. We show $[v_i,v_j]=0 \in \mathcal{G}$ for all $i,j$, and repeat this process to obtain $\mathcal{G}'_{2}$
\begin{align}\label{gradp3}\begin{split}u_1&:=[\partial_x,v_1]=[X,v_1]+L_{-1}=2e_{1220}+L_{-1}
\\ u_2&:=[\partial_x,v_2]=[X,v_2]+L_{0}=e_{0120}+L_0
\\ u_3&:=[\partial_x,v_3]=[X,v_3]+L_{0}=2e_{1342}+L_0
\\ u_4&:=[\partial_x,v_4]=[X,v_4]+L_{-1}=2f_{0011}+L_{-1}
\\ u_5&:=[\partial_x,v_5]=[X,v_5]+L_0=(2f_{0121}+2f_{1120})+L_0
\\ u_6&:=[\partial_x,v_6]=[X,v_6]+L_0=(2e_{1110}+2e_{0120})+L_0
\\ u_7&:=[\partial_x,v_7]=[X,v_7]+L_0=(e_{0111}+2e_{0120})+L_0
\\ u_8&:=[\partial_x,v_8]=[X,v_8]+L_1=2e_{0010}+L_1
\\ u_9&:=[\partial_x,v_9]=[X,v_9]+L_1=2e_{1232}+L_1
\\ u_{10}&:=[\partial_y,v_4]=[Y,v_4]+L_{-1}=f_{1242}+L_{-1}
\\ u_{11}&:=[\partial_y,v_5]=[Y,v_5]+L_0=2f_{2342}+L_0
\\ u_{12}&:=[\partial_y,v_6]=[Y,v_6]+L_0=f_{1120}+L_0
\\ u_{13}&:=[\partial_y,v_7]=[Y,v_7]+L_0=f_{0121}+L_0
\\ u_{14}&:=[\partial_y,v_8]=[Y,v_8]+L_1=f_{1221}+L_1
\\ u_{15}&:=[\partial_y,v_9]=[Y,v_9]+L_1=2e_{0001}+L_1
\\ u_{16}&:=[\partial_z,v_7]=[e,v_7]+L_0=2e_{1000}+L_1
\\ u_{17}&:=[\partial_z,v_8]=[e,v_8]=2e_{1122}
\\ u_{18}&:=[\partial_z,v_9]=[e,v_9]=2f_{0100}\end{split}
\end{align}
The possible $27$ brackets are reduced to these $18$ since $\partial_i\partial_j=\partial_j\partial_i$ for all $i,j$. We need to be careful again since this appears to give too many elements. Observe that $u_5,u_{12}$ and $u_{13}$ are not linearly independent in $\mathcal{G}$, and $u_{15}+u_{16}+e=u_8$. Since $e \in L_{(1)}$ this forces that $u_8$ is a linear combination of $u_{15}$ and $u_{16}$. Similarly $u_6, u_7$ and $u_2$ are only two linearly independent elements, in which case the set $\{u_i\}$ forms a $15$-dimensional vector space.

To make sure we obtain a grading we also need to verify that both $[u_i,u_j]\in \mathcal{G}'_{4}$ and $[\mathcal{G}'_{2},\mathcal{G}'_{3,4}]=0$. These are straightforward calculations using GAP. In \aref{gradingcheck} we provide the idea of how we may show $[\Gc'_2,\Gc'_2]=\Gc'_4$.

We continue to build this grading omitting any elements that live in the same coset despite initially appearing to be different to obtain a $15$-dimensional $\mathcal{G}'_1$.
\begin{align*}w_1&:=[\partial_x,u_4]=[X,u_4]+L_{-1}=e_{0100}+L_{-1}
\\ w_2&:=[\partial_x,u_5]=[X,u_5]+L_0=(2f_{0001}+2f_{1000})+L_0
\\ w_3&:=[\partial_x,u_6]=[X,u_6]+L_0=2e_{1221}+L_0
\\ w_4&:=[\partial_x,u_8]=[X,u_8]+L_1=(2e_{0121}+2e_{1120})+L_1
\\ w_5&:=[\partial_x,u_9]=[X,u_9]+L_1=2e_{2342}+L_1
\\ w_6&:=[\partial_x,u_{10}]=[X,u_{10}]+L_{-1}=2f_{1122}+L_{-1}
\\ w_7&:=[\partial_x,u_{11}]=[X,u_{11}]+L_0=2f_{1232}+L_0
\\ w_8&:=[\partial_x,u_{12}]=[X,u_{12}]+L_0=(f_{1000}+f_{0010})+L_0
\\ w_{9}&:=[\partial_x,u_{14}]=[X,u_{14}]+L_1=(2f_{0111}+2f_{1110})+L_1
\\ w_{10}&:=[\partial_x,u_{15}]=[X,u_{15}]+L_1=(e_{1111}+e_{0121})+L_1
\\ w_{11}&:=[\partial_y,u_{14}]=[Y,u_{14}]+L_1=f_{1342}+L_1
\\ w_{12}&:=[\partial_y,u_{15}]=[Y,u_{15}]+L_1=(f_{1110}+f_{0120})+L_1
\\ w_{13}&:=[\partial_x,u_{17}]=[X,u_{17}]=e_{0011}
\\ w_{14}&:=[\partial_x,u_{18}]=[X,u_{18}]=2e_{1242}
\\ w_{15}&:=[\partial_y,u_{17}]=[Y,u_{17}]=2f_{1220}
\end{align*}It is easy to verify that $[\mathcal{G}'_{1},\mathcal{G}'_{j}]$ are in the correct place. Finally, we construct $\mathcal{G}'_0:=[\mathcal{G}'_{-1},\mathcal{G}'_1]$.
\begin{align*}
x_1&:=[\partial_x,w_7]=[X,w_7]+L_0=(2f_{1121}+2f_{0122})+L_0
\\ x_2&:=[\partial_x,w_8]=[X,w_8]+L_0=(e_{1100}+e_{0110})+L_0
\\ x_3&:=[\partial_x,w_{13}]=[X,w_{13}]=e_{1121}+e_{0122}
\\ x_4&:=[\partial_x,w_{11}]=[X,w_{11}]+L_1=(f_{1231}+2f_{1222})+L_1
\\ x_5&:=[\partial_x,w_{12}]=[X,w_{12}]+L_1=(h_{\al_1}+h_{\al_3})+L_1
\\ x_6&:=[\partial_x,w_{15}]=[X,w_{15}]=f_{1100}+f_{0110}
\\ x_7&:=[\partial_x,w_{9}]=[X,w_{9}]+L_1=(h_{\al_1}+h_{\al_4})+L_1
\\ x_8&:=[\partial_x,w_{10}]=[X,w_{10}]+L_1=(e_{1231}+2e_{1222})+L_1
\end{align*}To see this $\mathcal{G}'_0$ is a classical simple Lie algebra, we can give the following isomorphism with a Lie algebra of type $A_2$ with simple roots $\al_1$ and $\al_2$ using the ordering of \cite{Bour02}.
\begin{align*}
x_1 &\mapsto f_{\al_1}, x_2 \mapsto e_{\al_2},
\\x_3 &\mapsto e_{\al_1}, x_4 \mapsto f_{\al_1+\al_2},
\\x_5 &\mapsto h_{\al_2}, x_6 \mapsto f_{\al_2},
\\x_7 &\mapsto h_{\al_1}, x_8 \mapsto e_{\al_1+\al_2}.
\end{align*}This is now a depth-one grading on a simple Lie algebra with classical $L_0$ component, and hence $\mathcal{G} \cong S(3;\underline{1})^{(1)}$ by \cite{ko95}.
\end{proof}

\section[Counterexamples to Morozov's theorem: $E_6$]{Counterexamples to Morozov's theorem: $E_6$}

We will explore the nilpotent orbit of type ${A_2}^2+A_1$ in the remaining exceptional Lie algebras, to begin consider $G$ an algebraic group of type $E_6$ with Lie algebra $\g=\Lie(G)$. Consider the nilpotent orbit denoted by $\mathcal{O}({A_2}^2+A_1)$ with representative $e:=\sum_{\al \in \Pi\setminus\{\al_4\} }e_{\al}$. Since this is the `same' representative for all $p$, which allows one to continue to make use of the associated cocharacter $\tau$ given by weighted diagram \[\subalign{2\quad2\quad-&5\quad2\quad2\\&2}\] from \cite[pg. 84]{LT11}.

This orbit was also considered in \cite[Section 4]{lmt08}, where the authors give non-zero linearly independent elements $X,Y \in \g_e(\tau,-1)$ with \begin{align*}
X &=e_{\subalign{11&100\\&0}}+2e_{\subalign{01&100\\&1}}+e_{\subalign{00&110\\&1}}+e_{\subalign{01&110\\&0}}+e_{\subalign{00&111\\&0}} \in \g_e(\tau,-1),
\\Y&=f_{\subalign{11&100\\&1}}+f_{\subalign{11&110\\&0}}+f_{\subalign{01&110\\&1}}+f_{\subalign{00&111\\&1}}+2f_{\subalign{01&111\\&0}}\in \g_e(\tau,-1).
\end{align*} This gives that $\Lie(G_e) \subseteq \bigoplus_{i\ge 0} \g_e(\tau,i)$ has dimension $24$. The tables of \cite{S16} say $\g_e$ has dimension $27$ for $p=3$, and a check initially appears to give a new element in the zero component.

However, for fields of characteristic three $E_6$ is not simple with a one-dimensional centre. Once we factor out this centre we are left with a $77$-dimensional simple Lie algebra, and we have that $\g_e/\mathfrak{z}(\g)$ has dimension $26$. We have \[\g_e=\bigoplus_{i \ge -1} \g_e(\tau,i) \quad \text{and}\quad \dim\,\g_e(\tau,-1)=2.\]

\begin{prop}Let $G$ be an algebraic group of type $E_6$ with Lie algebra $\g=\Lie(G)$ for $p=3$ and $e$ be a nilpotent orbit representative for ${A_2}^{2}+A_1$. The centraliser is $27$ dimensional with basis given by the existing $24$ elements of $\Lie(G_e)$, together with the centre of $E_6$ and the new elements,\begin{align*}
X&=e_{\subalign{11&100\\&0}}+2e_{\subalign{01&100\\&1}}+e_{\subalign{00&110\\&1}}+e_{\subalign{01&110\\&0}}+e_{\subalign{00&111\\&0}} \in \g_e(\tau,-1),
\\Y&=f_{\subalign{11&100\\&1}}+f_{\subalign{11&110\\&0}}+f_{\subalign{01&110\\&1}}+f_{\subalign{00&111\\&1}}+2f_{\subalign{01&111\\&0}}\in \g_e(\tau,-1).
\end{align*}\end{prop}
\begin{proof}The vectors contained in $\bigoplus_{i \ge 0} \g_e(\tau,i)$ from \cite[pg. 84]{LT11} can easily be verified to still be linearly independent for characteristic three. Hence, they continue to form a basis for $\Lie(G_e)$. We then check that the new elements are independent of each other. This gives the required basis.\end{proof}

We obtain the following information using GAP, with the final piece of information shown explicitly in \aref{AppE62}.
\begin{enumerate}
\item{The radical $A$ of $\g_e$ is of dimension $17$, with abelian derived subalgebra of dimension $8$;}
\item{the normalizer $\w_{E_6}:=N_{\g}(A)$ has dimension $35$;}
\item{the Lie algebra $M:=\w_{E_6}/A$ is simple and restricted of dimension $18$;}
\item{the set $L_{-1}:=\{x \in \g: [x,A] \subseteq \w_{E_6}\}$ is $44$-dimensional, and $\langle L_{-1}\rangle\cong \g$.}
\end{enumerate}

It is well-known that we can easily identify $F_4$ inside $E_6$ associating roots $e_{1000}\mapsto e_{\subalign{00&000\\&1}}, e_{0100}\mapsto e_{\subalign{00&100\\&0}}, e_{0010}\mapsto e_{\subalign{01&000\\&0}}+e_{\subalign{00&010\\&0}},$ and $e_{0001}\mapsto e_{\subalign{10&000\\&0}}+e_{\subalign{00&001\\&0}}$. This observation suggests the graded Lie algebra arising from the Weisfeiler filtration built using $\w_{E_6}$ maybe isomorphic to $\mathcal{S}_3(\underline{1},\omega_S)^{(1)}$ from \cite{Sk92}. We now prove the main result of this section, that the subalgebra $\w_{E_6}$ is a maximal subalgebra of $\g$.

\begin{thm}\label{none6}Let $G$ be an algebraic group of type $E_6$ over an algebraically closed field of characteristic three with Lie algebra $\g=\Lie(G)$. For any $e \in \mathcal{O}({A_2}^2+A_1)$ we have the following \begin{enumerate}[label=(\alph*)]
\item{\label{partae6}$A=\rad(\g_e)$ and $\g_e/A \cong H(2;\underline{1})$ as Lie algebras.}
\item{\label{partbe6}$A=\rad(\mathfrak{n}_e)$ and $\mathfrak{n}_e/A \cong CH(2;\underline{1})$ as Lie algebras.}
\item{$A=\rad(\w_{E_6})$ and $\w_{E_6}/A\cong W(2;\underline{1})$ as Lie algebras.}
\item{$\w_{E_6}$ is a maximal Lie subalgebra of $\g$.}
\item{\label{partee6}$\w_{E_6} \subseteq \g$ gives rise to a Weisfeiler filtration such that the corresponding graded Lie algebra is isomorphic to $\mathcal{S}_3(\underline{1},\omega_S)^{(1)}$}\end{enumerate}\end{thm}

\begin{proof}
We verify the radical using GAP, which is $17$-dimensional with an abelian derived subalgebra of dimension $8$. Let $M=\g_e/A$, then $M$ is a depth-one graded Lie algebra since $A \subseteq \g_e(\tau,\ge 2)$. This is easily checked taking a basis of $A$ and applying $\tau$ to them. Further, by \cite{LT11} we have $M_0 \cong \mathfrak{sl}(2)$.

We have that $M_{-1}$ is a $2$-dimensional $M_0$-module. It is easy to verify that $[M_{-1},x]=0$ implies that $x=0$ for $x \in M_{>0}$. Similarly one can verify $[M_1,x]=0$ implies $x=0$ for all $x \in M_{-1}$. We can then apply \tref{weakrec} in exactly the same way as in \tref{nonf4} to obtain $H(2;\underline{1})^{(2)} \subseteq L \subseteq H(2;\underline{1})$, and since $\g_e/A$ is $10$-dimensional we obtain part \ref{partae6}.

Since $\mathfrak{n}_e=\g_e \oplus \bbk h$, we have $\rad(\mathfrak{n}_e)=A$ and $(\mathfrak{n}_e)_0 \cong \mathfrak{sl}(2) \oplus \bbk h$. Hence $\mathfrak{n}_e \subseteq CH(2;\underline{1})$, giving part \ref{partbe6} since $\dim\, \mathfrak{n}_e/A=\dim\, CH(2;\underline{1})$ in exactly the same way as \tref{nonf4}.

Since $M:=\w_{E_6}/A$ is a simple Lie algebra $A=\rad(\w_{E_6})$. It follows that since $M_{-1}$ is $2$-dimensional with $M_0 \cong\mathfrak{gl}(2)$, $M \cong W(2;\underline{1})$ by \cite{ko95}.

For maximality set $L_0:=\w_{E_6}$, we find the set \[L_{-1}:=\{x \in \g: [x,A] \subseteq L_0\},\] and show $L_{-1}/L_0$ is an irreducible $L_0/A$-module such that $\langle L_{-1}\rangle=E_6$. Using GAP we obtain that $L_{-1}/L_0$ is an irreducible $L_0$-module, as we can compute all submodules of $\g$ as a $\mathfrak{w}_{E_6}$-module. We find $L_{-1}/\w_{E_6}$ is irreducible, and so any subalgebra properly containing $\w_{E_6}$ contains $L_{-1}$ by \rref{whymax} which generates $\g$.

For \ref{partee6}, we use $L_{-1}$ to build a filtration with $L_{-k}:=[L_{-k+1},L_{-1}]+L_{-k+1}$. This produces the following dimensions of each component in the filtration with the basis of each component given in \autorefs{tabe6} from \aref{AppE62}. \begin{enumerate}\item{$\dim \,L_{-4}=78$,}\item{$\dim\, L_{-3}=70$,}\item{$\dim\, L_{-2}=62$,}\item{$\dim\, L_{-1}=44$,}\item{$\dim\, L_{0}=\dim \mathfrak{w}_1=35$,}\item{$\dim\, L_1=\dim\, A=17$,}\item{$\dim\, L_2=\dim\, [A,A]=8$.}\end{enumerate}
Crucially, for a Weisfeiler filtration we need $\g$ to be a simple Lie algebra. In fields of characteristic three the Lie algebra $E_6$ has a one-dimensional centre. We quotient out by the centre to obtain the $77$-dimensional simple Lie algebra, and the above filtration changes slightly as $\dim\,L_i/\mathfrak{z}(\g)=\dim\, L_i-1$ except for the derived subalgebra $L_2$ of the radical of $L_0$.

For the corresponding graded Lie algebra $\mathcal{G}$, this gives \[\dim\, \mathcal{G}=\sum_{i=-4}^{2}\dim (L_{i}/L_{i+1})=8+8+18+9+18+8+8=77\] Computing all $L_0$-submodules of $E_6$ in GAP, we see each $\Gc_i$ is irreducible. If $M(\Gc)\ne 0$, then there is some minimal $k$ such that $\Gc_{k}\subseteq M(G)$. However, since $[L_{-i},L_1]=L_{-i+1}$ for all $i=2,3,4$ it follows that $[\Gc_1,\Gc_{-k}]\ne 0 \subseteq \Gc_{-k+1}\cap M(\mathcal{G})$ contradicting the fact $k$ is minimal. Hence $M(\mathcal{G})=0$, and $\Gc$ is a simple Lie algebra using \pref{simplegraded}.

Identifying $F_4$ inside $E_6$ by associating roots $e_{1000}\mapsto e_{\subalign{00&000\\&1}}, e_{0100}\mapsto e_{\subalign{00&100\\&0}}, e_{0010}\mapsto e_{\subalign{01&000\\&0}}+e_{\subalign{00&010\\&0}},$ and $e_{0001}\mapsto e_{\subalign{10&000\\&0}}+e_{\subalign{00&001\\&0}}$ we may generate a subalgebra in $E_6$ isomorphic to $F_4$ in GAP. We may then compute $F_4 \cap L_i$ in $\g$ using GAP, and see $\dim(F_4 \cap L_{2i})=\dim(F_4 \cap L_{2i+1})$.

It is then a straightforward check that the even components work precisely the same way as in \tref{weisf4} when we consider $\g$ of type $F_4$. Hence, in our case we regrade our even components $\mathcal{G}_{-2}+\mathcal{G}_0+\mathcal{G}_2+\mathcal{G}_4$ in the exact same way as \tref{weisf4} to obtain $\mathcal{G}'_{-2}+\mathcal{G}'_0+\mathcal{G}'_2+\mathcal{G}'_4+\mathcal{G}'_6+\mathcal{G}'_8$ found by doubling each degree from the $F_4$ case. Hence, this is isomorphic to $S(3;\underline{1})^{(1)}$ by \tref{weisf4}.

It follows from \cite[\S3]{Sk92} that the simple Lie algebra $\mathcal{S}:=\mathcal{S}_{3}(\underline{n},\omega_S)$ has a grading such that $\dim\,\mathcal{S}_{-1}=\dim\,\mathcal{S}_{-2}=3$, and $\dim\,\mathcal{S}_1=6$ with $\mathcal{S}_{2k}\cong (S(3;\underline{1})^{(1)})_{k}$. Using \cite[Theorem 5.2]{Sk92} this is actually enough to identify that $\mathcal{G} \cong \mathcal{S}$. Hence, we find elements in the odd components such that $[\mathcal{G}'_{-1},\mathcal{G}'_{-1}]=\mathcal{G}'_{-2}$. This gives \begin{align*}\mathcal{G}'_{-1}:=\spnd\{&(e_{\subalign{11&100\\&1}}+e_{\subalign{11&110\\&0}}+2e_{\subalign{00&111\\&1}}+e_{\subalign{01&111\\&0}})+L_2,
\\&(f_{\subalign{11&100\\&0}}+2f_{\subalign{01&100\\&1}}+2f_{\subalign{00&110\\&1}}+f_{\subalign{00&111\\&0}})+L_2, \\&(f_{\subalign{10&000\\&0}}+2f_{\subalign{01&000\\&0}}+f_{\subalign{00&010\\&0}}+2f_{\subalign{00&001\\&0}})+L_0\}.\end{align*}

Applying $\ad(\mathcal{G}'_{-1})$ to $\mathcal{G}'_2$ we obtain the following $6$-dimensional vector space \begin{align*}\mathcal{G}'_1:=\spnd\bigl\{&(e_{\subalign{10&000\\&0}}+e_{\subalign{01&000\\&0}}+2e_{\subalign{00&010\\&0}}+2e_{\subalign{00&001\\&0}})+L_2,
\\&(e_{\subalign{11&100\\&0}}+e_{\subalign{01&100\\&1}}+e_{\subalign{00&110\\&1}}+2e_{\subalign{00&111\\&0}})+L_0,
\\&(e_{\subalign{12&211\\&1}}+e_{\subalign{11&221\\&1}})+L_2,
\\&(f_{\subalign{11&000\\&0}}+f_{\subalign{00&011\\&0}})+L_{-2},
\\&(f_{\subalign{11&210\\&1}}+f_{\subalign{01&211\\&1}})+L_2,
\\&(2f_{\subalign{11&100\\&1}}+f_{\subalign{11&110\\&0}}+f_{\subalign{00&111\\&1}}+2f_{\subalign{01&111\\&0}})+L_0\bigr\}.\end{align*} This is a well-defined grading, and since all axioms of \cite[Theorem 5.2]{Sk92} are satisfied we conclude $\mathcal{G} \supseteq \mathcal{S}_{3}(\underline{1},\omega_S)$. By dimension reasons it follows that these are isomorphic as Lie algebras.
\end{proof}

\section[Counterexamples to Morozov's theorem: $E_7$]{Counterexamples to Morozov's theorem: $E_7$}

Consider the same orbit $\mathcal{O}({A_2}^2+A_1)$ with the same representative as in the $E_6$ case, that is $e:=\rt{1}+\rt{2}+\rt{3}+\rt{5}+\rt{6}$. We are continuing to use the associated cocharacter $\tau$ given by \[\subalign{2\quad2\quad-&5\quad2\quad2\quad{-2}\\&2}\] from \cite[pg. 98]{LT11}. By \cite[Section 4]{lmt08}, there are non-zero elements \begin{align*}
X &=e_{\subalign{11&1000\\&0}}+2e_{\subalign{01&1000\\&1}}+e_{\subalign{00&1100\\&1}}+e_{\subalign{01&1100\\&0}}+e_{\subalign{00&1110\\&0}} \in \g_e(\tau,-1),
\\Y&=f_{\subalign{11&1000\\&1}}+f_{\subalign{11&1100\\&0}}+f_{\subalign{01&1100\\&1}}+f_{\subalign{00&1110\\&1}}+2f_{\subalign{01&1110\\&0}}\in \g_e(\tau,-1),
\end{align*} and $\Lie(G_e) \subseteq \bigoplus_{i\ge 0} \g_e(\tau,i)$ has dimension $43$. The tables of \cite{S16} gives $\dim\,\g_e=45$ for $p=3$, and so \[\g_e=\bigoplus_{i \ge -1} \g_e(\tau,i) \quad \text{and}\quad \dim\,\g_e(\tau,-1)=2.\]

\begin{enumerate}\item{The radical $A$ of $\mathfrak{n}_e$ is abelian of dimension $8$.}
\item{We have $\w_{E_7}:=N_{\g}(A)$ is $53$-dimensional.}
\item{$M:=\w_{E_7}/A$ is a semisimple Lie algebra of dimension $45$.}
\item{The set $L_{-1}:=\{x \in \g: [x,A] \subseteq \w_{E_7}\}$ is $98$-dimensional, and $\langle L_{-1}\rangle\cong \g$. The basis of this space can be found in \autorefs{tabe7} from \aref{AppE72}.}\end{enumerate}

We can use the basis given \cite[pg. 98]{LT11}, and apply $\ad(\g_e(\tau,-1))$ to the elements $f_{\subalign{01&2111\\&1}}$, $e_{\subalign{11&0000\\&0}}+e_{\subalign{00&0110\\&0}}$ and $e_{\subalign{12&2221\\&1}}$ of $\g_e(\tau,4)$ to obtain an ideal $I$ of our Lie algebra $\w_{E_7}$ that is $35$-dimensional.

\begin{thm}\label{none7}Let $G$ be an algebraic group of type $E_7$ over an algebraically closed field of characteristic three with Lie algebra $\g=\Lie(G)$. For any $e \in \mathcal{O}({A_2}^2+A_1)$ we have the following \begin{enumerate}[label=(\alph*)]
\item{\label{partae7}$A=\rad (\g_e)$ and $\g_e/I \cong H(2;\underline{1})$ as Lie algebras.}
\item{$A=\rad(\mathfrak{n}_e)$ and $\mathfrak{n}_e/I \cong CH(2;\underline{1})$ as Lie algebras.}
\item{$A=\rad(\w_{E_7})$ and $\w_{E_7}$ is a maximal Lie subalgebra of $\g$.}
\item{$M\cong (\mathfrak{sl}(2)\otimes \mathcal{O}(2;\underline{1}))\rtimes (1_{\mathfrak{sl}_2}\otimes W(2;\underline{1}))$ as Lie algebras.}
\end{enumerate}\end{thm}
\begin{proof}Our ideal $I$ is such that $I/A$ is a $27$-dimensional ideal of $M$ with $24$-dimensional radical in $I/A$. It also follows that $\mathfrak{w}_{E_7}/I$ is a restricted simple Lie algebra with dimension $18$, and so $I/A$ is a maximal ideal of $M$. Using cocharacter $\tau$ the simple Lie algebra is depth-one graded with $4$-dimensional zero part isomorphic to $\mathfrak{gl}(2)$, and so $\mathfrak{w}_{E_7}/I \cong W(2;\underline{1})$ by \cite[Theorem 1]{ko95}.

We apply the Weak Recognition theorem stated in \tref{weakrec} to obtain $H(2;\underline{1})^{(2)} \subseteq  \mathfrak{g}_e/I\subseteq H(2;\underline{1})$, and so by dimension reasons $\g_e/I \cong H(2;\underline{1})$. This gives \ref{partae7}, and using the fact $\mathfrak{n}_e= \g_e\oplus \bbk h$ we find $\mathfrak{n}_e/I \cong CH(2;\underline{1})$ and $\rad(\g_e)=\rad(\mathfrak{n}_e)=A$. It follows that $\g_e/A \cong I/A \rtimes H(2;\underline{1})$ and $\mathfrak{n}_e/A \cong I/A \rtimes CH(2;\underline{1})$.

To confirm that $I/A \cong \mathfrak{sl}(2)\otimes \mathcal{O}(2;\underline{1})$ we need that $I/A$ is a minimal ideal of $M$. It is easy to check using GAP that $I/\rad(I)$ is a simple Lie algebra of dimension $3$ contained in $\g_e(\tau,0)$. Hence, $S\cong \mfsl(2)$ and since $I/A$ is irreducible as an $M$-module we deduce that $I/A$ is minimal. Using \tref{blocktheorem2}, $M \subseteq \Der(\mathfrak{sl}(2)\otimes \mathcal{O}(2;\underline{1}))=(\mathfrak{sl}(2)\otimes \mathcal{O}(2;\underline{1}))\rtimes (1_{\mathfrak{sl}_2}\otimes W(2;\underline{1}))$. By dimension reasons we are done. Since $M$ is semisimple it follows automatically that $\rad(\w_{E_7})=A$.

To prove maximality in this case is trickier than the $E_6$ case, as $L_{-1}/L_0$ is not irreducible for $L_0:=\w_{E_7}$. There is an $80$-dimensional submodule of $L_{-1}$ obtained by considering $N:=\{x \in \g: [x,A] \subseteq I\}$, where $I$ is the minimal ideal of $\w_{E_7}$. For full calculations we refer the reader to \autorefs{tabe7} in \aref{AppE72}. It is easy to verify that $\langle N\rangle \cong \g$, and that $N/L_0$ is an irreducible module. This is not quite enough to prove maximality, as a priori $L_{-1}/L_0$ could be the direct sum of at least two submodules.

We have $\dim\, N/L_0=27$, and so we may have an $18$-dimensional submodule arising from the $18$-dimensional composition factor obtained from $L_{-1}/L_0$. Considering a subalgebra generated by these $18$ elements and $L_0$ could produce a proper subalgebra of $\g$ that strictly contains $L_0$.

We show that $L_{-1}/L_0$ is an indecomposable module, and so such a submodule does not exists to prevent this from happening. Suppose for a contradiction that the $18$-dimensional composition factor occurs as a submodule of $L_{-1}/L_0$, and consider the inverse image $L'_{-1}$ of this in $L_{-1}$. There is an $(I/A)$-module homomorphism $(L_{-1}/L_0)\times A\to L_0/A$ from the Lie bracket of $\g$.

Since $[I,A]=0$ and no composition factor of $L'_{-1}/L_0$ regarded as an $I/A$ module is a composition factor of $L_0/A$ we have $[L'_{-1},A]\subseteq A$. Hence $L'_{-1}\subseteq \mathfrak{w}_{E_7}=L_0$ --- a contradiction. Hence, our module is indecomposable. Since $\langle N\rangle=\g$, we conclude $\w_{E_7}$ is maximal using \rref{whymax}.
\end{proof}

\section[Counterexamples to Morozov's theorem: $E_8$ $(1)$]{Counterexamples to Morozov's theorem: $E_8$ $(1)$}

Continuing with the same theme we look at $\mathcal{O}({A_2}^2+A_1)$ with the representative $e:=\rt{1}+\rt{2}+\rt{3}+\rt{5}+\rt{6}$. We use the associated cocharacter $\tau$ given by \[\subalign{2\quad2\quad-&5\quad2\quad2\quad{-2}\quad0\\&2}\] in \cite[pg. 128]{LT11}. By \cite[Section 4]{lmt08}, there are non-zero elements \begin{align*}
X &=e_{\subalign{11&10000\\&0}}+2e_{\subalign{01&10000\\&1}}+e_{\subalign{00&11000\\&1}}+e_{\subalign{01&11000\\&0}}+e_{\subalign{00&11100\\&0}} \in \g_e(\tau,-1),
\\Y&=f_{\subalign{11&10000\\&1}}+f_{\subalign{11&11000\\&0}}+f_{\subalign{01&11000\\&1}}+f_{\subalign{00&11100\\&1}}+2f_{\subalign{01&11100\\&0}}\in \g_e(\tau,-1),
\end{align*} and $\Lie(G_e) \subseteq \bigoplus_{i\ge 0} \g_e(\tau,i)$ has dimension $86$. The tables of \cite{S16} gives $\dim\,\g_e=88$ for $p=3$, and so \[\g_e=\bigoplus_{i \ge -1} \g_e(\tau,i) \quad \text{and}\quad \dim\,\g_e(\tau,-1)=2.\]
 Using GAP with full details given in \aref{AppE82}, we obtain the following information.
\begin{enumerate}\item{The radical $A$ of $\mathfrak{n}_e$ is abelian of dimension $8$.}
\item{We have $\w_{E_8}:=N_{\g}(A)$ is $96$-dimensional.}
\item{$M:=\w_{E_8}/A$ is a semisimple Lie algebra of dimension $88$.}
\item{The set $L_{-1}:=\{x \in \g: [x,A] \subseteq \w_{E_8}\}$ is $177$-dimensional, and $\langle L_{-1}\rangle\cong \g$.}\end{enumerate}

Using the tables of \cite[pg. 128]{LT11}, $\g_e(\tau,0)$ has type $G_2 \oplus A_1$. However, for characteristic three the exceptional Lie algebra of type $G_2$ contains a $7$-dimensional ideal $I_{G_2}$ generated by short roots with $I_{G_2} \cong G_2/I_{G_2} \cong \mathfrak{psl}(3)$.

Applying $\ad(\g_e(\tau,-1))$ to $\g_e(\tau,4)$ repeatedly we obtain an ideal $I$ of dimension $71$. Further, we may add the remaining elements from $G_2 \subseteq \g_e(\tau,0)$ to obtain $78$-dimensional ideal $J$.

\begin{thm}\label{none8}Let $G$ be an algebraic group of type $E_8$ over an algebraically closed field of characteristic three with Lie algebra $\g=\Lie(G)$. For any $e \in \mathcal{O}({A_2}^2+A_1)$ we have the following \begin{enumerate}[label=(\alph*)]
\item{\label{solide8}$A=\rad (\g_e)$ and $\g_e/J \cong H(2;\underline{1})$ as Lie algebras.}
\item{$A=\rad(\mathfrak{n}_e)$ and $\mathfrak{n}_e/J \cong CH(2;\underline{1})$ as Lie algebras.}
\item{\label{maximalbrill}$A=\rad(\w_{E_8})$ and $\w_{E_8}$ is a maximal Lie subalgebra of $\g$.}
\item{\label{brillianthopefully} We have $M\cong J\rtimes W(2;\underline{1})$ as Lie algebras, where $J$ lies in the short exact sequence \[0 \rightarrow \mathfrak{psl}(3)\otimes \mathcal{O}(2;\underline{1}) \rightarrow J \rightarrow \mathfrak{psl}(3)\rightarrow 0.\]}\end{enumerate}\end{thm}

\begin{proof}Our ideal $I$ is such that $I/A$ is a $63$-dimensional ideal of $M$ with $56$-dimensional radical in $I/A$. We show $I$ is a minimal ideal using methods described in the Appendix. Hence, we use \tref{blocktheorem2} along with \cite[pg. 128]{LT11} to see $I/A \cong \mathfrak{psl}(3) \otimes \mathcal{O}(2;\underline{1})$.

Using ideal $J$ we apply \tref{weakrec}, and see $H(2;\underline{1})^{(2)} \subseteq  \mathfrak{g}_e/J\subseteq H(2;\underline{1})$. Thus by dimension reasons $\g_e/J \cong H(2;\underline{1})$. This gives \ref{solide8}, and using the fact $\mathfrak{n}_e=\g_e \oplus \bbk h$ we find $\mathfrak{n}_e/J \cong CH(2;\underline{1})$ and $\rad(\g_e)=\rad(\mathfrak{n}_e)=A$. Since $M$ is semisimple we obtain that $\rad(\w_{E_8})=A$.

We approach \ref{maximalbrill} in the same way to $E_7$ to achieve maximality, since the module $L_{-1}/L_0$ is not irreducible. There is an $159$-dimensional submodule of $L_{-1}$ obtained by considering $N:=\{x \in \g: [x,A] \subseteq I\}$, where $I/A\cong \mathfrak{psl}(3) \otimes \mathcal{O}(2;\underline{1})$ is a minimal ideal of $\w_{E_8}$.

It is easy to verify that $\langle N\rangle \cong \g$, and $N/L_0$ is an irreducible module using both \autorefs{tabe8} and \aref{AppE82}. Using a similar argument to the $E_7$ case or GAP directly, we can show that $L_{-1}/L_{0}$ is indecomposable. It follows $\w_{E_8}$ is maximal since $\langle N\rangle=\g$ using the same reason as in \tref{none7}.

Using GAP we see that $\mathfrak{w}_{E_8}/J$ is a simple Lie algebra of dimension $18$ with a depth-one grading and zero component isomorphic to $\mathfrak{gl}(2)$. Hence, $\mathfrak{w}_{E_8}/J \cong W(2;\underline{1})$ and $M\cong J \rtimes W(2;\underline{1})$. This follows since we can obtain $W(2;\underline{1})$ as a complementary subalgebra, and $J/I$ is simple of dimension $7$ isomorphic to $\mathfrak{psl}(3)$.

However, there is no complementary subalgebra in $J$ isomorphic to $\mathfrak{psl}(3)$. This is a consequence of the fact the adjoint module of $G_2$ is indecomposable for $p=3$. We may consider the identity map $I \rightarrow J$, and the canonical map $J \rightarrow J/I$ to obtain the exact sequence as required and complete part \ref{brillianthopefully}.
\end{proof}

\subsection*{The Weisfeiler filtration}

Both \tref{none7} and \ref{none8} have very similar stories taking place when we look to produce an $L_0$-invariant subspace $L_{-1}$ in $\g$ where $L_{-1}/L_0$ is an irreducible $L_0/L_1$-module. In both settings we initially find the set $\{x \in \g: [x,A] \subseteq \w_{E_n}\}$, for $A$ the radical of $\w_{E_{7,8}}$, is indecomposable but not irreducible. Using \aref{AppE72} and \ref{AppE82} we are able to find $L'_{-1}$ which satisfies the initial point in building a Weisfeiler filtration.

In this section we give initial results about the corresponding graded Lie algebra, and show it differs from the cases for $F_4$ and $E_6$. This is because our $\mathcal{G}_0$ component is no longer a simple Lie algebra, but is a semisimple Lie algebra. Since $L_{-1}$ contains elements $x \in L_{-1}\setminus L'_{-1}$ such that $[x,A] \subseteq L_0$, we find that the Weisfeiler radical $M(\Gc)$ is non-zero.

\begin{prop}\label{weise7}Let $p=3$ and $n \in \{7, 8\}$. Suppose $\g$ has type $E_n$ with maximal subalgebra $\w_{E_n}$ and $\w_{E_n}$-invariant subspace $L'_{-1}$ from \tref{none7} and \tref{none8}.

Then, for the Weisfeiler filtration $\mathcal{F}$ associated to the pair $(L'_{-1}, \w_{E_n})$ with corresponding graded Lie algebra $\Gc$ we have that $\dim\,M(\Gc)=26$. Hence, $\Gc/M(\Gc)$ is a semisimple Lie algebra of dimension $\dim\,\g-26$.\end{prop}
\begin{proof}
In both cases we obtain the following dimensions in the Weisfeiler filtration. \begin{enumerate}\item{$\dim\,L'_{-3}=\dim\,\g$,}\item{$\dim \,L'_{-2}=\dim \, \g-8$,}
\item{$\dim \,L'_{-1}=\dim\,L_{-1}-18$,}
\item{$\dim \,L_{0}=\dim\,\w_{E_n}$,}
\item{$\dim\, L_1=\dim\, A=8$.}\end{enumerate}
Consider $x \in L_{-1}\setminus L'_{-1}$, in $\Gc$ we have $[x+L'_{-1},y]=[x,y]+L_0$ for all $y \in A$. However, $[x,A] \subseteq L_0$ and so $[x,y]+L_0=L_0$. In particular, $[\Gc_1,x]=0$ and so ${M(\Gc)}_{-2}:=\{x \in L'_{-2}:[x,A] \subseteq \w_{E_n}\} \subseteq M(\Gc)$ has dimension $18$. From above we have $\dim\,\Gc_3=8$, and by definition $[\Gc_3,A] \subseteq \Gc_2$.

Using GAP it is straightforward to check that $[\Gc_3,A] \subseteq {M(\Gc)}_2$. For this, we calculate $[L_{-1},L_{-1}]$ and show it is equal to $\g$. Using Tables \ref{tabe7} and \ref{tabe8} it follows that $[L'_{-1},L'_{-1}]=L'_{-2}$ has dimension $\dim\,\g-8$. Hence, $\Gc_3 \subseteq M(\Gc)$ and $\dim\,M(\Gc)=26$ as required.
\end{proof}

To finish, we give estimates for the structure of the corresponding graded Lie algebra in both cases. We have $\dim\,\bar{\Gc}=107$ in $E_7$, and that $\bar{\Gc}_0 \cong \mathfrak{sl}(2)\otimes \mathcal{O}(2;\underline{1})\rtimes 1_{\mathfrak{sl}_2}\otimes W(2;\underline{1})$. It should be noted that we have $\bar{\Gc}_{\ge 2}=0$, and $[A,[A,\bar{\Gc}_{-1}]]=0$. The second part is straightforward to see since $[I,A]=0$, and in producing $L'_{-1}$ we used elements $u$ such that $[u,A] \subseteq I$.

This starts to look like the \emph{degenerate} case of the Weisfeiler filtration from \autoreft{sec:Wei}.

\begin{conjecture}Let $p=3$ and suppose $\g$ has type $E_7$ with maximal subalgebra $\w_{E_7}$ and $\w_{E_7}$-invariant subspace $L'_{-1}$ from \tref{none7}.

Then, for the Weisfeiler filtration $\mathcal{F}$ associated to the pair $(L'_{-1}, \w_{E_7})$ with corresponding graded Lie algebra $\Gc$ we have that \[\mathfrak{sl}(2)\otimes \mathcal{O}(3;\underline{1})\subset\bar{\Gc} \subseteq (\mathfrak{sl}(2)\otimes \mathcal{O}(3;\underline{1}))\rtimes (1_{\mathfrak{sl}_2}\otimes W(3;\underline{1})).\]
\end{conjecture}
\begin{proof}[Some Remarks]
Since $\Gc_0 \cong \Der(\mathfrak{sl}(2)\otimes \mathcal{O}(2;\underline{1}))$, we should have that $I\cong A_0(\bar{\Gc})$. By simple dimension reasons it should be the case that $A(\bar{\Gc}) \cong \mathfrak{sl}(2)\otimes \mathcal{O}(3;\underline{1})$, as $\mathcal{O}(4;\underline{1})$ is already too big to be contained in $\bar{\Gc}$. This is supported by the numerology with $\dim\, \bar{\Gc}=107= \dim\,A(\bar{\Gc})+\dim\,W(2;\underline{1})+\dim\,A$.\end{proof}

For $\g$ of type $E_8$, $\dim\,\bar{\Gc}=222$. This is more complicated since $\Gc_0 \cong J\rtimes (1_{\mathfrak{sl}_2}\otimes W(2;\underline{1}))$ where $J$ lies in the short exact sequence \[0 \rightarrow \mathfrak{psl}(3)\otimes \mathcal{O}(2;\underline{1}) \rightarrow J \rightarrow \mathfrak{psl}(3)\rightarrow 0.\] We have $\bar{\Gc}_{\ge 2}=0$, and $[A,[A,\bar{\Gc}_{-1}]]=0$. The second part follows since $[J,A]=0$, and in producing $L'_{-1}$ we used elements $u$ such that $[u,A] \subseteq J$.

\begin{conjecture}Let $p=3$ and suppose $\g$ has type $E_8$ with maximal subalgebra $\w_{E_8}$ and $\w_{E_8}$-invariant subspace $L'_{-1}$ from \tref{none8}.

Then, for the Weisfeiler filtration $\mathcal{F}$ associated to the pair $(L'_{-1}, \w_{E_8})$ with corresponding graded Lie algebra $\Gc$ we have that \[\mathfrak{psl}(3)\otimes \mathcal{O}(3;\underline{1})\subset\bar{\Gc} \subseteq (G_2\otimes \mathcal{O}(3;\underline{1}))\rtimes (1_{\mathfrak{psl}_3}\otimes W(3;\underline{1})).\]
\end{conjecture}
\begin{proof}[Some Remarks]
We should have $I\cong A_0(\bar{\Gc})$, and by dimension reasons $A(\bar{\Gc}) \cong \mathfrak{psl}(3) \otimes \mathcal{O}(3;\underline{1})$ since it is a minimal ideal of $\bar{\Gc}_0$. This is supported when we consider that $\dim\, \bar{\Gc}=222$. We also have $\dim\,A(\bar{\Gc})=189$, $\dim\,W(2;\underline{1})=18$ and $\dim\,A=8$. Since $J$ extends $\mathfrak{psl}(3)$, this gives the remaining dimension of $7$.\end{proof}

\section[Counterexamples to Morozov's theorem: $E_8$ $(2)$]{Counterexamples to Morozov's theorem: $E_8$ $(2)$}

We focus on an interesting example of a non-semisimple maximal subalgebra in the exceptional Lie algebra of type $E_8$. This does not appear in the paper \cite{lmt08}, consider $e=\rt{1}+\rt{2}+\rt{3}+\rt{5}+\rt{6}+\rt{8}$ to be the standard representative for the orbit $\mathcal{O}({A_2}^{2}+{A_1}^{2})$.

For fields of characteristic $p>3$, we know $\dim\,\Lie(G_e)=80$ by \cite{LT11} but \cite{S16} gives $\dim\,\g_e=84$ for $p=3$. Since the representative is the same for all $p$ we continue to use the associated cocharacter $\tau$ given by \[\subalign{2\quad2\quad-&5\quad2\quad2\quad{-3}\quad2\\&2}\] from \cite[pg. 131]{LT11}.

\begin{prop}\label{e84}The centraliser $\me$ is $84$-dimensional with a depth-one grading induced by $\tau$. Further, the new elements have the following properties \begin{align*}
u_1 &=e_{\subalign{11&10000\\&0}}+2e_{\subalign{01&10000\\&1}}+e_{\subalign{00&11000\\&1}}+e_{\subalign{01&11000\\&0}}+e_{\subalign{00&11100\\&0}} \in \g_e(\tau,-1),
\\u_2 &=e_{\subalign{12&21110\\&1}}+e_{\subalign{11&22110\\&1}}+e_{\subalign{11&21111\\&1}}+e_{\subalign{01&22210\\&1}}+2e_{\subalign{01&22111\\&1}} \in \g_e(\tau,-1),
\\u_3&=f_{\subalign{11&10000\\&1}}+f_{\subalign{11&11000\\&0}}+f_{\subalign{01&11000\\&1}}+f_{\subalign{00&11100\\&1}}+2f_{\subalign{01&11100\\&0}}\in \g_e(\tau,-1),
\\u_4&=f_{\subalign{12&22110\\&1}}+2f_{\subalign{12&21111\\&1}}+2f_{\subalign{11&22210\\&1}}+2f_{\subalign{11&22111\\&1}}+2f_{\subalign{01&22211\\&1}}\in \g_e(\tau,-1).
\end{align*}\end{prop}

\begin{proof}The existing $80$ elements are still linearly independent, and the new elements satisfy $[u_i,e]=0$. It is clear that these new elements are linearly independent. \end{proof}

We obtain the following information using GAP \begin{enumerate}\item{The radical is one-dimensional and isomorphic to the line $\bbk e$.}\item{Both $L:=\mathfrak{g}_e/\bbk e$ and $M:=\mathfrak{n}_e/\bbk e$ are restricted and semisimple.}\item{The Lie algebra $(\mathfrak{n}_e/\bbk e)^{(2)}$ is simple and restricted.}\end{enumerate}

Using \cite[pg. 339]{Strade09} there are some basic facts about $H(4;\underline{1})^{(1)}$ to recall. The usual basis in characteristic three for $H(4;\underline{1})^{(1)}$ given by $\{D_H(x^{(a)}):0 \le a_i < 3, 0 < \sum_i a_i < 8\}$, and further \[H(4;\underline{1}) = H(4;\underline{1})^{(1)}\oplus \bbk D_H(x^{(2)})\oplus (\bbk {x_1}^2\partial_3+\bbk {x_2}^2\partial_4+\bbk {x_3}^2\partial_1+\bbk {x_4}^2\partial_2).\] It follows that \[\Der\,H(4;\underline{1})^{(1)}=H(4;\underline{1})\oplus\bbk\left(\sum_{i=1}^4 x_i\partial_i\right)\] since $CH(4;\underline{1})=\Der\,H(4;\underline{1})^{(1)}$ by \cite[Theorem 7.1.2(3)(c)]{Strade04}.

\begin{thm}\label{CH41}Let $G$ be an algebraic group of type $E_8$ over an algebraically closed field of characteristic three with Lie algebra $\g=\Lie(G)$. Let $e \in \mathcal{O}({A_2}^2+{A_1}^2)$ be a nilpotent representative, then the following are true
\begin{enumerate}[label=(\alph*)]
\item{\label{ch41a}$L^{(2)}\cong H(4;\underline{1})^{(1)}$ as Lie algebras.}
\item{\label{ch41b}$\nn_e/\bbk e\subseteq \Der\,H(4;\underline{1})^{(1)}$ as Lie algebras.}
\item{\label{ch41c}$H(4;\underline{1})^{(1)}\subseteq\g_e/\bbk e\subseteq H(4;\underline{1})$ as Lie algebras.}
\item{\label{ch41d}$\mathfrak{n}_e$ is maximal in $E_8$.}\end{enumerate}\end{thm}

\begin{proof}
By \pref{e84}, we have $\dim\,\me(\tau,-1)=4$, and so $L$ has a depth-one grading with $\dim\, L_{-1}=4$ since $\bbk e \in \g_e(\tau,2)$. It is easy to check the second derived subalgebra is depth-one graded and simple using GAP. By \cite[pg. 131]{LT11}, we have $L_0$ is simple of type $C_2$.

We check this in GAP by computing $\g_e(\tau,0)$, and see it is the same as in the good characteristic case. If we take the derived subalgebra, then we still obtain the same zero component. Hence, this is a depth-one graded simple Lie algebra of dimension $79$ with simple zero component of classical type. We may apply \cite[Theorem 1]{ko95} to obtain \ref{ch41a}.

Combining \cite[Lemma 5.2.3 and Proposition 2.7.3]{Strade04} to use \tref{weakrec} in the same way as \tref{nonf4} on $\g_e/\bbk e$ we have $H(4;\underline{1})^{(1)}\subseteq \g_e/\bbk e\subseteq H(4;\underline{1})$. Since $\dim\, \g_e/\bbk e=83$ we obtain \ref{ch41c}. We have $L^{(1)} \cong H(4;\underline{1})^{(1)}$, and $L(\tau,6)=0$ so we may deduce that \[L \cong H(4;\underline{1})^{(1)}\oplus (\bbk {x_1}^2\partial_3+\bbk {x_2}^2\partial_4+\bbk {x_3}^2\partial_1+\bbk {x_4}^2\partial_2).\]Part \ref{ch41b} follows since the action of $\mathfrak{n}_e/\bbk e$ on our $H(4;\underline{1})^{(1)}$ is faithful, and so $M$ embeds into the derivation algebra. Since $\mathfrak{n}_e=\g_e \oplus\bbk h$, it follows \[M \cong H(4;\underline{1})^{(1)}\oplus (\bbk {x_1}^2\partial_3+\bbk {x_2}^2\partial_4+\bbk {x_3}^2\partial_1+\bbk {x_4}^2\partial_2)\oplus \bbk\left(\sum_{i=1}^4x_i\partial_i\right).\]

Setting $M_0:=\nn_e$ we try to find $M_{-1}$ such that $M_{-1}/M_0$ is an irreducible $M_0$-module. The set $M'_{-1}:=\{x \in \g:[x,e] \subseteq M_0\}$ is $168$-dimensional with $\langle M'_{-1} \rangle \cong \g$ and basis given in \autorefs{tabe82}. It turns out that although $M'_{-1}/M_0$ is not irreducible, it is indecomposable using the methods described in \aref{sec:ch41}. To prove maximality using \rref{whymax} in the same way to previous sections we need to locate the required $M_{-1}$ such that $M_{-1}/M_0$ is irreducible as an $M_0/A$-module.

Using \aref{sec:ch41} we find an $M_0$-invariant space $M_{-1}$ of dimension $164$ such that $M_{-1}/M_0$ is irreducible. Then we are able to use our normal argument using \rref{whymax} with this new $M_{-1}$, as this plays the role of $M_{(-1)}$ in the usual definition of a Weisfeiler filtration. We then confirm that $\langle M_{-1}\rangle=\g$ in GAP as required.\end{proof}

\begin{rem}
It should be noted that finding this module was extremely fortuitous, as there is no obvious ideal to recreate the situations of \tref{none7} and \ref{none8}. Had we been unable to locate such a module, we could say that $M_{-1}$ contains at least one element from $M'_{-1}$. We then simply verify for every $v$ of $M'_{-1}$ that $\langle M_0,v \rangle=M_0$ or $\g$. In either case we obtain part \ref{ch41d}.\end{rem}

%\end{document}

\chapter[Maximal subalgebras in fields of characteristic two]{Maximal subalgebras over fields of characteristic two}\label{sec:dreams}
We consider maximal subalgebras in exceptional Lie algebras for fields of characteristic two. Many ideas we have used are still useful in this chapter, although to identify simple Lie algebras we may need to provide explicit isomorphisms. Luckily for us, many of the simple Lie algebras we encounter have reasonably small dimensions.

\section[Simple Lie algebras in characteristic two]{Simple Lie algebras in characteristic two}\label{sec:chartwo}

There is no detailed list of known simple Lie algebras in characteristic two, so in this section we say something about the classical and Cartan type Lie algebras. From \autorefs{badp} in \autoreft{sec:notation} we see that the exceptional Lie algebras $G_2$, $E_6$ and $E_8$ are simple for characteristic two --- but $F_4$ and $E_7$ are not.

In $F_4$, there is an ideal generated by the short roots. This is obtained by generating a subalgebra indexed by the following roots \begin{align}\begin{split}\{\pm\{0001\}, \pm\{0011\}, \pm\{0111\}&, \pm\{1111\}, \\ \pm\{0121\}, \pm\{1121\}, \pm\{1221\}&, \pm\{1231\},\\ \pm\{0110\}, \pm\{1110\}, \pm\{0010\}&, \pm\{1232\}\},\end{split}\end{align} to produce a $26$-dimensional simple Lie algebra. There is a grading on this Lie algebra $I_{F_4}$, setting the following degrees on each simple root
\begin{align}\label{f4newgrading}\begin{split}\degs(\al_1)&=1\\ \degs(\al_2)&=0 \\ \degs(\al_3)&=-1 \\ \degs(\al_4)&=1\end{split}\end{align} This gives a short grading $I_{F_4}=(I_{F_4})_{-1}\oplus (I_{F_4})_0 \oplus (I_{F_4})_1$, such that $(I_{F_4})_0$ is a simple Lie algebra of dimension $14$. It is possible to show that $(I_{F_4})_0$ from the above description of the ideal in $F_4$ is in fact isomorphic to $\mathfrak{psl}(4)$.

We also have $F_4/I_{F_4} \cong I_{F_4}$ as Lie algebras, and to see this consider the remaining roots of $F_4$ with the following degrees \begin{align}\label{f4new2grading}\begin{split}\degs(\al_1)&=1\\ \degs(\al_2)&=-1 \\ \degs(\al_3)&=\frac{1}{2} \\ \degs(\al_4)&=0.\end{split}\end{align} It is easy to verify this gives the exact same grading as above for $I_{F_4}$, and from here an isomorphism is obvious. This is almost identical to the situation for $G_2$ in characteristic three, where the adjoint module is indecomposable. For $E_7$ there is a one-dimensional centre, and $\g/\mathfrak{z}(\g)$ is simple of dimension $132$.

In fields of characteristic two only Lie algebras of type $A_n$ are simple, and this is only when $n$ is even. For $n$ odd we have a non-zero centre, and any Lie algebra of type $B_n$, $C_n$ or $D_n$ have non-zero centres. In Lie algebras of type $C_n$ we take derived subalgebras to obtain simple Lie algebras.

Consider the case $n=3$ with $\mathfrak{sp}(6)$ of dimension $21$. Taking the second derived subalgebra gives a simple Lie algebra ${\mathfrak{sp}(6)}^{(2)}$ of dimension $14$. Since all Lie algebras of type $C_n$ are subalgebras of $\mathfrak{sl}(n+1)$, we have that ${\mathfrak{sp}(6)}^{(2)}\subseteq \mathfrak{psl}(4)$. By dimension reasons we must have ${\mathfrak{sp}(6)}^{(2)}\cong \mathfrak{psl}(4)$.

\begin{rem}\label{spremark}For characteristic two, Lie algebras of type $C_n$ have one-dimensional centres and the first derived subalgebra of $\mathfrak{sp}(2n)$ has dimension $\binom{2n}{2}$. The second derived subalgebra gives a simple Lie algebra of dimension $\binom{2n}{2}-1$.\end{rem}

Repeating \cite[\S4.4]{P15} we may consider the algebraic group of type $G_2$ with Lie algebra $\g$. It follows that $\g \subseteq \mathfrak{sp}(6)$, and since $\g$ is simple it must be that $\g$ actually lies in the second derived subalgebra. By dimension reasons $\g\cong \mathfrak{psl}(4)$. The Cartan type Lie algebras are all still defined in characteristic two, where we note that $W(1;\underline{1})$ is no longer simple, $W(2;\underline{1}) \cong \mathfrak{sl}(3)$, and $H(4;\underline{1})^{(1)} \cong \mathfrak{psl}(4)$.

We finish this section with a result that links together $F_4$ and $H(6;\underline{1})^{(1)}$. We start by showing that there is a simple subalgebra in $H(6;\underline{1})^{(1)}$, and prove it is isomorphic to the ideal generated by short roots in $F_4$.

\begin{prop}\label{newsubpossible6}For algebraically closed fields of characteristic two, the simple Lie algebra $H(6;\underline{1})^{(1)}$ contains a simple subalgebra denoted by $\mathcal{H}(6;\underline{1})$.\end{prop}
\begin{proof}Consider the Lie algebra generated by $\partial_i$ and elements \begin{align}\label{newgrading}\begin{split}u_1&:=D_H(x_1x_2x_5)+D_H(x_1x_3x_6)\\ u_2&:=D_H(x_1x_2x_4)+D_H(x_2x_3x_6)\\ u_3&:=D_H(x_1x_3x_4)+D_H(x_2x_3x_5)\\u_4&:=D_H(x_1x_4x_5)+D_H(x_3x_5x_6)\\ u_5&:=D_H(x_1x_4x_6)+D_H(x_2x_5x_6)\\ u_6&:=D_H(x_2x_4x_5)+D_H(x_3x_4x_6)\end{split}\end{align} Such elements are easily obtained in GAP by insisting all relations are true such as $\partial_1(\partial_2(u_1))=\partial_2$.

It should be noted that we have $\partial_i=D_H(x_{i'})$ in the notation of \eqref{hamultip2}. As an example, we have that \begin{align*}D_H(x_1x_2x_5)&=\partial_1(x_1x_2x_5)\partial_4+\partial_2(x_1x_2x_5)\partial_5-\partial_5(x_1x_2x_5)\partial_2\\&=x_2x_5\partial_4+x_1x_5\partial_5-x_1x_2\partial_2.\end{align*}Hence, it is now an easy exercise to check that $\partial_1(\partial_2(u_1))=\partial_2$.

Generating a subalgebra in GAP gives a Lie algebra of dimension $26$, which we confirm is simple using the MeatAxe. For further details on obtaining the elements $u_i$ above we refer the reader to \aref{hamp2}.\end{proof}

\begin{thm}\label{linkbetweenf4new}For $p=2$, we have that $F_4/I_{F_4}\cong I_{F_4}\cong \mathcal{H}(6;\underline{1})$ as Lie algebras.\end{thm}
\begin{proof}Consider our grading on $L:=I_{F_4}$ from \eqref{f4newgrading}. This grading is such that $L=L_{-1}\oplus L_0 \oplus L_1$ where $L_0 \cong \mathfrak{psl}(4)$ and $\dim\,L_{-1}=\dim\,L_{1}=6$. In $\mathcal{H}(6;\underline{1})$ we have exactly the same situation, since $\mathcal{H}(6;\underline{1})_0 \cong \mathfrak{sp}(6)^{(2)}$. There is a basis for $\mathcal{H}(6;\underline{1})_1$ from \eqref{newgrading} with elements \begin{align*}\begin{split}u_1&:=D_H(x_1x_2x_5)+D_H(x_1x_3x_6)\\ u_2&:=D_H(x_1x_2x_4)+D_H(x_2x_3x_6)\\ u_3&:=D_H(x_1x_3x_4)+D_H(x_2x_3x_5)\\u_4&:=D_H(x_1x_4x_5)+D_H(x_3x_5x_6)\\ u_5&:=D_H(x_1x_4x_6)+D_H(x_2x_5x_6)\\ u_6&:=D_H(x_2x_4x_5)+D_H(x_3x_4x_6)\end{split}\end{align*}

We start the isomorphism mapping in the following way \begin{align*}\partial_1&\mapsto e_{0001}, \partial_2 \mapsto f_{0110} \\ \partial_3&\mapsto e_{1111}, \partial_4 \mapsto f_{0121} \\ \partial_5 &\mapsto f_{0010}, \partial_6 \mapsto f_{1231} \end{align*} To complete the isomorphism we just map the appropriate element of $I_{F_4}$ to each $u_i$. For example; we send $u_1$ to $f_{1111}$, and check the relations such as $\partial_1(\partial_2(u_1))=\partial_2$ are satisfied. This shows that $I_{F_4} \cong \mathcal{H}(6;\underline{1})$ as required. \end{proof}

\section[The nilpotent orbit $\mathcal{O}({A_1}^3)$ in the exceptional Lie algebras]{The nilpotent orbit $\mathcal{O}({A_1}^3)$ in the exceptional Lie algebras}

We explore the nilpotent orbit $\mathcal{O}({A_1}^3)$ in each of the exceptional cases $E_n$ for $n=6,7$ and $8$. Denoted by $\mathcal{O}({A_1'}^3)$ in $E_7$, they all have the same representative $e:=\rt{1}+\rt{4}+\rt{6}$ obtained by restricting $E_7$ and $E_8$ to a subsystem of type $E_6$. This element is the same for all $p$, and so we continue to use the associated cocharacters $\tau$ given on \cite[pg. 81, 94 and 124]{LT11}. These cocharacters $\tau$ are labelled as \begin{align*}&\subalign{2\quad{-2}\quad&{2}\quad{-2}\quad2\\-&{1}}\\&\subalign{2\quad{-2}\quad&2\quad{-2}\quad2\quad{-1}\\-&{1}} \\&\subalign{2\quad{-2}\quad&{2}\quad{-2}\quad2\quad{-1}\quad0\\-&{1}}\end{align*} in $E_6$, $E_7$, and $E_8$ respectively.

By \cite[Section 4]{lmt08} there are non-zero elements $X,Y \in \g_e(\tau,-1)$ with \begin{align}\label{xy}\begin{split}
 X&=e_{\subalign{11&110\\&1}}+e_{\subalign{01&210\\&1}}+e_{\subalign{01&111\\&1}} \in \g_e(\tau,-1),
\\Y&=f_{\subalign{11&210\\&1}}+f_{\subalign{11&111\\&1}}+f_{\subalign{01&211\\&1}}\in \g_e(\tau,-1).\end{split}
\end{align}

Then, $\Lie(G_e) \subseteq \bigoplus_{i\ge 0} \g_e(\tau,i)$ and \[\g_e=\bigoplus_{i \ge -1} \g_e(\tau,i) \quad \text{and}\quad \dim\,\g_e(\tau,-1)=2,\] as in the characteristic three case for the orbit $\mathcal{O}({A_2}^2+A_1)$. The elements of $\Lie(G_e)$ given on \cite[pg. 81, 94 and 124]{LT11} continue to be linearly independent in characteristic two --- although we adjust the coefficients modulo $2$. Together with $X,Y$ we may form a basis of $\g_e$ for all cases. Using GAP we find

\begin{enumerate}\item{The radical $A$ of $\mathfrak{n}_e$ is abelian of dimension $3$, unless $\g$ has type $E_7$ where $A$ is abelian of dimension $4$.}
\item{$M_n:=N_{\g}(A)$ has dimension $43$, $74$ and $141$ respectively.}
\item{$N_{\g}(A)/A$ is semisimple in all cases, and $\dim\,N_{\g}(A)/A=\dim \g_e$.}
\item{The set $L_{-1}:=\{x \in \g: [x,A] \subseteq N_{\g}(A)\}$ is such that $\dim\, L_{-1}=\dim\,\g$.}\end{enumerate}

\begin{thm}\label{p2newmax}Let $G$ be an algebraic group of type $E_n$ over an algebraically closed field of characteristic two with Lie algebra $\g=\Lie(G)$. For any $e \in \mathcal{O}({A_1}^3)$ we have that $M_n$ is a maximal subalgebra.\end{thm}

\begin{proof}Since $L_{-1}= \g$, if $L_{-1}/M_n$ is irreducible we are done. Unfortunately this does not happen, and appears to be because the semisimple Lie algebra $M_n/A$ contains a minimal ideal $I_n$. To construct $I_n$ we apply $\ad(\g_e(\tau,-1))$ to elements of $\g_e(\tau,2)$ and produce a subalgebra in $M_n$.

We then construct the set $N_n:=\{x \in \g: [x,A] \subseteq I_n\}$. In all cases $\dim \, N_n = \dim\,\g-3$, and that $N_n/M_n$ is an irreducible $M_n$-module with Tables \ref{tabe62}, \ref{tabe8p2} and \ref{tabe7p22} giving the details on how to obtain this in GAP.

Consider $\langle N_n \rangle$, which has dimension equal to $\dim \, \g$. We obtain maximality since $L_{-1}/L_0$ is indecomposable, and so any subalgebra strictly containing $M_n$ contains $N_n$. It could also be deduced since there are no subalgebras of codimension $3$ in the exceptional Lie algebras.
\end{proof}
We use the basis given \cite[pg. 81]{LT11}, and apply $\ad(\g_e(\tau,-1))$ to $\g_e(\tau,2)$. This gives an ideal $I_6$ of our Lie algebra $M_6$ that is $35$-dimensional.

\begin{thm}\label{none6p2}Consider $\g=\Lie(G)$ where $G$ has type $E_6$ over an algebraically closed field of characteristic two. For any $e \in \mathcal{O}({A_1}^3)$ we have the following \begin{enumerate}[label=(\alph*)]
\item{$I_6/A \cong \mathfrak{sl}(3)\otimes \mathcal{O}(2;\underline{1})$ as Lie algebras.}
\item{$M_6/A\cong (\mathfrak{sl}(3)\otimes \mathcal{O}(2;\underline{1}))\rtimes (1_{\mathfrak{sl}_3}\otimes W(2;\underline{1}))$ as Lie algebras.}
\end{enumerate}\end{thm}

\begin{proof}
We have that $\dim\,I_6=35$, and $L:=M_6/I_6$ is an $8$-dimensional Lie algebra with grading $L_{-1}\oplus L_0 \oplus L_1$ where $\dim\,L_{\pm 1}=2$ and $L_0 \cong \mathfrak{gl}(2)$. In characteristic two the grading of $W(2;\underline{1})$ is identical to this, and it is easy to verify $L \cong W(2;\underline{1}) \cong \mathfrak{sl}(3)$.

To determine the isomorphism class of $I:=I_6/A$, we note that using GAP $I$ has dimension $32$ with nilpotent radical of dimension $24$. To confirm that $I \cong \mathfrak{sl}(3)\otimes \mathcal{O}(2;\underline{1})$, we know by \tref{blocktheorem2} that $I \cong S \otimes \mathcal{O}(m;\underline{n})$ since $I$ is minimal. By \cite[pg. 81]{LT11} it follows that $S \cong \mathfrak{sl}(3)$ since $\mathfrak{sl}(3) \subseteq I_6$, and we can verify $\rad(I_6) \subseteq \g_e(\tau, >0)$. This forces $\mathcal{O}(m;\underline{n})=\mathcal{O}(2;\underline{1})$, and hence $M_6 \cong (\mathfrak{sl}(3)\otimes \mathcal{O}(2;\underline{1}))\rtimes (1_{\mathfrak{sl}_3}\otimes W(2;\underline{1}))$ by Block's theorem.
\end{proof}

For the exceptional Lie algebras of type $E_8$, we compute minimal ideal $I_8$ in GAP by applying $\ad(\g_e(\tau,-1))$ to $\g_e(\tau,2)$. This produces an ideal of dimension $107$, but following similar ideas to \tref{none8} we may ``add'' some elements from $\g_e(\tau,0)$ to obtain a bigger ideal $J_8$ of dimension $133$. The basis of both these ideals is given in \aref{AppE8new}.

\begin{prop}\label{none8p2}Let $\g=\Lie(G)$ where $G$ has type $E_8$ over an algebraically closed field of characteristic two. For any $e \in \mathcal{O}({A_1}^3)$ we have the following \begin{enumerate}[label=(\alph*)]
\item{\label{p2e8a}$M_8/J_8 \cong W(2;\underline{1})$, and $J_8/I_8 \cong \mathcal{H}(6;\underline{1})$ as Lie algebras.}
\item{\label{p2e8b}$I_8/A \cong \mathcal{H}(6;\underline{1})\otimes \mathcal{O}(2;\underline{1})$ as Lie algebras.}
\item{\label{p2e8c}$M_8/A\cong J\rtimes W(2;\underline{1})$ as Lie algebras, where $J$ lives in the short exact sequence \[0 \rightarrow \mathcal{H}(6;\underline{1})\otimes \mathcal{O}(2;\underline{1}) \rightarrow J \rightarrow \mathcal{H}(6;\underline{1})\rightarrow 0.\]}
\end{enumerate}\end{prop}

\begin{proof}We use the basis for $\g_e$ given on \cite[pg. 124]{LT11}, and apply $\ad(\g_e(\tau,-1))$ to elements of $\g_e(\tau,2)$ to obtain ideal $I_8$. After adding the remaining elements from $\g_e(\tau,0)$ we obtain a maximal ideal $J_8$ such that $\dim\,J_8=133$. It also follows that $L:=M_8/J_8$ is an $8$-dimensional simple Lie algebra with grading $L_{-1}\oplus L_0 \oplus L_1$ where $\dim\,L_{\pm 1}=2$ and $L_0 \cong \mathfrak{gl}(2)$. This is enough to deduce $L \cong W(2;\underline{1})$, proving \ref{p2e8a}.

Using GAP we can show $M:=(I_8/A)/(\rad(I_8/A))$ is a $26$-dimensional simple Lie algebra strictly contained in $\g_e(\tau,0)$. This follows since all elements of our ideal $I_8$ not contained in $\g_e(\tau,0)$ form a solvable ideal of $I_8$. By \cite{LT11}, $\g_e(\tau,0)=F_4\oplus \mathfrak{sl}(2)$. Since $M$ is simple of dimension $26$ strictly contained in $\g_e(\tau,0)$ we conclude that $M \cong I_{F_4}$. It follows from \tref{linkbetweenf4new} that $M \cong \mathcal{H}(6;\underline{1})$.

Since $I_8/A$ is a minimal ideal in $M_8/A$, we can use \tref{blocktheorem2} to see it has the form $S \otimes \mathcal{O}(m;\underline{n})$. We have that $\dim\,S=26$, thus by dimension reasons $m=2$ and $n=(1,1)$. It follows $I_8/A \cong \mathcal{H}(6;\underline{1})\otimes \mathcal{O}(2;\underline{1})$ to give \ref{p2e8b}.

Finally for \ref{p2e8c}, using $\g_e(\tau,-1)$ and certain elements of $\g_e(\tau,1)$ we may produce an $8$-dimensional simple subalgebra isomorphic to $W(2;\underline{1})$. Hence, $M_8/A$ has a complementary subalgebra to $J_8/A$ isomorphic to $\mathcal{H}(6;\underline{1})$, and using the obvious maps we see $J$ lives in the short exact sequence $0 \rightarrow \mathcal{H}(6;\underline{1})\otimes \mathcal{O}(2;\underline{1}) \rightarrow J \rightarrow \mathcal{H}(6;\underline{1})\rightarrow 0$ to give \ref{p2e7c}.
\end{proof}

We have left the final case for type $E_7$ until now because the structure of our maximal ideal is different. We compute $I_7$ in GAP by applying $\ad(\g_e(\tau,-1))$ to $\g_e(\tau,2)$, and obtain an ideal of dimension $59$.

We use the basis for $\g_e$ given on \cite[pg. 94]{LT11}, and apply $\ad(\g_e(\tau,-1))$ to elements of $\g_e(\tau,2)$ to obtain ideal $I_7$. However, unlike \tref{none8} we ``add'' elements from $\mathfrak{n}_e$, that generate a simple Lie algebra of dimension $8$ isomorphic to $W(2;\underline{1})$, to obtain an ideal $J_7$ with $\dim J_7=67$. The basis for $J_7$ can be found in \autorefs{e7J7}. This is a consequence of the structure of $\mathfrak{sp}(2n)$ in characteristic two, as we must take the derived subalgebra after factoring out the centre to obtain a simple Lie algebra.

\begin{prop}\label{none7p2}Let $\g=\Lie(G)$ where $G$ has type $E_7$ over an algebraically closed field of characteristic two. For any $e \in \mathcal{O}({A_1'}^3)$ we have the following \begin{enumerate}[label=(\alph*)]
\item{\label{p2e7a}$J_7/I_7 \cong W(2;\underline{1})$ as Lie algebras.}
\item{\label{p2e7b}$J_7/A \cong (\mathfrak{psl}(4)\otimes \mathcal{O}(2;\underline{1}))\rtimes (1_{\mathfrak{psl}_4}\otimes W(2;\underline{1}))$ as Lie algebras.}
\item{\label{p2e7c} $(M_7)^{(1)}/A \cong J_7/A$, and $M_7/J_7$ is a solvable $6$-dimensional Lie algebra.}
\end{enumerate}\end{prop}

\begin{proof}
Using GAP $J_7/I_7$ is $8$-dimensional and simple with grading $L_{-1}\oplus L_0 \oplus L_1$ where $\dim\,L_{\pm 1}=2$ and $L_0 \cong \mathfrak{gl}(2)$. In characteristic two the grading of $W(2;\underline{1})$ is identical to this, and we may conclude $L \cong W(2;\underline{1}) \cong \mathfrak{sl}(3)$. We obtain that $\dim\,J_7=67$. It also follows that $L:=J_7/I_7$ is an $8$-dimensional simple Lie algebra with precisely the same grading as $L$ above. This gives part \ref{p2e7a}.

Using the ideas from \aref{AppE72} we compute that $\dim\,\rad(I_7/A)=42$, and $M:=(I_7/A)/(\rad(I_7/A))$ is a simple Lie algebra of dimension $14$ strictly contained in $\g_e(\tau,0)$. Since $\g_e(\tau,0)=\mathfrak{psl}(4)\oplus \mathfrak{sl}(2)$, we deduce that $M \cong \mathfrak{psl}(4)$. Since $I_7/A$ is a minimal ideal in $M_7/A$, applying \tref{blocktheorem2} we have $I_7/A \cong \mathfrak{psl}(4)\otimes \mathcal{O}(2;\underline{1})$ as required. Then, we obtain part \ref{p2e7b} by dimension reasons since $J_7/A \subseteq \Der(\mathfrak{psl}(4)\otimes \mathcal{O}(2;\underline{1}))=(\mathfrak{psl}(4)\otimes \mathcal{O}(2;\underline{1}))\rtimes (1_{\mathfrak{psl}_4}\otimes W(2;\underline{1}))$.

It is clear that $\dim\, M_7/J_7 =6$, and using GAP we compute the derived series to see it is solvable. In this case $\g_e(\tau,0)$ has type $C_3+A_1$, and $\dim \g=21$ for $\g$ of type $C_3$. Hence, if we take the derived subalgebra of $\mathfrak{n}_e$ we obtain a Lie algebra of dimension $67$, and it is clear this must be equal to $J_7$ to finish part \ref{p2e7c}.
\end{proof}

\subsection*{The Weisfeiler filtration for $E_{6,7,8}$}

All the results, \tref{none6p2}, \ref{none8p2}, and \ref{none7p2} exhibit the same structure in the $M_n$-invariant subspace $L_{-1}$ in $\g$. We initially find that $L_{-1}:=\{x \in \g: [x,A] \subseteq M_n\}$ is indecomposable but not irreducible. Using Appendices \ref{AppE6p2}, \ref{AppE7p2} and \ref{AppE8new} we are able to calculate an $M_n$-invariant subspace $L'_{-1}$, such that $L'_{-1}/L_0$ is irreducible.

There are some results about the corresponding graded Lie algebra we are able to give, building a Weisfeiler filtration using $M_n$ and $L'_{-1}$. Since $L_{-1}$ contains some elements $x \notin L'_{-1}$ such that $[x,A] \subseteq L_0$, the Weisfeiler radical $M(\Gc)$ should be non-zero.

\begin{prop}\label{weise7p2}Let $p=2$ and $n \in \{6,8\}$. Suppose $\g$ has type $E_n$ with maximal subalgebra $M_n$ and $M_n$-invariant subspace $L'_{-1}$ from \tref{none6p2} and \tref{none8p2}.

Then, for the Weisfeiler filtration $\mathcal{F}$ associated to $(L'_{-1}, M_n)$ with corresponding graded Lie algebra $\Gc$ we have that $\dim\,M(\Gc)=3$. Hence, $\Gc/M(\Gc)$ is a semisimple Lie algebra of dimension $\dim\,\g-3$.\end{prop}
\begin{proof}
We obtain the following dimensions in the Weisfeiler filtration. \begin{enumerate}\item{$\dim \,L'_{-2}=\dim \, \g$,}
\item{$\dim \,L'_{-1}=\dim\,\g-3$,}
\item{$\dim \,L_{0}=\dim\,M_n$,}
\item{$\dim\, L_1=\dim\, A=3$.}\end{enumerate}
Suppose $x \in L_{-1}\setminus L'_{-1}$. In $\Gc$ we have $[x+L'_{-1},y]=[x,y]+L_0$ for all $y \in A$. However, $[x,A] \subseteq L_0$ and so $[x,y]+L_0=L_0$. In particular, $[\Gc_1,x]=0$ and so ${M(\Gc)}_{-2}:=\{x \in L'_{-2}:[x,A] \subseteq M_n\}=\Gc_{-2}$ has dimension $3$ as required.
\end{proof}

The corresponding graded Lie algebra then has $\dim\,\bar{\Gc}=\dim\,\g-3$. In $E_6$ we have that $\bar{\Gc}_0 \cong (\mathfrak{sl}(3)\otimes \mathcal{O}(2;\underline{1}))\rtimes (1_{\mathfrak{sl}_3}\otimes W(2;\underline{1}))$. It should be noted that $\bar{\Gc}_{\ge 2}=0$, and $[A,[A,\bar{\Gc}_{-1}]]=0$. The second part is straightforward to see since $[I,A]=0$. This again looks like the \emph{degenerate} case of the Weisfeiler filtration.

\begin{conjecture}Let $p=2$ and suppose $\g$ has type $E_6$ with maximal subalgebra $M_6$ and $M_6$-invariant subspace $L'_{-1}$ from \tref{none6p2}.

Then, for the Weisfeiler filtration $\mathcal{F}$ associated to the pair $(L'_{-1}, M_6)$ with corresponding graded Lie algebra $\Gc$ we have that \[\mathfrak{sl}(3)\otimes \mathcal{O}(3;\underline{1})\subset\bar{\Gc} \subseteq (\mathfrak{sl}(3)\otimes \mathcal{O}(3;\underline{1}))\rtimes (1_{\mathfrak{sl}_3}\otimes W(3;\underline{1})).\]
\end{conjecture}
\begin{proof}[Some Remarks]
Since $\bar{\Gc}_0 \cong \Der(\mathfrak{sl}(3)\otimes \mathcal{O}(2;\underline{1}))$, we should have that $I\cong A_0(\bar{\Gc})$. It appears to be the case that $A(\bar{\Gc}) \cong \mathfrak{sl}(3)\otimes \mathcal{O}(3;\underline{1})$. We then would have $\dim\, \bar{\Gc}=75= \dim\,A(\bar{\Gc})+\dim\,W(2;\underline{1})+\dim\,A$, which is precisely what we expect.\end{proof}

When $\g$ has type $E_8$, we have that $\dim\,\bar{\Gc}=245$. This is a slightly more complicated situation since $\bar{\Gc}_0 \cong J\rtimes W(2;\underline{1})$ where $J$ lies in the short exact sequence \[0 \rightarrow \mathcal{H}(6;\underline{1})\otimes \mathcal{O}(2;\underline{1}) \rightarrow J \rightarrow \mathcal{H}(6;\underline{1})\rightarrow 0.\] We have $\bar{\Gc}_{\ge 2}=0$, and $[A,[A,\bar{\Gc}_{-1}]]=0$.

\begin{conjecture}Let $p=2$ and suppose $\g$ has type $E_8$ with maximal subalgebra $M_8$ and $M_8$-invariant subspace $L'_{-1}$ from \tref{none8p2}.

Then, for the Weisfeiler filtration $\mathcal{F}$ associated to the pair $(L'_{-1}, M_8)$ with corresponding graded Lie algebra $\Gc$ we have that \[\mathcal{H}(6;\underline{1})\otimes \mathcal{O}(3;\underline{1})\subset\bar{\Gc} \subseteq (\Der(\mathcal{H}(6;\underline{1}))\otimes \mathcal{O}(3;\underline{1}))\rtimes (1_{\mathcal{H}(6;\underline{1})}\otimes W(3;\underline{1})).\]
\end{conjecture}
\begin{proof}[Some Remarks]
We should have that $I_8/A\cong A_0(\bar{\Gc})$, and by dimension reasons $A(\bar{\Gc}) \cong \mathcal{H}(6;\underline{1})\otimes \mathcal{O}(3;\underline{1})$ as $\mathcal{H}(6;\underline{1})\otimes \mathcal{O}(2;\underline{1})$ is a minimal ideal of $\Gc_0$. We would have that $\dim\, \bar{\Gc}=245$, with $\dim\,A(\bar{\Gc})=208$, $\dim\,W(2;\underline{1})=8$ and $\dim\,A=3$. The way that $J$ extends $\mathcal{H}(6;\underline{1})$ gives the remaining dimension of $26$.\end{proof}

When $\g$ has type $E_7$, we have $\dim\,\bar{\Gc}=130$. Since $E_7$ has a one-dimensional centre, and the Weisfeiler filtration is only defined for simple Lie algebras we should consider the question in $\g/\mathfrak{z}(\g)$. However, this is not a problem since the centre of $E_7$ is contained in the radical of our maximal subalgebra.

This slightly complicates matters, as $\mathfrak{sp}(2n)$ has strange properties for $p=2$. Our minimal ideal $I_7/A\cong \mathfrak{psl}(4)\otimes \mathcal{O}(2;\underline{1})$, and so \begin{align*}\bar{\Gc}_0 &\subseteq (\Der(\mathfrak{psl}(4))\otimes \mathcal{O}(2;\underline{1}))\rtimes(1_{\mathfrak{psl}_4}\otimes W(2;\underline{1}))\\&=(\mathfrak{sp}(6)\otimes \mathcal{O}(2;\underline{1}))\rtimes(1_{\mathfrak{psl}_4}\otimes W(2;\underline{1}))\end{align*} by \tref{none7p2}. We have $\bar{\Gc}_{\ge 2}=0$, and $[A,[A,\bar{\Gc}_{-1}]]=0$ in this case. Hence, \pref{weise7p2} continues to hold in our case, as $L_{-1}$ has dimension $133$ but $L'_{-1}$ has dimension $130$. Hence, we can conclude in exactly the same way that $\dim\,M(\Gc)=3$. If we factor our the centre, in $\g/\mathfrak{z}(\g)$ we only obtain $\mathfrak{sp}(6)^{(1)}$ when we look at $\bar{\Gc}_0$.

\begin{conjecture}\label{weise7mate}Let $p=2$ and suppose $\g$ has type $E_7$ with maximal subalgebra $M_7$ and $M_7$-invariant subspace $L'_{-1}$ from \tref{none7p2}.

Then, for the Weisfeiler filtration $\mathcal{F}$ associated to the pair $(L'_{-1}, M_7)$ with corresponding graded Lie algebra $\Gc$ we have that \[\mathfrak{psl}(4)\otimes \mathcal{O}(3;\underline{1})\subset\bar{\Gc} \subseteq (\mathfrak{sp}(6)\otimes \mathcal{O}(3;\underline{1}))\rtimes (1_{\mathfrak{psl}_4}\otimes W(3;\underline{1})).\]
\end{conjecture}
\begin{proof}[Some Remarks]
Since $\Gc_0 \subseteq \Der(\mathfrak{psl}(4)\otimes \mathcal{O}(2;\underline{1}))$, we should find $I_7/A\cong A_0(\bar{\Gc})$. It should then be the case that $A(\bar{\Gc}) \cong \mathfrak{psl}(4)\otimes \mathcal{O}(3;\underline{1})$ if we are in the degenerate situation. In this case we have $\dim\,A(\bar{\Gc})=112$ and $112+8+3=130-7$. This is supported by the fact $\dim \, \mathfrak{sp}(6)-\dim\,\mathfrak{psl}(4)=7$.\end{proof}

\section[A final case of non-semisimple maximal subalgebras in $E_7$ and $E_8$]{A final case of non-semisimple maximal subalgebras in $E_7$ and $E_8$}

Consider the orbit $\mathcal{O}({A_1}^4)$ over algebraically closed fields of characteristic $p=2$ in the exceptional Lie algebras of type $E_7$ and $E_8$. Throughout this thesis there have been examples of non-smooth nilpotent orbits such that $e^{[p]}=0$ and \[\g_e=\bigoplus_{i=-1}^n \g_e(\tau,i).\] Apart from one example so far we have always had $\dim\, \g_e(\tau,-1)=2$, it turns out that $\dim\, \g_e(\tau,-1)=6$ or $8$ in $E_7$ and $E_8$ respectively. Let $e:=\rt{2}+\rt{3}+\rt{5}+\rt{7}$ be the standard representative, which remains the same in all characteristics. Hence, we may continue to use the associated cocharacters $\tau$ given by \begin{align*}&\subalign{{-1}\quad{2}\quad-&{3}\quad{2}\quad{-2}\quad2\\&{2}} \\&\subalign{{-1}\quad{2}\quad-&{3}\quad{2}\quad{-2}\quad2\quad{-1}\\&{2}}\end{align*} respectively from \cite{LT11}.

In $E_7$ there is a one-dimensional centre when $p=2$, and this produces another element in the zero degree component of the centraliser. We have that $\dim\,\g_e=70$ due to the tables of \cite{S16} whereas $\dim \, \Lie(G_e)=63$.

\begin{prop}Let $G$ be an algebraic group of type $E_7$ with Lie algebra $\g=\Lie(G)$ for $p=2$ and $e$ be a nilpotent orbit representative for ${A_1}^{4}$. The centraliser has basis the existing $63$ elements and the centre $E_7$ together with the following $6$ elements:\begin{align}\begin{split}
u_1 &=e_{\subalign{00&1000\\&1}}+e_{\subalign{01&1000\\&0}}+e_{\subalign{00&1100\\&0}} \in \g_e(\tau,-1),
\\u_2&=e_{\subalign{00&1110\\&1}}+e_{\subalign{01&1110\\&0}}+e_{\subalign{00&1111\\&0}}\in \g_e(\tau,-1),
\\u_3&=e_{\subalign{12&2110\\&1}}+e_{\subalign{11&2210\\&1}}+e_{\subalign{11&2111\\&1}}\in \g_e(\tau,-1),
\\u_4&=f_{\subalign{01&1000\\&1}}+f_{\subalign{00&1100\\&1}}+f_{\subalign{01&1100\\&0}}\in \g_e(\tau,-1),
\\u_5&=f_{\subalign{01&1110\\&1}}+f_{\subalign{00&1111\\&1}}+f_{\subalign{01&1111\\&0}}\in \g_e(\tau,-1),
\\u_6&=f_{\subalign{12&2210\\&1}}+f_{\subalign{12&2111\\&1}}+f_{\subalign{11&2211\\&1}}\in \g_e(\tau,-1).\end{split}
\end{align}
\end{prop}

The existing elements from \cite{LT11} are linearly independent, and clearly the new elements are also linearly independent. This is an extremely strange case where $\g_e=\mfn_e$. Hence, there is no element $h$ such that $\mathfrak{n}_e=\g_e\oplus \bbk h$. Computing the radical using GAP we have that $A$ has dimension $2$, and $\w:=\mfn_{\g}(A)$ is $71$-dimensional. The new element $f \in \w$ is contained in $\g(\tau,-2)$, and $[e,f]=\mathfrak{z}(\g)$.

\begin{thm}\label{e17a4}Let $G$ be an algebraic group of type $E_7$ over an algebraically closed field of characteristic two with Lie algebra $\g=\Lie(G)$. The following are true for any $e \in \mathcal{O}({A_1}^4)$:
\begin{enumerate}[label=(\alph*)]\item{$\w$ is a maximal subalgebra of $\g$.}
\item{\label{isom1} ${\me}'\cong {\me}^{(2)}\cong H(6;\underline{1})^{(1)}$ as Lie algebras.}
%\item{\label{isom2}$\g_e \cong H(6;\underline{1})^{(1)}\oplus (\sum_{i=1}^6\bbk x_i\partial_{i'}) \subseteq H(6;\underline{1})$ as Lie algebras.}

\end{enumerate}
\end{thm}

\begin{proof}For maximality, consider the set \[L_{-1}:=\{x \in \g: [x,A] \subseteq \w\},\] and show $L_{-1}/\w$ is an irreducible $\w/A$ module. In this case we have $L_{-1}=\g$, and using GAP we find there can be a submodule of dimension $132$. We can obtain this submodule in GAP, and show that it does not contain $f \in \g(\tau,-2)$. In particular, this submodule does not contain $\w$. It follows that $L_{-1}/\w$ is irreducible, and hence $\w$ is a maximal subalgebra of $E_7$.

For \ref{isom1}, we know $H(6;\underline{1})^{(1)}$ has a depth-one grading with $\dim\,L_{-1}=6$ and $L_0 \cong \mathfrak{sp}(6)^{(1)}$. From \cite[pg. 110]{LT11} we observe that $\g_e$ has the correct dimension for the minus one space, and has $\g_e(\tau,0)\cong \mathfrak{sp}(6)$. Consider the map \[\theta:\me(\tau,-1)\mapsto (H(6;\underline{1})^{(1)})_{-1}.\] We map each element of $\g_e(\tau,-1)$ to one $\partial_i$ to fix an order, and consider mapping the elements of $\g_e(\tau,3)$ to an appropriate choice in $H_3:=(H(6;\underline{1})^{(1)})_3$.

To make this appropriate choice we consider how the multiplication of $\partial_i$ works on $D_H(f)$ for $f \in H_3$ inside $H(6;\underline{1})^{(1)}$. Then, we can do the same in $E_7$ and check that the necessary relations hold. This produces the necessary isomorphism.

Since the subalgebra $\g_e'$ generated by $\g_e(\tau,-1)$ and $\g_e(\tau,1)$ is a subalgebra of $\g_e^{(2)}$ that is simple and has dimension $62$ (checking using GAP) by dimension reasons we have $\g_e^{(2)}=\langle\g_e(\tau,\pm 1) \rangle \cong H(6;\underline{1})^{(1)}$ giving part \ref{isom1}.

\end{proof}

The Lie algebra $\g_e/A$ is $68$-dimensional, and has derived subalgebra simple of dimension $62$. Recall from \rref{hamstructure} that the elements $x_j\partial_{j'}$ with $j'$ given in \eqref{hamstuff} do not lie in $H(2n;\underline{1})^{(1)}$. Using GAP we may check that the elements we lose in the derived subalgebra all have degree zero. Since $\g_e$ has no elements of degree $4$ we may deduce that $\g_e/A \cong H(6;\underline{1})^{(1)}\oplus \bigoplus_i(\bbk x_j\partial_{j'})$ as Lie algebras since the elements $x_j\partial_{j'}$ are elements of $\mathfrak{sp}(6)$.

Our final example of a non-semisimple maximal subalgebra will be in the exceptional Lie algebra of type $E_8$. Consider the nilpotent orbit with the same label from $E_7$ with standard representative $e:=\rt{2}+\rt{3}+\rt{5}+\rt{7}$. The tables of \cite{S16} give $\dim\,\Lie(G_e)=120$ and $\dim \,\g_e=128$.

A quick check reveals that $\g_e$ contains the same $120$ elements as $\Lie(G_e)$ from \cite[pg. 172]{LT11} with coefficients reduced modulo $2$ together with the following new $8$ elements\begin{align}\begin{split}
u_1 &=e_{\subalign{00&10000\\&1}}+e_{\subalign{01&10000\\&0}}+e_{\subalign{00&11000\\&0}}\in \g_e(\tau,-1),
\\u_2&=e_{\subalign{00&11100\\&1}}+e_{\subalign{01&11100\\&0}}+e_{\subalign{00&11110\\&0}}\in \g_e(\tau,-1),
\\u_3&=e_{\subalign{12&21100\\&1}}+e_{\subalign{11&22100\\&1}}+e_{\subalign{11&21110\\&1}}\in \g_e(\tau,-1),
\\u_4&=e_{\subalign{12&32211\\&2}}+e_{\subalign{12&33211\\&1}}+e_{\subalign{12&32221\\&1}}\in \g_e(\tau,-1),
\\u_5&=f_{\subalign{01&10000\\&1}}+f_{\subalign{00&11000\\&1}}+f_{\subalign{01&11000\\&0}}\in \g_e(\tau,-1),
\\u_6&=f_{\subalign{01&11100\\&1}}+f_{\subalign{00&11110\\&1}}+f_{\subalign{01&11110\\&0}}\in \g_e(\tau,-1),
\\u_7&=f_{\subalign{12&22100\\&1}}+f_{\subalign{12&21110\\&1}}+f_{\subalign{11&22110\\&1}}\in \g_e(\tau,-1),
\\u_8&=f_{\subalign{12&33211\\&2}}+f_{\subalign{12&32221\\&2}}+f_{\subalign{12&33221\\&1}}\in \g_e(\tau,-1).\end{split}
\end{align}

The normaliser $\mfn_e:=\g_e \oplus \bbk h$, where $[h,e]=e$. The radical of both will be exactly the same, and we can show that $A$ has dimension $1$. Taking the second derived subalgebra of $\g_e$ we find a simple Lie algebra of dimension $118$. Hence, $\rad(\g_e)$ is at most $10$-dimensional. Using GAP we verify none of these elements together with $e$ forms a solvable ideal in $\g_e$.

\begin{prop}\label{newsubpossible8}For algebraically closed fields of characteristic two, the simple Lie algebra $H(8;\underline{1})^{(1)}$ contains a simple subalgebra denoted by $\mathcal{H}(8;\underline{1})$.\end{prop}
\begin{proof}We use the same idea as \pref{newsubpossible6}, and consider the $\partial_i$ along with elements \begin{align}\label{newgrading2}\begin{split}v_1&:=D_H(x_1x_2x_3x_6x_7)+D_H(x_1x_2x_4x_6x_8)+D_H(x_1x_3x_4x_7x_8)\\ v_2&:=D_H(x_2x_1x_3x_5x_7)+D_H(x_2x_1x_4x_5x_8)+D_H(x_2x_3x_4x_7x_8)\\ v_3&:=D_H(x_3x_1x_4x_5x_8)+D_H(x_3x_2x_4x_6x_8)+D_H(x_3x_1x_2x_5x_6)\\u_4&:=D_H(x_4x_1x_2x_5x_6)+D_H(x_4x_2x_3x_6x_7)+D_H(x_4x_1x_3x_5x_7)\\ v_5&:=D_H(x_5x_2x_3x_6x_7)+D_H(x_5x_2x_4x_6x_8)+D_H(x_5x_3x_2x_7x_6)\\ v_6&:=D_H(x_6x_1x_3x_5x_7)+D_H(x_6x_1x_4x_5x_8)+D_H(x_6x_3x_4x_7x_8)\\ v_7&:=D_H(x_7x_1x_4x_5x_8)+D_H(x_7x_2x_4x_6x_8)+D_H(x_7x_1x_2x_5x_6)\\ v_8&:=D_H(x_8x_1x_2x_5x_6)+D_H(x_8x_2x_3x_6x_7)+D_H(x_8x_1x_3x_5x_7)\end{split}\end{align} in $(H(8;\underline{1})^{(1)})_3$. Using GAP we see this generates a simple Lie algebra of dimension $118$ with the details given in \aref{hamp2}. \end{proof}

\begin{thm}\label{a14e8}Let $G$ be an algebraic group of type $E_8$ over an algebraically closed field of characteristic two with Lie algebra $\g=\Lie(G)$. The following are true for any $e \in \mathcal{O}({A_1}^4)$:
\begin{enumerate}[label=(\alph*)]
\item{$L_{0}:=\mathfrak{n}_e$ is a maximal subalgebra of $\g$.}
\item{\label{intereste8}$\mathfrak{n}_e^{(3)}/\bbk e \cong \mathcal{H}(8;\underline{1})$ as Lie algebras.}
\end{enumerate}
\end{thm}

\begin{proof}We know that $\dim \, L_0=129$. Consider $L^{+}:= \bigoplus_{i=1}^3 \g_e(\tau,i)$, and the vector space defined by \[L'_{-1}:=\{x \in \g:[x,L^{+}] \subseteq L_0\}.\] Using GAP we find $M$ is $137$-dimensional with $L'_{-1}/L_0$ an $8$-dimensional irreducible $\g_e(\tau,0)$-module. Take any subalgebra $N$ containing $L_0$. Since $L'_{-1}/L_0$ is irreducible, it follows that $N$ must contain these $8$ elements. We find that $\langle N \rangle =\g$, thus $L_0$ is a maximal subalgebra of $\g$.

Alternatively, we can compute \[L_{-1}:=\{x \in \g: [x,e] \subseteq L_0\}.\] It turns out again that $L_{-1}=\g$, and $L_{-1}/L_0$ is indecomposable. A check on GAP finds a submodule of dimension $247$. Since $E_8$ has no subalgebras of codimension $1$ it follows $L_0$ is a maximal subalgebra of $E_8$. The basis for the $247$-dimensional submodule is given in \aref{appendix:weirda1e8}.

For part \ref{intereste8} we check that the third derived subalgebra of $L_0$ is a Lie algebra of dimension $119$ with a one-dimensional centre. It follows that $M:=\mathfrak{n}_e^{(3)}/\bbk e$ is a $118$-dimensional simple Lie algebra. Using cocharacter $\tau$ we have $\dim\,M_{-1}=8$, $M_0\cong \mathfrak{sp}(8)^{(1)}$, and $M_3$ is $8$-dimensional.

The same situation occurs in the subalgebra $\mathcal{H}(8;\underline{1})$ of $H(8;\underline{1})^{(1)}$. There is a basis for $(\mathcal{H}(8;\underline{1}))_3$ spanned by the elements \begin{align*}v_1&:=D_H(x_1x_2x_3x_6x_7)+D_H(x_1x_2x_4x_6x_8)+D_H(x_1x_3x_4x_7x_8)\\ v_2&:=D_H(x_2x_1x_3x_5x_7)+D_H(x_2x_1x_4x_5x_8)+D_H(x_2x_3x_4x_7x_8)\\ v_3&:=D_H(x_3x_1x_4x_5x_8)+D_H(x_3x_2x_4x_6x_8)+D_H(x_3x_1x_2x_5x_6)\\
v_4&:=D_H(x_4x_1x_2x_5x_6)+D_H(x_4x_2x_3x_6x_7)+D_H(x_4x_1x_3x_5x_7)\\ v_5&:=D_H(x_5x_2x_3x_6x_7)+D_H(x_5x_2x_4x_6x_8)+D_H(x_5x_3x_2x_7x_6)\\ v_6&:=D_H(x_6x_1x_3x_5x_7)+D_H(x_6x_1x_4x_5x_8)+D_H(x_6x_3x_4x_7x_8)\\ v_7&:=D_H(x_7x_1x_4x_5x_8)+D_H(x_7x_2x_4x_6x_8)+D_H(x_7x_1x_2x_5x_6)\\ v_8&:=D_H(x_8x_1x_2x_5x_6)+D_H(x_8x_2x_3x_6x_7)+D_H(x_8x_1x_3x_5x_7)\end{align*} using \eqref{newgrading2}. We map $\partial_i$ to each element of $\g_e(\tau,-1)$, and then make appropriate choices to map $(\mathcal{H}(8;\underline{1}))_3$ to $L_3$. To make this appropriate choice we consider how the multiplication of $\partial_i$ works on $D_H(f)$ for $f \in (\mathcal{H}(8;\underline{1}))_3$ inside $\mathcal{H}(8;\underline{1})$. Then, we look for the same relations among elements of $L_3$ and $L_{-1}$. This gives the necessary isomorphism.\end{proof}

The Lie algebra $\g_e/\bbk e$ is $127$-dimensional, and has second derived subalgebra simple of dimension $118$. By \cite{LT11} we know that $\g_e/A$ contains $\mathfrak{sp}(8)$.

This is like the idea behind $H(8;\underline{1})$ having simple subalgebra $H(8;\underline{1})^{(1)}$ in characteristic two, since elements of the form $x_j\partial_{j'}$ are not contained in the simple Lie algebra by \rref{hamstructure}. However, they all have degree zero so it seems plausible there is a subalgebra in $H(8;\underline{1})$ such that the derived subalgebra is equal to $\mathcal{H}(8;\underline{1})$.

\section[A possible link to Lie superalgebras]{A possible link to Lie superalgebras}\label{sec:special}

We consider the tables of \cite{LT11, S16} and nilpotent orbits $\mathcal{O}(E_8(a_2))$ and $\mathcal{O}(E_8(a_4))$ with standard representatives \begin{align*}e_1&:=e_{\subalign{10&00000\\&0}}+e_{\subalign{00&00000\\&1}}+e_{\subalign{01&00000\\&0}}+e_{\subalign{00&10000\\&1}}+e_{\subalign{00&11000\\&0}}+e_{\subalign{00&01100\\&0}}+e_{\subalign{00&00110\\&0}}+e_{\subalign{00&00001\\&0}} \\e_2&:=e_{\subalign{11&00000\\&0}}+e_{\subalign{00&10000\\&1}}+e_{\subalign{01&10000\\&0}}+e_{\subalign{00&11000\\&0}}+e_{\subalign{00&01100\\&0}}+e_{\subalign{00&00110\\&0}}+e_{\subalign{00&00011\\&0}}+e_{\subalign{01&11000\\&0}}\end{align*} respectively. In both cases the dimension of the centraliser $\g_e$ increases from the centraliser for $e$ when considered in $\g$ over a field of good characteristic. These centralisers now contain the $p$-th powers of each $e_i$.

\begin{thm}\label{specialmax}Let $G$ be an algebraic group of type $E_8$ over an algebraically closed field of characteristic two with Lie algebra $\g=\Lie(G)$. For $f_1:=f_{\subalign{12&22110\\&1}}+f_{\subalign{11&22111\\&1}}+f_{\subalign{12&32110\\&1}}+f_{\subalign{12&22210\\&1}}+f_{\subalign{12&21111\\&1}}+f_{\subalign{12&32100\\&2}}+f_{\subalign{01&22221\\&1}}$
we have $L_1:= \langle e_1,f_1 \rangle$ is a maximal subalgebra of dimension $124$.
Similarly for $f_2:=f_{\subalign{12&32211\\&1}}+f_{\subalign{23&43221\\&2}}+f_{\subalign{13&43321\\&2}}+f_{\subalign{12&44321\\&2}}$ we have $L_2:=\langle e_2,f_2 \rangle$ is a maximal subalgebra of dimension $124$.\end{thm}

\begin{proof} We may consider $\g$ as an $L_i$-module, and show there are no submodules other than $L_i$ itself. We construct matrices in GAP to describe the action of $L_i$ on $\g$. For each basis element of $L_i$ we obtain a $\dim\,\g \times \dim\,\g$ matrix with each row a list of coefficients. Forming the module with these matrices in the MeatAxe of \cite{GAP4} we ask for all submodules of $\g$. For full details on how to do this we refer the reader to \aref{weirdp2algapp}. This shows $\g/L_i$ is an irreducible $L_i$-module providing maximality.\end{proof}

It is worth noting that the dimension of the simple Lie algebra $S(5;\underline{1})^{(1)}$ is also $124$ for $p=2$, and so it may well be the case that both $L_1$ and $L_2$ are isomorphic to a Lie algebra of special type. We obtain the following information using GAP.

\begin{enumerate}\item{$x_i:={e_i}^{[8]} \subseteq L_i$ has type ${A_1}^4$.}
\item{$\mathfrak{c}_{L_i}(x_i):=\mathfrak{c}_{\g}(x_i) \cap L_i$ is $64$-dimensional.}
\item{\label{key22s}$\w:={\mathfrak{c}_{L_i}(x_i)}^{(2)}$ is a $59$-dimensional simple restricted Lie algebra.}\end{enumerate}
We will use this to explain a possible link with Lie superalgebras. To begin we need the definition of a Lie superalgebra for characteristic two, using \cite{leico10} as our main reference.

\begin{defs}Consider the superspace $\g=\g_{\bar{0}}\oplus \g_{\bar{1}}$ where $\g_{\bar{0}}$ is a Lie algebra, $\g_{\bar{1}}$ is a $\g_{\bar{0}}$-module. We define on $\g_{\bar{1}}$ a map such that $x\mapsto x^2$ with $(ax)^2=a^2x^2$ for any $x \in \g_{\bar{1}}$ and $a \in \bbk$, and $(x+y)^2-x^2-y^2$ is a bilinear form on $\g_{\bar{1}}$ with values in $\g_{\bar{0}}$. \end{defs}

We define the bracket on `odd' elements as $[x,y]:=(x+y)^2-x^2-y^2$, and if $x,y \in \g_{\bar{0}}$ we use the usual Lie bracket. For $u \in \g_{\bar{1}}$, define $[x,u]=-[u,x]$ as the left and right action of $\g_{\bar{0}}$ on $\g_{\bar{1}}$ respectively. The Jacobi identity becomes $[x^2,y]=[x,[x,y]]$ for all cases of more than one odd element.

The definition of derived subalgebras must be modified slightly because of the odd elements. Set $\g^{(0)}:=\g$, and consider \[\g^{(1)}:=[\g,\g]+\spnd\{x^2:x \in \g_{\bar{1}}\}.\] We are taking the usual derived subalgebra and adding in elements $x^2$.

Take the analogous definitions for polynomial rings in the super case denoted as $\mathcal{O}(m;\underline{n}|m)$, and the Lie superalgebra version of Cartan type Lie algebras from \cite{leico10}. The idea is the same, but we have to consider what happens with the odd elements.

For any homogenous $f \in \mathcal{O}(m;\underline{n}|m)$, consider \[Le_{f}:= \sum_{i \le m}(\partial_i(f)\nu_i+(-1)^{p(f)}\nu_i(f)\partial_i),\] where $\partial_i$ are the usual partial differentiation, $\nu_i$ the odd analogous and $p(f)$ is the parity of $f$. Then, define a Lie superalgebra with basis \[\mathfrak{le}(m;\underline{n}|m)=\spnd\{Le_f:f \in \mathcal{O}(m;\underline{n}|m)\}.\]

By \cite[Theorem 5.1]{leico10} this superalgebra is not simple, but the first derived subalgebra is simple. Treating both $\partial_i$ and $\nu_i$ as the usual differentiation on polynomials in $\mathcal{O}(2m;\underline{n})$, we obtain a Lie algebra with basis given by elements of the form $\sum_{j=1}^{2m}(\partial_j(f)\partial_{j'})$ with $j'$ as in \eqref{hamstuff}. Therefore forgetting the superstructure of $\mathfrak{le}(m;\underline{n}|m)$ we obtain the Hamiltonian Lie algebra $H(2m;\underline{n})$.

There is a subalgebra with elements $Le_f$ such that $\sum_{i} \partial_i(f)\nu_i(f)=0$. If we forget the superstructure of this particular superalgebra we obtain a Lie algebra denoted by $\mathfrak{sh}(2m;\underline{n})$ in \cite[Table 1.18]{leico10}. Together with \cite[Lemmas 2.2.2 (4) and 2.2.3 (2)]{KoLei92} we have the next result.

\begin{thm}For $p=2$, there exists a class of simple Lie algebras denoted $\mathfrak{sh}(2m;\underline{n})$ of dimension $2^{2m-1}-2^{m-1}-2$. For $m=4$ there is a simple subalgebra of dimension $59$.\end{thm}

From Propositions \ref{newsubpossible6} and \ref{newsubpossible8} we have a simple subalgebra in the Hamiltonian Lie algebras over fields of characteristic two denoted by $\mathcal{H}(2n;\underline{1})$. The dimension is equal to $2^{2n-1}-2^{n-1}-2$ in both cases of $n=3$ and $n=4$.

Consider $H(2n;\underline{1})$ and its standard grading in the case of $p \ge 3$, then $L_1$ is an irreducible $L_0$-module coming from the third symmetric power of the standard representation of $\mathfrak{sp}(2n)$. If $p=2$ then $L_1$ is third exterior power of the standard representation of $\mathfrak{sp}(2n)$.

It is a result of \cite{PS83} that this contains a submodule $V(\omega_3)$ where $\omega_3$ is the third fundamental Weyl weight such that $\dim\, V(\omega_3)=\binom{2n}{3}-\binom{2n}{1}$. We also have a condition on $n$ for this module being irreducible --- if $n$ is even we have $V(\omega_3)$ is irreducible, otherwise there is a submodule.

Where $n$ is even occurs in \pref{newsubpossible8}, and we can calculate that $\dim\,L_1=48=\binom{8}{3}-\binom{8}{1}$. For \pref{newsubpossible6} we have $n$ is odd, and the degree $1$ component has dimension $6\ne \binom{6}{3}-\binom{6}{1}=\dim\,V(\omega_3)$.

There is a simple subalgebra of dimension $59$ in $\mathcal{H}_8$, and so we should not be surprised if $\mathcal{H}_8\cong \mathfrak{sh}(8;\underline{1})$. Since $L_i \cap \mathcal{H}_8$ is a 59-dimensional simple Lie algebra we may find a link between these cases. We end with a conjecture summarising the key ideas of this section.

\begin{conjecture}Let $L=H(2n;\underline{1})^{(1)}$ be a Lie algebra of Hamiltonian type with $n \ge 3$. If $p=2$, then there exists a simple subalgebra denoted by $\mathcal{H}(2n,\underline{1})$. The simple Lie algebra $\mathcal{H}(2n;\underline{1})$ is isomorphic to $\mathfrak{sh}(2n;\underline{1})$.
\end{conjecture}
\begin{comment}
\begin{proof}[Some ideas]The standard grading on $L$ is given by $L_{-1}\oplus L_0 \oplus L_1 \oplus \ldots L_{r}$ where $L_0$ has type $C_n$. We may then obtain a description of our submodule $V(\omega_3)$ in terms of basis elements for $H(2n;\underline{1})$. We would then hope to show that $[L_{-1},V(\omega_3)]=L_0^{(1)}$, which would give any ideal of the subalgebra $\langle L_{-1},V(\omega_3) \rangle$ contains $L_{-1}+L_0^{(1)}+S$. It should then become clear that we have simplicity at least in cases where we have irreducible $V_(\omega_3)$. When it is not irreducible we should be able to repeat the same process taking its irreducible submodule as in the case of $n=3$.\end{proof}
\end{comment}
%\end{document}

\appendix
\chapter[Using GAP]{Using GAP}\label{GAP}
Many calculations in this thesis were done using \cite{GAP4}, and although the majority could be done by hand it becomes likely we make mistakes given the tedious nature of some calculations. In this Appendix we give full details on all the commands used, and then in \aref{GAP2} we provide the reader with a detailed account of our use of GAP in many of the results.

\section[Root systems of the exceptional Lie algebras in GAP]{Root systems of the exceptional Lie algebras in GAP}\label{appendix:gaproot}
All our work involves the exceptional Lie algebras denoted by $\g$, and in \seref{sec:notation} we remarked about the ordering of our roots in the same sense as Bourbaki. In this section we outline the definitions for the root systems of the exceptional Lie algebras using \cite{Bour02}.

We start by defining the root systems of the exceptional Lie algebras $E_6$, $E_7$ and $E_8$. These are slightly easier than $G_2$ and $F_4$ as there are no short roots to consider. For the exceptional Lie algebras $E_n$ for $n \in \{6,7,8\}$, we only need to consider $E_8$. This follows since the root systems of both $E_6$ and $E_7$ are subsystems of the $E_8$ root system. Hence, we obtain all three at the same time.

Throughout this section, let $\{e_i\}$ be the standard orthonormal basis for $\mathbb{R}^n$. That is, for each $i$, we let $e_i$ be the vector with $1$ in the $i$-th position and zeros everywhere else. Consider the standard scalar product on $\mathbb{R}^n$ denoted by $(x|y)$ for $x,y \in \mathbb{R}^n$. Define the $\bbz$-span of this to be $I$.

For $E_8$, consider the subgroup of $I'=I+\bbz((e_1+\ldots+e_8)/2)$ in $\mathbb{R}^8$ denoted as $I''$ consisting of elements $\sum_i c_ie_i+c/2(e_1+\ldots+e_8)$ such that $\sum_i c_i$ is an even integer. The root system $\Phi$ is then defined as $\Phi=\{\alpha \in I'':(\alpha|\alpha)=2\}$. This gives the roots $\pm (e_i\pm e_j)$ for $i \ne j$ and \[1/2\sum_{i=1}^8(-1)^{k(i)}e_i,\] where the $k(i)\in\{0,1\}$ add up to an even integer. Then, define the simple roots as $\{1/2(e_1+e_8-(e_2+\ldots+e_7),e_1+e_2,e_2-e_1,e_3-e_2,e_4-e_3,e_5-e_4,e_6-e_5,e_7-e_6)\}$, which we now label in the same order as $\{\al_1,\ldots,\al_8\}$.

For $E_7$, we delete $e_7-e_6$ and change $\al_1$ to $\frac{1}{2}(e_1+e_7-(e_2+\ldots+e_6))$. A similar idea produces the simple roots in $E_6$. To obtain these exceptional Lie algebras in GAP by means of the Chevalley basis, for $n \in \{6,7,8\}$ we use:

\begin{lstlisting}[basicstyle=\small]
gap> g:=SimpleLieAlgebra(``E'',n,finite field);;
gap> b:=Basis(g);;
\end{lstlisting}

For $F_4$, we consider $\mathbb{R}^4$ and the subgroup $I'=I+\bbz((e_1+\ldots+e_4)/2$. Define $\Phi$ as the set $\{\alpha \in I':(\alpha|\alpha)=1 \,\,\text{or}\,\, 2\}$. This gives the roots $\pm e_i$, $\pm (e_i\pm e_j)$ for $i \ne j$, and $\pm \frac{1}{2}(e_1\pm e_2\pm e_3\pm e_4)$ where we can choose the signs independently. Finally, we define the simple roots as $\{e_2-e_3,e_3-e_4,e_4,\frac{1}{2}(e_1-e_2-e_3-e_4)\}$, which we now label in the same order as $\{\al_1,\ldots,\al_4\}$.

For $F_4$ we obtain a slightly different ordering in GAP --- differing by a permutation on the ordering in \cite{Bour02} using the commands:
\begin{lstlisting}[basicstyle=\small]
gap> g:=SimpleLieAlgebra(``F'',4,field);;
gap> b:=Basis(g);;
\end{lstlisting}

This will give the basis with the order given in \autorefs{posroots} where the table is to be read from left to right, and top to bottom precisely the same as in \cite[Table 8]{de08}.

\begin{table}[H]\caption{List of positive roots for $F_4$ in GAP}\phantomsection\label{posroots}\centering
\begin{tabular}{|c|c|c|c| }\hline
$0001$&$1000$&$0010$&$0100$\\ \hline
$0011$&$1100$&$0110$&$0111$\\ \hline
$1110$&$0120$&$1111$&$0121$\\ \hline
$1120$&$1121$&$0122$&$1220$\\ \hline
$1221$&$1122$&$1231$&$1222$\\ \hline
$1232$&$1242$&$1342$&$2342$\\ \hline
\end{tabular}\end{table}

Finally, for $G_2$ we consider simple roots in $\mathbb{Z}^3$ as $e_1-e_2$ and $-2e_1+e_2+e_3$. It is straightforward to calculate the Dynkin diagrams given in \eqref{dynkindiagrams}-\eqref{dynkindiagrams5}. Each simple root $\al_i$ corresponds to a vertex, and the number of edges between vertices is given by \begin{equation*}\frac{2(\al_i|\al_j)}{(\al_i|\al_i)}\frac{2(\al_j|\al_i)}{(\al_j|\al_j)}\end{equation*} for $i \ne j$. Checking that the ordering in this section matches the Dynkin diagrams from \autoreft{sec:notation} is left as an exercise to the reader.

For further details on the root systems we refer the reader to \cite{Bour02}, and the full lists of the corresponding order of the basis in GAP are found in the tables of \cite{de08}. It should be noted that in all cases except for $F_4$ they are ordered in the same way, with the GAP ordering for $F_4$ described in \autorefs{posroots} above.

\section[Some Basics]{Some Basics}

We calculate many properties about certain subalgebras of our Lie algebra. In this section we give a brief outline of all the basic commands used to do this. To define a simple Lie algebra, and consider the derived series we do the following:

\begin{lstlisting}[basicstyle=\small, caption={Obtaining the exceptional Lie algebras in GAP}]
gap> g:=SimpleLieAlgebra(``F'',4,GF(5));;
gap> b:=Basis(g);;
gap> LieDerivedSeries(g);
gap> LieDerivedSubalgebra(g);
gap> LieCentre(g);
\end{lstlisting}

It must be stressed that all our GAP calculations are done over the finite field $GF(p)$ for $p$ a bad prime of $\g$. However, our results are true over algebraically closed fields of the same characteristic. This is mainly due to the fact that the exceptional Lie algebras are defined using a Chevalley basis from \tref{ChevBas}. This allows us to obtain the ``same'' basis over any field just with the structure constants reduced modulo $p$. Hence, using a finite field in GAP is good enough to illustrate the case in algebraically closed fields.

We may compute the adjoint module of a Lie algebra $\g$. As an example, we will consider $\g$ of type $F_4$ to illustrate this.
\begin{lstlisting}[basicstyle=\small,caption={The adjoint module in GAP}]
gap> Mats:=List(b,x->AdjointMatrix(b,x))
gap> gm:=GModuleByMats(Mats,GF(5));;
\end{lstlisting}

This produces a list of matrices that represent the action of $\g$ on a basis for $\g$ where the `AdjointMatrix' command does this automatically for each basis element of $\g$. Hence, the collection of matrices actually represents the adjoint action of $\g$. Finally, we ask GAP to consider this as a module with `GModuleByMats'.

To take a subalgebra of $\g$, and to find centralisers and normalisers of elements in GAP we use the following:

\begin{lstlisting}[basicstyle=\small,caption={Centralisers and normalisers in GAP}]
gap> g:=SimpleLieAlgebra(``F'',4,GF(5));;
gap> b:=Basis(g);;
gap> e:=b[1]+b[2]+b[3]+b[4];
gap> Subalgebra(g,[b[3],b[7]]);
gap> C:=LieCentraliser(g,Subalgebra(g,[e]));;
gap> N:=LieNormaliser(g,Subalgebra(g,[e]));;
\end{lstlisting}

Note that, $e$ is the nilpotent element $\rt{1}+\rt{2}+\rt{3}+\rt{4}$ in a Lie algebra of type $F_4$ over a finite field of characteristic five. To take the Lie solvable radical or the nilpotent radical we input

\begin{lstlisting}[basicstyle=\small,caption={Computing the radical of a Lie algebra in GAP}]
gap> rad:=LieSolvable(Nil)Radical(N);
\end{lstlisting}

We produce Lie algebras of Cartan type using the same `SimpleLieAlgebra' command, where we specify our vector $n \in \mathbb{N}^m$ instead of the rank to consider $X(m;\underline{n})$. For example, to consider the non-restricted Witt algebra $W(1;\underline{2})$ we use

\begin{lstlisting}[basicstyle=\small,caption={Cartan type Lie algebras in GAP}]
gap> SimpleLieAlgebra(``W'',[2],GF(5))
\end{lstlisting}

\section[Dealing with modules]{Dealing with modules}\label{appendix:gapmod}

In this section we give the basic outline of the procedures we use for determining simplicity of a Lie algebra or indecomposability of a module. We select certain nilpotent orbits to study, using \cite{de08} for the complete list of all basis elements in the exceptional Lie algebras. For example, to consider the nilpotent orbit $A_4+A_3$ in the exceptional Lie algebra of type $E_8$ over algebraically closed fields of characteristic five we consider

\begin{lstlisting}[basicstyle=\small,caption={The orbit $\mathcal{O}(A_4+A_3)$ in $E_8$}]
gap> g:=SimpleLieAlgebra(``E'',8,GF(5));;
gap> b:=Basis(g);;
gap> e:=b[1]+b[2]+b[3]+b[4]+b[6]+b[7]+b[8];
v_1+v_2+v_3+v_4+v_6+v_7+v_8
\end{lstlisting}

This gives the exceptional Lie algebra of type $E_8$, with orbit $\mathcal{O}(A_4+A_3)$ using the standard representative $e=\sum_{\al \in \Pi\setminus \{\al_5\}}$ from \cite[pg. 148]{LT11}. We may consider the centraliser, normaliser and subalgebras generated by elements of $\g$. This allows us to obtain the radical of certain subalgebras --- to obtain the information in \cite[\S 4.3]{P15} use

\begin{lstlisting}[basicstyle=\small, caption={Centraliser and normaliser of nilpotent element of type $\mathcal{O}(A_4+A_3)$}]
gap> C:=LieCentralizer(g,Subalgebra(g,[e]));;
gap> N:=LieNormalizer(g,Subalgebra(g,[e]));
Lie algebra of dimension 51 over GF(5)
gap> rad:=LieSolvableRadical(N);
Lie algebra of dimension 24 over GF(5)
gap> IsLieAbelian(rad);
true
gap> w:=LieNormalizer(g,rad);
Lie algebra of dimension 74 over GF(5)
gap> w/rad;
Lie algebra of dimension 50 over GF(5)
\end{lstlisting}

We use the well-known MeatAxe package from \cite{hr94} to verify the simplicity of our subalgebras amongst other things. For this, we input the adjoint module as a set of matrices $\{e_i\}$ using the ``AdjointMatrix'' and ``Mats'' commands to obtain this in GAP where the collection of all $\{e_i\}$ represents the adjoint action. We can then ask GAP to check the irreducibility of a module or whether the module is indecomposable.

As a word of warning, the original intention was for studying groups. We are allowed to use this to check whether the adjoint module is irreducible, because for $\{e_i\}$ there is $\la_i \in \mathbbm{k}$ such that $\la_i \mathrm{I}+e_i$ is an invertible matrix. Hence, we are generating subgroups of $GL(V)$ where the original module is irreducible under $\la_i \mathrm{I} +e_i$ if and only it is under $e_i$.

This allows us to use these commands in GAP for Lie algebras, so when GAP checks the irreducibility of such a module it is actually checking it is irreducible under $\la_i \mathrm{I} +e_i$. For further details we refer the reader to \cite{hr94}.

This confirms simplicity since any ideal of our Lie algebra $\g$ is a submodule in the adjoint module. We must be slightly cautious since GAP by default multiplies matrices on the right, rather than our usual left. The best way to verify what we obtain is the correct way around is take a known ideal, normally the radical and check there is a submodule of the same dimension. Using the default we may see the radical as a composition factor.

The following set of commands allows us to check whether our module is irreducible, indecomposable and what the submodules are. This will be used throughout the Appendix, where we calculate this kind of information for many different subalgebras.

\begin{lstlisting}[basicstyle=\small,caption={Checking the irreducibility of a module},label={lst:irreducible}]
gap> b:=Basis(g);;
gap> Mats:=List(b,x->AdjointMatrix(b,x));;
gap> gm:=GModuleByMats(Mats,GF(5));;
gap> MTX.IsAbsolutelyIrreducible(gm);
true
\end{lstlisting}

To check the indecomposability of modules, a list of all submodules, or obtain a proper submodule we use

\begin{lstlisting}[basicstyle=\small,caption={Obtaining submodules in GAP}]
gap> MTX.IsIndecomposable(gm);
true
gap> MTX.BasesSubmodules(gm);
gap> t:=MTX.ProperSubmoduleBasis(gm);;
gap> tm:=MTX.InducedActionSubmodule(FactorModule)(gm,t);;
\end{lstlisting}

\begin{rem}\label{indecoofmodule}
For us, we check whether certain factor modules are indecomposable to help prove maximality. For this, we use ``t'' and ``tm'' as above to obtain the correct factor module. Then we ask GAP if this is indecomposable. \end{rem}

We are not able to take a tensor product of two modules --- the action used is incompatible for Lie algebras. The check to see our algebra is absolutely simple is used because simple over a finite field does not always imply simple over algebraically closed fields.

The algorithm is given in detail in \cite{hr94}, where the MeatAxe initially produces a submodule of our module $M$ generated from a random element $v$ of $M$. Then the process runs through several procedures, to calculate characteristic polynomials. For all the factors of the characteristic polynomial we evaluate at $v$, and consider the null space of this. Consider $u$ in the null space and the submodule of $M$ generated by $u$.

If this gives a proper submodule then GAP will return false. Otherwise the algorithm continues, and then considers the transpose of the matrices used to produce $M$ in the MeatAxe. It then repeats the process to verify this does not produce a proper submodule.

To see this is enough to prove irreducibility we refer the reader to \cite[\S2.2]{hr94}. The idea is that for any reducible module the algorithm must produce a submodule, and if not we may conclude that $M$ must be irreducible. In GAP our modules are taken as $\bbk G$-modules where $\bbk=GF(q)$ is a field and $G$ a finite group. Consider the $\bbk$-algebra $End_{\bbk G}(M)$ of endomorphisms of $M$. This is a field extension of $\bbk$ called the \emph{centralising field}.

To consider the absolute irreducibility of $M$, GAP calculates the centralising field $GF(p^e)$ of $M$. The algorithm then checks to see that $e=1$, and returns that our module is absolutely irreducible in this case. For full details, and argument as to why the algorithm is sufficient we refer the reader to \cite[pg. 50]{hro} and \cite[\S3]{hr94}.

Taking the adjoint module of the exceptional Lie algebra of type $E_7$, and asking whether the module is irreducible gives false when $p=2$. Doing the exact same thing for $\g/\mathfrak{z}(\g)$ returns true --- exactly what we would expect to happen.

For \tref{sec:e8p5non} we need to show that $\w/\rad(\w)$ is a simple Lie algebra of dimension $50$. This is done using a straightforward application of the above.

\begin{lstlisting}[basicstyle=\small, caption={Confirming information from \tref{sec:e8p5non}}]
gap> n:=Basis(w/rad);;
gap> Mats:=List(n,x->AdjointMatrix(n,x));;
gap> gm:=GModuleByMats(Mats,GF(5));;
gap> MTX.IsAbsolutelyIrreducible(gm);
true
\end{lstlisting}

The last thing we consider in \tref{sec:e8p5non} is the subalgebra generated by elements of $\g_e(\tau,\pm 1)$, and show this is a $47$-dimensional Lie subalgebra. We list all elements of each space in GAP using \cite{LT11, de08} and generate a subalgebra by them. We then compute the radical, and verify that the subalgebra factored out by its radical is a $24$-dimensional simple Lie algebra.

\section[Non-semisimple subalgebras in GAP]{Non-semisimple subalgebras in GAP}\label{nonsemigap}

The crucial use of GAP for us concerns the study of non-semisimple maximal subalgebras. We consider a subalgebra $L_0:=N_{\g}(A)$, where $A$ is the radical of a normaliser for some nilpotent element $e \in \g$. We write some procedures in GAP to obtain the necessary information for maximality. To this end, we find an $L_0$-module $L_{-1}$ such that $L_{-1}/L_0$ is irreducible. Then, show $L_{-1}$ generates $\g$ as a Lie algebra.

For this section we give the majority of procedures we use in GAP to obtain the necessary details in any of our results. In \aref{GAP2} we use all of these, and give examples of them in action for some of our results on non-semisimple maximal subalgebras.

Consider $\g$ an exceptional Lie algebra with nilpotent element $e \in \g$, and calculate the normaliser $\mathfrak{n}_e$ along with its radical.

\begin{lstlisting}[basicstyle=\small,caption={Studying a nilpotent orbit in GAP}]
gap> g:=SimpleLieAlgebra("exceptionaltype",rank,field);;
gap> b:=Basis(g);;
gap> e is a nilpotent element for g
gap> N:=LieNormaliser(g,Subalgebra(g,[e]));;
gap> rad:=LieSolvableRadical(N);;
gap> t:=Basis(rad);;
gap> d:=Dimension(rad);;
\end{lstlisting}

Then, we look at the normaliser of the radical and label the basis elements as follows
\begin{lstlisting}[basicstyle=\small,caption={Finding maximal non-semisimple subalgebras}]
gap> Nrad:=LieNormalizer(g,rad);
gap> n:=Basis(Nrad);;
gap> x(1):=t[1]...x(d):=t[d];;
\end{lstlisting}

Denote the radical by $A$, we need to compute $L_{-1}:=\{x \in \g: [x,A] \subseteq N_{\g}(A)\}$, and show $L_{-1}/L_0$ is irreducible. Let $L_0:=N_{\g}(A$, we start by making a list of the elements of $L_0$.

\begin{lstlisting}[basicstyle=\small,caption={Obtaining a list of basis elements for $L_0$},label={lst:vectorspace}]
gap> W:=List([]);;
gap> for j in [1..Dimension(Nrad)] do
if \in(n[j],W)=false then
R:=List([n[j]]);
Append(W,R);
continue;
fi;od;
\end{lstlisting}

We then write a small procedure to find basis elements of $\g$ that are contained in $L_{-1}$ but not in $L_0$, and produce a vector space generated by these elements along with our elements from $L_0$.

\begin{lstlisting}[basicstyle=\small,caption={Obtaining some elements of $L_{-1}$},label={lst:lminus}]
gap> for i in [1..Dimension(g)] do
 y:=function(i)
 a(1):=\in(b[i]*x1,Nrad);...;a(s):=\in(b[s]*xs,Nrad);
a(s+1):=\in(b[i],Nrad);
if a(s+1)=true or a1=false or as=false then
return 0;else return i;
fi;end;
if y(i)=i then Append(W,[b[i]]);fi;od;od;
gap> V:=VectorSpace(GF(p),W);;
gap> v:=Basis(V);
\end{lstlisting}

We may not have completely obtained everything in $L_{-1}$, as this only considers basis elements of $\g$. Since $L_{-1}$ is $L_0$-invariant we apply $\ad\,L_0$ to our space repeatedly until we obtain no new elements.

\begin{lstlisting}[basicstyle=\small,caption={Obtaining the remaining basis elements of $L_{-1}$},label={lst:fullminus}]
gap> for j in [1..Dimension(Nrad)] do
for s in [1..Dimension(V)] do sj:=function(j)
a1:=\in(n[j]*v[s],V);
a2:=n[j]*v[s];if a1=false then
return a2;
else
return 1;
fi;
end;
if sj(j)=a2 then Append(W,[sj(j)]);
fi;od;od;
gap> V:=VectorSpace(GF(p),W);;
gap> Dimension(V);
\end{lstlisting}

We check that the subalgebra generated by $L_{-1}$ is isomorphic to $\g$. All that remains is to confirm that $L_{-1}/L_0$ is irreducible --- we achieve this by computing $L_{-1}$ as an $L_0$-module and checking it has no proper submodules strictly containing $L_0$ other than $L_0$ itself. To confirm $\langle L_{-1}\rangle = \g$ and obtain the module in GAP we use \autoref{lst:modulelminus}.

\begin{lstlisting}[basicstyle=\footnotesize, caption={Writing $L_{-1}$ as an $L_0$-module},label={lst:modulelminus}]
gap> h:=Subalgebra(g,W);;h=g;
true
gap> for l in [1..Dimension(Nrad)] do cof:=function(l);
t1:=[Coefficients(v,n[l]*v[1]),...,
Coefficients(v,n[l]*v[Dimension(V)])];;
return t1;end;od;
gap> Mats:=List([1..Dimension(Nrad)],cof);;
gap> gm:=GModuleByMats(Mats,GF(p));;
gap> MTX.BasesSubmodules(gm);
\end{lstlisting}

There is a nice use of GAP to build the Weisfeiler filtration from our irreducible module $L_{-1}$ which should always be assumed to be denoted as $V$ in our GAP calculations, unless otherwise stated.

\begin{lstlisting}[basicstyle=\small,caption={Obtaining $L_{-2}$ for a Weisfeiler filtration},label={lst:lminus2}]
gap> v:=Basis(V);;
gap> W:=List([]);;
gap> for j in [1..Dimension(V)] do
for k in [1..Dimension(V)] do
if \in(v[j]*v[k],W)=false then R:=List([v[j]*v[k]]);
Append(W,R);
continue;fi;od;od;
gap> VV:=VectorSpace(GF(p),W);;Dimension(VV);vv:=Basis(VV);
\end{lstlisting}

Using \autoref{lst:lminus2} we obtain $L_{-2}$, and note in some cases we may need to ``add'' $L_{-1}$ as in the definition of the Weisfeiler filtration from \autoreft{sec:Wei}. This can be adapted to obtain $L_{-3}:=[L_{-1},L_{-2}]$, and so on to produce a Weisfeiler filtration of $\g$.

\subsection*{Minimality of ideals}
In some cases we may need to find a minimal ideal $I$ to find the necessary $L_0$-invariant subspace $L_{-1}$ such that $L_{-1}/L_0$ is irreducible as an $L_0/A$-module.

To obtain the minimality of $I$ we consider them as $L_0$-modules, and show the only proper submodule is the radical $A$ of $L_0$. This implies that $I/A$ is irreducible as an $L_0/A$-module. This gives that $I/A$ is a minimal ideal of the Lie algebra $L_0/A$, and hence we may use \tref{blocktheorem2} to help identify the isomorphism class of $I/A$.

In some cases $I$ will also be maximal, but there will be cases where this is not true. For these situations, $I$ will always be the minimal ideal to be used with \tref{blocktheorem2}, and $J$ will be used to denote a maximal ideal.

\chapter[GAP calculations]{GAP calculations}\label{GAP2}

We now provide many of the GAP calculations in detail that were performed to prove many of the results in this thesis. This part of the Appendix should allow the reader to completely reproduce and check the results for themselves.

\section[The Ermolaev algebra in $F_4$ using GAP]{The Ermolaev algebra in $F_4$ using GAP}

We begin with the information in \autoreft{sec:ermoax}, and study the exceptional Lie algebra of type $F_4$ over fields of characteristic three. In particular, the nilpotent orbit $\mathcal{O}(F_4(a_1))$ with standard representative $e:=e_{0100}+e_{1000}+e_{0120}+e_{0001} \in F_4$. Recall from \tref{ermax} that $f:=f_{1242}$, and so to check $L:=\langle e,f \rangle$ we calculate
\begin{lstlisting}[basicstyle=\small,caption={The Ermolaev algebra in $F_4$}]
gap> g:=SimpleLieAlgebra(``F'',4,GF(3));;
gap> b:=Basis(g);;
gap> e:=b[1]+b[2]+b[4]+b[10];;
gap> f:=b[45];;
gap> L:=Subalgebra(g,[e,f]);;
gap> Dimension(L);
26
\end{lstlisting}
\subsection{Proposition \ref{ermoisom}}\label{appprop}

We show how we obtain the module $V$ used in \pref{ermoisom} to show $L$ is isomorphic to the Ermolaev algebra. To start, we take a basis of $L$ in GAP and consider the element $v:=e_{0111}-e_{1110}$. We can either take the basis element of $L$ equal to $v$ or just write $v$ in terms of our usual basis for $F_4$ as follows
\begin{lstlisting}[basicstyle=\small,caption={Obtaining the module $V$ in \pref{ermoisom}}]
gap> l:=Basis(L);;
gap> v:=l[15];;
gap> vv:=b[8]-b[9];;
gap> v=vv;
true
\end{lstlisting}

Next, we obtain the subalgebra $W$ so that we can consider the action of $W$ on $v$. Take $ff:=f_{1222}-f_{1242}$, and calculate the subalgebra generated by $e$ and $ff$.
\begin{lstlisting}[basicstyle=\small]
gap> ff:=b[44]-b[46];;
gap> W:=Subalgebra(g,[e,ff]);;
gap> Dimension(W);
18
gap> w:=Basis(W);
\end{lstlisting}

Now we are able to calculate the module $V$ with the following in GAP:
\begin{lstlisting}[basicstyle=\small]
V:=VectorSpace(GF(3),[w[2]*v,w[3]*v,w[4]*v,
w[5]*v,w[6]*v,w[7]*v,w[8]*v,w[9]*v]);;
Dimension(V);
8
\end{lstlisting}

Note that we should take $w[i]\cdot v$ for all $i=1,\ldots,18$, but a quick calculation in GAP shows the above is sufficient to obtain all of $V$. It is a simple exercise to show that $V \cap W =0$, and that $V+W=L$ as required in \pref{ermoisom}. Finally, we compute the Lie algebra generated by $[V,V]$ using the following:

\begin{lstlisting}[basicstyle=\small,caption={Generating $W$ in \pref{ermoisom}}]
gap> vv:=Basis(V);;
gap> R:=List([]);;
gap> for j in [1..Dimension(V)] do
for k in [1..Dimension(V)] do
if \in(vv[j]*vv[k],R)=false then S:=List([vv[j]*vv[k]]);
Append(R,S);continue;fi;od;od;
gap> WW:=Subalgebra(g,R);;
gap> Dimension(WW);
18
gap> WW=W;
true
\end{lstlisting}

To follow our labelling above, note that $WW$ is the Lie algebra generated by all elements of $R$ where $R$ is used to denote the collection of all elements $[v_i,v_j]$ for $v_i, v_j \in V$. By checking that $WW$ and $W$ are the same we confirm that $\langle [V,V]\rangle \cong W$.

This also provides us with enough information to check that the grading defined in the proof of \pref{ermoisom} is well-defined. Since every element $x$ of $W$ is a Lie bracket of element of $V$, each homogeneous element of $L$ has a degree. From here, we just have to check that no element has been given more than one degree.

\subsection{Maximality of the Ermolaev algebra in GAP}

Using \autoref{lst:irreducible} we may confirm that $L$ is a restricted simple Lie algebra of dimension $26$. In \tref{ermax} we gave a non-GAP proof of maximality, but here we may achieve maximality completely in GAP. Consider $\g$ as a $L$-module using an easy algorithm given as follows:

\begin{lstlisting}[basicstyle=\small,caption={Obtaining a subalgebra $L$ as a $\g$-module},label={lst:checkingmaximalsubalgebra}]
gap> bh:=Basis(L);
gap> for i in [1..Dimension(g)] do
s:=function(i);
m1:=[Coefficients(b,bh[1]*b[i]),...,
Coefficients(b,bh[Dimension(L)]*b[i])];
return m1;end;od;
gap> Mats:=List([1..Dimension(g)],s);;
\end{lstlisting}

Using \autoref{lst:irreducible} we can ask for all the submodules of $\g$ as an $L$-module. This returns just two non-zero submodules, namely $L$ and $\g$ themselves. From this we conclude $\g/L$ is irreducible. Hence, any subalgebra strictly containing $L$ must contain $\g$ proving maximality. This is the exact idea is used in \tref{specialmax} to prove the existence of two maximal subalgebras in exceptional Lie algebras of type $E_8$.

\section[Examples of maximal non-semisimple subalgebras]{Examples of maximal non-semisimple subalgebras}

Recall \aref{nonsemigap}, where we gave many procedures that were used in calculating many details in our examples of maximal non-semisimple subalgebras. We start by fixing our notation in the same way as \aref{nonsemigap}. Let $\g$ be a simple Lie algebra of exceptional type and $e$ a nilpotent element of $\g$.

For $\mathfrak{n}_e$, let $A$ be the radical of this normaliser and $L_0=N_{\g}(A)$. For some of our examples we need to calculate a minimal ideal $I$ in $L_0$. The entirety of this section is devoted to computing $L_{-1}$ in each case, along with such minimal ideals $I$ as the become necessary.

\subsection[Theorem \ref{nonf4}]{Theorem \ref{nonf4}}\label{AppF42}

We begin with \tref{nonf4} regarding the nilpotent orbit $\mathcal{O}(\widetilde{A_2}+A_1)$ in $\g$ of type $F_4$, with representative $e:=e_{1000}+e_{0010}+e_{0001}$. We use \cite[Table 8]{de08} to look up the basis elements in GAP to obtain the normaliser of such a nilpotent element and calculate the radical.
\begin{lstlisting}[basicstyle=\small,caption={The nilpotent orbit $\mathcal{O}(\widetilde{A_2}+A_1)$ in $F_4$}]
gap> g:=SimpleLieAlgebra("F",4,GF(3));;
gap> b:=Basis(g);;
gap> e:=b[1]+b[2]+b[3];;
gap> N:=LieNormaliser(g,Subalgebra(g,[e]));;
gap> rad:=LieSolvableRadical(N);;
gap> IsLieAbelian(rad);
true
gap> t:=Basis(rad);;
gap> Nrad:=LieNormalizer(g,rad); n:=Basis(Nrad);;
<Lie algebra of dimension 26 over GF(3)>
gap> x1:=t[1];;x2:=t[2];;x3:=t[3];;x4:=t[4];;
gap> x5:=t[5];;x6:=t[6];;x7:=t[7];;x8:=t[8];;
\end{lstlisting}

This gives $19$-dimensional $\mathfrak{n}_e$ for the nilpotent element $e$, with an $8$-dimensional solvable radical $A$ with basis
\[A:=\langle e_{1000}+e_{0010}+e_{0001}, e_{0011}, e_{1121}+e_{0122}, e_{1122}, e_{1242}, f_{0100}, f_{1100}+f_{0110}, f_{1220}\rangle.\]
It is easy to check this is abelian, and that $N_{\g}(A)$ is $26$-dimensional. We consider $N_{\g}(A)/A$, and see this is an $18$-dimensional simple restricted Lie algebra.
\begin{lstlisting}[basicstyle=\small]
gap> bh:=Basis(Nrad/rad);;
gap> Mats:=List(bh,x->AdjointMatrix(bh,x));;
gap> gm:=GModuleByMats(Mats,GF(3));;
gap> MTX.IsAbsolutelyIrreducible(gm);
true
\end{lstlisting}

Next, we compute $L_{-1}:=\{x \in \g: [x,A] \subseteq N_{\g}(A)\}$, and show $L_{-1}/L_0$ is irreducible. Let $L_0:=N_{\g}(A)$, we make our list of elements in $L_0$ using \autoref{lst:vectorspace}.
\begin{lstlisting}[basicstyle=\small]
gap> W:=List([]);;
gap> for j in [1..26] do
if \in(n[j],W)=false then
R:=List([n[j]]);
Append(W,R);
continue;
fi;od;
\end{lstlisting}

Using \autoref{lst:lminus} we produce a vector space generated by the elements along with our elements from $L_0$.
\begin{lstlisting}[basicstyle=\small]
gap> for i in [1..52] do
 s:=function(i)
 a1:=\in(b[i]*x1,Nrad);
 a2:=\in(b[i]*x2,Nrad);
 a3:=\in(b[i]*x3,Nrad);
 a4:=\in(b[i]*x4,Nrad);
 a5:=\in(b[i]*x5,Nrad);
 a6:=\in(b[i]*x6,Nrad);
 a7:=\in(b[i]*x7,Nrad);
 a8:=\in(b[i]*x8,Nrad);
a9:=\in(b[i],Nrad);
if a9=true or a1=false or a2=false or a3=false or a4=false
or a5=false or a6=false or a7=false or a8=false then
return 0;else return i;
fi;end;
if s(i)=i then Append(W,[b[i]]);fi;od;od;
gap> V:=VectorSpace(GF(3),W);;
gap> v:=Basis(V);
\end{lstlisting}

We have not obtained everything in $L_{-1}$, and so we apply $\ad\,L_0$ to our space repeatedly using \autoref{lst:fullminus} until we obtain no new elements.
\begin{lstlisting}[basicstyle=\small]
gap> for j in [1..26] do
for s in [1..Dimension(V)] do sj:=function(j)
a1:=\in(n[j]*v[s],V);
a2:=n[j]*v[s];if a1=false then
return a2;
else
return 1;
fi;
end;
if sj(j)=a2 then Append(W,[sj(j)]);
fi;od;od;
gap> V:=VectorSpace(GF(3),W);;
gap> Dimension(V);
44
\end{lstlisting}
Fortunately, for $\g$ of type $F_4$ we only need to do this once. To see this we could try to repeat the process, but would only obtain the exact same $V$. The basis is given in \autorefs{tabf4}, it is left to the reader to look up the corresponding elements of $F_4$ using \cite[Table 8]{de08}.

We check that the subalgebra generated by $L_{-1}$ is isomorphic to $\g$. To confirm that $L_{-1}/L_0$ is irreducible, we compute $L_{-1}$ as an $L_0$-module and check it has no proper submodules strictly containing $L_0$ other than $L_0$ itself using \autoref{lst:modulelminus}. In this case we look at \autoref{lst:f4example} which does this for this example.

\begin{table}[H]\caption{$L_{-1}$ in $F_4$ for use in \tref{nonf4}}\phantomsection\label{tabf4}\footnotesize\centering
\begin{align*}L_{-1}:=\spnd\{&v.1+v.2+v.3, v.5, v.14+v.15, v.18, v.22, v.28, v.30+v.31, \\&v.40,
  v.15, v.31, v.2+v.3, v.41, \\&v.43-v.44, v.33+v.34, v.47, v.35-v.36-v.37, v.11+v.12, v.21,\\&v.3,
  v.50+v.51, v.32-v.34, v.19-v.20, v.24, \\&v.12+v.13, v.8-v.9-v.10, v.49-v.51,
  v.9+v.10, v.10, \\&v.13, v.17, v.20, v.23, v.34, v.36+v.37, v.37, v.44,\\& v.45, v.48, v.51, v.52, v.25-v.27,
  v.26+v.27, v.6+v.7, v.38+v.39 \} \end{align*}
\end{table}

\begin{landscape}

\begin{lstlisting}[basicstyle=\footnotesize, caption={$L_{-1}$ as an $L_0$-module for $F_4$ example},label={lst:f4example}]
gap> h:=Subalgebra(g,W);;h=g;
true
gap> for l in [1..26] do cof:=function(l);
t1:=[Coefficients(v,n[l]*v[1]),Coefficients(v,n[l]*v[2]),Coefficients(v,n[l]*v[3]),Coefficients(v,n[l]*v[4]),
Coefficients(v,n[l]*v[5]),Coefficients(v,n[l]*v[6]),Coefficients(v,n[l]*v[7]),Coefficients(v,n[l]*v[8]),
Coefficients(v,n[l]*v[9]),Coefficients(v,n[l]*v[10]),Coefficients(v,n[l]*v[11]),Coefficients(v,n[l]*v[12]),
Coefficients(v,n[l]*v[13]),Coefficients(v,n[l]*v[14]),Coefficients(v,n[l]*v[15]),Coefficients(v,n[l]*v[16]),
Coefficients(v,n[l]*v[17]),Coefficients(v,n[l]*v[18]),Coefficients(v,n[l]*v[19]),Coefficients(v,n[l]*v[20]),
Coefficients(v,n[l]*v[21]),Coefficients(v,n[l]*v[22]),Coefficients(v,n[l]*v[23]),Coefficients(v,n[l]*v[24]),
Coefficients(v,n[l]*v[25]),Coefficients(v,n[l]*v[26]),Coefficients(v,n[l]*v[27]),Coefficients(v,n[l]*v[28]),
Coefficients(v,n[l]*v[29]),Coefficients(v,n[l]*v[30]),Coefficients(v,n[l]*v[31]),Coefficients(v,n[l]*v[32]),
Coefficients(v,n[l]*v[33]),Coefficients(v,n[l]*v[34]),Coefficients(v,n[l]*v[35]),Coefficients(v,n[l]*v[36]),
Coefficients(v,n[l]*v[37]),Coefficients(v,n[l]*v[38]),Coefficients(v,n[l]*v[39]),Coefficients(v,n[l]*v[40]),
Coefficients(v,n[l]*v[41]),Coefficients(v,n[l]*v[42]),Coefficients(v,n[l]*v[43]),Coefficients(v,n[l]*v[44])];;
return t1;end;od;
\end{lstlisting}
\end{landscape}

\begin{lstlisting}[basicstyle=\footnotesize]
gap> Mats:=List([1..Dimension(Nrad)],cof);;
gap> gm:=GModuleByMats(Mats,GF(3));;
gap> MTX.BasesSubmodules(gm);
[< immutable compressed matrix 8x44 over GF(3) >,
< immutable compressed matrix 26x44 over GF(3) >,
< immutable compressed matrix 44x44 over GF(3) >]
\end{lstlisting}

This leaves us with proper submodules of dimension $8$ and $26$ only, and hence $L_{-1}/L_0$ is an $18$-dimensional irreducible module as required. We build the Weisfeiler filtration from our irreducible module $L_{-1}$ using \autoref{lst:lminus2} as follows:

\begin{lstlisting}[basicstyle=\small,caption={Obtaining $L_{-2}$ in $F_4$}]
gap> v:=Basis(V);;
gap> W:=List([]);;
gap> for j in [1..44] do
for k in [1..44] do
if \in(v[j]*v[k],W)=false then R:=List([v[j]*v[k]]);
Append(W,R);
continue;fi;od;od;
gap> VV:=VectorSpace(GF(3),W);;Dimension(VV);
52
\end{lstlisting}

This is used in \tref{weisf4} where we find a Weisfeiler filtration $\mathcal{F}$ such that $\gr(\mathcal{F})\cong S(3;\underline{1})^{(1)}$.

\subsubsection[The grading in Theorem \ref{weisf4}]{The grading in Theorem \ref{weisf4}}\label{gradingcheck}

Having obtained $L_{-1}$, we are able to do some computations for the Weisfeiler filtration in \tref{weisf4}. In the proof we obtained $3$ elements $f_1, f_2$ and $f_3$ representing $x^2yz^2\partial_x-xy^2z^2\partial_y$, $y^2x^2z\partial_y-x^2yz^2\partial_z$ and $xy^2z^2\partial_z-x^2y^2z\partial_x \in S(3;\underline{1})^{(1)}$.

Since we are looking in $\Gc$, all our spaces of quotient spaces since $\Gc_i=L_{i}/L_{i+1}$ from \seref{sec:Wei}. This creates some small issues using GAP, and so we initially compute the necessary Lie brackets. Then, we add $L_{i+1}$ afterwards. This is because when taking the Lie bracket in $\Gc$ we have $[u+L_{i+1},v+L_{j+1}]=[u,v]+L_{i+j-1}$. We set our elements $\partial_{\{x,y,z\}}$ from \tref{weisf4} in GAP using the following:

\begin{lstlisting}[basicstyle=\small]
gap> dz:=e;
gap> dx:=b[8]+b[9]+2*b[10];
gap> dy:=2*b[35]+b[37]+b[36];
\end{lstlisting}

We may take a basis for $\Gc_{-2}$ spanned by the elements \begin{align*}\{&(2f_{1000}+2f_{0001}+f_{0010})+L_{-1}, (e_{1100}+2e_{0110})+L_{-1}, \\&f_{0122}+L_{-1}, 2e_{1220}+L_{-1}, \\&2f_{0011}+L_{-1}, f_{1242}+L_{-1}, \\&2e_{2342}+L_{-1}, 2f_{1122}+L_{-1}\}.\end{align*} Using \cite{de08} we can write each of these elements as $x_i$ for some $i=1,\ldots,8$. That is, consider $x_1:=2f_{1000}+2f_{0001}+f_{0010}$. We then do the obvious computations to satisfy \eqref{relations} in \seref{blahblah}.

\begin{lstlisting}[basicstyle=\small]
gap> x1:=
gap> dx*(dx*(dz*(dz*(dy*x1))))=dx;
true
gap> dy*(dy*(dz*(dz*(dx*x1))))=dy;
true
\end{lstlisting}

Hence, we may conclude that $x_1+L_{-1}$ is our candidate for $x^2yz^2\partial_x-xy^2z^2\partial_y$. A similar idea is used to find $f_2$ and $f_3$, since using our basis for $L_{-1}$ we can consider elements of $\Gc_{-1}$. Then, we check the obvious relations in a similar way to the above to find our candidates for $y^2x^2z\partial_y-x^2yz^2\partial_z$ and $xy^2z^2\partial_z-x^2y^2z\partial_x$. 

We create our grading using $f_i$ and $\partial_{\{x,y,z\}}$, this ensures automatically that $[\Gc'_{-1},\Gc'_i]=\Gc'_{i-1}$ for all $i$. To check the rest of the grading is well-defined we compute all the necessary Lie brackets to check that $[\Gc'_i,\Gc'_j]=\Gc'_{i+j}$ for all $i,j$. 

Our main issue is cases where $[u,v] \ne 0$, but we require $[u+L_{i+1},v+L_{j+1}]=0$. Since the multiplication in $\Gc$ is defined as $[u,v]+L_{i+j-1}$, we must show that $[u,v] \in L_{i+j-1}$ to solve this issue. In fact, this is used to show $[\Gc'_4,\Gc'_4]=0$ in the proof of \tref{weisf4}.

For an example, we check some of the calculations in showing $[\Gc'_2,\Gc'_2]=\Gc'_4$ in this Appendix. The exact same idea is used for the grading in \tref{none6} for the $E_6$ case, with the obvious changes to the elements needed to be made.

Take the basis for $\Gc'_2$ from \seref{blahblah}, \eqref{gradp3}. We need to put these elements of $\g$ into a basis in GAP, taking $L_{-1}$ and maximal subalgebra $L_0=N_{\g}(A)$ as above in the Weisfeiler filtration. In GAP we have labelled these as $V$ and Nrad respectively. To obtain $u_i$ for some of our $i$ in \eqref{gradp3} simply input

\begin{lstlisting}
gap> u1:=b[16]+V;
gap> u2:=b[10]+Nrad;
gap> u3:=b[23]+Nrad;
gap> u5:=(2*b[36]+2*b[37])+Nrad;
gap> u12:=b[36]+Nrad
\end{lstlisting}

Then, we can consider $[u_1,u_2]=[2e_{1220}+L_{-1},e_{0120}+L_0]=0 \in \Gc$ since $[2e_{1220},e_{0120}]=0$. For a non-zero example, consider \begin{align*}[u_2,u_5+u_{13}]&=[e_{0120}+L_0,(2f_{0121})+L_0]\\&=[e_{0120},2f_{0121}]+L_{-1}\end{align*}. In GAP we can easily compute this Lie bracket in $F_4$ to give \[[u_2,u_5+u_{12}]=(f_{0001})+L_{-1}.\] Note, we take $u_5+u_{12}$ as $u_5, u_{12}$ and $u_{13}$ are not linearly independent as explained in the proof of \tref{weisf4}.

 Checking the list of elements of $L_{-1}$ from \autorefs{tabf4} we see that both $f_{0001}-f_{0010}$ and $f_{1000}+f_{0010}$ are elements of $L_{-1}$. Hence, $(f_{0001})+L_{-1}$ is equivalent to $f_1=(2f_{1000}+2f_{0001}+f_{0010})+L_{-1}$ as required. Continuing with all such Lie brackets will show that $[\Gc'_2,\Gc'_2]=\Gc'_4$ as required. To check the remainder of the grading is a simple case of repeating the above.

\subsection[Theorem \ref{none6}]{Theorem \ref{none6}}\label{AppE62}

Consider \tref{none6} regarding the nilpotent orbit $\mathcal{O}({A_2}^2+A_1)$ in $\g$ of type $E_6$ with representative $e:=\sum_{\al \in \Pi\setminus\{\al_4\} }e_{\al}$. To obtain the necessary information we use \cite[Table 9]{de08} to look up the basis elements in GAP.

\begin{lstlisting}[basicstyle=\small,caption={The nilpotent orbit $\mathcal{O}({A_2}^2+A_1)$ in $E_6$}]
gap> g:=SimpleLieAlgebra("E",6,GF(3));
gap> b:=Basis(g);
gap> e:=b[1]+b[2]+b[3]+b[5]+b[6];
gap> N:=LieNormaliser(g,Subalgebra(g,[e]));;
gap> rad:=LieSolvableRadical(N);;
gap> t:=Basis(rad);;
gap> Nrad:=LieNormalizer(g,rad);
\end{lstlisting}
This produces the $28$-dimensional normaliser $\mathfrak{n}_e$ for nilpotent element $e$, and $17$-dimensional solvable radical $A$. This contains $\mathfrak{z}(\g)$ and the following $16$ elements\begin{align*}\{&e_{\subalign{10&000\\&0}}-e_{\subalign{00&000\\&1}}+e_{\subalign{01&000\\&0}}, e_{\subalign{00&000\\&1}}-e_{\subalign{00&010\\&0}}-e_{\subalign{00&001\\&0}}, e_{\subalign{11&000\\&0}}, e_{\subalign{00&011\\&0}},\\& e_{\subalign{11&100\\&1}}+e_{\subalign{11&110\\&0}}-e_{\subalign{00&111\\&1}}+e_{\subalign{01&111\\&0}}, e_{\subalign{11&110\\&1}}+e_{\subalign{11&111\\&0}}, \\& e_{\subalign{11&111\\&0}}-e_{\subalign{01&111\\&1}}, e_{\subalign{11&111\\&1}}, e_{\subalign{12&211\\&1}}+e_{\subalign{11&221\\&1}}, e_{\subalign{12&221\\&1}},\\& f_{\subalign{00&100\\&0}}, f_{\subalign{00&100\\&1}}-f_{\subalign{01&100\\&0}}, f_{\subalign{00&110\\&0}}+f_{\subalign{01&100\\&0}} \\& f_{\subalign{11&100\\&0}}-f_{\subalign{01&100\\&1}}-f_{\subalign{00&110\\&1}}-f_{\subalign{11&100\\&0}}, f_{\subalign{01&210\\&1}}, f_{\subalign{11&210\\&1}}+f_{\subalign{01&211\\&1}}\},\end{align*}
We verify that this has $8$-dimensional abelian derived subalgebra, and obtain $N_{\g}(A)$ is $35$-dimensional. Then, we check $N_{\g}(A)/A$ is a $18$-dimensional simple restricted Lie algebra.

Using \autoref{lst:lminus} and \autoref{lst:fullminus}, we find $\dim\, L_{-1}=44$ in this case. Similar ideas are used to verify $\langle L_{-1}\rangle \cong \g$, and that $L_{-1}/L_0$ is irreducible. For the remainder of the Weisfeiler filtration in \eqref{weisfeiler} we repeatedly use and adapt \autoref{lst:lminus2}.

This gives $\dim\, L_{-2}=62$, $\dim\,L_{-3}=70$ and $\dim\, L_{-4}=78$. This is used in \tref{none6}\ref{partee6} to find a Weisfeiler filtration $\mathcal{F}$ such that $\gr(\mathcal{F})\cong \mathcal{S}_3(\underline{1},\omega_S)^{(1)}$. The basis for each of the spaces $L_{-i}$ are given in \autorefs{tabe6}, with \cite[Table 10]{de08} to be used to read off the elements from $E_6$.

\begin{landscape} \begin{table}[H]\caption{$L_{-1}$, $L_{-2}$ and $L_{-3}$ in $E_6$ for \tref{none6}}\phantomsection\label{tabe6}\footnotesize\centering\begin{align*}L_{-1}:=\{&v.1+v.2+v.3+v.5+v.6, v.7-v.11,
  v.22-v.23-v.25, v.27, v.34, v.40, v.44+v.45-v.46, v.60,\\& v.73-v.75+v.77-v.78,
  v.45-v.46, v.64, v.62, v.65+v.66+v.67, v.71, v.3+v.5, v.48+v.49-v.50+v.52, \\&v.49+v.52,
  v.53+v.54+v.55+v.56-v.57, v.23, v.11, v.2, v.46, v.50+v.51+v.52, \\&v.18+v.19-v.21,
  v.74+v.75-v.77-v.78, v.29+v.30+v.31, v.32, v.33, v.36, v.25, v.19+v.20-v.21,\\& v.17+v.21,
  v.12-v.13+v.14+v.15+v.16, v.5+v.6, v.77-v.78, v.6, v.51-v.52, \\&v.20+v.21, v.66-v.67,
  v.75-v.78, v.30-v.31, v.54-v.55+v.56, v.13+v.15-v.16, v.37-v.39+v.41-v.42\},
\end{align*}\begin{align*}
L_{-2}:=\{&v.40, v.60, v.7-v.11, v.44+v.45-v.46, v.27,
  v.22-v.23-v.25, v.1+v.2+v.3+v.5+v.6, v.34, v.11, v.45+v.46, \\&v.23-v.25, v.62+v.64,
  v.2-v.5-v.6, v.32+v.33, v.48-v.49-v.50-v.52,v.17+v.18-v.20+v.21,
  \\&v.73-v.75+v.77-v.78, v.64, v.46, v.25, v.3+v.5,  v.5, v.49+v.50, v.71, v.6, v.50-v.52,
  v.66-v.67, \\&v.65-v.67, v.77-v.78, v.18-v.19+v.20+v.21, v.51-v.52,
 v.53+v.54+v.55+v.56-v.57, v.75-v.78,\\& v.54-v.55+v.56, v.33,  v.19+v.20-v.21, v.20+v.21,
  v.74-v.78, v.12-v.13+v.14+v.15+v.16, \\&v.29+v.30+v.31,v.13+v.15-v.16,
 v.37-v.39+v.41-v.42, v.36, v.30-v.31, v.52, v.21, v.78, v.67, \\&v.76,
  v.55-v.56-v.57, v.31, v.14+v.15-v.16, v.72, v.56+v.57, v.68-v.69, v.15,
 v.41-v.42, v.35, v.26-v.28, \\&v.38+v.39+v.42, v.58-v.59-v.61, v.8+v.9-v.10 \}.\end{align*}
\begin{align*}L_{-3}:=\{&v.60, v.40, v.27, v.7-v.11, v.11, v.45+v.46, v.44+v.46,
  v.62+v.64, v.1+v.3-v.5-v.6, v.48-v.49-v.50-v.52, \\&v.34,
  v.22-v.23-v.25, v.23-v.25, v.2-v.5-v.6, v.17+v.18-v.20+v.21,
  v.73-v.75+v.77-v.78, v.32+v.33, \\&v.46, v.25, v.5+v.6, v.64, v.3-v.6, v.49+v.50+v.51-v.52,
  v.33, v.18-v.19+v.20+v.21,v.71, \\& v.50+v.51+v.52, v.65+v.66+v.67, v.19+v.20-v.21, v.77-v.78,
  v.36, v.29+v.30+v.31, v.74+v.75+v.78, \\&v.53+v.54+v.55+v.56-v.57, v.12-v.13+v.14+v.15+v.16, v.6,
  v.51-v.52, v.52, v.20-v.21, v.66,\\& v.75-v.76+v.78, v.67, v.21, v.76-v.78, v.72,
  v.56+v.57, v.78, v.54-v.55-v.57, v.55, v.68-v.69, \\&v.13-v.14,
  v.37-v.39-v.41+v.42,v.58-v.59-v.61, v.30, v.15, v.38+v.39+v.41, v.31, v.14-v.16,
   \\&v.41-v.42, v.26-v.28, v.8+v.9-v.10, v.35, v.57, v.16, v.69, v.39-v.42, v.28,
  v.59-v.61, v.9+v.10, v.43+v.47\}\end{align*}\end{table}\end{landscape}

\subsection[Theorem \ref{none7}]{Theorem \ref{none7}}\label{AppE72}

Now consider \tref{none7} which looks at the nilpotent orbit $\mathcal{O}({A_2}^2+A_1)$ in $\g$ of type $E_7$ with representative $e:=\sum_{\al \in \Pi\setminus\{\al_4\} }e_{\al}$. We use \cite[Table 10]{de08} to look up the basis elements in GAP.
\begin{lstlisting}[basicstyle=\small,caption={The nilpotent orbit $\mathcal{O}({A_2}^2+A_1)$ in $E_7$}]
gap> g:=SimpleLieAlgebra("E",7,GF(3));
gap> b:=Basis(g);
gap> e:=b[1]+b[2]+b[3]+b[5]+b[6];
gap> N:=LieNormaliser(g,Subalgebra(g,[e]));;
gap> rad:=LieSolvableRadical(N);;
gap> t:=Basis(rad);;
gap> Nrad:=LieNormalizer(g,rad);
gap> x1:=t[1];;x2:=t[2];;x3:=t[3];;x4:=t[4];;
gap> x5:=t[5];;x6:=t[6];;x7:=t[7];;x8:=t[8];;
\end{lstlisting}
This gives a $46$-dimensional normaliser of nilpotent element $e$ in $\g$ such that the radical is an $8$-dimensional abelian radical $A$. This is generated by \begin{align*}\langle&e_{\subalign{10&0000\\&0}}+e_{\subalign{00&0000\\&1}}+e_{\subalign{01&0000\\&0}}+e_{\subalign{00&0000\\&1}}+e_{\subalign{00&0100\\&0}}+e_{\subalign{00&0010\\&0}}, \\&e_{\subalign{11&0000\\&0}}-e_{\subalign{00&0110\\&0}},e_{\subalign{12&2210\\&1}}, e_{\subalign{11&1110\\&1}}, f_{\subalign{00&1000\\&0}}, f_{\subalign{01&2100\\&1}}\\& f_{\subalign{00&1000\\&1}}-f_{\subalign{01&1000\\&0}}+f_{\subalign{00&1100\\&0}},e_{\subalign{11&1100\\&1}}-e_{\subalign{11&1110\\&0}}-e_{\subalign{01&1110\\&1}} \rangle\end{align*}This shows that $N_{\g}(A)$ is $53$-dimensional, and $M:=N_{\g}(A)/A$ is a $35$-dimensional semisimple restricted Lie algebra.

This is the first case where we have a minimal ideal $I$ in $M$, and we refer to computing such an ideal above \tref{none7}. Using \autoref{lst:lminus} and \autoref{lst:fullminus} we obtain $L'_{-1}:=\{x \in \g: [x,A] \subseteq N_{\g}(A)\}$ of dimension $98$. However, this has a proper submodule of dimension $80$ that strictly contains $L_0$.

\begin{lstlisting}[basicstyle=\small,caption={Finding the ideal $I$ for \tref{none7}}]
gap> nn:=Basis(N);;w1:=nn[6];;w2:=nn[31];;
gap> y1:=b[104];;y2:=b[8]+b[12];;y3:=b[55];;
\end{lstlisting}

This gives the elements of $\g_e(\tau,-1)$ and $\g_e(\tau,4)$ that are not contained in $A$ used to compute $I$. These elements denoted $v_{30}:=f_{\subalign{01&2111\\&1}}$, $v_{31}=e_{\subalign{11&0000\\&0}}+e_{\subalign{00&0110\\&0}}$ and $v_{32}=e_{\subalign{12&2221\\&1}}$ in \cite[pg. 98]{LT11}. We then compute $I$, and the radical of $I/A$ with the following commands.
\begin{lstlisting}[basicstyle=\footnotesize]
gap> I:=Subalgebra(g,
[w2*(w2*(w1*(w1*y1))),w2*(w2*(w1*(w1*y2))),
w2*(w2*(w1*(w1*y3))),w2*(w2*(w1*y1)),
w2*(w2*(w1*y2)),w2*(w2*(w1*y3)),
w1*(w2*(w1*y1)),w1*(w2*(w1*y2)),
w1*(w2*(w1*y3)),w1*y1,w1*y2,w1*y3,
w2*y1,w2*y2,w2*y3,w1*(w1*y1),w1*(w1*y2),
w1*(w1*y3),w2*(w1*y1),w2*(w1*y2),
w2*(w1*y3),w2*(w2*y1),w2*(w2*y2),w2*(w2*y3)]);;
gap> Dimension(I);
35
gap> LieSolvableIdeal(I/A);
<Lie algebra of dimension 24 over GF(3)>
\end{lstlisting}

The basis for our ideal $I$ is given in \autorefs{tabeI7}. To compute the irreducible module we adapt the \autoref{lst:lminus}, and use $I$ to produce the required $L_0$-invariant subspace $L_{-1}$ such that $L_{-1}/L_0$ is irreducible.
\begin{lstlisting}[basicstyle=\small]
gap> for i in [1..133] do
s:=function(i)
a1:=\in(b[i]*x1,I);a2:=\in(b[i]*x2,I);
a3:=\in(b[i]*x3,I);a4:=\in(b[i]*x4,I);
a5:=\in(b[i]*x5,I);a6:=\in(b[i]*x6,I);
a7:=\in(b[i]*x7,I);a8:=\in(b[i]*x8,I);
a9:=\in(b[i],Nrad);
if a9=true or a1=false or a2=false or a3=false or a4=false
or a5=false or a6=false or a7=false or a8=false then
return 0;else return i;
fi;end;
if s(i)=i then Append(W,[b[i]]);fi;od;od;
gap> V:=VectorSpace(GF(3),W);; v:=Basis(V);
\end{lstlisting}
We keep repeating as in \aref{AppE62} until we find an $L_0$-invariant subspace of dimension $80$, and obtain the basis of our $L'_{-1}$ given in \autorefs{tabe7}. We are able to check $L_{-1}/L_0$ has dimension $27$ and is irreducible. To finish the proof of maximality we showed $L'_{-1}/L_0$ is indecomposable, this is easily achieved in GAP using
\begin{lstlisting}[basicstyle=\small]
gap> Mats:=List([1..Dimension(Nrad)],cof);;
gap> gm:=GModuleByMats(Mats,GF(3));;
gap> MTX.IsIndecomposable(gm);
true\end{lstlisting}where ``cof'' is as in \autoref{lst:modulelminus}, adapted to this particular case in $E_7$. If the module is indecomposable, then the factor module is indecomposable --- hence this gives us what we require.

We compute the remaining $L_{-i}$ for a Weisfeiler filtration, and allows us to make the conclusions in \rref{weise7mate} about the corresponding graded Lie algebra. We obtain $\dim\, L_{-2}=125$ and $L_{-3}=\g$, with basis found in \autorefs{tabe7}.

\begin{table}[H]\caption{$I$ in $E_7$ for \tref{none7}}\phantomsection\label{tabeI7}\footnotesize\centering
\begin{align*}I=\langle &v.110+v.111+v.112, v.127-v.129+v.131-v.132,
  v.47+v.48+v.49, \\&v.119-v.120, v.77-v.78-v.79-v.81, v.34-v.36,v.93-v.94,
  \\&v.20+v.21-v.23+v.24, v.58-v.59, v.88, v.26+v.29, v.60, \\&v.117, v.73+v.74, v.39, v.70, v.42+v.43, \\&v.63,
  v.107+v.108, v.1+v.3-v.5-v.6, \\&v.51+v.52, v.124, v.95+v.98, v.19, v.2-v.5-v.6, v.104,
v.12, \\&v.91, v.27-v.29, v.72+v.74, v.8, v.55, v.67, v.46, v.33 \rangle
  \end{align*}\end{table}

\begin{landscape}
\begin{table}[H]\caption{$L'_{-1}$ and $L'_{-2}$ in $E_7$ for \tref{none7}}\phantomsection\label{tabe7}\footnotesize\centering
\begin{align*}L_{-1}:=\{&v.1-v.2+v.3, v.2-v.5-v.6, v.8, v.12, v.19,
  v.20+v.21-v.23+v.24, v.26+v.27, v.27-v.29, v.33, \\&v.34-v.36, v.39, v.42+v.43, v.46,
  v.47+v.48+v.49, v.51+v.52, v.55, v.58-v.59, v.60, v.63, v.67, v.70,\\& v.72-v.73, v.73+v.74,
  v.77-v.78-v.79-v.81, v.88, v.91, v.93-v.94, v.95+v.98, v.104, v.107+v.108,
  v.110+v.111+v.112, \\&v.117, v.119-v.120, v.124, v.127-v.129+v.131-v.132, v.29,v.74,
  v.3-v.5+v.6, v.98, v.100+v.101+v.103, v.78-v.79-v.80, \\& v.113, v.83+v.84+v.85+v.86-v.87,
  v.21-v.23, v.43, v.5+v.6, v.128+v.129+v.131, v.79+v.80+v.81, v.37+v.38+v.40, \\&v.53, v.22+v.23-v.24,
  v.14-v.15+v.16+v.17+v.18, v.131-v.132, v.6, v.13, v.76, v.94, v.52, v.36,\\& v.23+v.24, v.38-v.40,
  v.108, v.80-v.81, v.101-v.103, v.111-v.112, v.120, v.97+v.99, v.48-v.49,\\& v.59,
  v.30+v.31, v.129-v.132, v.62,v.15+v.17-v.18, v.125,  \\&v.84-v.85+v.86, v.121+v.122,
  v.114-v.115, v.56+v.57, v.44-v.45, v.64-v.66+v.68-v.69\}\end{align*}\end{table}

\begin{table}[H]\footnotesize\centering
\begin{align*}L_{-2}:=\{&v.26-v.27-v.29, v.33, v.39, v.46, v.51+v.52, v.55,
 v.60, v.67, v.72+v.73-v.74, v.88,  \\&v.91, v.104, v.107+v.108, v.117, v.8, v.95+v.98, v.73+v.74,
  v.77-v.78-v.79-v.81, v.27-v.29, v.1-v.2+v.3, \\&v.42+v.43,
  v.20+v.21-v.23+v.24, v.29, v.43, v.74, v.98, v.108, v.93, v.52, \\&v.34, v.2, v.63, v.22+v.23-v.24,
  v.124, v.79+v.80+v.81, v.120, v.111-v.112, v.58, v.47-v.48, v.127-v.129, v.12, \\&v.5+v.6, v.19,
  v.70, v.94, v.36, v.119, v.110-v.112, v.59, v.48-v.49, v.131-v.132, \\&v.6, v.78+v.80, v.113,
  v.100-v.101, v.13, v.53, v.3, v.80-v.81, v.76, v.21-v.23, \\&v.37+v.38+v.40, v.128+v.129+v.132,
  v.97+v.99, v.62, v.56+v.57,v.14-v.15+v.16+v.17+v.18,  \\&v.23+v.24, v.129-v.132, v.101-v.103,
  v.84-v.85+v.86, v.30+v.31, v.38-v.40, v.15+v.17-v.18, v.83-v.85-v.87,
  \\&v.64-v.66+v.68-v.69, v.44-v.45, v.125, v.121+v.122, v.114-v.115, v.24, v.81, v.112,
  v.49, v.132, v.133, v.103, \\&v.85+v.86, v.31, v.116, v.106, v.7, v.40, v.16+v.17, v.99, v.50, v.32, v.82, v.130,
  v.17, v.86, v.57, \\&v.18, v.61, v.35, v.122, v.87, v.126, v.105, v.66-v.69, v.115, v.65-v.68, v.9+v.11,
  v.102,\\& v.68-v.69, v.45, v.89+v.90, v.25, v.54, v.10+v.11, v.123, v.90-v.92, v.71+v.75, v.118, v.41 \},\end{align*}\end{table}
  \end{landscape}

\subsection[Theorem \ref{none8}]{Theorem \ref{none8}}\label{AppE82}
We continue to focus on the nilpotent orbit $\mathcal{O}({A_2}^2+A_1)$, but now in $\g$ of type $E_8$ with representative $e:=\sum_{\al \in \Pi\setminus\{\al_4\} }e_{\al}$. We use \cite[Table 11]{de08} to look up the basis elements in GAP.
\begin{lstlisting}[basicstyle=\small,caption={The nilpotent orbit $\mathcal{O}({A_2}^2+A_1)$ in $E_8$}]
gap> g:=SimpleLieAlgebra("E",8,GF(3));
gap> b:=Basis(g);
gap> e:=b[1]+b[2]+b[3]+b[5]+b[6];
gap> N:=LieNormaliser(g,Subalgebra(g,[e]));;
gap> rad:=LieSolvableRadical(N);;
gap> t:=Basis(rad);;
gap> Nrad:=LieNormalizer(g,rad);
gap> x1:=t[1];;x2:=t[2];;x3:=t[3];;x4:=t[4];;
gap> x5:=t[5];;x6:=t[6];;x7:=t[7];;x8:=t[8];;
\end{lstlisting}
This produces an $89$-dimensional $\mathfrak{n}_e$ for representative $e$, with an $8$-dimensional abelian radical $A$ generated by \begin{align*}\langle&e_{\subalign{10&00000\\&0}}+e_{\subalign{00&00000\\&1}}+e_{\subalign{01&00000\\&0}}+e_{\subalign{00&00000\\&1}}+e_{\subalign{00&01000\\&0}}+e_{\subalign{00&00100\\&0}}, \\&e_{\subalign{11&00000\\&0}}-e_{\subalign{00&01100\\&0}},e_{\subalign{12&22100\\&1}}, e_{\subalign{11&11100\\&1}}, f_{\subalign{00&10000\\&0}}, f_{\subalign{01&21000\\&1}}\\& f_{\subalign{00&10000\\&1}}-f_{\subalign{01&10000\\&0}}+f_{\subalign{00&11000\\&0}},e_{\subalign{11&11000\\&1}}-e_{\subalign{11&11100\\&0}}-e_{\subalign{01&11100\\&1}} \rangle\end{align*}We obtain that $N_{\g}(A)$ is $96$-dimensional, and consider $M:=N_{\g}(A)/A$ to show this is an $88$-dimensional semisimple restricted Lie algebra.

Using \autoref{lst:lminus} and \autoref{lst:fullminus} we obtain $L_{-1}:=\{x \in \g: [x,A] \subseteq N_{\g}(A)\}$ of dimension $177$. However, this has a proper submodule of dimension $159$ that strictly contains $L_0$. Hence, we follow similar techniques to the previous subsection.

First, we need to find the ideal $I$ obtained with the elements of $\g_e(\tau,-1)$ and $\g_e(\tau,4)$ that are not contained in $A$.
\begin{lstlisting}[basicstyle=\small, caption={Finding the ideal $I$ in \tref{none8}}]
gap> nn:=Basis(N);;
gap> w1:=nn[7];;w2:=nn[54];;
gap> y1:=b[219];;y2:=b[176];;
gap> y3:=b[169];;y4:=b[9]+b[13];;
gap> y5:=b[71];;y6:=b[78];;y7:=b[109];;
\end{lstlisting}

For this particular case we can form a bigger ideal, adding in the additional elements:
\begin{lstlisting}[basicstyle=\small]
gap> u1:=b[240];u2:=b[239];u3:=b[119];u4:=b[120];
\end{lstlisting}
which produces the ideal $J$ of dimension $78$ with basis given in \autorefs{E8J}, and using simple commands we confirm $M/J$ is an $18$-dimensional simple Lie algebra. We then simply repeat the process from \aref{AppE72} using ideal $I$ in $E_8$. This gives $\dim\, L'_{-1}=159$, $\dim \, L'_{-2}=240$ and $L'_{-3} \cong \g$ with bases given in \autorefs{tabe8}.

  \begin{table}[H]\caption{Basis of $J$ for \tref{none8}}\phantomsection\label{E8J}\footnotesize\centering
\begin{align*}
J=\langle &v.240, v.120, v.119, v.239, v.224-v.225+v.226,
  v.185+v.187+v.188, \\&v.178+v.179+v.181, v.231+v.232, v.201-v.203, \\&v.195-v.196, v.211-v.212,
  v.162-v.163,\\& v.154-v.155, v.208, v.156, v.148,\\& v.230, v.197, v.190, v.194, v.135,v.127,
  v.222-v.223,\\&v.180+v.182, v.173+v.175, v.236, v.214, v.209, \\&v.241-v.243+v.245-v.246,
  v.58+v.59+v.61,v.65+v.67+v.68, \\&v.136-v.137-v.138-v.140, v.39-v.41, v.47-v.50,
  \\&v.23+v.24-v.26+v.27, v.80-v.82,v.86-v.87, v.30+v.33, v.85, \\&v.90,v.131+v.132, v.46, v.54,
  v.51+v.52, \\&v.97,v.101, v.1+v.3-v.5-v.6, v.64+v.66, v.72+v.73,\\&v.157+v.160, v.21, v.29, v.115,v.100,
  v.118, v.84, \\&v.113+v.114, v.95-v.96, v.104-v.105+v.106, v.248, \\&v.128, v.8, v.107-v.108, v.219,
  v.176,v.169, \\&v.9, v.152, v.31-v.33, v.2-v.5-v.6, \\&v.130+v.132, v.13, v.71, v.78,\\&v.124, v.57,
  v.38, v.109 \rangle \end{align*}\end{table}

\begin{landscape}
\begin{table}[H]\caption{$L'_{-1}$ and $L'_{-2}$ in $E_8$ for \tref{none8}}\phantomsection\label{tabe8}\footnotesize\centering
\begin{align*}
L'_{-1}:=\{&v.1-v.2+v.3, v.2-v.5-v.6, v.8, v.9, v.13,
  v.21, v.23+v.24-v.26+v.27, v.29, v.30+v.31, \\&v.31-v.33, v.38, v.39-v.41, v.46, v.47-v.50,
  v.51+v.52, v.54, v.57, v.58+v.59+v.61, v.64+v.66, \\&v.65+v.67+v.68, v.71, v.72+v.73, v.78, v.80-v.82, v.84,
 v.85, v.86-v.87, v.90, v.95-v.96,  \\&v.97, v.100, v.101, v.104-v.105+v.106, v.107-v.108,
 v.109, v.113+v.114, v.115, v.118, v.119, v.120, v.124, \\&v.127, v.128, v.130-v.131, v.131+v.132, v.135,
  v.136-v.137-v.138-v.140, v.148, v.152, v.154-v.155, \\& v.156, v.157+v.160,
  v.162-v.163, v.169, v.173+v.175, v.176, v.178+v.179+v.181, v.180+v.182, \\& v.185+v.187+v.188, v.190, v.194,
  v.195-v.196, v.197, v.201-v.203, v.208, v.209, v.211-v.212, v.214, v.219, v.222-v.223,\\&
  v.224-v.225+v.226, v.230, v.231+v.232, v.236, v.239, v.240, v.241-v.243+v.245-v.246, v.248,
  \\&v.132, v.33, v.242+v.243+v.246, v.16-v.17+v.18+v.19+v.20, v.3+v.5, v.52, v.138+v.139+v.140, v.69,
  \\&v.24+v.25-v.27, v.44+v.45+v.48, v.143+v.144+v.145+v.146-v.147, v.160, v.5+v.6, v.183,
 \\& v.25+v.26-v.27, v.137+v.140, v.164+v.165+v.168, v.245-v.246, v.6, v.14, v.79, v.73, v.82, v.87,
 \\& v.155, v.22, v.41, v.96, v.108, v.66, v.114, v.26+v.27, v.163, v.175, v.139-v.140, v.50, v.165-v.168,
 \\& v.182, v.212, v.223, v.203, v.218, v.237, v.213, v.34+v.35, v.91+v.92, v.93, v.98, v.134, v.59-v.61,
  \\&v.105+v.106, v.67-v.68, v.179-v.181, v.196, v.42+v.43, v.232, v.117, v.45-v.48,
  \\&v.243-v.246, v.142, v.187-v.188, v.144-v.145+v.146, v.199, v.225+v.226, v.206+v.207,
  \\&v.233-v.234, v.200+v.202, v.53-v.55, v.102+v.103, v.75+v.76, v.111-v.112, v.81+v.83,
  \\&v.159+v.161, v.60-v.62, v.17+v.19-v.20, v.167+v.170, v.121-v.123+v.125-v.126,
 \\& v.215+v.216, v.192-v.193, v.227+v.228, v.184-v.186 \} \end{align*}
 \end{table}
 \begin{table}[H]\centering\footnotesize
 \begin{align*}L'_{-2}:=\{&v.30-v.31-v.33, v.38, v.46, v.54, v.57, v.64+v.66,
  \\&v.71, v.72+v.73, v.78, v.85, v.90, v.100, v.107-v.108, v.109, v.115, v.124, v.130+v.131-v.132, v.148,
  \\&v.152, v.156, v.169, v.173+v.175, v.176, v.180+v.182, v.190, v.197, v.208, v.219, v.222-v.223, v.230,
  \\&v.1-v.2+v.3, v.23+v.24-v.26+v.27, v.9, v.31-v.33, v.51+v.52, v.136-v.137-v.138-v.140,
    \\&v.131+v.132, v.157+v.160, v.84, v.33, v.132, v.160, v.214, v.209, v.39, v.95, v.97, v.101, v.66,
  \\&v.108, v.73, v.175, v.47, v.52, v.2, v.182, v.138+v.139+v.140, v.194, v.223, v.203, v.232, v.196, v.58-v.59,
  \\&v.104+v.105, v.80, v.114, v.86, v.154, v.65-v.67, v.25+v.26-v.27, v.162, v.241-v.243, v.212,
  \\&v.187-v.188, v.225+v.226, v.179-v.181, v.5+v.6, v.13, v.21, v.29, v.236, v.41, v.96, v.127, v.50,
  \\&v.118, v.135, v.201, v.231, v.195, v.59-v.61, v.105+v.106, v.82, v.113, v.87, v.155, v.67-v.68,
  \\&v.163, v.245-v.246, v.211, v.185-v.188, v.224-v.226, v.178-v.181, v.120, v.248, v.239,
  \\&v.8, v.22, v.213, v.42+v.43, v.98, v.134, v.200+v.202, v.60-v.62, v.81+v.83, v.159+v.161, v.184-v.186,
  \\&v.6, v.183, v.137+v.139, v.79, v.14, v.218, v.164-v.165, v.3, v.139-v.140, v.69, v.24-v.26,
  \\&v.242+v.243+v.246, v.75+v.76, v.111-v.112, v.93, v.117, v.44+v.45+v.48, v.142,
  \\&v.16-v.17+v.18+v.19+v.20, v.199, v.167+v.170, v.215+v.216, v.26+v.27, v.165-v.168, v.243-v.246,
  \\&v.91+v.92, v.34+v.35, v.206+v.207, v.144-v.145+v.146, v.45-v.48, v.143-v.145-v.147,
  \\&v.17+v.19-v.20, v.102+v.103, v.53-v.55, v.192-v.193, v.121-v.123+v.125-v.126,
  \\&v.128, v.119, v.240, v.237, v.233-v.234, v.227+v.228, v.140, v.61, v.106, v.68, v.27, v.246, v.188, v.226,
  \\&v.181, v.247, v.168, v.189, v.92, v.145+v.146, v.74, v.35, v.7, v.207, v.221, v.172, v.244, v.43, v.146+v.147,
  \\&v.171, v.15, v.48, v.112, v.76, v.170, v.19-v.20, v.83, v.161, v.18, v.216, v.202, v.217, v.147, v.234,
  \\&v.238, v.20, v.103, v.122-v.123-v.125+v.126, v.88, v.55, v.28, v.193, v.210, v.150+v.151, v.62,
  \\&v.123-v.125, v.151-v.153, v.36, v.63, v.116, v.89, v.149, v.40, v.94, v.141, v.37, v.204,
  \\&v.125-v.126, v.186, v.228, v.205, v.235, v.110, v.10+v.12, v.99, v.70, v.49, v.174, v.198, v.129+v.133,
  \\&v.77, v.11+v.12, v.56, v.166, v.220, v.191, v.229\}\end{align*}\end{table}\end{landscape}

\subsection[Theorem \ref{CH41}]{Theorem \ref{CH41}}\label{sec:ch41}

The final case for characteristic three was the orbit $\mathcal{O}({A_2}^2+{A_1}^2)$ in exceptional Lie algebras of type $E_8$ from \tref{CH41}. We use \cite[Table 11]{de08} to look up the basis elements in GAP.
\begin{lstlisting}[basicstyle=\small,caption={The nilpotent orbit $\mathcal{O}({A_2}^2+{A_1}^2)$ in $E_8$}]
gap> g:=SimpleLieAlgebra("E",8,GF(3));
gap> b:=Basis(g);
gap> e:=b[1]+b[2]+b[3]+b[5]+b[6]+b[8];
gap> N:=LieNormaliser(g,Subalgebra(g,[e]));;
gap> C:=LieCentralizer(g,Subalgebra(g,[e]));;
gap> LieDerivedSubalgebra(C);
\end{lstlisting}

This gives that $\dim\,\mathfrak{n}_e=85$ along with the radical equal to $\bbk e$. We can repeatedly take derived subalgebras, and check whether each one is simple or not. Eventually, we show the second derived subalgebra of $\mathfrak{g}_e/\bbk e$ is simple of dimension $79$.

For $M_0:=\mathfrak{n}_e$, there is no ideal to be used to obtain the necessary $M_{-1}$ vector space such that $M_{-1}/M_0$ is irreducible. Initially, we find a module of dimension $168$ which we called $M'_{-1}$ in \tref{CH41}. We check that the module $M'_{-1}/\mathfrak{n}_e$ is indecomposable.

Checking on GAP we can see that as an $M_0$-module this has multiple different submodules of slightly lower dimensions, but only one of dimension $164$. Hence the basis in \autorefs{tabe82} is the correct one that we need. This $M_{-1}$ is such that $M_{-1}/M_0$ is irreducible, and generates $E_8$ as a Lie algebra allowing us to complete \tref{CH41}.

\begin{landscape}\begin{table}[H]\caption{$M_{-1}$ in $E_8$ for \tref{CH41}}\phantomsection\label{tabe82}\footnotesize\centering
\begin{align*}M_{-1}:=\{&v.2, v.1+v.3, v.5+v.6, v.8, v.9, v.13,
  v.16-v.17+v.18+v.19+v.20, \\&v.21+v.22, v.23+v.25+v.26, v.24-v.25+v.26+v.27, v.29, v.30+v.31,
  v.31-v.33, v.38,\\& v.39-v.41+v.42+v.43, v.46+v.47, v.44+v.45+v.48, v.47-v.50, v.51+v.52, v.54,
  v.57, \\&v.58+v.59+v.60+v.61-v.62, v.64+v.65+v.66-v.67, v.65+v.67+v.68, v.69, v.71-v.72,
  v.72+v.73, v.78, \\&v.80+v.81-v.82+v.83, v.84, v.85-v.86, v.86-v.87, v.90, v.95-v.96,
 \\& v.97-v.98, v.100, v.101, v.104-v.105+v.106, v.107-v.108, v.109, v.113+v.114, v.115, v.118,
 \\& v.120, v.124, v.127, v.130-v.131, v.131+v.132, v.134+v.135, v.136+v.137+v.138-v.139,
  \\&v.137+v.138+v.139-v.140, v.143+v.144+v.145+v.146-v.147, v.148, v.152, v.154-v.155,
  \\&v.155+v.156, v.157+v.160, v.159+v.161+v.162-v.163, v.164+v.165+v.168, v.169, v.173+v.175, v.175-v.176,
 \\& v.178+v.179+v.181, v.179-v.180-v.181-v.182, v.183, \\&v.184-v.185-v.186-v.187-v.188, v.190, v.194, v.195-v.196, v.196-v.197, v.200-v.201+v.202+v.203,
  \\&v.208, v.209, v.211-v.212, v.213-v.214, v.219, v.222-v.223, v.224-v.225+v.226, v.230,
 \\& v.231+v.232, v.236, v.239, v.241-v.243, v.245-v.246, v.242+v.243+v.246+v.248, v.73, v.212, v.3, v.6,
\\&  v.87, v.199, v.25+v.27, v.26+v.27, v.197, v.33, v.180+v.181-v.182, v.50, v.176, v.52, v.108, v.161+v.162,
  \\&v.66-v.67, v.67-v.68, v.156, v.135, v.132, v.223, v.225+v.226, v.138-v.140, v.139-v.140,
  \\&v.232, v.160, v.181+v.182, v.98, v.45-v.48, v.162-v.163, v.114, v.141-v.142,
  \\&v.81-v.82-v.83, v.82-v.83, v.215+v.216, v.243-v.246+v.248, v.248, v.214, v.246, v.96,
  \\&v.201+v.202, v.22, v.17-v.18, v.105+v.106, v.185+v.186-v.187, v.41+v.42-v.43, v.42+v.43,
  \\&v.202+v.203, v.233-v.234, v.165-v.168, v.237, v.117, v.93+v.94, v.18+v.19-v.20,
  \\&v.186+v.187-v.188, v.37-v.40, v.111-v.112, v.166-v.167-v.170,
  \\&v.59-v.61-v.62, v.60-v.62, v.227+v.228, v.144-v.145+v.146,
  \\&v.145-v.146-v.147, v.79, v.217+v.218, v.121-v.123-v.125+v.126, v.91+v.92,
  \\&v.205+v.206+v.207, v.14-v.15, v.75+v.76+v.77, v.125-v.126, v.10+v.11-v.12, v.102+v.103,
  \\&v.191+v.192-v.193, v.34+v.35-v.36, v.171-v.172, v.150-v.151-v.153,v.53-v.55-v.56\}\end{align*}\end{table}\end{landscape}

\subsection[Theorem \ref{none6p2}]{Theorem \ref{none6p2}}\label{AppE6p2}

We move to the results in characteristic two, and start by considering the nilpotent orbit $\mathcal{O}({A_1}^3)$ in $\g$ of type $E_6$ with representative $e:=\rt{1}+\rt{4}+\rt{6}$. We use \cite[Table 9]{de08} to look up the basis elements in GAP.
\begin{lstlisting}[basicstyle=\small,caption={The nilpotent orbit $\mathcal{O}({A_1}^3)$ in $E_6$}]
gap> g:=SimpleLieAlgebra("E",6,GF(2));;
gap> b:=Basis(g);;
gap> e:=b[1]+b[4]+b[6];;
gap> N:=LieNormaliser(g,Subalgebra(g,[e]));;
gap> rad:=LieSolvableRadical(N);;t:=Basis(rad);;
gap> Nrad:=LieNormalizer(g,rad); x1:=t[1];;x2:=t[2];;x3:=t[3];;
\end{lstlisting}
This gives a $41$-dimensional Lie normaliser of $e$ with a $3$-dimensional abelian radical $A$ generated by \[\langle e, e_{\subalign{11&211\\&1}},f_{\subalign{01&110\\&1}}\rangle\] We obtain $N_{\g}(A)$ is $43$-dimensional, and can consider $M:=N_{\g}(A)/A$ --- a $40$-dimensional semisimple restricted Lie algebra.

Both \autoref{lst:lminus} and \autoref{lst:fullminus} are used to find $L_{-1}:=\{x \in \g: [x,A] \subseteq N_{\g}(A)\}$ of dimension $78$. However, this has a proper submodule of dimension $75$ that strictly contains $L_0$, but is indecomposable. We follow similar techniques to obtain the ideal $I_6$ using the elements of $\g_e(\tau,-1)$ and $\g_e(\tau,2)$ that are not contained in $A$.
\begin{lstlisting}[basicstyle=\small, caption={Finding the ideal $I_6$ in \tref{none6p2}}]
gap> nn:=Basis(N);;
gap> w1:=nn[13];;w2:=nn[33];;
gap> y1:=b[23];;y2:=b[16];;y3:=b[12];;y4:=b[4]+b[6];;
gap> y5:=b[4]+b[1];y6:=b[39];;y7:=b[41];;y8:=b[51];;
\end{lstlisting}
This produces the ideal $I_6$ of dimension $35$ with solvable radical of dimension $27$ with basis given in \autorefs{tabeI62}. We repeat \aref{AppE72}, using $E_6$ in characteristic two, $e$ as above and ideal $I_6$. This gives $\dim\, L'_{-1}=75$, $L'_{-2}=\g$ with basis for $L'_{-1}$ given in \autorefs{tabe62}.\begin{landscape}
\begin{table}[H]\caption{$L'_{-1}$ in $E_6$ for \tref{none6p2}}\phantomsection\label{tabe62}\small\centering
\begin{align*}
L'_{-1}:=\{&v.1, v.4, v.6, v.8, v.7+v.9, v.10+v.11, v.12, v.16, v.17, v.20,
  v.18+v.21, v.23, \\&v.26+v.27, v.27+v.28, v.30, v.32, v.33, v.35, v.38, v.39, v.41, v.43+v.45, v.46+v.47, \\&v.49, v.50,
  v.51, v.55, v.54+v.57, v.58+v.60, \\&v.60+v.61, v.65, v.67, v.70,  v.73+v.76, v.76+v.78, v.28, v.61, v.62+v.63+v.64,
  v.74+v.75+v.77+v.78,\\& v.72, v.22+v.24+v.25, v.36, v.78, v.2, v.3, v.5, v.9, v.45, v.44, v.11, v.13, v.14,\\& v.75+v.77,
  v.15, v.77, v.21, v.24+v.25, v.25,\\&  v.29, v.31, v.34, v.37+v.40, v.53, v.40+v.42, v.57, v.56, v.47, v.52, v.48,
  v.63, v.64, \\& v.69, v.68, v.59, v.71\} \end{align*}
 \end{table}

 \begin{table}[H]\caption{$I_6$ in $E_6$ for \tref{none6p2}}\phantomsection\label{tabeI62}\small\centering
\begin{align*}
I_6=\langle &v.54+v.57, v.8, v.70, v.38, v.35, v.33, v.32, v.49, v.50,
  v.18+v.21, v.10+v.11, v.7+v.9, \\&v.26+v.27, v.27+v.28, v.20, v.60+v.61, v.58+v.61, v.65, v.76+v.78, v.73+v.78,
  v.43+v.45, v.17, \\& v.67, v.46+v.47, v.4, v.39, v.41, v.51, v.23, v.16, v.12, v.55, v.30, v.1, v.6 \rangle \end{align*}
 \end{table}\end{landscape}

 \subsection[Theorem \ref{none8p2}]{Theorem \ref{none8p2}}\label{AppE8new}

Now consider the nilpotent orbit $\mathcal{O}({A_1}^3)$ in $\g$ of type $E_8$ with representative $e:=\rt{1}+\rt{4}+\rt{6}$. We use \cite[Table 11]{de08} to look up the basis elements in GAP.
\begin{lstlisting}[basicstyle=\small,caption={The nilpotent orbit $\mathcal{O}({A_1}^3)$ in $E_8$}]
gap> g:=SimpleLieAlgebra("E",8,GF(2));;
gap> b:=Basis(g);;
gap> e:=b[1]+b[4]+b[6];;
gap> N:=LieNormaliser(g,Subalgebra(g,[e]));;
gap> x1:=e;;x2:=b[145];;x3:=b[45];
rad:=Subalgebra(g,[x1,x2,x3]);;
gap> Nrad:=LieNormalizer(g,rad);;
\end{lstlisting}
This produces the $139$-dimensional Lie normaliser $\mathfrak{n}_e$ of nilpotent element $e$, and gives an $3$-dimensional abelian radical $A$ generated by \[\langle e, e_{\subalign{11&21100\\&1}},f_{\subalign{01&11000\\&1}}\rangle\] We obtain that $N_{\g}(A)$ is $141$-dimensional, and can consider $M:=N_{\g}(A)/A$ --- a $138$-dimensional semisimple restricted Lie algebra.

Both \autoref{lst:lminus} and \autoref{lst:fullminus} gives $L_{-1}:=\{x \in \g: [x,A] \subseteq N_{\g}(A)\}$ of dimension $248$. However, this has a proper submodule of dimension $245$ that strictly contains $L_0$, but is indecomposable. We follow similar techniques and obtain the ideal $I_8$ using elements of $\g_e(\tau,-1)$ and $\g_e(\tau,2)$ that are not contained in $A$.
\begin{lstlisting}[basicstyle=\small,caption={Finding the ideal $I_8$ in \tref{none8p2}}]
gap> nn:=Basis(N);;
gap> w1:=nn[17];;w2:=nn[88];;
gap> y1:=b[31];;y2:=b[19];;y3:=b[16];;
y4:=b[1]-b[4];;y5:=b[4]-b[6];;y6:=b[125];;
y7:=b[122];;y8:=b[111];;y9:=b[139];;
gap> I:=Subalgebra(g,[w1*(w2*y9),w1*(w2*y8),w1*y9,w1*y8,
w2*y8,w2*y9,w1*(w2*(w1*y3)),w1*y1,w1*y2,w1*y3,w2*y1,w2*y2,
w2*y3,w2*(w1*y1),w2*(w1*y2),w2*(w1*y3),w1*y4,w1*y5,w1*y6,
w2*y4,w2*y5,w2*y6,w2*(w1*y4),w2*(w1*y5),w2*(w1*y6),w1*y7,
w2*y7]);;
\end{lstlisting}

This only gives a subalgebra of dimension $38$, but we are looking for an ideal of $L_0$. In particular, it must be $L_0$ invariant and so we skip some steps with the following procedure.

\begin{lstlisting}[basicstyle=\footnotesize]
gap> W:=List([]);;
gap> for j in [1..Dimension(I)] do
if \in(i[j],W)=false then
R:=List([i[j]]);
Append(W,R);continue;fi;od;
gap> for j in [1..Dimension(Nrad)] do
for s in [1..Dimension(I)] do sj:=function(j)
a1:=\in(i[s]*n[j],I);
a2:=n[j]*i[s];if a1=false then
return a2;
else
return 1;
fi;end;
if sj(j)=a2 then
Append(W,[sj(j)]);fi;od;od;
gap> R:=Subalgebra(g,W);;
gap> Dimension(R);
107
\end{lstlisting}
This gives us the ideal $I_8$ that we use during \tref{none8p2}. We can add the remaining elements from $\g_e(\tau,0)$ to obtain a bigger ideal, to begin we add just one or two and then use a similar idea to the above to produce an ideal.

\begin{lstlisting}[basicstyle=\footnotesize,caption={Finding the ideal $J_8$ in \tref{none8p2}}]
gap> Append(W,[b[240],b[120],b[119],b[239]]);;
gap> R:=Subalgebra(g,W);
gap> Dimension(R);r:=Basis(R);;
gap> for j in [1..Dimension(Nrad)] do
for s in [1..Dimension(R)] do sj:=function(j)
a1:=\in(r[s]*n[j],I);
a2:=n[j]*r[s];if a1=false then
return a2;else
return 1;fi;end;
if sj(j)=a2 then
Append(W,[sj(j)]);
fi;od;od;
gap> J:=Subalgebra(g,W);;
\end{lstlisting}
which produces the ideal $J$ of dimension $133$ with basis given in \autorefs{E8J8}. We confirm that $M/J$ is an $8$-dimensional simple Lie algebra, and find that $\dim\, L'_{-1}=245$ and $L'_{-2}=\g$. The basis for $L'_{-1}$ is given in \autorefs{tabe8p2}.\begin{table}[H]\caption{Basis of $J_8$ in \tref{none8p2}}\phantomsection\label{E8J8}\footnotesize\centering
\begin{align*}
J_8=\langle &v.144+v.147, v.109+v.110, v.10, v.118, v.96, v.177, v.63, v.51, v.122, \\&v.138, v.24+v.27, v.9+v.11,v.38+v.40,
  v.37+v.40, v.23, \\&v.150+v.152, v.152+v.153, v.168, v.241+v.244, v.244+v.246, \\&v.132+v.133, v.99+v.100, v.113, v.79,
  v.116, v.88, v.104+v.105, \\&v.103, v.4, v.125, v.139, v.111, v.107, v.31,\\& v.16, v.45, v.145, v.1+v.6, v.129+v.131,
  v.26, v.164, v.52, v.137, \\&v.12+v.13, v.20, v.6, v.123, v.76, v.39,\\& v.53, v.14, v.83, v.47, v.60, v.22, v.66, v.28,
  v.89, v.80, v.59+v.61, \\&v.46+v.49, v.73, v.36, v.94, v.86, v.67+v.68, \\&v.54+v.56, v.97, v.70+v.71, v.101, v.77+v.78,
  v.155, v.184, v.161, \\&v.175, v.127, v.141, v.163, v.192, v.170, v.182,\\& v.135, v.149, v.179+v.181, v.205,
  v.187+v.188, v.210, \\&v.190+v.191, v.213, v.195, v.166+v.169, v.197+v.198, v.218, v.201,\\& v.174+v.176, v.212, v.228,
  v.215, v.194, v.222, \\&v.204, v.224+v.225, v.235, v.232, v.219+v.220, v.229+v.230, \\&v.237, v.240, v.120, v.119, v.239,
  v.8, v.248, v.234, \\&v.42, v.34, v.226, v.65, v.58, v.211, v.87, v.82, v.207, \\&v.202, v.91, v.185, v.178, v.106,
  \\&v.162, v.154, v.114, v.128, v.242+v.245+v.247 \rangle \end{align*}
\end{table}\begin{landscape}

\begin{table}[H]\caption{$L'_{-1}$ in $E_8$ for \tref{none8p2}}\phantomsection\label{tabe8p2}\footnotesize\centering
\begin{align*}
L'_{-1}:=\{& v.1, v.4, v.6, v.8, v.10, v.9+v.11,
  v.12+v.13, v.14, v.16, v.20, v.22, v.23, v.26, v.24+v.27, v.28, v.31, v.34, v.36,
 \\& v.37+v.38, v.39, v.38+v.40, v.42, v.45, v.47, v.46+v.49, v.51, v.52, v.53, v.54+v.56,
 \\& v.58, v.60, v.59+v.61, v.63, v.65, v.66, v.67+v.68, v.70+v.71, v.73, v.76, v.77+v.78,
 \\& v.79, v.80, v.82, v.83, v.86, v.87, v.88, v.89, v.91, v.94, v.96, v.97, v.99+v.100,
\\&  v.101, v.103, v.104+v.105, v.106, v.107, v.109+v.110, v.111, v.113, v.114, v.116,
\\&  v.118, v.119, v.120, v.122, v.123, v.125, v.127, v.128, v.129+v.131, v.132+v.133,
\\&  v.135, v.137, v.138, v.139, v.141, v.145, v.144+v.147, v.149, v.150+v.152,
\\&  v.152+v.153, v.154, v.155, v.161, v.162, v.163, v.164, v.168, v.166+v.169, v.170,
\\& v.175, v.174+v.176, v.177, v.178, v.179+v.181, v.182, v.184, v.185, v.187+v.188,
 \\& v.190+v.191, v.192, v.194, v.195, v.197+v.198, v.201, v.202, v.204, v.205, v.207,
\\&  v.210, v.211, v.212, v.213, v.215, v.218, v.219+v.220, v.222, v.224+v.225, v.226,
\\&  v.228, v.229+v.230, v.232, v.234, v.235, v.237, v.239, v.240, v.241+v.244,
 \\& v.244+v.246, v.242+v.245+v.247, v.248, v.40, v.153, v.157+v.158+v.160,
  v.243+v.246+v.247, v.189, v.30+v.32+v.33, v.69, \\&v.246, v.2, v.3, v.5, v.7, v.55, v.56,
  v.57, v.11, v.61, v.13, v.15, v.62, v.17, v.18, v.19, v.64, v.21, v.68, v.71, \\&v.27,
  v.29, v.74, v.78, v.32+v.33, v.75, v.33, v.35, v.41, v.85, v.43, v.44, v.49, v.48,
\\&  v.50, v.90, v.92, v.93, v.95, v.98, v.100, v.72, v.102, v.105, v.81, v.84, v.108,
  v.110, v.112, v.115, v.117, v.245+v.247, v.247,  \\&v.121+v.124, v.124+v.126, v.131,
 v.130, v.133, v.134, v.140, v.136, v.142, v.147, v.143, v.146, v.148, v.156,
 \\& v.158+v.160, v.160, v.159, v.169, v.167, v.172, v.171, v.173, v.176, v.181, v.180,
\\&  v.186, v.183, v.188, v.191, v.193, v.198, v.196, v.151, v.199, v.200, v.203, v.206,
 \\& v.208, v.209, v.214, v.216, v.220, v.223, v.221, v.225, v.227, v.230, v.231, v.233,
  v.236, v.238, v.217\} \end{align*}\end{table}
\end{landscape}

\subsection[Theorem \ref{none7p2}]{Theorem \ref{none7p2}}\label{AppE7p2}
Since $E_7$ is very different we have left this until now, so let $\g$ be of type $E_7$ and use nilpotent orbit representative $e:=\rt{1}+\rt{4}+\rt{6}$. We use \cite[Table 10]{de08} to look up the basis elements in GAP.
\begin{lstlisting}[basicstyle=\small,caption={The nilpotent orbit $\mathcal{O}({A_1}^3)'$ in $E_7$}]
gap> g:=SimpleLieAlgebra("E",7,GF(2));;
gap> b:=Basis(g);;
gap> e:=b[1]+b[4]+b[6];;
gap> N:=LieNormaliser(g,Subalgebra(g,[e]));;
gap> rad:=LieSolvableRadical(N);; t:=Basis(rad);;
gap> Nrad:=LieNormalizer(g,rad);x1:=t[1];;x2:=t[2];;
x3:=t[3];;x4:=t[4];;
\end{lstlisting}
This produces the $72$-dimensional Lie normaliser $\mathfrak{n}_e$ we are interested in with $4$-dimensional abelian radical $A$. This is generated by \[\langle e, e_{\subalign{11&2110\\&1}},f_{\subalign{01&1100\\&1}},\mathfrak{z}(\g)\rangle.\] Further, we check $N_{\g}(A)$ is $74$-dimensional and $M:=N_{\g}(A)/A$ is a $70$-dimensional semisimple Lie algebra.

We obtain $L_{-1}:=\{x \in \g: [x,A] \subseteq N_{\g}(A)\}$ of dimension $133$ using \autoref{lst:lminus} and \autoref{lst:fullminus}. However, this has a proper submodule of dimension $130$ that strictly contains $L_0$, but is indecomposable. We follow similar techniques and obtain the ideal $I_7$ using elements of $\g_e(\tau,-1)$ and $\g_e(\tau,2)$ that are not contained in $A$.

\begin{lstlisting}[basicstyle=\small,caption={Finding the ideal $I_7$ in \tref{none7p2}}]
gap> nn:=Basis(N);;
gap> w1:=nn[15];;w2:=nn[52];;
gap> y1:=b[14];;y2:=b[27];;y3:=b[80];;
y4:=b[1]-b[4];;y5:=b[4]-b[6];;y6:=b[65];;
y7:=b[68];;y8:=b[57];;y9:=b[114];;
gap> I:=Subalgebra(g,[w1*(w2*y9),w1*(w2*y8),w1*y9,w1*y8,
w2*y8,w2*y9,w1*(w2*(w1*y3)),w1*y1,w1*y2,w1*y3,w2*y1,w2*y2,
w2*y3,w2*(w1*y1),w2*(w1*y2),w2*(w1*y3),w1*y4,w1*y5,w1*y6,
w2*y4,w2*y5,w2*y6,w2*(w1*y4),w2*(w1*y5),w2*(w1*y6),w1*y7,
w2*y7]);;
\end{lstlisting}
This only gives a subalgebra of dimension $38$, but as we are looking for an ideal we can use the same idea as the previous section. This produces the ideal $I_7$ of dimension $59$ with $45$-dimensional solvable radical.

We form a bigger ideal in this case using elements that produce $W(2;\underline{1})$. This gives a $67$-dimensional ideal $J_7$ with basis given in \autorefs{e7J7}. It is an easy check in GAP that $M/J_7$ is a solvable $7$-dimensional Lie algebra. We then produce the required submodule of dimension $130$ with basis given in \autorefs{tabe7p22}.
\begin{table}[H]\caption{$J_7$ in $E_7$ for \tref{none7p2}}\phantomsection\label{e7J7}\footnotesize\centering
\begin{align*}
J_7:=\langle &v.117+v.118, v.54+v.55, v.94, v.63,
  v.34, v.125, \\&v.42, v.50, v.9, v.79, v.65, v.109, v.8+v.10, v.21+v.24, \\&v.84+v.87,
  v.33+v.35,v.32+v.35, v.89+v.91, v.91+v.92, \\&v.127+v.130, v.130+v.132, v.20, v.103,
  v.82, v.70,\\& v.123, v.119, v.111+v.112,v.102+v.104, v.58, \\&v.13, v.39+v.41, v.1, v.108,
  v.99, \\&v.114, v.80, v.44, v.57, v.14, v.27, v.38,\\& v.85,v.68, v.4, v.74+v.75, v.6,
  v.23, v.100, \\&v.71+v.73, v.43,v.78, v.11+v.12, \\&v.61, v.25, v.48+v.49, v.66, \\&v.52,
  v.18, v.92, v.116, \\&v.128+v.129+v.131, v.35, v.95+v.96+v.98, \\&v.132, v.53,
  v.26+v.28+v.29\rangle \end{align*}
 \end{table}
\begin{landscape}
 \begin{table}[H]\caption{$L'_{-1}$ in $E_7$ for \tref{none7p2}}\phantomsection\label{tabe7p22}\footnotesize\centering
\begin{align*}L'_{-1}:=\{&v.117+v.118, v.54+v.55, v.94, v.63,
  v.34, v.125, v.42, v.50, v.9, v.79, v.65, v.109, v.8+v.10, \\&v.21+v.24, v.84+v.87,
  v.33+v.35, v.32+v.35, v.89+v.91, v.91+v.92, v.127+v.130, \\&v.130+v.132, v.20, v.103,
  v.82, v.70, v.123, v.119, v.111+v.112, v.102+v.104, v.58, v.13, v.39+v.41,\\&v.1, v.108,
  v.99, v.114, v.80, v.44, v.57, v.14, v.27, v.38, v.85, v.68, v.4, v.74+v.75, v.6,
 \\& v.23, v.100, v.71+v.73, v.43, v.78, v.11+v.12, v.61, v.25, v.48+v.49, v.66, v.52,
  v.18, v.92, v.2,\\& v.3, v.5, v.7, v.40, v.41, v.10, v.12, v.15, v.16, v.17, v.46, v.47,
  v.19, v.49, v.24, v.26, v.53, v.29,\\& v.30, v.31, v.55, v.35, v.36, v.37, v.59, v.45,
  v.60, v.51, v.62, v.56, v.128, v.129, v.131, \\&v.133, v.64+v.67, v.67+v.69, v.28, v.73,
  v.72, v.75, v.76, v.81, v.77, v.87, v.83, v.86, v.88, v.93, v.95, v.98,v.132, v.97,
  v.96, \\&v.104, v.106, v.105, v.107, v.112, v.110, v.115, v.113, v.118, v.116, v.120,
  v.122, v.121, v.124, v.90, v.126 \}\end{align*}\end{table}\end{landscape}

\subsection[Theorem \ref{e17a4} and Theorem \ref{a14e8}]{Theorem \ref{e17a4} and Theorem \ref{a14e8}}\label{appendix:weirda1e8}

Our final case using many GAP calculations concerns the orbit $\mathcal{O}({A_1}^4)$ in exceptional Lie algebras of type $E_7$ and $E_8$ when $p=2$.

For $E_7$, we consider the subalgebra $\w$ and $L_{-1}:=\{x \in \g: [x,A] \subseteq \w\}$ where $A$ is the $2$-dimensional radical of $\mathfrak{n}_e$. In this case we have $L_{-1}=\g$, but $L_{-1}$ has a $132$-dimensional submodule.

Using \autoref{lst:lminus} we can find this submodule of dimension $132$ with a bit of trial and error. It turns out this does not contain our $f \in \g(\tau,-2)$ such that $[e,f]=\mathfrak{z}(\g)$. In particular, this submodule does not contain $\w$. It follows that $L_{-1}/\w$ is irreducible, and hence $\w$ is a maximal subalgebra of $E_7$.

For $E_8$, we set $L_0:=\mathfrak{n}_e$. The crucial step is to find the irreducible submodule of dimension $247$ since $L_{-1}:=\{x \in \g: [x,A] \subseteq L_0\}$ is equal to $\g$ and not irreducible. To find our $247$-dimensional submodule we use a trial and error method of generating $L_0$-invariant submodules.

We ask GAP to give some basis elements of $\g$ that lie in $L_{-1}:=\{x \in \g: [x,e]\in L_0\}$, and then repeatedly apply $(\ad\,L_0)$ to obtain an $L_0$-invariant subspace. This produces \autorefs{taba14e8}, and adding in any of the elements $f_{\al_2}$, $f_{\al_3}$, $f_{\al_5}$ or $f_{\al_7}$ to the generating set for this submodule will give all of $E_8$.

Adapting all our previous ideas we can verify that this $247$-dimensional module has no submodules strictly bigger than $L_0$, and hence it must be the case that $L_{-1}/L_0$ is an irreducible module. This allows us to prove maximality. To avoid the need for checking whether the module is indecomposable we could ask for a subalgebra generated by $L_0$ and each element above. In all cases this gives a Lie algebra of dimension $248$.

\begin{landscape}
\begin{table}[H]\caption{$247$-dimensional $L_{-1}$ for \tref{a14e8}}\phantomsection\label{taba14e8}\footnotesize\centering
\begin{align*}L_{-1}:=\{&v.2, v.3, v.5, v.7, v.9, v.10+v.11+v.12, v.13+v.14, v.15,
  \\&v.17+v.18, v.18+v.19, v.21, v.23+v.24, v.25, v.26+v.27+v.28, v.29, v.30, v.32, v.33+v.34, v.34+v.35, v.38+v.39,
  \\&v.41, v.42+v.43, v.44, v.46, v.48+v.49, v.50, v.51+v.52+v.53, v.54, v.55, v.57+v.58, v.58+v.59, v.61, v.62, v.64,
  \\&v.65+v.67, v.69+v.70, v.71, v.72, v.74, v.75, v.78+v.79, v.81, v.80+v.82, v.84, v.85, v.86+v.87+v.88, v.90+v.91,
  \\&v.91+v.92, v.93, v.95, v.97, v.98+v.99, v.100, v.102, v.104, v.105+v.106, v.108, v.109, v.112, v.113+v.114, v.115,
  \\&v.117, v.119, v.120, v.121, v.124, v.126, v.128, v.130+v.131, v.131+v.132, v.133+v.134, v.136, v.137+v.138+v.139,
 \\& v.140, v.142, v.143+v.144, v.146+v.147, v.147+v.148, v.151, v.152, v.153+v.154+v.155, v.156, v.157, v.158+v.159,
 \\& v.160, v.162+v.163, v.165, v.167, v.168+v.169, v.171+v.172, v.172+v.173, v.176, v.177+v.178+v.179, v.180, v.181,
  \\&v.183, v.186, v.185+v.187, v.188, v.189+v.190, v.193, v.196, v.197, v.198+v.199, v.200+v.202, v.203, v.206+v.207,
  \\&v.207+v.208, v.209, v.210+v.211+v.212, v.214, v.216, v.217, v.218+v.219, v.221, v.223, v.225+v.226, v.227, v.230,
  \\&v.231, v.233+v.234, v.236, v.238, v.240, v.242, v.243, v.245, v.247, v.244+v.248, v.1, v.4, v.6, v.11+v.12, v.12,
  \\&v.19, v.14, v.16, v.24, v.20, v.27+v.28, v.22, v.28, v.35, v.31, v.37, v.40, v.36, v.43, v.39, v.49, v.45, v.53,
  \\&v.47, v.56, v.52, v.60, v.59, v.63, v.70, v.66, v.67, v.68, v.73, v.76, v.82, v.77, v.79, v.83, v.87+v.88, v.89,
  \\&v.88, v.92, v.94, v.96, v.99, v.101, v.103, v.106, v.107, v.110, v.111, v.114, v.116, v.118, v.241, v.248, v.246,
  \\&v.122+v.123, v.123+v.125, v.125+v.127, v.132, v.129, v.134, v.135, v.138, v.139, v.141, v.148, v.144, v.145, v.149,
  \\&v.150, v.154, v.155, v.159, v.163, v.161, v.169, v.164, v.166, v.170, v.175, v.174, v.178+v.179, v.179, v.182,
  \\&v.184, v.187, v.191, v.190, v.192, v.194, v.195, v.199, v.201, v.204, v.205, v.202, v.211+v.212, v.208, v.212,
  \\&v.213, v.215, v.219, v.220, v.222, v.226, v.224, v.228, v.229, v.232, v.235, v.234, v.237, v.173, v.8, v.239 \}\end{align*}\end{table}\end{landscape}

\section[Maximal subalgebras in $E_8$ for characteristic two]{Maximal subalgebras in $E_8$ for characteristic two in \tref{specialmax}}\label{weirdp2algapp}
We give the necessary GAP input to obtain the result \tref{specialmax}. For this section, we will only consider the $E_8(a_2)$ orbit as the details are almost identical in $E_8(a_4)$, with the only change needed when writing $e_2$ and $f_2$ in terms of the basis in GAP.

\begin{lstlisting}[basicstyle=\footnotesize,label={lst:p2max},caption={Maximal subalgebras in $E_8$ for $p=2$}]
gap> g:=SimpleLieAlgebra("E",8,GF(2));;
gap> b:=Basis(g);;
gap> e1:=b[1]+b[2]+b[3]+b[8]+b[10]+b[12]+b[13]+b[14];;
gap> f1:=b[184]+b[185]+b[187]+b[189]+b[190]+b[191]+b[194];;
gap> L1:=Subalgebra(g,[e1,f1]);
gap> Dimension(L1);
124
\end{lstlisting}

We can now apply \autoref{lst:checkingmaximalsubalgebra}, to obtain $\g$ as an $L_i$-module for $i=1,2$ in terms of a collection of matrices. To finish the proof of \tref{specialmax} we have to verify that the factor module $\g/L_i$ is irreducible. For this, we ask GAP for all the composition factors of $\g$ considered as an $L_{i}$-module. Since the only non-trivial submodule is precisely $L_i$ we may conclude $L_{i}$ is maximal for both $i=1$ and $i=2$.

\begin{lstlisting}[basicstyle=\small, label={maxcheckp2}, caption={Checking the maximality in \tref{specialmax}}]
gap> bh:=Basis(L1);
gap> for i in [1..Dimension(g)] do
s:=function(i);
m1:=[Coefficients(b,bh[1]*b[i]),...,
Coefficients(b,bh[Dimension(L1)]*b[i])];
return m1;end;od;
gap> Mats:=List([1..Dimension(g)],s);;
gap> gm:=GModuleByMats(Mats,GF(2));;
gap> MTX.BasesSubmodules(gm);
\end{lstlisting}

\section[Subalgebras of the Hamiltonian Lie algebra for $p=2$]{Subalgebras of the Hamiltonian Lie algebra for $p=2$}\label{hamp2}

In \pref{newsubpossible6} and \pref{newsubpossible8} we produced simple subalgebras of $H(6;\underline{1})^{(1)}$ and $H(8;\underline{1})^{(1)}$ for $p=2$. For this short section, we give a method of obtaining the simple subalgebra of dimension $118$ in $H(8;\underline{1})^{(1)}$.

The basis given for $H(8;\underline{1})^{(1)}$ in GAP seems complicated on first viewing. However, it is easy to recognise the elements $D_H(f)$ from \eqref{hamultip2}. For example; $D_H(x_i)=\partial_{i'}$ is a basis element consisting of one `v.j' in terms of the GAP basis. To illustrate this further we define the $\partial_i$ for $i=1,\ldots,8$ in $H(8;\underline{1})^{1}$.

\begin{lstlisting}[basicstyle=\footnotesize,caption={The basis elements of the Hamiltonian Lie algebra in GAP}]
gap> g:=SimpleLieAlgebra(``H'',[1,1,1,1,1,1,1,1],GF(2));
gap> b:=Basis(g);
gap> d1:=b[1];
v.1025
gap> d2:=b[2];
v.1281
gap> d3:=b[4];
v.1537
gap> d4:=b[8];
v.1793
gap> d5:=b[16];;d6:=b[32];;
d7:=b[64];;d8:=b1[128];;
\end{lstlisting}

To obtain a simple subalgebra $L$ of this, we try to find candidates for elements `v.j+v.k+v.m' which lie in $L_1$. Or, we could use our workings in \pref{newsubpossible8}, and consider the multiplications $[D_H(x_i),D_H(x_ix_jx_kx_lx_m)]$ for $i,j,k,l,m \in \{1,\ldots,8\}$ in GAP to try identify our $L_3$ component. Here we give the elements we found to work, since we are just confirming there is a simple Lie algebra of the required dimension.

\begin{lstlisting}[basicstyle=\small,caption={Obtaining a simple subalgebra of the Hamiltonian Lie algebra}]
gap> x1:=b[208];
v.193+v.657+v.849
gap> x2:=b[224];
gap> x3:=b[14];
gap> x4:=b[13];
gap> L:=Subalgebra(g,[d1,...,d8,x1,...,x4]);;
gap> Dimension(L);
118
\end{lstlisting}

To verify this is simple we use \autoref{lst:irreducible}. It should be noted that our calculation gives a simple subalgebra of an algebraically closed field of characteristic two. To see this, we just check $L$ is absolutely simple.

From this, we can take a basis of $L$ and identify the top component in the grading we take on $L$ inherited from $H(8;\underline{1})^{(1)}$. This then allows us to check the necessary isomorphism in \tref{a14e8}. A similar approach to all of the above will provide the necessary details for \pref{newsubpossible6}.

\section[The exceptional Lie algebra of type $E_8$ in characteristic five]{The exceptional Lie algebra of type $E_8$ in characteristic five}\label{sec:weirdwesicalculations}

We bring this Appendix to a close by considering the exceptional Lie algebra of type $E_8$ over fields of characteristic five, and provide the details of how we can calculate all the information used in \chref{sec:newchapter} in GAP.

\subsection[The nilpotent orbit $\mathcal{O}(A_4+A_3)$ in $E_8$]{The nilpotent orbit $\mathcal{O}(A_4+A_3)$ in $E_8$}

We start by applying some of the procedures from \aref{nonsemigap} to the nilpotent orbit $\mathcal{O}(A_4+A_3)$ in $E_8$ for $p=5$. This was already considered in \aref{appendix:gapmod} with representative $e=\sum_{\al \in \Pi\setminus \{\al_5\}}$ from \cite[pg. 148]{LT11}. For \tref{weise8}, we find a $\w$-invariant subspace $L_{-1}$ with basis given in \autorefs{l1fore8}.

\begin{lstlisting}[basicstyle=\small]
gap> g:=SimpleLieAlgebra(``E'',8,GF(5));;
gap> b:=Basis(g);;
gap> e:=b[1]+b[2]+b[3]+b[4]+b[6]+b[7]+b[8];
v_1+v_2+v_3+v_4+v_6+v_7+v_8
gap> N:=LieNormalizer(g,Subalgebra(g,[e]));
gap> A:=LieSolvableRadical(N);
gap> w:=LieNormalizer(g,A);
Lie algebra of dimension 74 over GF(5)
\end{lstlisting}

We apply all our steps from \aref{nonsemigap} to obtain $L_{-1}$, and then continue the process of taking $L_{-2}:=[L_{-1},L_{-1}]+L_{-1}$ and so on to obtain the whole Weisfeiler filtration. We find that we obtain up to $L_{-3}$ in our filtration, where $L_{-3}=\g$.

To check these modules are irreducible we ask for all submodules of $\g$ as an $L_0$-module. The only proper submodules are precisely $A, L_0, L_{-1}, L_{-2}$ and $L_{-3}$. For completeness we give the basis of the radical $A$ is given in \autorefs{radicalA}.

\begin{table}[H]\caption{The radical $A$ for \tref{weise8}}\phantomsection\label{radicalA}\footnotesize\centering
\begin{align*}
A:=\langle &v.1+v.2+v.3+v.4+v.6+v.7+v.8, v.9+4v.10+v.11+3v.14+3v.15, \\&v.16+4v.17+4v.22,
v.23, v.44+v.45+3v.46+2v.47+4v.49+v.50, \\&v.51+v.53+3v.54+4v.56, v.58+v.60,
v.65, v.80+3v.81+v.83+4v.84, \\&v.86+4v.88, v.91, v.104+4v.105, v.107, v.116, \\&v.125,
v.132+v.133, v.138+3v.139+4v.140+4v.141, \\&v.144+3v.145+3v.146+4v.147+2v.148+2v.149, v.168,
v.172+v.175, \\&v.177+v.179+4v.181+2v.182, v.202, v.205+4v.207, v.226 \rangle \end{align*}
 \end{table}

 \begin{table}[H]\caption{$L_{-1}$ in $E_8$ for \tref{weise8}}\phantomsection\label{l1fore8}\scriptsize\centering
\begin{align*}
L_{-1}:=\{&v.1+v.2+v.3+v.4+v.6+v.7+v.8, v.9+4v.10+v.11+3v.14+3v.15, v.16+4v.17+4v.22,
  \\&v.23,v.44+v.45+3v.46+2v.47+4v.49+v.50,v.51+v.53+3v.54+4v.56, v.58+v.60,
  v.65, \\&v.80+3v.81+v.83+4v.84, v.86+4v.88,v.91,\\& v.104+4v.105, v.107, v.116, v.125,
  v.132+v.133, v.138+3v.139+4v.140+4v.141, \\&v.144+3v.145+3v.146+4v.147+2v.148+2v.149,
v.168, v.172+v.175, \\&v.177+v.179+4v.181+2v.182, v.202,
  v.205+4v.207, v.226, \\&v.60, v.133, v.17+4v.22, v.175, v.139+4v.140+2v.141, \\&v.207,
  v.179+v.181+v.182, v.228, v.229+4v.230, \\&v.209+3v.210+3v.212, v.237,
  v.213+4v.214+v.215+4v.216, v.53+2v.56, v.88, v.22, \\&v.10+2v.11+3v.14+v.15,
  v.140+2v.141, v.145+3v.146+2v.147+3v.148+v.149, \\&v.181+4v.182,
  v.183+2v.184+2v.186+3v.187+4v.188, \\&v.210+3v.212,
  v.189+4v.190+3v.191+3v.192+3v.193+3v.194, v.81+4v.84, \\&v.105,
  v.54+v.56, v.45+2v.46+3v.47+4v.49+2v.50, v.14+v.15, \\&v.2+2v.3+4v.4+4v.6+v.8,
v.146+3v.147+3v.148+2v.149, \\&v.151+4v.152+4v.153+2v.154+4v.155+3v.156, v.184+2v.187+4v.188,
\\&v.157+4v.158+4v.159+v.160+3v.161+3v.162+2v.163,v.101+2v.102, v.115, \\&v.83+3v.84,
  v.75+2v.77+4v.78, v.46+2v.47+v.49+4v.50, \\&v.37+3v.38+v.41+2v.43,
  v.6+v.7+v.8, \\&v.241+2v.242+2v.244+3v.246+4v.247+3v.248,
  v.150+3v.153+4v.155+2v.156,v.113+v.114, \\&v.120,v.102+3v.103,
  v.97+4v.98+v.99+4v.100, v.76+3v.77+2v.78+3v.79,
  \\&v.69+4v.70+v.71+v.72+2v.73+2v.74,v.38+2v.39+v.40+4v.41+2v.42+2v.43,
  \\&v.30+2v.31+v.32+4v.33+2v.34+2v.35+2v.36, v.242+v.243+3v.244+v.246+3v.247+v.248,
\\&v.11+v.15,v.15, v.103, v.182, v.141, v.56, v.47+4v.49+v.50, v.84,
  \\&v.147+4v.148+4v.149, v.148+3v.149, v.77+4v.78, v.212, v.49+v.50, v.114,
  \\&v.3+v.4+v.7+3v.8, v.78+3v.79, v.186+4v.188, v.187+4v.188, v.39+3v.40+v.42,
  \\&v.230, v.7+2v.8, v.98+v.99+4v.100, v.119, v.152+2v.154+4v.155+v.156,
  \\&v.40+v.41+3v.42+v.43, v.99+2v.100, v.214+2v.216, v.215+2v.216,
  \\&v.243+4v.244+3v.247+2v.248, v.238, \\&v.153+v.155+2v.156, v.70+4v.71+v.72+v.73+v.74,
\\&v.111+2v.112, v.190+2v.191+3v.192+4v.193+4v.194, \\&v.246+3v.247+v.248,
v.71+3v.72+2v.73+4v.74, v.231+2v.232, v.159+4v.160+2v.162,
  \\&v.191+3v.192+v.193+2v.194, v.31+4v.32+v.33+4v.34+v.35+4v.36,
  \\&v.93+3v.94+3v.95+v.96, v.217+3v.218+3v.219+v.220,
  \\&v.158+v.160+4v.161+2v.162+4v.163, v.32+v.33+3v.34+v.35+4v.36,
  \\&v.121+2v.122+2v.124+3v.126+4v.127+3v.128, v.63+3v.64+2v.66+2v.67+2v.68, \\&v.195+4v.196+4v.197+4v.198+2v.199,
  v.122+v.123+3v.124+v.126+3v.127+v.128, \\&v.24+2v.25+2v.26+2v.27+3v.28+4v.29, v.164+2v.165+2v.166+v.167+3v.169+2v.170\} \end{align*}
 \end{table}

\subsection[The maximality of $L_{[p]}$ in $E_8$ from Theorem \ref{nonrestrictedwitt}]{The maximality of $L_{[p]}$ in $E_8$ from Theorem \ref{nonrestrictedwitt}}\label{w12e8}
Consider the orbit $\mathcal{O}(E_8(a_2))$ with representative $e:=e_{\al_1}+e_{\al_2}+e_{\al_2+\al_4}+e_{\al_3+\al_4}+e_{\al_5}+e_{\al_6}+e_{\al_7}+e_{\al_8}$ in \autoreft{sec:2nwitt}. To obtain the Witt algebra $W(1;\underline{2})$ we use $f_{\widetilde{\al}}$ and $e$ to generate a $25$-dimensional non-restricted simple Lie algebra along with its normalizer.

\begin{lstlisting}[basicstyle=\small, label={lst:wittp5e8},caption={Finding a subalgebra of type $W(1;2)$ in $E_8$}]
gap> g:=SimpleLieAlgebra("E",8,GF(5));;
gap> b:=Basis(g);;
gap> e:=b[1]+b[2]+b[5]+b[6]+b[7]+b[8]+b[10]+b[11];;
gap> f:=b[240];;
gap> h:=Subalgebra(g,[e,f]);
gap> Dimension(h);IsRestrictedLieAlgebra(h);
25
false
gap> LieNormalizer(g,h);
<Lie algebra of dimension 26 over GF(5)>
gap> ep:=PthPowerImage(e);;
gap> \in(ep,LieNormalizer(g,h));
true
\end{lstlisting}

Checking the simplicity of our Lie algebra is a straightforward application of \autoref{lst:irreducible}. To calculate the size of the second Jordan block in \tref{nonrestrictedwitt} we find $u \in \g$ such that $[e,u]=e^{[p]}$. For this, use GAP to produce an element $u:=\sum_{i=1}^{248} x[i] \cdot b[i]$ in terms of the GAP basis given in \cite[Table 10]{de08}. We then insist that $[e,u]=e^{[p]}$, and show $u$ is precisely as in \eqref{keyujordan}.

\begin{lstlisting}[basicstyle=\small, caption={Defining $E_8$ over a polynomial ring}, label={lst:easy}]
gap> gr:=SimpleLieAlgebra("E",8,Rationals);;
gap> r:=PolynomialRing(Rationals,Dimension(gr));;
gap> g:=SimpleLieAlgebra("E",8,r);;
gap> b:=Basis(g);;
gap> x:=IndeterminatesOfPolynomialRing(r);;
gap> e:=b[1]+b[2]+b[5]+b[6]+b[7]+b[8]+b[10]+b[11];;
gap> ep2:=2*b[38]+b[39]+b[41]+b[42]
+b[43]+3*b[44]+3*b[45]+3*b[48]+2*b[49];;
\end{lstlisting}

This procedure also gives the nilpotent element $e$ along with its $p$-th power. We can then take a sum of all basis elements and consider $[e,u]-e^{[p]}=0$. This gives many restrictions on $u$ and leaves a linear algebra problem to solve which ultimately gives \eqref{keyujordan}.
\begin{lstlisting}[basicstyle=\footnotesize,caption={Obtaining $u$ for \eqref{keyujordan}}]
Take generic u;
gap> e*u-ep2;
Implement relations such as x[121]=0, x[1]=x[9] and so on;
gap> u;
(x_1)v.1+(x_1)v.2+(x_1)v.5+(x_1)v.6+(x_1)v.7+(x_1)v.8+(x_1)v.10
+(x_1)v.11+v.29+3v.30+3v.31+3v.33-v.34+2v.35-3v.36+9v.37+3v.40
+(x_54)v.54+(x_54)v.57-(x_54)v.58+(x_54)v.59-(2x_54)v.60
+(x_54)v.61+(-x_54)v.62+(x_54)v.63+(x_75)v.71+(-x_75)v.72
-(2x_75)v.73+(-2x_75)v.74+(x_75)v.75-(2x_75)v.77-(2x_87)v.84
-(x_87)v.86+(x_87)v.87+(-x_87)v.88+(-x_87)v.89+(x_95)v.95
-(x_95)v.96-(x_95)v.97+(x_95)v.98+(x_95)v.99+(x_100)v.100
-(x_100)v.101+(x_100)v.102+(x_111)v.111+(x_111)v.112
+(x_115)v.115+(x_120)v.120
gap> e*u-ep2;
(-5)*v.38
\end{lstlisting}
Since we are working in characteristic five, this is zero as required. It should be noted that we have $u$ from \eqref{keyujordan} in addition with some elements. However, observe that each of these is simply an element of $\g_e$ using \cite{LT11, de08}. For example; $b[120]=e_{\widetilde{\al}}=v_{10}$ in \pref{newcent}. Hence, we ignore these elements as it will not change the size of $\langle u,e,f\rangle$.

\begin{lstlisting}[basicstyle=\footnotesize,label={lst:pthpower}]
gap> g:=SimpleLieAlgebra("E",8,GF(5));;
gap> b:=Basis(g);;
gap> e:=b[1]+b[2]+b[5]+b[6]+b[7]+b[8]+b[10]+b[11];;
gap> u:=b[29]+(3)*b[30]+(3)*b[31]+(3)*b[33]+(-1)*b[34]+
(2)*b[35]+(-3)*b[36]+(9)*b[37]+(3)*b[40];
gap> e*u=PthPowerImage(e);
true
\end{lstlisting}

\subsection[The uniqueness of $L_{[p]}$ in $E_8$ from Theorem \ref{nonrestrictedwitt}]{The uniqueness of $L_{[p]}$ in $E_8$ from Theorem \ref{nonrestrictedwitt}}\label{uniquew12}

For the uniqueness of $W(1;2)$ in $E_8$ we use \autoref{lst:easy} to define $E_8$ over a polynomial ring with ``enough'' coefficients, and define \[\text{a n} \,\,\text{for}\,\, n=4 \,\,\text{or}\,\, 1 \mod 5\] in GAP as the sum of all basis elements $x \in \g$ such that $x \in \g(\tau,4)$. Note that we use $1 \mod 5$ to resemble the elements of $\g_e(\tau,-1 \mod 5)$. Then, consider the sum of all these elements in `fnew' using

\begin{lstlisting}[basicstyle=\footnotesize,label={lst:unique},caption={Uniqueness of $W(1;2)$ in $E_8$ for $p=5$}]
gap> f:=b[240];
gap> e:=b[1]+b[2]+b[5]+b[6]+b[7]+b[8]+b[10]+b[11];;
gap> a4:=x[9]*b[9]+x[13]*b[13]+x[14]*b[14]+x[15]*b[15]
+x[16]*b[16]+x[17]*b[17]+x[18]*b[18]+x[19]*b[19];;
gap> a14:=x[54]*b[54]+x[57]*b[57]+x[63]*b[63]+
x[62]*b[62]+x[61]*b[61]+x[60]*b[60]+x[59]*b[59]+x[58]*b[58];;
gap> a24:=x[90]*b[90]+x[91]*b[91]+x[92]*b[92]+
x[93]*b[93]+x[94]*b[94];;
gap> a34:=x[109]*b[109]+x[111]*b[111]+x[112]*b[112];
gap> a44:=x[119]*b[119];
gap> a36:=x[232]*b[232]+x[233]*b[233];;
gap> a26:=x[215]*b[215]+x[216]*b[216]+x[217]*b[217]
+x[218]*b[218]+x[219]*b[219];;
gap> a16:=x[184]*b[184]+x[187]*b[187]+x[190]*b[190]+
x[185]*b[185]+x[186]*b[186]+x[188]*b[188]+x[189]*b[189];;
gap> a4:=x[20]*b[20]+x[9]*b[9]+x[13]*b[13]+x[14]*b[14]+
x[15]*b[15]+x[16]*b[16]+x[17]*b[17]+x[18]*b[18]+x[19]*b[19];;
gap> a6:=x[141]*b[141]+x[142]*b[142]+x[143]*b[143]+x[144]*b[144]
+x[145]*b[145]+x[146]*b[146]+x[147]*b[147]+x[148]*b[148]
+x[152]*b[152];;
gap> fnew:=b[240]+a4+a6+a14+a16+a24+a26+a34+a36+a44;;
\end{lstlisting}

Note that `fnew' is our candidate for a non-conjugate to $f$. We may begin by considering elements of the centraliser of $e$, and check that for $f+v$ the Lie algebra generated by $e$ and $f+v$ is equal to the Lie algebra generated by $e$ and $f$ for $v \in \g_e(\tau,4)$. Since this is satisfied we substitute in these relations in the coefficients. For example; we set $x[232]:=4\,x[233]$.

We redefine `fnew' with these new relations, and consider $[\text{fnew},(\ad\,e)^22(\text{fnew})]$. In $W(1;2)$ this must be zero, so we are left with a collection of linear equations all needed to be zero. After repeated substitutions of necessary relations we are led to the conclusion the maximal subalgebra from \tref{nonrestrictedwitt} is unique up to conjugation.

\begin{lstlisting}[basicstyle=\footnotesize, caption={Final checks for uniqueness of $W(1;2)$ in $E_8$}]
gap> x1:=(ad e)^22(fnew);;
gap> x1*fnew;
\end{lstlisting}

\subsection[The first Witt algebra in $E_8$]{The first Witt algebra in $E_8$}\label{firstwittgap}

In this final subsection, we consider the GAP calculations performed in \autoref{sec:witty} used to obtain some partial results about subalgebras of type $W(1;\underline{1})$ in the exceptional Lie algebra of type $E_8$. To obtain \pref{nilptbc} we initially search for nilpotent orbits $\mathcal{O}$ such that for an orbit representative $e$, we have $(\ad\,e)^4(f)=\la e$. For a positive example, we consider the nilpotent orbit $\mathcal{O}(A_3)$ with standard representative $e:=\rt{1}+\rt{3}+\rt{4}$ and apply $\ad\,e$ four times.

\begin{lstlisting}[basicstyle=\footnotesize,caption={The nilpotent orbit $\mathcal{O}(A_3)$ in $E_8$, and the corresponding Witt algebra}]
gap> g:=SimpleLieAlgebra("E",8,GF(5));;
gap> b:=Basis(g);;
gap> e:=b[1]+b[3]+b[4];;
gap> e*(e*(e*(e*(b))));
gap> f:=b[136];;
gap> h:=Subalgebra(g,[e,f]);
gap> Dimension(h);
5
\end{lstlisting}

Observe that for $f_{\subalign{11&10000\\&0}}=b[136]$, we have $(\ad\,e)^4(f_{\subalign{11&10000\\&0}})=e$. Hence, this orbit must be considered as a possibility for a $W(1;\underline{1})$ subalgebra. In the negative examples we repeat exactly as above, changing $e$ and will see none of these have the property $e \in \im((\ad\,e)^4)$.

\subsection[Verifying our choice for $X\partial$ is unique]{Verifying our choice for $X\partial$ is unique}\label{xdunique}

For \pref{nomaxwitt}, to prove that the subalgebras of type $W(1;\underline{1})$ are not maximal in types $A_3$ and $A_4$ we first need to show that our representative $X\partial$ is unique ``enough''. In both cases we take an element $v \in \g(\tau,-2)$ and consider $[e,v]$. Forcing this to be zero gives that $v=0$ or $v \in \mathfrak{z}(\mathfrak{l})$ where $\mathfrak{l}$ has type $A_4$. We could build a full list of $\g_e(\tau,0) \cap \im(\ad\,e)$ for all nilpotent orbits.

Take the orbit $\mathcal{O}(A_4)$, with Lie algebra representative $3h_{\al_1}+3h_{\al_2}+2h_{\al_3}+2h_{\al_4}$ for cocharacter $\tau$. This will also be taken to be our choice for $X\partial$. Then, compute the zero weight space for $X\partial$. For this we write some code to do this automatically:

\begin{lstlisting}[basicstyle=\footnotesize,caption={Confirming $\g_e(\tau,0)\cap \im\,\ad\,e=0$}]
gap> e:=b[1]+b[2]+b[3]+b[4];;
gap> h:=3*b[241]+3*b[242]+2*b[243]+2*b[244];
gap> h*e=-e;
true
gap> for i in [1..Dimension(C)] do
s:=function(i)
a1:=c[i]*h=0*b[1];
if a1=true then
return i;
else
return 0;
fi;
end;
if s(i)=i then
Append(W,[c[i]]);
fi;od;
gap> g0:=Subalgebra(g,W);;
gap> Dimension(g0);
24
gap> C1:=LieCentre(g0);;c1:=Basis(C1);;
\end{lstlisting}

We are able to verify that this Lie algebra $\g(\tau,0)$ is $24$-dimensional such that $\g(\tau,0)/\mathfrak{z}(\g(\tau,0))$ is $23$-dimensional and simple. It is easy to check this must have type $A_4$, and so we have obtained $\g_e(\tau,0)$. For the intersection with $\im (\ad\,e)$ we use GAP to produce a vector space generated by such an image.

\begin{lstlisting}[basicstyle=\footnotesize]
gap> V:=VectorSpace(GF(5),e*b);;
gap> Intersection(V,g(0));
\end{lstlisting}

We could use everything described in this section to obtain a complete list of $\g_e(\tau,0) \cap \im(\ad\,e)$ for all standard nilpotent orbits in any characteristic. It seems plausible that nothing too dramatic happens, and very similar results to \cite{Jan04} should be expected. For example, in the non-smooth orbit $\mathcal{O}(A_4+A_3)$ we have $\g_e(\tau,0) \cap \im(\ad\,e)=\g_e(\tau,0)$.

In all our cases, despite $\g_e$ having a bigger dimension than $\Lie(G_e)$ the majority of new elements are in $\g_e(\tau,-1)$. Hence, we will usually find that that $\g_e(\tau,0) \cap \im(\ad\,e)$ is exactly the same as it was for nilpotent $e$ from \cite{Jan04} when considered in the good characteristic case.
 \subsection{Finding non-zero fixed vectors for Proposition \ref{nomaxwitt}}\label{fixedvectorapp}
The final step to rule out maximal subalgebras isomorphic to $W(1;1)$ we show that there is a non-zero fixed vector $v$ for $W(1;1)$, that is a vector such that $[W(1;1),v]=0$.

Recall that we have unique choices for $\partial=e$ and $X\partial=h_{\tau}$ up to a centre in the Levi factor for nilpotent orbit with representative $e$. Using \cite[Appendix]{HS14} we obtain a method for showing any possible $W(1;1)$-subalgebra has a non-zero fixed vector.

By \autoref{lst:easy} we may obtain $E_8$ as a Lie algebra over a polynomial ring, and define our $e$ and $X\partial$. Note that we only consider the nilpotent orbit of type $A_3$. For the orbit of type $A_4$ the reader needs to replace $e$ appropriately, and then consider all the possible choices for $X\partial$.

\begin{lstlisting}[basicstyle=\footnotesize]
gap> gr:=SimpleLieAlgebra("E",8,Rationals);;
gap> r:=PolynomialRing(Rationals,Dimension(gr));;
gap> g:=SimpleLieAlgebra("E",8,r);;
gap> b:=Basis(g);;
gap> x:=IndeterminatesOfPolynomialRing(r);;
gap> e:=b[1]+b[3]+b[4];;
gap> Xd:=b[241]+b[244]+3*b[243];;
\end{lstlisting}

Since $W(1;1)$ is generated as a Lie algebra by the elements $\partial$ and $X^3\partial$ we need a candidate for $X^3\partial$. To begin we consider a generic element $u:=\sum_{i \in I}\text{x[i]}\text{b[i]} \in \g$, such that $[e,u]=X\partial$ and $[X\partial,u]=u$. For this, we use

\begin{lstlisting}[basicstyle=\footnotesize]
gap> u:=x[1]*b[1]+x[2]*b[2]+...+x[248]*b[248]
gap> Xd*u-u;
gap> e*u-Xd;
\end{lstlisting}

Since $[X\partial,u]-u=0$ and $[e,u]-X\partial=0$ we reduce the number of coefficients for $u$. After some deductions we obtain

\begin{lstlisting}[basicstyle=\footnotesize]
gap> u:=
x[23]*b[23]+x[24]*b[24]
+x[31]*b[31]+x[39]*b[39]+
x[47]*b[47]+
x[63]*b[63]+x[70]*b[70]+
x[76]*b[76]+x[77]*b[77]+
x[83]*b[83]+x[88]*b[88]+
x[97]*b[97]+
x[101]*b[101]+x[104]*b[104]+x[107]*b[107]+
x[116]*b[116]+
b[121]+x[122]*b[122]+3*b[123]+b[124]+x[125]*b[125]+
x[133]*b[133]+
x[141]*b[141]+x[149]*b[149]+x[150]*b[150]+
-x[150]*b[150]+x[168]*b[168]+x[175]*b[175]+
x[181]*b[181]+x[182]*b[182]+x[188]*b[188]+
x[194]*b[194]+
x[205]*b[205]+x[210]*b[210]+
x[215]*b[215]+x[220]*b[220]+
x[232]*b[232];;
gap> Xd*u-u;
(-15)*v.123
gap> e*u-Xd;
0*v.1
\end{lstlisting}

Since we work in characteristic five, we see both are now zero. Hence, this choice of $u$ is sufficiently generic and satisfies the conditions to be $X^2\partial$. We then play the same game to find a candidate for $X^3\partial$, starting with a generic $v$ insisting that $[e,v]=u$ and $[X\partial,v]=2v$. After some work, we obtain the following choice in GAP.
\begin{lstlisting}[basicstyle=\footnotesize]
gap> v:=
x[16]*b[16]-x[23]*b[17]-x[24]*b[19]+
-x[31]*b[27]-x[39]*b[35]+
-x[47]*b[43]+x[63]*b[57]+
x[70]*b[64]+x[76]*b[71]+x[77]*b[72]+
x[83]*b[78]+x[88]*b[84]+
x[97]*b[93]+x[101]*b[98]+
x[104]*b[102]+x[107]*b[105]+
-x[116]*b[115]+-b[129]-x[122]*b[130]+
b[131]+x[125]*b[132]+x[133]*b[140]+
x[141]*b[148]+x[149]*b[156]+
x[168]*b[172]+x[175]*b[179]+
x[181]*b[186]+x[182]*b[187]+
x[188]*b[193]+x[194]*b[199]+
-x[205]*b[209]-x[210]*b[214]-x[215]*b[219]+
-x[220]*b[223]+-x[232]*b[233];;
\end{lstlisting}

To find a fixed vector $w$, we take a generic $w \in \g$ insisting that $[e,w]=[Xd,w]=0$. This forces $w \in \g_e(\tau,0)$ for $\tau$ the associated cocharacter from \pref{nilptbc} for $e$ of type $A_3$.
\begin{lstlisting}[basicstyle=\footnotesize,caption={Finding a fixed vector for $W(1;1)$ subalgebras in $\mathcal{O}(A_3)$ orbit}]
gap> w:=x[6]*b[6]+x[7]*b[7]+x[8]*b[8]+
x[14]*b[14]+x[15]*b[15]+
x[22]*b[22]+x[32]*b[30]+
x[32]*b[32]+x[38]*b[38]+x[38]*b[40]+
x[49]*b[46]+x[49]*b[49]+
x[54]*b[54]+x[54]*b[56]+x[69]*b[69]+
x[75]*b[75]+x[80]*b[80]+
x[81]*b[81]+x[82]*b[82]+x[86]*b[86]+x[87]*b[87]+
x[91]*b[91]+x[92]*b[92]+x[96]*b[96]+
x[109]*b[109]-x[109]*b[110]+
x[117]*b[117]+x[118]*b[118]+x[119]*b[119]+x[120]*b[120]+
x[126]*b[126]+x[127]*b[127]+x[128]*b[128]+
x[134]*b[134]+x[135]*b[135]+
x[142]*b[142]+x[152]*b[150]+
x[152]*b[152]+x[158]*b[158]+x[158]*b[160]+
x[169]*b[166]+x[169]*b[169]+
x[176]*b[174]+x[176]*b[176]+x[189]*b[189]+
x[195]*b[195]+x[200]*b[200]+
x[201]*b[201]+x[202]*b[202]+x[206]*b[206]+x[207]*b[207]+
x[211]*b[211]+x[212]*b[212]+x[216]*b[216]+
x[229]*b[229]-x[229]*b[230]+
x[237]*b[237]+x[238]*b[238]+x[239]*b[239]+x[240]*b[240]+
x[241]*b[241]+x[242]*b[242]+2*x[241]*b[243]+3*x[241]*b[244]+
(4*x[241]-x[242])*b[245]+x[246]*b[246]+x[247]*b[247]+x[248]*b[248];;
gap> e*w=0;
true
gap> Xd*w=0;
true
gap> v*w;
\end{lstlisting}

Using the Lie bracket $[v,w]=0$ gives a collection of linear equations that must always have a non-trivial solution. This follows because the number of linear equations is always strictly less than the number of indeterminates for $w$. Hence, we achieve the conclusion of \pref{nomaxwitt}. The same idea was used for the nilpotent orbit of type $A_4$.

\clearpage
\addcontentsline{toc}{chapter}{Index}
\printindex
\bibliographystyle{alpha}
\bibliography{g2}

% Comment the following THREE lines if you do NOT have an Appendix
%\appendix
%\chapter{A Long Proof}
%.........
% If you need more than one appendix, then just use another \chapter command
%\chapter{Yet Another Appendix}

\end{document}